\documentclass[11pt]{report}
\usepackage{amssymb,amsfonts,amsmath,amsthm}
\usepackage[latin1]{inputenc}
\usepackage[all]{xy}
\usepackage[pdftex]{graphicx}
\usepackage{indentfirst}
\usepackage{verbatim}
\usepackage{array}

\newtheorem{theo}{Theorem}[section]
\newtheorem{theorem}[theo]{Theorem}
\newtheorem{prop}[theo]{Proposition}
\newtheorem{lemma}[theo]{Lemma}
\newtheorem{coro}[theo]{Corollary}
\newtheorem{conj}[theo]{Conjecture}

\theoremstyle{definition}
\newtheorem{defin}[theo]{Definition}
\newtheorem{example}[theo]{Example}

\theoremstyle{remark}
\newtheorem{remark}[theo]{Remark}

\newcommand{\C}{\mathbb{C}}

\newcommand{\F}{\mathbb{F}}
\newcommand{\A}{\mathbb{A}}
\renewcommand{\P}{\mathbb{P}}

\newcommand{\mS}{\mathcal{S}}
\newcommand{\X}{\mathcal{X}}
\newcommand{\Xp}{\mathcal{X}^\prime}
\newcommand{\Xpp}{\mathcal{X}^{\prime\prime}}
\newcommand{\wX}{\widehat{\mathcal{X}}}
\renewcommand{\l}{\ell}
\renewcommand{\L}{\mathcal{L}}
\newcommand{\wL}{\widehat{\mathcal{L}}}
\newcommand{\Y}{\mathcal{Y}}
\newcommand{\mC}{\mathcal{C}}
\newcommand{\mQ}{\mathcal{Q}}
\newcommand{\mO}{\mathcal{O}}
\DeclareMathOperator{\tor}{\dasharrow}
\DeclareMathOperator{\Sec}{Sec}

\DeclareMathOperator{\Sing}{Sing}
\DeclareMathOperator{\Pic}{Pic}
\DeclareMathOperator{\mult}{mult}
\DeclareMathOperator{\Bl}{Bl}
\DeclareMathOperator{\II}{II}

\begin{document}

\begin{titlepage}
\begin{center}

\includegraphics[width=0.13\textwidth]{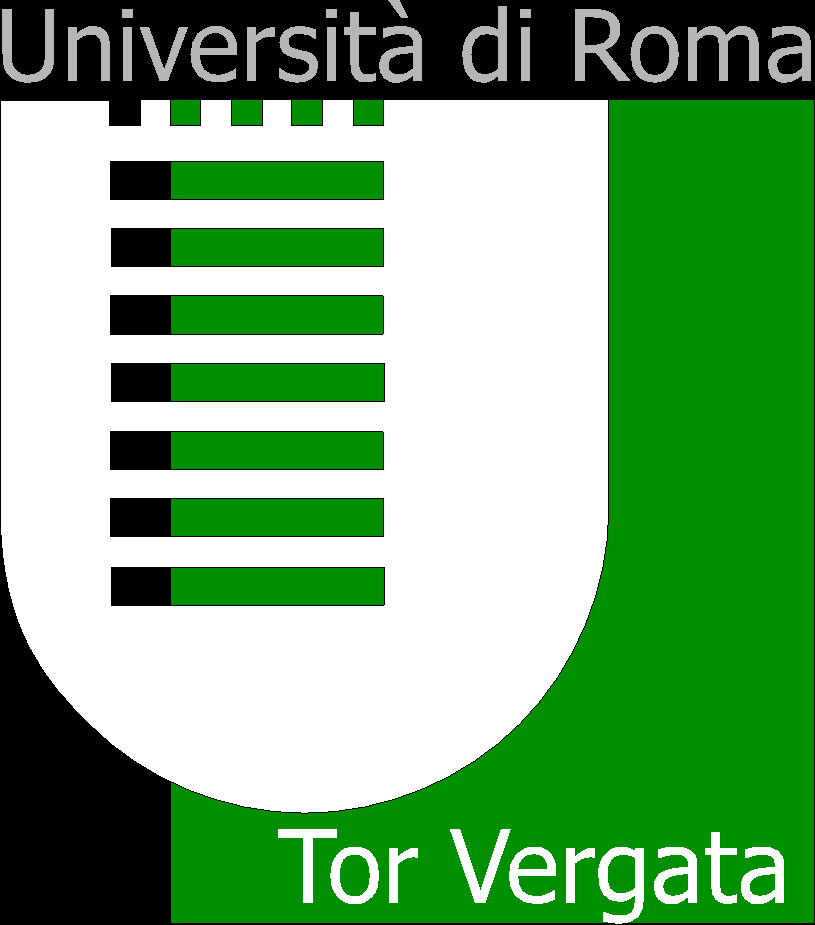}

\LARGE \textbf
{ \mbox{UNIVERSITÀ DEGLI STUDI DI ROMA} \\
``TOR VERGATA"}

 \vskip 8mm

\Large
FACOLTÀ DI MATEMATICA

 \vskip 8mm

DOTTORATO DI RICERCA IN MATEMATICA

 \vskip 8mm

CICLO XXV 

 \vskip 28mm

\LARGE
Threefolds with one apparent double point

 \vskip 28mm

\Large
Vitalino Cesca Filho

 \vskip 5mm

A.A. 2013/2014

 \vskip 6mm

\end{center} 
\begin{flushleft}

Docente Guida/Tutor: Prof. Ciro Ciliberto

 \vskip 3mm

Coordinatore: Prof. Carlo Sinestrari

\end{flushleft}
\end{titlepage}

\newpage
\tableofcontents

\chapter*{Introduction}
\addcontentsline{toc}{chapter}{Introduction}

Given an irreducible variety $X$ of dimension $n$ in $\P^{2n+1}$, the number of \emph{apparent double points} of $X$ is the number of secant lines to $X$ passing through a general point of $\P^{2n+1}$. If this number is 1, then $X$ is called a \emph{variety with one apparent double point}, or OADP variety. Roughly, $X$ is said to be \emph{secant defective} if through a general point of $\P^{2n+1}$ there is no secant line to $X$. Therefore OADP varieties are considered the simplest non defective varieties. However, their classification is very rich and challenging. 

In dimension one, the only OADP variety is the rational normal cubic (see for instance \cite[Proposition 2.2]{cr}). The classification in dimension two started with Severi. In \cite{severi},  he stated that the OADP surfaces having at most a finite number of singularities were degree four rational normal scrolls and weak Del Pezzo surfaces. A gap in his arguments was fixed in \cite{r_severi}, which considers only the smooth case. More recently, Ciliberto and Russo  classified all OADP surfaces, including the Verra Varieties to Severi's list, see \cite{cr}. 

The classification of smooth OADP threefolds was done by Ciliberto, Mella and Russo, in \cite{cmr}. In the same work, the authors have also shown an important property of a smooth OADP variety $X$: if $x$ is a general point of $X$, then the projection $\tau:X\tor \P^n$ from the projective tangent space $T_xX$ is birational. This also holds for singular OADP varieties, as remarked in \cite{cr}. The inverse of $\tau$ contracts to the point $x$ a hypersurface $V\subset \P^n$, which  the authors called \emph{fundamental hypersurface}. They studied the case in which $V$ is a hyperplane, giving a partial classification of OADP varieties for arbitrary dimension $n$.

The aim of the present work is to study singular OADP threefolds such that the fundamental surface $V$ is not a plane. This will be done through the analysis of linear systems $\X$ in $\P^3$ that define the inverse of a tangential projection $\tau$ of $X$.

This approach classifies the so-called \emph{Bronowski varieties}, i.e., varieties such that the general tangential projection $\tau$ is birational. As explained above, OADP varieties are Bronowski. The converse is known as the \emph{Bronowski conjecture} and it has been first claimed in \cite{bronowski}. All examples of Bronowski varieties that appear in our study are also OADP.

Some remarks made in \cite{cr} allow us to simplify the study of Bronowski threefolds. The first remark is that $V$ is a projection of the Veronese quartic surface to $\P^3$. The case in which $V$ is a quadric is studied in Chapter \ref{ed_chapter}. In Chapter \ref{sl_chapter}, the cubic case is considered. Finally, in Chapter \ref{sc_chapter}, the case in which $V$ is a quartic surface is analysed. 

The second remark is that the linear system $\X$ which defines the map $\tau^{-1}:\P^3\tor X\subset \P^7$ has degree $d\geq 2\delta+1$, where $\delta=\deg V$. Here, we focus in the cases in which equality holds, that is, we assume the following hypothesis:
\begin{itemize}
\item[(H)] Let $X$ be a normal Bronowski threefold such that the linear system $\X$ defining the inverse of the tangential projection at a general point $x\in X$ has degree $d=2\delta+1$, where $\delta$ is the degree of the fundamental surface.
\end{itemize}

The hypothesis on the degree of $\X$ may seem unnatural at a first glance, but it will be explained later that, for instance, it holds in the smooth cases (see page \pageref{pr_hipoteseH}).

Other properties of $\X$ explained in \cite{cr} are used in the classification (see Lemma \ref{pr_X'eX''} below).

By understanding the different linear systems $\X$ that can define a map $\tau^{-1}$, we can study the varieties $X$ given by the images of this map. The obvious way to do that is to blow up the base locus of $\X$. This is simplified by the fact that, in most of the cases, the curves in the base locus of $\X$ are rational curves. Another option is to look for a family of surfaces in $X$ that span projective spaces of low dimension. Then we can use the fact that $X$ is \emph{linearly normal} (see page \pageref{pr_lin_normal}) together with  a result on linearly normal varieties from \cite{eh_minimal} to prove that it lies on a rational normal scroll.

The singularities of $X$ can be understood with the help of Lemma \ref{pr_truqueconetangente} and Lemma \ref{pr_truque_secao_tgente}. In addition, the fact that $\tau$ is a projection implies that most of these singular points are mapped to points in $\P^3$ (more precise statements are given in Lemma \ref{ed_singularidades_vem_de_C4}, Lemma \ref{sl_singularidades_vem_de_C6} and Lemma \ref{sc_singularidades_vem_de_C6}).

Most of the threefolds considered in this thesis are degenerations of known OADP threefolds. The ones in Chapter \ref{ed_chapter} are degenerations of the Edge variety of degree seven. In 6 of the 19 degenerations, the fundamental surface is a cone. The varieties in Chapter \ref{sl_chapter} are degenerations of the degree eight smooth scroll in lines described in Example \ref{pr_exemplo_sl}. In all of them, the fundamental surface is the general cubic scroll with a double line.

All the threefolds that appear in Chapter \ref{sc_chapter} are singular, so they do not appear in the classification in \cite{cmr}. They have degree nine and all of them have a point of multiplicity four. In the general threefold, the tangent cone at this singular point is a cone over the quartic Veronese surface. The fundamental surface can be any of the three quartic Steiner surfaces, each producing different OADP varieties.

Considering the classification of smooth OADP threefolds presented in \cite{cmr}, these examples of degree nine are quite unexpected. This shows that not only smooth OADP varieties are interesting, but also the normal ones.

The main results of this thesis are the following:

In Chapter \ref{ed_chapter}, a full classification of the case in which $V$ is a quadric is given, with a description of the singularities, in Theorem \ref{ed_the_classification}. In the same result there is a description of these varieties as residual intersections of quadrics with a cone over the Segre embedding of $\P^1\times\P^2$, which allow us to prove that those are OADP varieties. In Table \ref{ed_tabela}, the 19 varieties fitting in this case are listed.

In Chapter \ref{sl_chapter}, the number of varieties produced is very high, so an explicit list is not given. This is due to the number of possible degenerations of a curve in the base locus of $\X$. A hypothesis on the base locus of $\X$ was added, namely that it has pure dimension one. This is due to the fact that we do not study in detail its scheme structure, so there could be unexpected embedded points. The types of singularities of $X$ are explained and a geometric description is given (see Theorem \ref{sl_the_classification}). The fact that these threefolds are OADP is also proved.
 
In Chapter \ref{sc_chapter}, a similar hypothesis is also made on the base locus of $\X$. The singularities of $X$ are explained in Proposition \ref{sc_singularidades_de_X}. We prove that the general variety is OADP, by describing it as an intersection of divisors in a certain cone.

\chapter{Preliminaries}

\section{Notations and basic definitions}

Let $X\subset \P^N$ be a \emph{variety}, that is, a reduced projective scheme over $\C$, and $x$ a point in $X$. In what follows, $T_xX$ denotes the \emph{embedded projective space} to $X$ at $x$.

Let $\varepsilon:\Bl_x(X)\to X$ be the blow of $X$ in $x$, with exceptional divisor $E$. If $y$ is a point in $E$, we say that $y$ is \emph{infinitely near to} $x$, writing $x\prec y$.  When doing blow-ups, I refer in the same way to varieties and their strict transforms, if there is no danger of confusion.\label{ciro_transf_estrita}

More generally, we say that $y$ is \emph{infinitely near to $x$ of order $n$} if there is a sequence of blow ups:
\[ X_n \xrightarrow{\,\varepsilon_n\,} X_{n-1}\xrightarrow{\,\varepsilon_{n-1}\,} \cdots \xrightarrow{\,\varepsilon_2\,} X_1\xrightarrow{\,\varepsilon_1\,} X_0=X \]
where $\varepsilon_i$ is the blow up at $x_{i-1}$, such that $\varepsilon_i(x_i)=x_{i-1}$ for $i=1,\ldots,n$, $x=x_0$ and $\varepsilon_n(y)=x_{n-1}$.

If a point $x\in X$ is not infinitely near to any other point in $X$, we say that $x$ is \emph{proper}.

I will now introduce some notations.

For rational normal scrolls, the same notation of \cite{eh_minimal} will be used. So given $0 \leq a_0\leq \dots \leq a_d$, $S(a_0,\ldots,a_d)$ is the image of the projective bundle $\P(\mathcal{E})$ via $\vert \mO_{\P(\mathcal{E})}(1)\vert$, where:
\[ \mathcal{E} = \bigoplus_{i=0}^d \mO_{\P^1}(a_i) \]
In other words, it is the scroll described by the $\P^d$'s joining corresponding points of $d+1$ rational normal curves of degrees $a_0,\ldots,a_d$ (where the rational normal curve of degree zero is a point). It has dimension $d+1$, degree $a_0+\dots +a_d$ and spans a space of dimension $a_0+\dots+a_d+d$.

For a curve $C\subset \P^3$, the exceptional divisor of its blow up is denoted by $E_C$ and the intersection of the strict transform of a surface (or a linear system) $S$ with $E_C$ is denoted $S_C$. The same notation is used in the blow up of points.

Moreover, if $C$ is a rational curve and it is the complete intersection of two surfaces $C=S_1\cap S_2$, then its normal bundle is:
\[ N_C=N_{C/\P^3}\cong \mO_{\P^1}(a)\oplus\mO_{\P^1}(b) \]
where $C^2=a$ in $S_1$ and $C^2=b$ in $S_2$. Then, blowing up $C$, it follows that $E_C\cong\F_n$, where $n=\vert b-a \vert$, since $E_C$ is isomorphic to $\P(N_C)$.

If $C$ is a curve and $E_C\cong \F_0$, I will denote the divisor classes of $E_C$ as $(a,b)$, where $(1,0)$ are the fibers over points of $C$. If $E_C\cong \F_n$, for $n\geq 1$, I'll write its divisor classes as $ae_n+bf_n$, where $e_n$ is the $(-n)$-section and $f_n$ is a fiber.

Sometimes it is usefull to map curves in $E_C$ birationally to $\P^2$. This will be done in the following way:
\begin{lemma}\label{pr_imagemF1F2}
Let $D$ be a curve in $E\cong \F_n$. Then there is a birational map that maps $D$ to a curve $C\subset \P^2$ such that:
\begin{itemize}
\item[(i)] If $n=0$ and $D\equiv (a,b)$, then $C$ has degree $a+b$, one point of multiplicity $a$ and one point of multiplicity $b$. 
\item[(ii)] If $n>0$ and $D\equiv ae_n+bf_n$, then $C$ has degree $b$, a point $p$ of multiplicity $b-a$ and $n-1$ points of multiplicity $a$ infinitely near to $p$. The points of $D$ lying on $e_n$ are mapped to points infinitely near to $p$.
\end{itemize}
\end{lemma} 
\begin{proof}
In item $(i)$, just consider the map that blows up a general point of $\F_0$ and contracts the two lines through it.

To prove $(ii)$, for $n=1$ take the map contracting the $(-1)$-section. For $n>1$, consider the elementary transformation $\varepsilon:\F_n\tor\F_{n-1}$, consisting of the blow up at a general point not contained in $e_n$ and the contraction of the strict transform of the fiber through this point. It maps a curve of class $ae_n+bf_n$ to a curve of class $ae_{n-1}+bf_{n-1}$ and it creates a point of multiplicity $a$ in $e_{n-1}$. Then the result follows.
\\ \end{proof}

After mapping curves in $E_C$ to curves in $\P^2$, we often perform suitable Cremona transformations, in order to obtain simpler curves. Those that will be used the most are the \emph{standard quadratic transformations}:

\begin{defin}\label{pr_def_stdquad}
Let $p_1,p_2,p_3$ be three points in $\P^2$. The standard quadratic transformation based in $p_1,p_2,p_3$ is the Cremona transformation in $\P^2$ determined by the linear system of conics through these points. It blows up $p_1,p_2,p_3$ and contracts the three lines containing pairs of these points. 

A curve of degree $d$ and multiplicity $m_i$ in  $p_i$ is mapped via such transformation to a curve of degree $2d-m_1-m_2-m_3$ with multiplicity $d-m_j-m_k$ in the contraction of the line through $p_j$ and $p_k$.
\end{defin}

We now recall the definition of tangent cone. Let $Y\subset \mathbb{A}^n$ be an affine variety defined by an ideal $I$ and let $y$ a point in $Y$. Consider all the polynomials in $I$ expanded around $y$ and let $I^*$ be the ideal generated by the \emph{leading forms} of these polynomials, that is, the homogeneous parts of lower degree. Then the \emph{affine tangent cone} of $Y$ in $y$ is the scheme defined by the ideal $I^*$, that is, it is the spectrum of the ring:
\[ \C[X_1,\ldots,X_n]/I^* \cong \bigoplus_{n\geq 0} \frac{m_y^n}{m_y^{n+1}} \]
where $X_1,\ldots,X_n$ are the coordinates in $\A^n$ and $(\mO_y,m_y)$ is the local ring of regular functions of $Y$ in $y$.

If $Y\subset \P^n$ is a projective variety, the \emph{projective tangent cone} of $Y$ in $y$, denoted $\mC_yY$ is the closure of the respective affine tangent cone in an affine chart containing $y$. We will use the term \emph{tangent cone} for the projective definition, as the varieties studied here are usually projective.

The dimension of $\mC_yY$ is equal to the local dimension of $Y$ in $y$ and its degree is equal to the multiplicity of $Y$ in $y$. Since it is defined by homogeneous polynomials, it can be projectivized. Its projectivization is isomorphic to the exceptional divisor of the blow up of $Y$ in $y$.

The following is an easy result on tangent cones:

\begin{lemma}\label{pr_interseccao_conetangente}
Let $Y,Z$ be two hypersurfaces in $\P^n$ and let $q\in Y\cap Z$ be a point such that:
\[ \mult_q(Y\cap Z)=(\mult_q Y)\cdot (\mult_q Z) \]

Then:
\[ \mC_q(Y\cap Z) = \mC_qY \cap \mC_qZ \] 
\end{lemma} 
\begin{proof}
Let $\A^n$ be an affine chart of $\P^n$ containing $q$. Then $Y$ and $Z$ are defined in $\A^n$ respectively as the zero locus of polynomials $f$ and $g$ in $n$ variables. Let $f^*$ and $g^*$ be the leading forms of these polynomials expanded around $q$ and let $I$ be the ideal generated by $f^*$ and $g^*$.

Then $\mC_qY$ is defined by the ideal $(f^*)$, $\mC_qZ$ is defined by $(g^*)$ and $\mC_q Y\cap \mC_q Z$ is defined by $I$. The tangent cone of $Y\cap Z$ is generated by the leading forms of polynomials in the ideal $(f,g)$. These include the polynomials in $I$. Therefore: \[\mC_q(Y\cap Z) \subset( \mC_qY \cap \mC_qZ) \]

The hypothesis implies that the $\mC_qY$ and $\mC_qZ$ do not have common components. Therefore:
\[ \dim (\mC_qY \cap \mC_qZ)=n-2=\dim \mC_q(Y\cap Z) \]

Moreover, the hypothesis also implies that:
\[ \deg \mC_q(Y\cap Z)=\deg\mC_qY \cdot \deg \mC_qZ = \deg(\mC_qY \cap \mC_qZ)\]
Hence the result is proved.
\\ \end{proof}

Let $\phi:\P^n\tor X\subset \P^N$ be a birational map defined by a linear system $\L$ in $\P^n$. Suppose there is a point $p\in \P^n$ of multiplicity $m>0$ for $\L$ that is mapped to a point $x_p\in X$. In other words, blowing up $p$, the strict transform of $\L$ intersects $E_p\cong \P^{n-1}$ in a fixed divisor of degree $m$. The following Lemma gives a method to understand the projectivized tangent cone of $X$ at $x_p$, that is,  the exceptional divisor of the blow up of $X$ in $x_p$.

\begin{lemma}\label{pr_truqueconetangente}
Keeping the above notation, let $\wL$ be the linear system of hypersurfaces in $\L$ having multiplicity $m+1$ in $p$. Let $C$ be the image of $E_p$ via the map defined by the moving part of $\wL\cap E_p$. If $E_p$ is not contracted by this map, that is, if $C$ has dimension $n-1$, then it is projectively equivalent to a component of the projectivized tangent cone of $X$ at $x_p$. 
\end{lemma}
\begin{proof}
The hypersurfaces in $\L$ correspond, in $\P^N$, to the hyperplane sections of $X$. Since $p$ is mapped to $x_p$, $\wL$ corresponds to those hyperplane sections which contain $x_p$. Then $\wL$ defines a map $\widehat{\phi}:\P^n\tor X^\prime\subset \P^{N-1}$, where $X^\prime$ is the image of $X$ via the internal projection $\pi$ from $x_p$. Note that $\pi^{-1}(C)=x_p$. Let $\tilde{\pi}$ be the resolution of $\pi$ by the blow up $\tilde{X}\to X$ of $X$ in $x_p$, with exceptional divisor $E\subset \P^{N-1}$, where $\P^{N-1}$ is the exceptional divisor of the blow up of $\P^N$ in $x_p$. See the diagram below:

\[ \xymatrix{ 
& 	&	&\tilde{X} \ar@{->}[dl]_{\Bl_{x_p}} \ar@{->}[ddl]^{\tilde{\pi}} \\
&\P^n\ar@{-->}[r]^\phi	&X\ar@{-->}[d]^\pi	& \\
&						&X^\prime\ar@{<--}[ul]^{\widehat{\phi}}	&
} \]

The exceptional divisor $E$ is projectively equivalent to its image $C^\prime\subset X^\prime$ via $\tilde{\pi}$, that is, $C^\prime$ is projectively equivalent to the projectivized tangent cone of $X$ in $x_p$. Note that the cone over $C^\prime$ with vertex $x_p$ is actually the tangent cone of $X$ in $x_p$, since it consists of the lines having an extra tangency with $X$ in $x_p$. 

Since $\pi^{-1}$ contracts $C$ to $x_p$, it follows that $C\subset C^\prime$. By hypothesis, $C$ has dimension $n-1$, so it is a component of $C^\prime$ and the result follows.
\\ \end{proof}

To illustrate this Lemma, we'll give an example:

\begin{example}
Let $\L$ be the linear system of cubic curves in $\P^2$ having two base points: $p$ and $q$, with $p\prec q$. It defines a map $\phi:\P^2\tor X\subset\P^7$. It is known that $X$ has a singularity of type $A_1$, which is locally of the form $x^2+y^2+z^2=0$. This singular point $x_p$ is the image of $p$.

Let $\wL$ be the linear system of curves in $\L$ having multiplicity two in $p$. Note that $p$ is the only base point of $\wL$. It defines the map $\widehat{\phi}:\P^2\tor X^\prime\subset \P^6$, where $X^\prime=\pi(X)$ is the projection of $X$ from $x_p$. The point $x_p$ is the contraction via the inverse of $\pi$ of a conic $C\subset X^\prime$: the cone over $C$ with vertex $x_p$ is the tangent cone of $X$ in $x_p$.

Blowing  up the point $p$, the restriction of $\wL$ to $E_p$ has degree two and no base points, so it is mapped by $\widehat{\phi}$ to a conic. By Lemma \ref{pr_truqueconetangente}, this conic is $C$.

Now let $\mathcal{M}$ be the linear system of cubic curves in $\P^2$ having three infinitely near base points, that is, $p\prec q\prec q^\prime$. It defines a map $\psi:\P^2\tor Y\subset\P^6$ which maps $p$ to a point $y_p$.

Let $\widehat{\mathcal{M}}$ be the linear system of curves in $\mathcal{M}$ having multiplicity two in $p$. These curves have a further base point in $q$. Therefore, blowing up $p$, the moving part of $\widehat{\mathcal{M}}\cap E_p$ maps it to a line. By Lemma \ref{pr_truqueconetangente}, the projectivized tangent cone of $Y$ in $y_p$ contains a line, that is, this tangent cone contains a plane.

This agrees with the fact that $y_p$ is a singularity of type $A_2$, which is locally of the form $x^2+y^2+z^3=0$. Its tangent cone is $x^2+y^2=0$, which factors as the union of two planes.
\end{example}

\section{Bronowski and OADP varieties}

Let $X\subset \P^N$ be an irreducible, non degenerate, projective variety of dimension $n$. Let $X(2)$ be the twofold symmetric product of $X$. Consider the incidence correspondence:
\[ I(x):=\{ ((x_1,x_2),p)\in X(2)\times \P^N : x_1\neq x_2, p\in <x_1,x_2> \} \]
Define the \emph{abstract secant variety} of $X$ to be the closure $\mS(X)$ of $I(X)$ in $X(2)\times \P^N$. The image $\Sec(X)$ of the projection of $\mS(X)$ to $\P^N$ is the \emph{secant variety} of $X$. In other words, $\Sec(X)$ is the closure in $\P^N$ of the union of the \emph{secant lines} of $X$, that is, the lines spanned by two distinct points in $X$.

Since $\mS(X)$ has dimension $2n+1$, the surjective map $p_X:\mS(X)\to \Sec(X)$ gives:
\[ \dim(\Sec(X))\leq \min\{2n+1,N\} \]
and $X$ is said to be \emph{secant defective} if the strict inequality holds. According to \emph{Terracini's Lemma}, $X$ is secant defective if and only if:
\[ \dim(T_{x_1}X\cap T_{x_2}X) > \min\{2n-N,-1\} \]

A \emph{general tangential projection} of $X\subset \P^{2n+1}$ is a projection $\tau_x:X\tor \P^n$ from the tangent space $T_xX$ at a general point $x\in X$. Then $X$ is said to be a \emph{Bronowski variety} if its general tangential projection is birational. 

A Bronowski variety is not secant defective. Indeed, let $C$ be the join of $T_xX$ and $X$. Then by \cite[equation 1.2.8]{r_tangent}:
\[ \dim(C)=\dim(T_xX)+\dim(X)-\dim(T_xX\cap T_yX) \] 
with $y$ a general point in $X$. If $X$ is Bronowski, $\tau_x$ is dominant and $\dim(C)=2n+1$, giving $\dim(T_xX\cap T_yX)=-1$. Then the result follows by  Terracini's Lemma. 

Suppose now that $X\subset \P^{2n+1}$ is not secant defective, that is, $\Sec(X)=\P^{2n+1}$. Then we say that $X$ is a \emph{variety with one apparent double point}, shortly \emph{OADP variety}, if  $p_X:\mS(X)\to \P^{2n+1}$ is birational, that is, if through a general point of $\P^{2n+1}$ there is a unique secant line to $X$. 

Using the same argument of \cite[\S 3]{cmr}, we have that \emph{OADP varieties are Bronowski varieties}. The following is known as the \emph{Bronowski conjecture}:
\begin{conj}
Any Bronowski variety is OADP.
\end{conj}

\section{The tangential projection}

Let $X\subset \P^7$ be a Bronowski threefold. The results in this section (except for Lemma \ref{pr_truque_secao_tgente}) were presented in \cite{cr} for arbitrary $n=\dim X$.

Let $x$ be a general point of $X$ and let $\tau=\tau_x:X\tor \P^3$ be the associated tangential projection. Consider the blow up $\pi:\Bl_x(X)\to X$ of $X$ at $x$ with exceptional divisor $E\cong \P^2$ and set $\tilde{\tau}=\tau \circ \pi$. Let $V$ be the image of $E$ via $\tilde{\tau}$. Then the map:
\[ \bar{\tau} = \tilde{\tau}_{\vert E}:E\tor V \subset \P^3 \]
is birational (see \cite[(5,7)]{gh}). It is defined by a linear system of conics, the \emph{second fundamental form} $\II_{X,x}$ of $X$ at $x$. See \cite[(3.2)]{iveylands} for other definitions of the second fundamental form.

The surface $V$ is called the \emph{fundamental surface} of $X$. Since it is the image of $\P^2$ via a linear system of conics, $V$ is a birational projection to $\P^3$ of the Veronese surface of degree four in $\P^5$. Hence:
\begin{lemma}
If $X$ is a Bronowski threefold, then one of the following holds:
\begin{itemize}
\item[(i)] $\II_{X,x}$ has a fixed line and $V$ is a plane;
\item[(ii)] $\II_{X,x}$ has three base points and $V$ is a plane;
\item[(iii)] $\II_{X,x}$ has two base points and $V$ is a smooth quadric surface or a quadric cone;
\item[(iv)] $\II_{X,x}$ has one base point and $V$ is a cubic surface, projection of a cubic scroll from an external point;
\item[(v)] $\II_{X,x}$ has no base points and $V$ is a Steiner quartic surface, projection of the Veronese quartic from an external point.
\end{itemize}
\end{lemma}

Cases $(i)$ and $(ii)$ have been completely described in \cite{cmr}. Case $(iii)$ will be considered in Chapter \ref{ed_chapter}, case $(iv)$ in Chapter \ref{sl_chapter} and case $(v)$ in Chapter \ref{sc_chapter}.

Let $\sigma$ be the inverse of $\tau$, defined by a linear system $\X$, of dimension $7$ and degree $d\leq \deg(X)$. It contracts $V$ to the point $x\in X$, so $\X$ has fixed intersection with $V$.

Set $\delta=\deg(V)$. Let $\X^\prime$ be the linear system of surfaces $F$ of degree $d-\delta$ such that $F+V\in \X$. The moving part of $\X^\prime$ corresponds, via $\tau$, to hyperplane sections of $X$ through $x$.

And let $\X^{\prime\prime}$ be the linear system of surfaces $F$ of degree $d-2\delta$ such that $F+2V\in \X$. Its moving part corresponds, via $\tau$, to tangent hyperplane sections of $X$ at $x$, that is, hyperplane sections of $X$ containing $T_xX$. In conclusion:
\begin{lemma}\label{pr_X'eX''}
Let $\X^\prime$ and $\X^{\prime\prime}$ be as defined above. Then:
\begin{itemize}
\item[(i)] The image of $\P^3$ via the map defined by $\X^\prime$ is the image of $X$ by the internal projection $\pi_x:X\to X^\prime\subset \P^6$ from $x$. 
\item[(ii)] The restriction of $\X^\prime$ to $V$ defines the map $\bar{\sigma}$, inverse of $\bar{\tau}$
\item[(iii)] The moving part of $\X^{\prime\prime}$ is the linear system of planes in $\P^3$, which correspond, via $\tau$, to the tangent hyperplane sections of $X$ at $x$.
\item[(iv)] If $\Pi\subset\P^3$ is a plane containing a base curve $C$ of $\X$, then it corresponds to a tangent hyperplane section of $X$ containing the image of the blow up at $C$.
\item[(v)] The fixed part of $\X^{\prime\prime}$ consists of the exceptional divisors corresponding to the blow up at the indeterminacy locus of $\tau$ out of $x$ and not contracted by $\tau$.
\end{itemize}
\end{lemma}

Item $(iii)$ implies that $d\geq 2\delta+1$ and equality holds if and only if $\X^{\prime\prime}$ has no fixed part. In this thesis, the following hypothesis will be made on $X$:\label{pr_hipoteseH}
\begin{itemize}
\item[(H)] Let $X$ be a normal Bronowski threefold such that the linear system $\X$ defining the inverse of the tangential projection at a general point $x\in X$ has degree $d=2\delta+1$, where $\delta$ is the degree of the fundamental surface.
\end{itemize}

This is a working hypothesis, so we are restricting to a class of normal Bronowski threefolds satisfying this condition on the degree of $\X$. We now make some remarks in order to clarify its meaning.

As noted in Lemma \ref{pr_X'eX''}, $d>2\delta+1$ if and only if $\Xpp$ has a fixed part, which consists of the exceptional divisors corresponding to the blow up at the indeterminacy locus of $\tau$ out of $x$. So, one has to analyse the intersection scheme $T_xX\cap X$ and decide if its components, with the exception of $x$, are contracted to curves or points in $\P^n$.

Since $\dim X=3$, by \cite[Proposition 5.2]{cmr} $T_xX\cap X$ has a component of dimension two if and only if $X$ is either a hypersurface or a scroll over a curve. Moreover, it was explained in \cite[Proposition 5.5]{cr} that a normal Bronowski threefold that is a scroll over a curve is a smooth rational normal scroll. In this case, $\X$ is the linear system of cubic surfaces having a double and a simple fixed lines, both lying in the plane $V$. Therefore, if $X$ is a normal Bronowski threefold, hypothesis (H) does not hold only if $\dim (T_xX\cap X)\leq 1$. 

Note also that (H) is satisfied when $X$ is a smooth OADP threefold. This can be proven by analysing the tangential projection of the threefolds given in the classification in \cite{cmr}. Another option is to use the fact that in this case the support of $T_xX\cap X$ is either $x$ or a bunch of lines through $x$ (see \cite{cmr}, Lemma 5.3 and Proposition 5.4). Moreover, by \cite[Proposition 6.5]{cmr}, a line of $X$ is mapped by $\tau$ to a line in $\P^3$. Therefore, the image of the blow up of the indeterminacy locus of $\tau$ off $x$ is one-dimensional, so $d=2\delta+1$.

If these considerations can be generalized for a smooth Bronowski threefold, then the results in this thesis imply that the Bronowski conjecture holds for smooth threefolds.

The condition $d=2\delta+1$ is also satisfied when $X$ is a normal Bronowski variety in which the fundamental surface is a plane, that is, $\delta=1$. This follows from the remarks made above on scrolls over a curve and from \cite[Proposition 2.3 and Proposition 5.7]{cr}.

\quad

Before proceeding to the next result, we'll say some words on the dual variety. If $X\subset \P^r$ is an irreducible variety, the \emph{dual variety} of $X$, denoted $X^*$ is the image in ${\P^r}^*$ of the \emph{conormal scheme}, which is the Zariski closure in $X\times {\P^r}^*$ of the irreducible scheme:
\[ \{(x,H):x\in X \setminus \Sing(X), H\in{\P^r}^*, T_xX\subset H \} \]
of dimension $r-1$. Therefore $\dim(X^*)\leq r-1$, and $X$ is said to be \emph{dual defective} if strict inequality holds. According to \cite[(3.5)]{gh}, $X$ is dual defective if and only if at a general point $x\in X$, the second fundamental form $\II_{X,x}$ is a singular linear system, that is, every member is singular.

These definitions can be extended, component-wise, for reducible varieties.

The following Lemma will be useful in understanding the singularities of $X$.

\begin{lemma}\label{pr_truque_secao_tgente}
Let $X$ be a Bronowski threefold, $x\in X$ a general point and $\sigma:\P^3\tor X$ be the inverse of $\tau$, as defined above. Suppose that $\II_{X,x}$ has general smooth member. Let $q\in \P^3$ be a point that is the preimage via $\sigma$ of a point $x_q\in X$. If the image of a general plane through $q$ has a singularity of multiplicity $m$ in $x_q$, then $X$ has a singularity of multiplicity $m$ in $x_q$.
\end{lemma}
\begin{proof}
By Lemma \ref{pr_X'eX''}, planes in $\P^3$ are mapped by $\sigma$ to tangent hyperplane sections of $X$ at $x$. 
Since $x_q$ is projected to $q$, tangent hyperplane sections of $X$ at $x$ containing $x_q$ are projected to planes in $\P^3$ through $q$.

Since the second fundamental is not singular at a general point of $X$, it follows that $\dim(X^*)=6$. The tangent hyperplane sections of $X$ (at any point) that contain $x_q$ correspond to points in a hyperplane section $\Sigma\cap X^*=\Sigma_X$ of $X^*$, with $\Sigma\subset (\P^7)^*$. 

By hypothesis, the image of a general plane through $q$ has a singularity of multiplicity $m$ in $x_q$. Then this is a point of $X$ of multiplicity $m^\prime\leq m$. Suppose that equality does not hold, that is $m^\prime<m$. Then a general tangent hyperplane section of $X$ at $x$ through $x_q$ is tangent to the tangent cone $C$ of $X$ in $x_q$. Since $x$ is general, it can be replaced by any general point, that is, this holds for a hyperplane section corresponding to a general point in $\Sigma_X$. 

The threefold $C$ is a cone over a (possibly reducible) surface $Y$ with vertex $x_q$. Therefore a hyperplane is tangent to $C$ in $x_q$ if and only if it contains $x_q$ and is tangent to $Y$.  Such hyperplanes correspond to points in the hyperplane section $\Sigma\cap Y^*=\Sigma_Y$ of $Y^*$. Note that $\dim(\Sigma_Y)\leq 5$.

Since a general tangent hyperplane section of $X$ through $x_q$ is also tangent to $C$, it follows that $\Sigma_X\subset \Sigma_Y$. Therefore $\Sigma_X$ is a component of $\Sigma_Y$, that is, there is a component $Z$ of $Y$ such that $\Sigma_X=\Sigma_Z:=\Sigma\cap Z^*$.

Hence, tangent hyperplane sections of $X$ through $x_q$ coincide with hyperplane sections of $X$ through $x_q$ that are tangent to $Z\subset Y$. Remember that the cone over $Y$ with vertex $x_q$ is the tangent cone of $X$ at this point. Then, if $X^\prime$ is the projection of $X$ from $x_q$ to a $\P^6$ containing $Y$, it follows that $Y\varsubsetneq X^\prime$, since $X^\prime$ is a threefold ($X$ is not a cone). Moreover, tangent hyperplane sections of $X^\prime$ coincide with hyperplane sections of $X^\prime$ that are tangent to the surface $Z\subset X^\prime$. This implies $(X^\prime)^*=Z^*$.

Dualizing this equality gives $X^\prime=Z$ (see \cite{harris}, Theorem 15.24 and Example 16.20), a contradiction. Thus $m=m^\prime$.
\\ \end{proof}

\section{OADP threefolds}

I will now quickly recall two examples of OADP threefolds that will be used in the following chapters. A more detailed description of these and other examples, as well as the proofs that these are OADP threefolds, can be found in \cite{cmr}.

\begin{example}[Degree seven Edge threefolds]\label{pr_exemplo_edge}
Let $Y\subset \P^7$ be the Segre embedding of $\P^1\times \P^3$. Let $\Pi$ be a general $\P^3$ of the ruling of $Y$ and $Q$ a quadric of $\P^7$ containing $\Pi$. Then the residual intersection of $Q$ with $Y$ is a smooth threefold, called the \emph{Edge threfold of degree seven}. It is part of a series of OADP varieties with dimension $n\geq 2$ and degree $2n+1$ defined by Edge, in \cite{edge}, and Babbage, in \cite{babbage}. In the same paper, Edge also defines another series of OADP varities, which have degree $2n$.

The degree seven Edge threefold has two lines through a general point.
In Chapter \ref{ed_chapter}, it will be showed that the fundamental surface of $X$ is a smooth quadric.
\end{example}

\begin{example}[Degree eight scroll in lines]\label{pr_exemplo_sl}
This example was first proven to be an OADP variety by Alzati and Russo, in \cite{ar_oadp}. See also \cite{cmr}. Let $Y$ be the Segre embedding of $\P^1\times \P^2$ and let $F\subset \P^7$ be the cone over $Y$ having a line as vertex. Define $X$ to be the residual intersection of $F$ with two general quadrics containing a $\P^4$ of its ruling.

This threefold has degree eight and is ruled by lines. 
In Chapter \ref{sl_chapter}, we'll see that $\II_{X,x}$ has one base point and the fundamental surface $V$ is a cubic.
\end{example}

Note that these two OADP threefolds are defined by the residual intersection of quadrics with a rational normal scroll. Related to this, is the following definition: An irreducible non-degenerate variety $X\subset \P^N$ is \emph{linearly normal}\label{pr_lin_normal} if one of the following equivalent properties holds:
\begin{itemize}
\item [(i)] there is no variety $X'\subset \P^{N'} $ with $N'>N$ and a linear projection $\phi: \P^ {N'}\dasharrow \P^ {N}$ inducing an isomorphism $\phi: X'\to X$;
\item [(ii)] the linear system $\vert \mO_{X}(1)\vert$ has dimension $N$.
\end{itemize}

\begin{remark}\label{pr_oadpLN}
It was proven in \cite{cr} that OADP varieties are linearly normal. In the case of a Bronowski variety $X\subset\P^{2n+1}$, let $\X$ be the linear system defining the inverse of a general tangent projection. Suppose $\X$ is \emph{relatively complete}, that is, it is the complete linear system on $\P^n$ with the same class in $\Pic(\P^n)$ and with the same base locus of $\X$. Then the strict transform of $\X$ via $\tau$, which is the complete linear system $\vert \mO_X(1)\vert$, has dimension $2n+1$. Then $X$ is linearly normal.
\end{remark}

In order to understand the varieties constructed here as subvarieties of a rational normal scroll, the following theorem (see \cite{eh_minimal}) will be used:

\begin{theorem}\label{pr_contidonumscroll}
Let $X\subset \P^N$ be a linearly normal variety, and $D\subset X$ a divisor. If $D$ moves in a pencil $\{D_\lambda\vert \lambda\in\P^1\}$ of linearly equivalent divisors, then writing $\overline{D}_\lambda$ for the linear span of $D_\lambda$ in $\P^N$, the variety:
\[S = \bigcup_\lambda \overline{D}_\lambda \]
is a rational normal scroll.
\end{theorem}

\chapter{The quadric case} 
\label{ed_chapter}

  This chapter treats the case in which the fundamental surface $V$ is a quadric. This is the situation of the Edge threefold of degree seven, described in Example \ref{pr_exemplo_edge}. 

\section{First Considerations}

\subsection{The base locus of $\X$}
\label{ed_blocus_sec}

 The main result of this section is the following:  

\begin{lemma}\label{ed_baselocus}
The linear system $\X$ is defined as degree five surfaces having:
\begin{itemize}
\item multiplicity two in $C_4$, a quadric section of $V$;
\item multiplicity two at a smooth point $p$ of $V$;
\item multiplicity one in $\l$ and $\l^\prime$, the two (possibly infinitely near) lines in $V$ through $p$.
\end{itemize}
  In particular:
\[ \X \cap V=2C_4 +\l + \l^\prime \]

If $V$ is a cone with vertex $q$, then $\X$ has multiplicity three or four at $q$. If it is three,  there is a  quadric containing $C_4$ which is smooth in $q$. If it is four, all such quadrics are singular in this point.  
\end{lemma}
\begin{proof}
Since $\deg V=2$, the map $\bar{\tau}:E\tor V$ is defined by conics with two fixed points. Its inverse is defined by plane sections of $V$ through a fixed smooth point $p$. The quadric $V$ is a cone if and only if the two base points of $\bar{\tau}$ are infinitely near. 

By hypothesis (H), $\deg(\X)=5$. Then the restriction of $\X^\prime$ to $V$ gives cubic sections of $V$. The moving part is made of plane sections through $p$. So the fixed part is a quadric section $C_4$ of $V$. 

Being a fixed curve of $\X^\prime$, $C_4$ is a double curve of $\X$. So the restriction of $\X$ to $V$ is $C_4$ with multiplicity two and a fixed plane section of $V$.

Since $p$ is in the base locus of $\X^\prime$, it is a double point of $\X$. Note that, removing $V$ twice from $\X$ gives the linear systems of planes in $\P^3$, so $\X$ has not multiplicity greater than two in $C_4$ or in $p$.

If $p\notin C_4$, the degree two curve in $\X\cap V$ is a pair of (possibly coincident) lines through $p$.

If $p\in C_4$, then $\X^\prime \cap V$ has at least a double point at $p$. But, as noted above, $p$ cannot be a double point of $\X^\prime$. So $\X^\prime$ is tangent to $V$ at $p$ and, hence, $\X$ has multiplicity two in $p\in C_4$ and in an infinitely near point to $p$. Then again the degree two curve is  a pair of lines through $p$. 

 This proves the first part. Note that the case where $p\in C_4$, the point infinitely near to $p$ cannot lie on $C_4$. Indeed, this would imply that $\Xp$ has multiplicity three in $p$, which would then be a base point of $\Xpp$. This contradicts hypothesis (H).  

Suppose now that $V$ is a cone and let $q$ be its vertex. So $\l=\l^\prime=<p,q>$ (remember that $p$ is a smooth point of $V$). Since $\X^\prime$ must desingularize $V$, $\mult_q \X^\prime\geq 1$. Reattaching $V$ we find that:
\[ \mult_q \X\geq 3 \]
On the other hand, it can't be greater than four, since $\X^{\prime\prime}$ has no base points.

If $\mult_q \X=4$, then $\mult_q \X^\prime=2$, so $C_4$ has multiplicity four in $q$ and breaks into four lines. In particular, a quadric intersecting $V$ in $C_4$ is another cone with vertex $q$. 
 
 If $\mult_q \X=3$, then $q$ is a double point of $C_4$. In this case, a quadric intersecting $V$ in  $C_4$ is not a cone with vertex $q$, otherwise $C_4$ would be the union of four lines through $q$ and $\X$ would have multiplicity four at $q$. Then there is a quadric containing $C_4$ which is smooth in $q$. 
\\ \end{proof}

If $V$ is a smooth quadric, $\l$ and $\l^\prime$ are obviously the pair of lines through $p$ in $V$. So, in this case, different Bronowski threefolds are determined by the choice of $C_4$ and $p$ (they define a unique quadric $V$).  This is also the case if $V$ is a cone, with the difference that the lines through $p$ are infinitely near. 

Since $C_4$ is a quadric section of $V$, it is the base locus of a pencil of quadrics $\mQ$. 

\begin{lemma}\label{ed_imagemquadricas}
Each quadric of $\mQ$ is mapped isomorphically by $\sigma$ to a quadric contained in $X$. The point $p$ is mapped to a quadric $S_2^x$ in $X$ through $x$, isomorphic to $V$. These form a family $\mQ^\prime$ of disjoint quadric surfaces that cover $X$.
\end{lemma}	
\begin{proof}
The restriction of $\X$ to a quadric of $\mQ$ (except for $V$) is the double curve $C_4$ and moving plane sections. So it is mapped in $\P^7$ to a isomorphic quadric contained in $X$. Two quadrics of $\mQ$ are disjoint after the blow up at $C_4$.

Blowing up  $p$, $\X_p$ is a linear system of conics with two base points, given by the directions of $\l$ and $\l^\prime$ at $p$. These points are infinitely near if and only if $V$ is a cone. So $p$ is mapped to a quadric isomorphic to $V$. Since $V_p$ is the line through the two base points of $\X_p$, this quadric contains $x$. The only quadric of $\mQ$ through $p$ is $V$, so the quadrics in $\mQ^\prime$ are disjoint.

Given a point $y\in X$, it is the image of a point in $y^\prime\in \P^3$ or it lies on the exceptional locus of $\sigma$. In the first case, there is a quadric of $\mQ$ through $y^\prime$ which is mapped to a quadric of $\mQ^\prime$ through $y$. If $y$ is the image of $y^\prime\in E_{C_4}$, again there is a quadric of $\mQ$ through this point (if this quadric is $V$, then $y=x$). The point $p$ is mapped to $S_2^x$ and, as it will be noted in Lemma \ref{ed_imagem_p_e_retas}, $\l$ and $\l^\prime$ are mapped to lines in $S_2^x$. This finishes the proof.
\\ \end{proof}

Remember that planes in $\P^3$ correspond to tangent hyperplane sections of $X$ at $x$. $\X$ intersects a general plane in curves of degree five having four double and two simple points, mapping it to a surface of degree:
\[ 25 - 4\cdot 4 -2 = 7 \]
As a consequence, $X$ has degree seven.

\subsection{Equations for $\X$}

We now give general equations for $\X$. Suppose $V$ is given by $(g=0)$ and take another quadric of $\mQ$ with equation $(g^\prime=0)$. An equation for $\X$ is then:
\[ (g)^2(a_0x_0+a_1x_1+a_2x_2+a_3x_3) + gg^\prime\cdot P_p + a_7g^\prime h =0 \]
where $P_p=0$ is a general equation for the linear system of planes through $p$ (so it has three parameters) and $(h=0)$ is a cubic surface not containing $V$ that has a double point at $p$ and contains $C_4$. Such a cubic can be taken from the linear system:
\[ g(b_0x_0+b_1x_1+b_2x_2+b_3x_3)+g^\prime (c_0x_0+c_1x_1+c_2x_2+c_3x_3)=0 \]
by imposing the condition of singularity in $p$.

Note that (after removing the common factor) the first four monomials give $\X^{\prime\prime}$, and the first seven give $\X^\prime$.

Setting $p=(1:0:0:0)$, the map $\sigma$ can be written as:
\[ \sigma=(g^2x_0:g^2x_1:g^2x_2:g^2x_3:gg^\prime x_1:gg^\prime x_2:gg^\prime x_3:g^\prime h) \]

\subsection{A Cremona transformation}
\label{ed_jonqui_sec}

We now start with a curve $C_4$ and a point $p\notin C_4$ (possibly infinitely near to a point of $C_4$) in $\P^3$, and suppose that $C_4$ is the base locus of a pencil of quadrics $\mQ$ with general smooth member. Suppose also that the quadric $V\in\mQ$ containing $p$ is irreducible   and smooth in $p$  . If $V$ is smooth, let $\l$ and $\l^\prime$ be the lines in $V$ through $p$. If $V$ is a cone let $\l=\l^\prime$ be the line joining $p$ and the vertex of $V$.

Let $\L$ be the linear system of cubic surfaces containing $C_4$ and having a double point at $p$. If $p$ is infinitely near to $\hat{p}\in C_4$, define $\L$ containing $C_4$, having a double point at $\hat{p}$ and with tangent cone at $\hat{p}$ containing the tangent plane of $V$ in $\hat{p}$.

In both cases, those cubic surfaces contain the lines $\l$ and $\l^\prime$ and intersect $V$ at $C_4+\l+\l^\prime$.

\begin{lemma}\label{ed_cremona}
Suppose that the  conditions imposed on $\L$ are independent. Then $\L$ defines a Cremona transformation:
\[f_\L = f:\P^3\tor \P^3 \]

It maps $p$ to a quadric $V^\prime$ isomorphic to $V$ and contracts $V$ to a point $p^\prime\in V$. It maps the other quadrics of $\mQ$ to quadrics of a pencil having a base curve $C_4^\prime$ of the same type of $C_4$. 
\end{lemma}
\begin{proof}
Cubic surfaces containing $C_4$ have equations $(gf_1+g^\prime f_2=0)$, where $(g=0)$ and $(g^\prime =0)$ are equations for two quadrics of $\mQ$ and $f_1,f_2$ are linear forms; so there are eight parameters. If $p\notin C_4$, the number of conditions for surfaces to have a double point in $p$ is four. If $p$ is infinitely near to $\hat{p}\in C_4$, two conditions are needed for a double point in $\hat{p}$ and two more conditions for the tangent cone of $\L$ in $\hat{p}$ to contain $T_{\hat{p}}V$. Since the conditions are independent, this linear system has projective dimension three. So $f$ maps to $\P^3$.

The union of $V$ and  planes through $p$ give three independent cubics in $\L$. Hence $\L$ is generated by these three cubics plus another cubic not containing $V$. So the intersection of two surfaces of $\L$ is made of $C_4,\l,\l^\prime$ and plane sections of this cubic through $p$. And two plane sections through $p$ intersect at a single point besides $p$ itself. Hence $f$ is a Cremona transformation.

The contraction of $V$ is obvious, since $\L$ has fixed intersection with it. Blowing up  $p$, $\L$ cuts $E_p$ in conics with two base points (just like what happened with $\X$). So $f$ maps $p$ to a quadric surface $V^\prime$, which is a cone if and only if $V$ is.

A quadric of $\mQ$ is cut by $\L$ in the fixed curve $C_4$ and a linear system of plane sections. So it is mapped again to an isomorphic quadric surface. Hence, the image of $\mQ$ is a pencil having a base curve $C_4^\prime$ of the same type of $C_4$.
\\ \end{proof}

The map $f$ is a Cremona transformation of De Jonquieres type: it restricts to planes through $p$ in  the linear system of cubics with one double point and four simple points. These planes are mapped to planes through $p^\prime$.

The intersection of $\L$ with a general plane consists of cubic curves with six simple points. Therefore this plane is mapped by $f$ to a cubic surface,   that is, $f^{-1}$ is also defined by cubic surfaces.  

The linear system defining $f^{-1}$ has fixed intersection with $V^\prime$ and a double point in $p^\prime=f(V)$. This fixed intersection is $C_4^\prime$ and the two lines in $V^\prime$ through $p$. Indeed, $f^{-1}$ maps each quadric of the pencil having base locus $C_4^\prime$ to a quadric in $\mQ$, so this curve lies on the base locus of $f^{-1}$. Therefore:

\begin{lemma}
The rational map $f^{-1}$ is defined by the linear system of cubic surfaces containing $C_4^\prime$ and having a double point in $p^\prime$, the image of $V$ via $f$.
\end{lemma}

\subsection{The Segre symbol of a pencil of quadrics}
\label{ed_segre_sec}

Since each quadric of the pencil $\mQ$ is mapped isomorphically to a quadric of $\mQ^\prime$, a good tool to understand $X$ is the \emph{Segre symbol} of $\mQ$. It is due to a classification of C. Segre of pencils of quadrics, done in \cite{segresymbol}. The following quick revision is based on a paper of D. Avritzer and R. Miranda \cite{avritzer}.

Suppose the pencil $\mQ$ is given by:
\[ \lambda F_1 + \mu F_2 = 0\]

Let $A_1$ and $A_2$ be the two $4\times 4$ symmetric matrices defining the quadratic forms $F_1$ and $F_2$ and consider the equation:
\begin{equation}\label{ed_det}
\det( \lambda A_1+\mu A_2 )= 0
\end{equation}

Suppose there is a smooth quadric in $\mQ$. This implies that the left-hand side of \eqref{ed_det} is not identically zero, so it has $4$ roots in $(\lambda:\mu)$, counting multiplicities. These roots correspond to the singular quadrics in $\mQ$, being in general $4$ cones.

In some cases, a root of \eqref{ed_det} is also a solution of the equations on the $3\times 3$ minors, producing a pair of planes. If it's also a solution of the equations on all $2\times 2$ minors, a double plane is produced.

Given a root of \eqref{ed_det}, let $k$ be the dimension of the singular locus of the quadric determined by this root (so it's a root for the minors of order $4-k$) and let  $l_i$ be the minimum multiplicity with which a root appears in the equations of subdeterminants of order $4-i$, for $i=0,\ldots,k$. Define $e_i:=l_i-l_{i+1}\geq 0$ (set $l_{k+1}=0$ in order to define $e_k$). Note that $\sum e_i = l_0$ is the multiplicity of the given root in \eqref{ed_det}.

This gives, for each root  $(\lambda_j:\mu_j)$ a sequence $(e^j_0 \cdots e^j_{k_j})$, with $k_j\leq 2$. Let $r$ be the number of distinct roots of \eqref{ed_det}.
\begin{defin}
The Segre symbol of the pencil $\mQ$ is:
\[ [ (e^1_0 \cdots e^1_{k_1}) , \ldots , (e^r_0 \cdots e^r_{k_r}) ] \]
When $k_j=0$, the parenthesis are omitted.
\end{defin} 

For example, a pencil with Segre symbol $[1,1,1,1]$ has four cones, which is the maximum in this case. While the singular members of a pencil having Segre symbol $[(11),2]$ is a pair of planes and a cone.

Since the quadrics, contained in $X\subset \P^7$, of the family $\mQ^\prime$ are isomorphic to the ones of $\mQ$, the Segre symbol of $\mQ$ will give us information on the singular members of $\mQ^\prime$. It will help us to describe the singularities of $X$ and to study the different possibilities for the base locus $C_4$ of $\mQ$. The computations on the Segre symbol and singular elements of $\mQ$ done is this thesis can be performed by hand. Alternatively, it can be found in \cite[p. 305]{hodge}.

In \cite{segresymbol}, Segre studies the properties of the base curve of a pencil of quadrics considering its Segre symbol. The following result is a simplified version of this investigation made by Segre. It can also be obtained by a direct analysis of the possible Segre symbols in $\P^3$, also done in \cite{hodge}.
\begin{prop}\label{ed_singsegre}
Let $q$ be a singular point of the base locus of a pencil $\mQ$ of quadric surfaces in $\P^3$ with general smooth member. Then $q$ is a singular point of a quadric in $\mQ$ corresponding to a root of \eqref{ed_det} with multiplicity greater than $1$.
\end{prop}

\section{The smooth case}
\label{ed_smooth_sec}

Suppose that the fundamental surface $V$ is a smooth quadric. As noted in Section \ref{ed_blocus_sec}, the threefold $X$ is determined by the choices of $C_4$ and $p$.  

\begin{lemma}

 Let $f$ be the Cremona transformation of Section \ref{ed_jonqui_sec} associated to a choice of $C_4$ and $p$, and let $C_4^\prime$ be as in Lemma \ref{ed_cremona}.  The image of $\X$ via $f$ is the linear system of cubic surfaces containing $C_4^\prime$. In the general case ($C_4^\prime$ smooth), this is the rational representation of $X$ described by Edge, in \cite{edge}, and Babbage, in \cite{babbage}.

In particular, the threefold $X$ depends only on the choice of $C_4$.
\end{lemma}\label{ed_sistema_linear_simplificado}
\begin{proof}
To find the image of $\X$ via $f$, first look at $\X^{\prime\prime}$, the planes in $\P^3$. Their image is the linear system defining the inverse $f^{-1}$, that is, cubic surfaces containing $C_4^\prime$ and having a double point at $p^\prime$. Since $V$ is contracted to $p^\prime$ and   $\X^{\prime\prime}=\{S\; ;\; S+2V\in\X\}$,   it follows that $\X$ is mapped to cubic surfaces containing $C_4^\prime$. 
\\ \end{proof}

This Lemma allows us to concentrate our study to the simpler linear system  $\Y$ of cubic surfaces containing $C_4$, since Lemma \ref{ed_cremona} asserts that $C_4^\prime$ is of the same type of $C_4$. We will choose to not follow this path here, since we are interested in the description of the inverse of the general tangential projection of $X$. 

However it is worth to mention the meaning of the map defined by $\Y$. As it will be proved in Theorem \ref{ed_the_classification}, the threefold $X$ is contained in the Segre embedding $Y$ of $\P^1\times\P^3$, being the residual intersection of $Y$ with a quadric containing a $\P^3$ of its ruling. As noted by Edge, the map defined by $\Y$ is the inverse  of the birational projection of $X$ from a $\P^3$  of the ruling of $Y$ to another $\P^3$  of the ruling. This projection contracts a degree eight scroll, described by the directrix lines of $Y$ that are contained in $X$, to the curve $C_4$ (see \cite{edge} and \cite{babbage}). On the other hand, it will be shown later that the image of the exceptional divisor of $C_4$ via $\sigma$ is a degree twelve scroll in conics.

\quad

In what follows, we will first study properties of the map $\sigma$ which are common to all different choices of $C_4$. This includes the image of $p$ and the study of conics in $X$.

Next, we will study how the possible irreducible components of $C_4$ are blown up and mapped by $\sigma$.

After that, the singular locus of $X$ will be analysed. As it will be seen, it arises from the singularities of $C_4$. Finally, in Section \ref{ed_sum_sec}, the results for which $V$ is smooth will be collected.

\subsection{Some properties}

As noted before, we can suppose that $p$ is a general point in $V$. We can also suppose that $\l$ and $\l^\prime$ cut $C_4$ transversally. If $C_4$ has only reduced components, $\l$ and $\l^\prime$ intersect $C_4$ in two distinct points each.

\begin{lemma}\label{ed_imagem_p_e_retas}
The point $p$ is blown up and mapped to a smooth quadric surface $S_2^x$ through $x$ and the lines $\l$ and $\l^\prime$ are mapped respectively to $\l_x$ and $\l^\prime_x$, the lines in $S_2^x$ through $x$.
\end{lemma}
\begin{proof}
The image of $p$ is explained in Lemma \ref{ed_imagemquadricas}.

Blowing up  $\l$ gives an exceptional divisor $E_\l\cong \F_0$. A plane through $\l$ cuts $\X$ in $\l$ itself plus quartic curves. These cut $\l$ in one moving and three fixed points: two in $C_4$ (name them $q$ and $q^\prime$) and one in $p$. So:
\[ \X_p = M + F_q + F_{q^\prime} + F_p \equiv (4,1) \]
where $M\equiv (1,1)$ is the moving part of $\X_p$ and $F_q, F_{q^\prime}, F_p$ are fixed lines, the fibers over these points. The moving part has two base points, lying on $F_q$ and $F_{q^\prime}$, associated to the tangent directions of $C_4$ at $q$ and $q^\prime$. Using Lemma \ref{pr_imagemF1F2}, $M$ corresponds, in $\P^2$, to conics through four points, so $E_\l$ is mapped back to a line. The fixed parts have no moving intersection with $\X_p$, except for $F_p$, so they are contracted. 

Blowing up  $p$, and then  $\l$ and $\l^\prime$, we see that these lines are mapped to the lines in $S_2^x$ through $x$.

If $C_4$ has a non reduced component, it may happen that the two base points in $E_\l$ become infinitely near. However, the result will still be the same.
\\ \end{proof}

Also note that a general plane through $p$ is cut by $\X$ in quintics with five double points. The conic through these points is the intersection of $V$ with this plane. So it is mapped to a quintic (possibly weak) Del Pezzo surface through $x$ (this map is the inverse of the tangential projection from $x$). Remember that planes correspond to tangent hyperplane sections of $X$ in $x$. Planes through $p$ correspond to those reducible sections containing $S_2^x$.

\quad

Now I make a small remark about conics in $X$. Since through a general point $x^\prime \in X$ there is a quadric surface contained in $X$, there is a two-dimensional family of conics through $x^\prime$: the hyperplane sections of this quadric. The image via $\tau$ of such a conic is a conic in the quadric of $\mQ$ through $\tau(x^\prime)$, cutting $C_4$ in four points. 

 There is also an one-dimensional family of conics through $x^\prime$ parametrized by $C_4$.  Fixed a point in $C_4$, consider the plane containing this point, $p$ and $\tau(x^\prime)$. It cuts $C_4$ in three other points. Then there is a conic in this plane passing through these three points, $\tau(x^\prime)$ and $p$. It is mapped to a conic through $x^\prime$, since $\X$ has multiplicity two in $p$ and along $C_4$.

The conics through $x$ of the first family are mapped to the point $p$, since it is the image of $S_2^x$. The conics through $x$ of the second family are mapped to the points of $C_4$.

Obviously, any of these conics can be a pair of lines or a double line.

\subsection{The image of irreducible components of $C_4$}
\label{ed_componentes_C4}

We now study how the irreducible components of $C_4$ are blown up and mapped by $\sigma$. We'll suppose we are not in the general case (where $C_4$ is an elliptic quartic and $X=E_{3,7}$), so   either $C_4$ is an irreducible rational quartic or each component is a smooth rational curve. Then, after the desingularization, the exceptional divisor of the blow up at $C_4$ is a union of Hirzebruch surfaces.

Since $C_4$ is a $(2,2)$ curve in $V$, the possible irreducible components are: a rational quartic curve (nodal or cuspidal), a twisted cubic, a conic, a line, a double conic and a double line. Note that the singularities (and double components) of $C_4$ mean tangency conditions, and not singularities for surfaces in $\X$.

\subsubsection{A rational quartic}

If $C_4$ is a rational quartic, it has a double point $q$. It can be a nodal or a cuspidal point of $C_4$. If it is a nodal point, $\mQ$ has Segre symbol $[2,1,1]$, if it is cuspidal, the symbol is $[3,1]$. This can be proven with direct computations. In both cases, the special cone (corresponding to the numbers $2$ and $3$) of $\mQ$ has vertex $q$.

\begin{prop}\label{ed_ratquartic}
Suppose $C_4$ is a rational quartic curve. Then it is mapped to a scroll in conics $S_{12}^x$ of degree $12$ having multiplicity four in $x$, a double conic $C_{x_q}$ and two double lines. It cuts the quadric $S_2^x$ in the two double lines through $x$ and cuts the other quadrics of $\mQ^\prime$ in rational quartic curves with a double point. It is a cusp if and only if $C_4$ is cuspidal in $q$. These quartics are met by four lines in $S_{12}^x$, each cutting one of the two double lines through $x$. The conic $C_{x_q}$ is described by the singular points of these quartic curves.

  To distinguish both cases, denote by $^nS_{12}^x$ the image of a nodal quartic and by $^cS_{12}^x$ the image of a cuspidal quartic. A sketch of $^nS_{12}^x$ is presented in Figure \ref{ed_figura_S12}.  
\end{prop}
\begin{proof}
The blow ups are represented in Figure \ref{ed_figura_C4}.  

\begin{figure}
\centering
\includegraphics[trim=10 600 100 0,clip,width=1\textwidth]{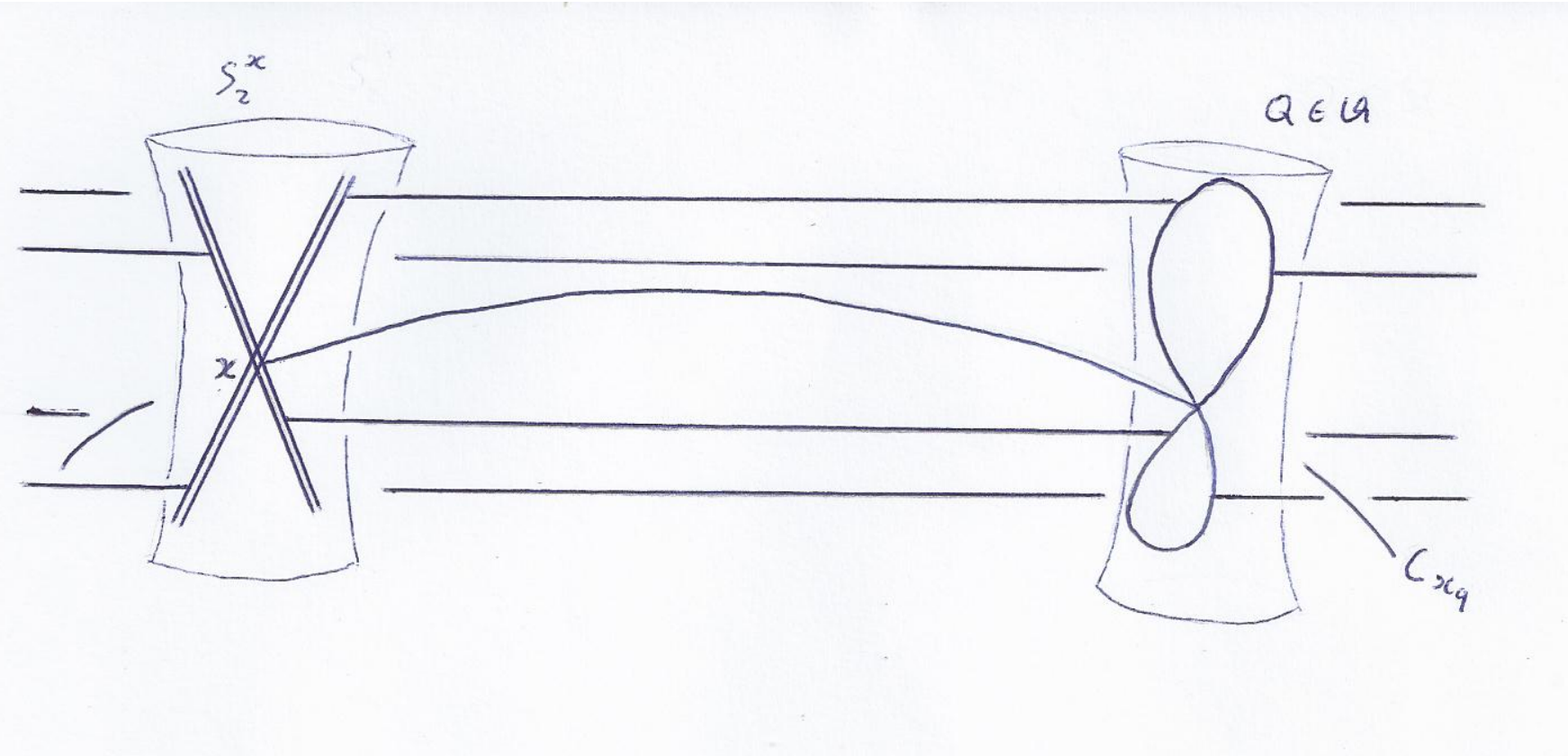}
\caption{The surface $^nS_{12}^x$}\label{ed_figura_S12}
\end{figure} 

\begin{figure}[tb]
\centering
\includegraphics[trim=10 5 240 470,clip,width=0.9\textwidth]{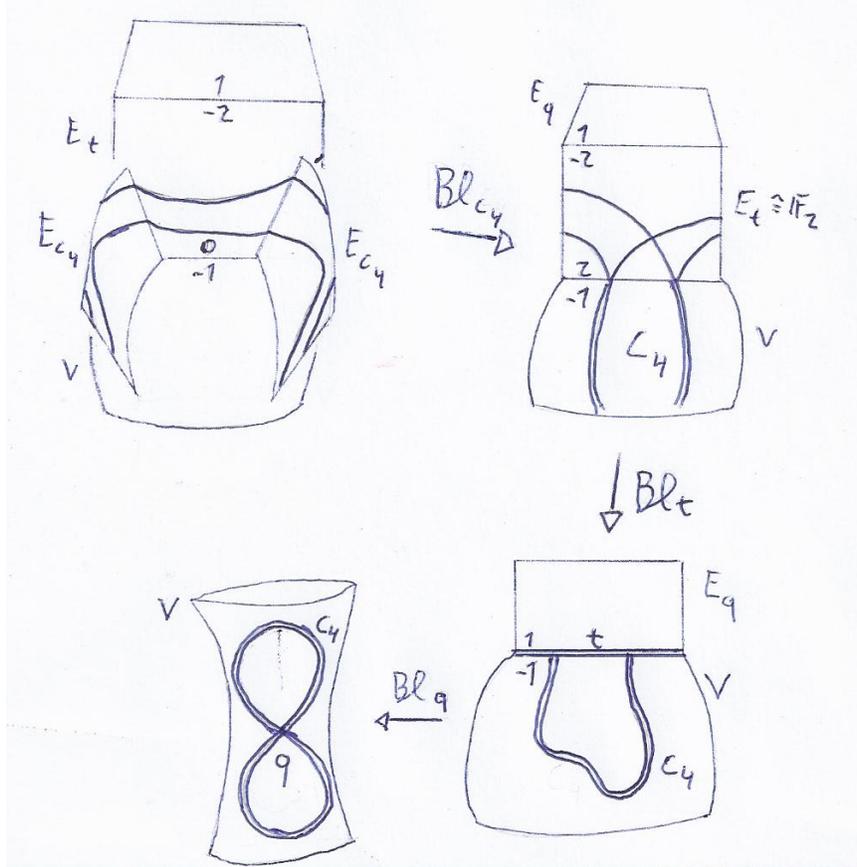}
\caption{$\X$ and the blow up at $C_4$ in the nodal case}\label{ed_figura_C4}
\end{figure} 

Before blowing up the curve itself, one must blow up $q$. $\X_q$ has two double points in $E_q\cong \P^2$, corresponding to the tangent cone of $C_4$ in $q$. These points are infinitely near if $q$ is a cusp. So $\X_q$ is a fixed double line connecting these two points, name it $t$.

The line $t$ has to be blown up too. It is the complete intersection of $E_q$ and $V$. In $E_q$, $t^2=1$; in $V$, $t^2=-1$, since it comes from the blow up of $V$ in $q$. Hence, the normal bundle of $t$ is:
\[ N_t = \mO_{\P^1}(1)\oplus \mO_{\P^1}(-1) \]

So blowing up  $t$ gives $E_t\cong \F_2$, where $V_t\equiv e_2+2f_2$ and $(E_q)_t\equiv e_2$. The linear system $\X_t$ must cut each fiber $f_2$ in two points, $E_q$ in none and $V$ in two fixed double points. So:
\[ \X_t \equiv  2e_2+4f_2 \]
with two double base points.  As showed in Lemma \ref{pr_imagemF1F2}, it birationally corresponds in $\P^2$ to degree four curves having  two double points, a third double point and a fourth double point infinitely near to this one. A quadratic standard transformation maps it to a system of conics with one double point, that is, pairs of lines through a point. Thus, $E_t$ is sent to a conic in $X$.

Note that, in the original linear system on $\F_2$, $\X_t$ has fixed intersection with $V_t$, so the conic passes through $x$. Moreover, the plane $E_q$ has no intersection with $\X$ at all, so it is mapped to a point $x_q\in X$ contained in this conic. Name the conic $C_{x_q}$.

Now let's find the normal bundle of the strict transform of $C_4$ after these two blow ups. $C_4$ is a complete intersection of two quadrics and its self-intersection in each of these quadrics is $8$. After blowing up  $q$, it decreases to $4$. Both quadrics contain $t$, since their tangent plane in $q$ coincide. Then, after the blow in $t$, $C_4$ is still a complete intersection. Its self intersection number continues to be $4$, since $t$ lied in the quadrics. Hence, after these blow ups, $C_4$ has normal bundle:
\[ N_{C_4} = \mO_{\P^1}(4)\oplus\mO_{\P^1}(4) \]

Now blow up  $C_4$, so $E_{C_4}\cong \F_0$. Lines of type $(0,1)$ correspond to quadrics containing $C_4$, so, in particular, $V_{C_4}\equiv (0,1)$.

The linear system $\X$ has multiplicity two in $C_4$ and cuts a quadric through it in $2C_4$ plus moving plane sections. So $\X_{C_4}$ is a $(4,2)$ curve.

The lines $\l$ and $\l^\prime$  cut $E_{C_4}$ in two points each, all of them lying on $V_{C_4}$. These are the base points of $\X_{C_4}$. So the image of $E_{C_4}$ by $\sigma$ has degree: 
\[ 4\cdot 2+2\cdot 4-2-2=12 \]
and it contains $x$ (the contraction of $V_{C_4}$), name it $S_{12}^x$. It is a scroll in conics, since it maps a general line of type $(1,0)$ to a conic through $x$.

To find the multiplicity of $S_{12}^x$ in $x$, note that $V_{C_4}$ has self-intersection $0$ in $E_{C_4}$ and $\X_{C_4}$ intersects it in four fixed simple points, corresponding to the lines $\l$ and $\l^\prime$. Then, after the blow up at the base points, $(V_{C_4})^2=-4$ and it is contracted to a point of multiplicity four of $S_{12}^x$.

To prove that $S_{12}^x$ is also singular along the conic $C_{x_q}$, note that:
\[ E_t\cap E_{C_4} \equiv  (1,0)+(1,0)\]
in $E_{C_4}$, that is, $E_t$ cuts $E_{C_4}$ in two fibers (which are infinitely near if $C_4$ is cuspidal in $q$), each is mapped to a conic. But since $E_t$ is mapped to $C_{x_q}$, both fibers are mapped to $C_{x_q}$. Therefore it is a double conic of $S_{12}^x$, since $\X_{C_4}$ restricted to $E_t\cap E_{C_4}$ defines a $2$-to-$1$ map.

A general quadric $Q$ of $\mQ$ intersects $E_{C_4}$ in a general line of type $(0,1)$, which is mapped to a rational quartic curve with a double point in $C_{x_q}$ (image of the intersection with $E_t\cap E_{C_4}$). This curve is the intersection of a quadric of $\mQ^\prime$, image of $Q$, with $S_{12}^x$. This double point is cuspidal if and only if the two intersections with $E_t\cap E_{C_4}$ are infinitely near, that is, if an only if $q$ is cuspidal. Since $E_t$ is mapped to $C_{x_q}$, this conic is described by the singular points of such quartic curves.

The four lines of type $(1,0)$ containing the base points of $\X_{C_4}$ are mapped to lines, cutting the quadrics of $\mQ^\prime$ in one point each.

Since $\l$ and $\l^\prime$ are mapped to the two lines through $x$, the four base points of $\X_{C_4}$ are blown up and mapped to these two lines. Each line comes from two points, so they are double lines of $S_{12}^x$ in $S_2^x$. This completes the proof.
\\ \end{proof}

\subsubsection{A twisted cubic}

\begin{prop}
A twisted cubic contained in $C_4$ is mapped to a degree nine scroll in conics $S_9^x$, having multiplicity three in $x$. It contains one double and one simple line through $x$, contained in $S_2^x$. It cuts the other quadrics of $\mQ^\prime$ in twisted cubics. Through a general point of $S_9^x$ there is a conic intersecting the twisted cubics and containing $x$. There are also three lines in $S_9^x$ meeting the twisted cubics.
\end{prop}
\begin{proof}
Let $C_3$ be the twisted cubic. Being a rational normal curve, its normal bundle is $\mO_{\P^1}(5)\oplus\mO_{\P^1}(5)$, so blowing up $C_3$ gives $E_{C_3}\cong \F_0$. 

Two quadrics containing $C_3$ intersect each other in $C_3$ plus a line cutting it in two points. So, if these quadrics intersect $E_{C_3}$ in curves of type $(a,b)$, it follows that:
\[ 2=(a,b)^2=2ab \]
Hence they are of type $(1,1)$.

If $Q$ is a quadric containing the twisted cubic, then:
\[ \X\cap Q\equiv 2(1,2)+(3,1) \equiv 2C_3+(3,1) \]
So the residual intersection of $\X$ with $Q$ cuts $C_3$ again in seven points.

Therefore $\X_{C_3}$ is of type $(5,2)$, since it cuts $(1,1)$-curves in seven points. It has three simple base points, corresponding to the intersections with $\l$ and $\l^\prime$, and two double points (possibly infinitely near), corresponding to the intersections with the other component (a line) of $C_4$. It intersects $V_{C_3}$ in these base points, so $V_{C_3}$ is contracted to $x$.

Then $\X_{C_3}$ maps $E_{C_3}$ to a surface of degree $10+10-3-4-4=9$ through $x$, call it $S_9^x$. It cuts a general quadric of $\mQ^\prime$ in twisted cubics, since the quadrics of $\mQ$ cut $E_{C_3}$ in curves of bidegree $(1,1)$ containing the two double points of $\X_{C_3}$. A general line of type $(1,0)$ is mapped to a conic through $x$ cutting each twisted cubic in one point.

The simple base points are mapped to the two lines through $x$ in $S_2^x$. Since those are three points, one of them is a double line of $S_9^x$. The three $(1,0)$-lines containing these points are mapped to lines meeting the twisted cubics in $\mQ^\prime$.

Since $V_{C_3}\equiv (1,1)$, it has self-intersection $2$ in $E_{C_3}$. On the other hand, $\X_{C_3}$ intersects this curve in its five base points. Then, after the blow up at the base points, $(V_{C_3})^2=2-5=-3$. Therefore $x$ is a triple point of $S_9^x$. 
\\ \end{proof}

\subsubsection{A conic}

\begin{prop}\label{ed_conic}
Let $C_2$ be a conic in $C_4$. Then it is mapped to a weak Del Pezzo surface of degree six $S_6^x$, having a double point in $x$. It contains the two lines through $x$ and cuts quadrics of $\mQ^\prime$ in conics. There are two lines in $S_6^x$ intersecting these conics, and each line intersects one of the lines through $x$. 

The $\P^6$ containing $S_6^x$ cuts $X$ in a tangent hyperplane section through $x$. It is the union of $S_6^x$ and the image of the plane $\Sigma$ containing $C_2$.
\end{prop}
\begin{proof}
As $C_2$ is a complete intersection, its normal bundle is:
\[ \mO_{\P^1}(2)\oplus\mO_{\P^1}(4) \]
so blowing up $C_2$ gives $E_{C_2}\cong \F_2$. 

If $\Sigma$ is the plane containing $C_2$, then $\Sigma_{C_2}\equiv e_2$. Quadrics containing $C_2$ intersect $E_{C_2}$ in sections $e_2+2f_2$.

$\X$ intersects $\Sigma$ in $2C_2$ plus moving lines, so $ \X_{C_2}\cdot e_2=2$ and $\X_{C_2}\cdot f_2=2$. 
Hence $\X_{C_2}\equiv 2e_2+6f_2$. 

Remember that $C_4$ is of type $(2,2)$ in $V$, so a conic cuts the other components in two (possibly infinitely near) points. And it cuts the lines through $p$ in one point each. Therefore the base points of $\X_{C_2}$ are two double and two simple points, all four in $V_{C_2}$. The double points can be infinitely near.  We can also compute the degree of the image surface:
\[ (2e_2+6f_2)^2-2\cdot 4-2 = -8+24-8-2 = 6\]
So name it $S_6^x$ ($V_{C_2}$ is contracted to $x$).

  Following Lemma \ref{pr_imagemF1F2}, $\X_{C_2}$ corresponds,  in $\P^2$, to degree six curves with  two double points (from the other components of $C_4$), two simple points (from $\l$ and $\l^\prime$), one point of multiplicity four and one double point infinitely near to this last point.

Note that $V_{C_2}$ corresponds to a conic with simple points in all six points above (so they are not in general position).

Performing two standard quadratic transformations, we map $V_{C_2}$ to a point and $\X_{C_2}$ to cubics with a base point, a second base point in the contraction of $V_{C_2}$ and a third base point infinitely near to it. The result is the same if the two double points are infinitely near or not.

Hence, the conic is mapped to a weak Del Pezzo surface of degree six $S_6^x$, having a double point in $x$.

Quadrics of $\mQ$ intersect $E_{C_2}$ in sections $e_2+2f_2$ containing the two double points, so they are mapped to conics in $\mQ^\prime$. These conics don't contain $x$, except for the quadric $S_2^x$, which intersects $S_6^x$ in the two lines through $x$ (images of the simple base points of $\X_{C_2}$).

The conics in $\mQ^\prime$ are intersected by conics through $x$ (image of the fibers $f_2$) and by two lines (image of the fibers through the simple base points).

  Note that the plane $\Sigma$ is mapped again to a plane. 
So, according to Lemma \ref{pr_X'eX''}, $\Sigma$ corresponds to the tangent hyperplane section made of $S_6^x$ and a plane. The plane intersects $S_6^x$ in a conic (the image of $e_2\subset E_{C_2}$). This conic is contained in a quadric of $\mQ^\prime$: the image of the reducible quadric of $\mQ$ containing $\Sigma$.  
\\ \end{proof}

\subsubsection{A line}

\begin{prop}\label{ed_line}
A line contained in $C_4$ is mapped to a scroll $S(1,2)$ through $x$, denoted $S_3^x$. The lines of its ruling are the intersections with quadrics of $\mQ^\prime$, including a line through $x$. Together with quartic scrolls, it forms tangent hyperplane sections of $X$ at $x$.

This cubic scroll is represented in Figure \ref{ed_figura_S3}.
\end{prop}
\begin{proof}
Let $r$ be this line. Blowing up $r$ gives $E_r\cong \F_0$, and lines of type $(0,1)$ correspond to intersections with planes through $r$.

  A general plane containing $r$ is cut by $\X$ in $r$ with multiplicity two and cubics with free intersections with $r$.  So $\X_r$ is of type $(3,2)$. It has two (possibly infinitely near) double points in the intersections with  the other components of $C_4$. It also has one base point in the intersection with one of the lines through $p$. These three base points lie on the curve $V_r$ of type $(1,1)$, which is contracted to $x$.

  In view of Lemma \ref{pr_imagemF1F2},   such linear system can be birationally mapped to $\P^2$ as quintics with one triple, three double and one simple points. With two standard quadratic transformations, we get a system of conics through one point. So $r$ is mapped to a cubic scroll $S(1,2)$, name it $S_3^x$. The result is the same if the two double points are infinitely near.

Looking directly in $E_r$, we see that the directrix line of $S_3^x$ comes from the fibre through the simple base point of $\X_r$, while the blow up at this point is mapped to one of the lines through $x$. The other lines of the ruling of $S_3^x$ come from $(1,1)$ curves through the double points, which come from the quadrics in $\mQ$.

A general plane containing $r$ is cut by $\X$ in $r$ with multiplicity two and cubics with one double point (at the intersection with the other component of $C_4$ - a cubic) and one simple point (at the intersection with one of the lines through $p$), none of them in $r$. These cubics define the inverse of the tangential projection of a scroll $S(2,2)$. So such a plane corresponds to a tangent section of $X$ made of the cubic scroll $S_3^x$ and a quartic scroll, both through $x$. They intersect each other in a twisted cubic, the image of a general $(0,1)$-curve in $E_r$. Note that special planes through $r$ may produce scrolls of type $S(1,3)$.
\\ \end{proof}

\begin{figure}
\centering
\includegraphics[trim=10 590 50 0,clip,width=1\textwidth]{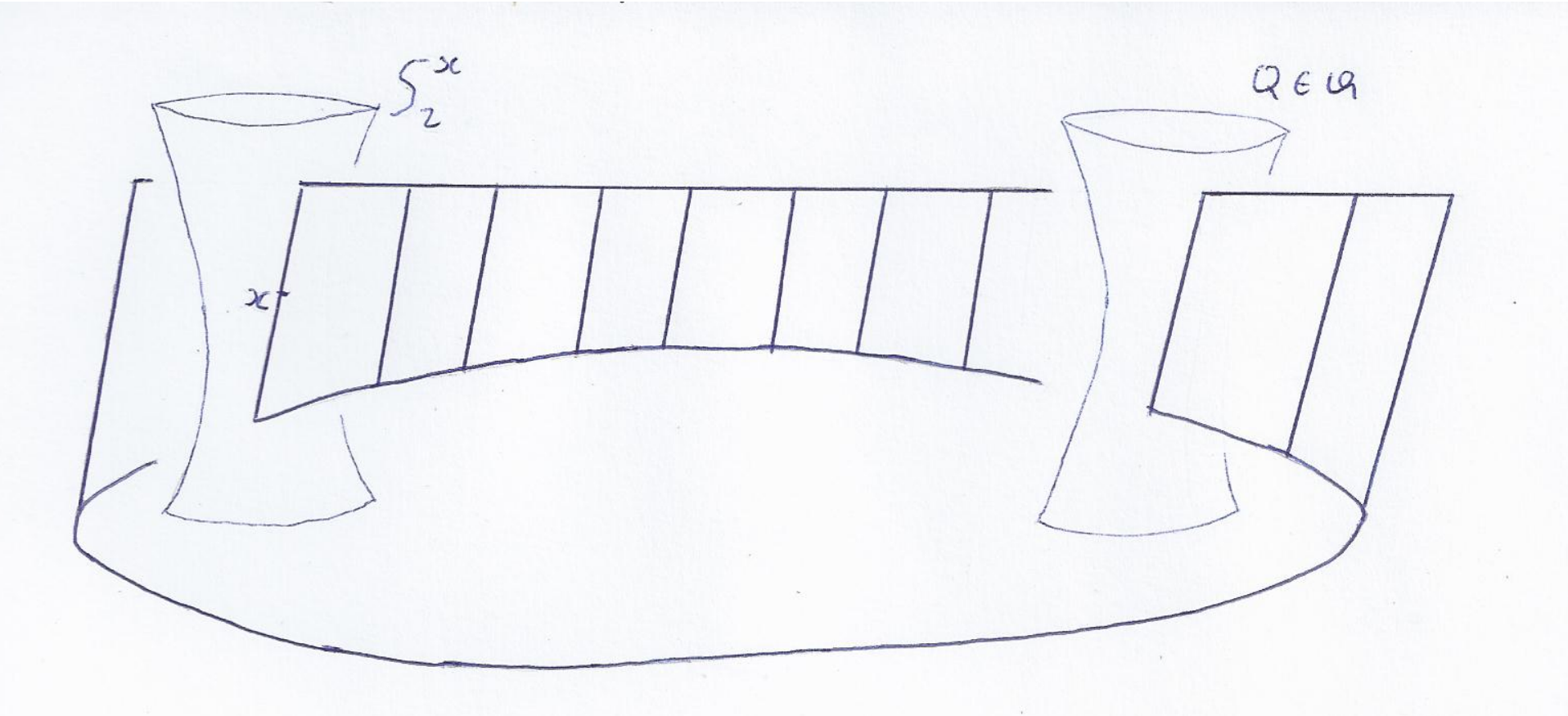}
\caption{The surface $S_3^x$}\label{ed_figura_S3}
\end{figure}

\subsubsection{A double conic}

\begin{prop}\label{ed_doubleconic}
Let $C_2$ be a double conic in $C_4$, so that $\X$ has multiplicity two in a second conic $t$ infinitely near to $C_2$. Then $C_2$ is mapped to a conic $C_2^\prime$ and $t$ is mapped to a degree six weak Del Pezzo surface $S_6^{x}$, singular at $x$. It contains $C_2^\prime$, being its intersection with the double plane (image of $\Sigma$) of $\mQ^\prime$.
\end{prop}
\begin{proof}
Let $\Sigma$ the plane containing $C_2$. In this situation, the Segre symbol of $\mQ$ is $[(111),1]$. This means that there are two singular quadrics in $\mQ$: a cone and the double plane $\Sigma$.

The linear system $\X$ has multiplicity two in this conic and its tangent cone at a general point of $C_2$ is a double plane: the tangent plane of $V$ at this point.

As before, blowing up $C_2$ gives $E_{C_2}\cong\F_2$, with $e_2$ being the intersection with the plane $\Sigma$. Again, $t=V_{C_2}$ is a section of type $e_2+2f_2$ in $E_{C_2}$, and $\X_{C_2}$ has type $2e_2+6f_2$. But now, the tangency condition implies that  in each fibre of $E_{C_2}$, $\X_{C_2}$ has a double point at the intersection of $t$ with this fibre. Hence:
\[ \X_{C_2} \equiv  2e_2+6f_2 \equiv 2t + \{ 2f_2 \} \] 
So it maps $E_{C_2}$ to a conic $C_2^\prime\subset X$. Note that $t$ is the other component of $C_4$, but now it is infinitely near to $C_2$.

Now blow up $t$. It's a complete intersection of $V$ and $E_{C_2}$. In $V$, $t^2=(C_2)^2=2$; in $E_{C_2}$, $t^2=(e_2+2f_2)^2=2$. Hence its normal bundle is:
\[ N_t=\mO_{\P^1}(2)\oplus \mO_{\P^1}(2) \]
and $E_t\cong \F_0$.

Both $E_{C_2}$ and $V$ intersect $E_t$ in $(0,1)$ lines. Since the moving part of $\X_{C_2}$ intersected $t$ in two points, it follows that $\X_t$ is made of curves of type $(2,2)$. It has two simple base points at the intersections with $\l$ and $\l^\prime$. 

  In Figure \ref{ed_figura_C2} a sketch of the blow up at $C_2$ is given.  

\begin{figure}
\centering
\includegraphics[trim=17 4 40 670,clip,width=1\textwidth]{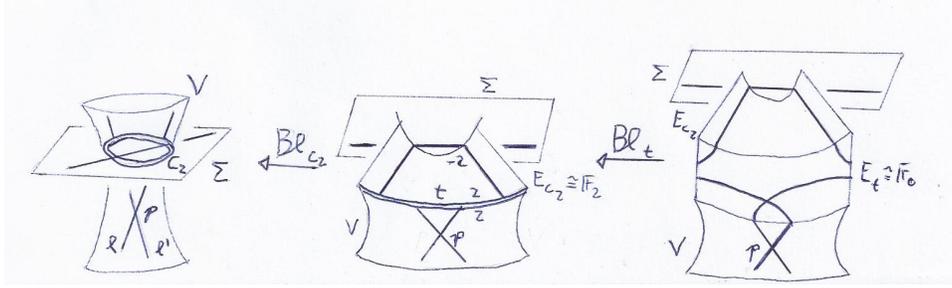}
\caption{Blow up at $C_2$}\label{ed_figura_C2}
\end{figure} 

$\X_t$ is birationally equivalent to quartic curves in $\P^2$ with two double points and two simple points. Here, $V_t$ is a line through one of the double points and both simple points. With a quadratic standard transformation, one maps $\X_t$ to plane cubic with two base points and a third one infinitely near to one of them (the image of $V_t$).

This is the same linear system found in Proposition \ref{ed_conic}. Therefore $E_t$ is mapped to $S_6^x$, a degree six weak Del Pezzo surface with a double point in $x$.

Since $\Sigma_{C_2}\equiv e_2$, this plane has no intersection with $E_t$. With the contraction of $E_{C_2}$ to $C_2^\prime$, the image of $\Sigma$ intersects $S_6^x$ in $C_2^\prime$.
\\ \end{proof}
\quad

In Section \ref{ed_singularities_sec} it will be proven that $C_2^\prime$ is a double conic of $X$.

\subsubsection{A double line}

\begin{prop}\label{ed_doubleline}

 Let $r$ be a double line of $C_4$, so $\X$ has multiplicity two in a second line $t$ infinitely near to $r$. Then $t$ is mapped to the surface $S_3^{x}\cong S(1,2)$, defined in Proposition \ref{ed_line}. The image of $r$ is a line $r^\prime$ of the ruling of $S_3^x$. 
\end{prop}
\begin{proof}
As in the previous case, the tangent cone of $\X$ at a general point of $r$ is the tangent plane of $V$ at this point, with multiplicity two. 

If $r\equiv (1,0)$ in $V$, the other components of $C_4$ are a $(1,0)$-line infinitely near to $r$ and two other lines of type $(0,1)$.

As before, blowing up $r$ gives $E_r\cong\F_0$. $\X_r$ has degree $(3,2)$ and has two double points (which can be infinitely near) and one simple base point. But in each line $(1,0)$ the system $\X_r$ has a double point in the intersection of this line with $t=V_r\equiv (1,1)$. So:
\[ \X_r\equiv (3,2)\equiv 2t+(1,0) \]
and $r$ is mapped back to a line $r^\prime\subset X$.   Note that $t$ is a line of $V$ infinitely near to $r$.  

Since $t$ is a complete intersection of $V$ and $E_r$, and since $t^2=0$ in $V$ and $t^2=2$ in $E_r$, its normal bundle is:
\[ N_t=\mO_{\P^1}(0)\oplus\mO_{\P^1}(2) \] 
so blowing up $t$ gives $E_t\cong\F_2$. The system $\X_t$ cuts $E_r\cap E_t\equiv e_2$ in one moving point and each fiber in two points. Hence:
\[\X_t\equiv 2e_2+5f_2\]

It has two double points (in the intersection with the other components of $C_4$) and one simple point (in the intersection with a line through $p$). All of them lie on $V_t\equiv e_2+2f_2$.

  The blow up at $r$ is represented in Figure \ref{ed_figura_r}.  

\begin{figure}[tb]
\centering
\includegraphics[trim=10 30 170 450,clip,width=0.9\textwidth]{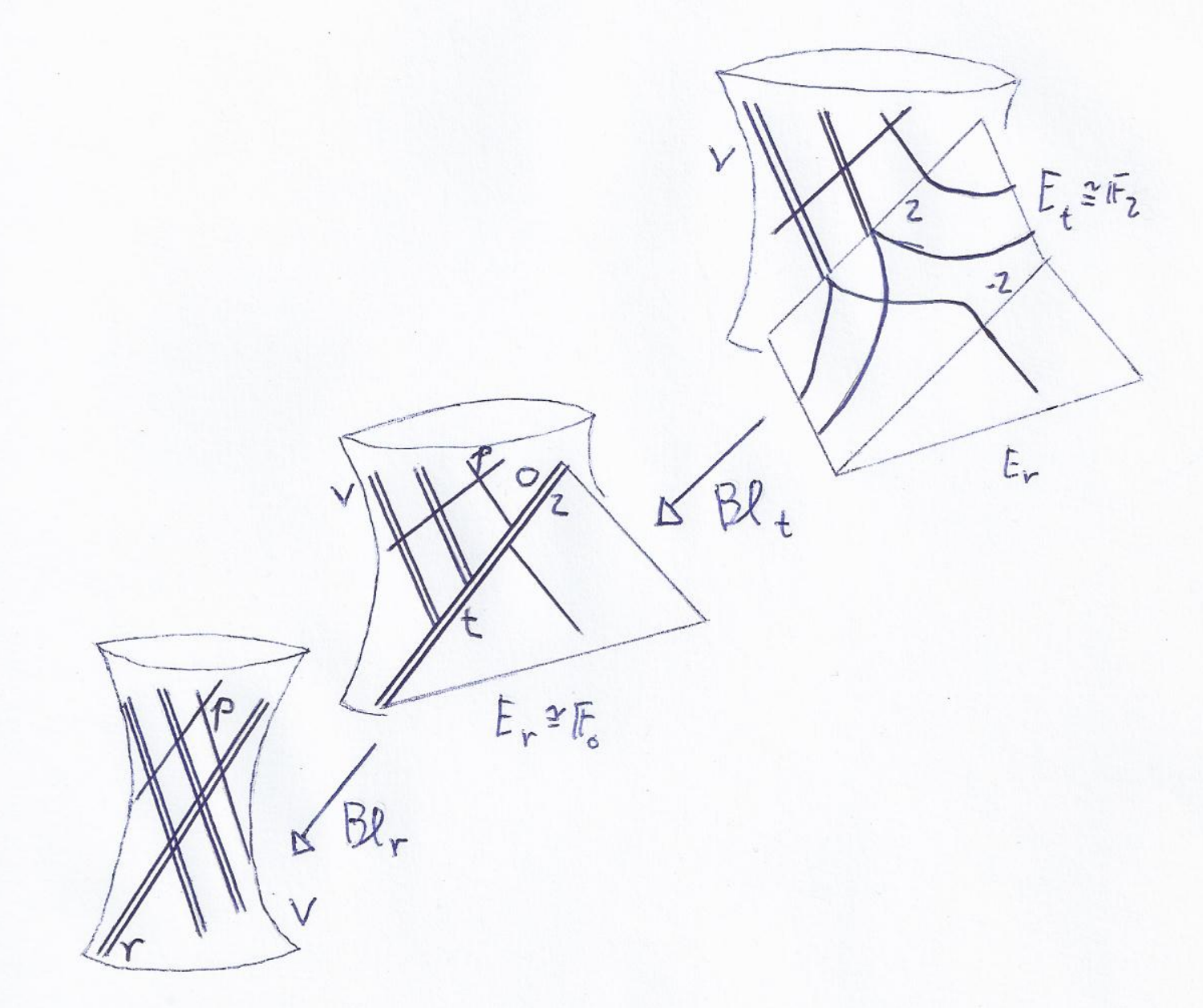}
\caption{Blow up at $r$}\label{ed_figura_r}
\end{figure} 

Using Lemma \ref{pr_imagemF1F2}, $\X_t$ is mapped, in $\P^2$, to curves of degree five with  two double points, one simple point, one point with multiplicity three and one double point infinitely near to this point.

Moreover, $V_t\equiv e_2+2f_2$ is mapped to a conic through all those five points. After two standard quadratic transformations, $\X_t$ is a system of conics with one base point, while $V_t$ is another point in $\P^2$.  Then $E_t$ is mapped to a cubic scroll $S_3^x$ through $x$.

In $E_t$, the sections of type $e_2+2f_2$ containing the two double points are mapped to lines of the ruling of $S_3^x$. This includes the reducible section consisting of the union of $e_2=E_r\cap E_t$ with the two fibers through the double points. In particular, $r^\prime$ is a line of the ruling of $S_3^x$.
\\ \end{proof}

\subsection{Singularities of $X$}
\label{ed_singularities_sec}

  In this section we will study the singularities of $X$.   The map $\sigma$ is an isomorphism outside $V$, which is contracted to $x$, so these singular points lie on the images of exceptional divisors. By Lemma \ref{ed_imagem_p_e_retas}, only $C_4$ can produce singularities.  Moreover:

\begin{lemma}\label{ed_singularidades_vem_de_C4}
A singularity of $X$ is the image of a singular point of $C_4$. This includes points in non reduced components of $C_4$. 
\end{lemma}
\begin{proof}
Since $x$ is a general point in $X$, Terracini's Lemma implies that an isolated singular point does not lie on $T_xX$. 
Then it is projected by $\tau$ to a point $q$. As just noted above, $q\in C_4$. 

But singular points of $X$ are contractions, via $\sigma$, of surfaces (or curves) which have fixed intersection with the strict transform of $\X$ after the blow up at its base locus. 

If $q$ is a smooth point in $C_4$, the tangent cone of $\X$ at $q$ is a pair of moving planes containing the tangent line of $C_4$ in $q$. So after the blow up at $q$, $\sigma$ maps $E_q$ to a conic. Therefore $q$ is not the projection of an isolated singular point of $X$.

Then $q$ is a singular point of $C_4$. 

On the other hand, if the singularity lies on a singular curve of $X$, this curve is the contraction of an exceptional divisor of the blow of $\X$ in its base locus. By Lemma \ref{ed_imagem_p_e_retas}, this exceptional divisor comes from $C_4$. And by the analysis made in Section \ref{ed_componentes_C4} the associated component of $C_4$ is either a double line or a double conic. In both cases, a point in a singular curve of $X$ is mapped to a singular point of $C_4$.
\\ \end{proof}

Note that if $q$ is the projection of a singular point of $X$, the tangent cone of $\X$ in $q$ is a fixed double plane: the tangent plane of $V$ at $q$. After the blow up at $q$, $\X$ cuts the exceptional plane $E_q$ in a fixed double line, so this plane is re-contracted to a point $x_q\in X$. 

By the above Lemma and Proposition \ref{ed_singsegre}, it follows:
\begin{coro}\label{ed_singsegreX}
Each singularity of $X$ is also a singularity of a quadric of $\mQ^\prime$ corresponding to a root of multiplicity greater than $1$ in the Segre symbol.
\end{coro}

The multiplicity of $X$ in $x_q$ can be easily computed using Lemma \ref{pr_truque_secao_tgente}.

\begin{lemma}\label{ed_singularidades_sao_duplas}
Any singular point of $X$ is of multiplicity two.
\end{lemma}
\begin{proof}
Let $x_q$ be a singular point of $X$. Then by Lemma \ref{ed_singularidades_vem_de_C4}, $x_q$ is the image via $\sigma$ of a singular point $q$ of $C_4$. More precisely, after the blow up at $q$, $\X$ intersects $E_q$ in a fixed double line, namely $V_q$. The plane $E_q$ is then contracted to $x_q$. Conversely, the preimage of $x_q$ via $\sigma$ is the point $q$ alone.

Let $\Omega$ be a general plane through $q$. Then $\X\cap\Omega$ has multiplicity two in $q$. Blowing up $q$, $\X\cap \Omega$ intersects the exceptional curve $e_q=E_q\cap\Omega$ in a fixed double point. Blowing up this point, we get  $(e_q)^2=-2$ and $e_q$ has no intersection with $\X\cap\Omega$. Then it is contracted to a double point, namely $x_q$, of $\sigma(\Omega)$. By Lemma \ref{pr_truque_secao_tgente}, $x_q$ is a double point of $X$.
\\ \end{proof}

In order to understand better the singularity of $X$  at $x_q$, we will use the fact that the blow up of $\P^3$ along a singular curve is a singular threefold. Then one can understand the singularities of $X$ by considering the singularities of the blow up of $\P^3$ along $C_4$. 

By studying the different possible singular points $q$ of $C_4$, we blow up $\C^3$ along a curve having the same type of singularity in the origin. Then this blow up gives a singular threefold that represents locally the singularity of $X$ in $x_q$.

We will now consider the possible singularities of $C_4$.

\paragraph*{A transversal intersection of two simple branches of $C_4$:}
In this case, $q$ is a nodal point of $C_4$. Then one needs to analyse the blow up of $\C^3$ along a curve with a nodal point in the origin. For instance, consider the curve $xy=z=0$. The blown up threefold is the subvariety of $\C^3\times \P^1$  with coordinates $(x,y,z),(u:v)$ given by the equation $uz-xyv=0$. In one of the affine charts $\C^4=\C^3\times \C$, it is given by $z-xyv=0$ and it is smooth. In the other affine chart, its equation is $uz-xy=0$, which is a rank four quadric cone in $\C^4$. Therefore this is the local equation of the singularity $x_q\in X$.

\paragraph*{A cuspidal point:}
This case can be explained by the blow up of $\C^3$ along the curve $z=x^3-y^2=0$. This produces the threefold $uz-v(x^3-y^2)=0$ in $\C^3\times\P^1$. Its singularity in $(0,0,0),(0:1)$ is given locally by $uz-x^3+y^2=0$. Then the tangent cone is $uz+y^2=0$, which has rank 3.

Blowing up $uz-x^3+y^2=0$ in the origin, one sees that there is no singular point infinitely near to it.

\paragraph*{A contact of two branches of $C_4$:}
A local representation of this singularity is given by the curve $z=x(x-y^2)=0$. The blow up of $\C^3$ along this curve produces the threefold $uz-vx(x-y^2)=0$ having a double point in $(0,0,0),(0:1)$ given locally by $uz-x^2+xy^2=0$. The tangent cone is $uz-x^2=0$, which also has rank 3.
 
Blowing up the origin, one of the affine charts is given by:
\[ x=\mu y \qquad z=\nu y \qquad u=\rho y \]
In this chart, the strict transform of $uz-x^2+xy^2=0$ is $\nu\rho-\mu^2+\mu y=0$, which has a double point at the origin. In the other charts there are no further singular points.

Therefore, $x_q$ is a double point with a second double point infinitely near to it.

\paragraph*{A general point of a double curve:} 
This singularity can be explained by the blow up of $\C^3$ along the curve $z=x^2=0$, which gives $uz-x^2v=0$. The double point is given locally by $uz-x^2=0$. This is also the equation of the tangent cone, which has rank 3, that is, it has a double line.

As  remarked in Proposition \ref{ed_doubleconic} and in Proposition \ref{ed_doubleline}, a double line of $C_4$ is mapped to a line in $X$, and a double conic is mapped to a conic in $X$. Then the double point actually moves in a double curve of $X$.

\paragraph*{A transversal intersection of three branches of $C_4$:} 
Locally, this singular point is given by $z=xy(x+y)=0$. Blowing up $\C^3$ along this curve gives a threefold with a double point locally given by $uz-xy(x+y)=0$. The tangent cone is $uz=0$, that is, a pair of three-dimensional planes.

Now, blow up the origin and consider the affine chart given by:
\[ y=\eta x \qquad z=\nu x \qquad u=\rho x \]
The strict transform of $uz-xy(x+y)=0$ is $\rho\nu-x\eta-x\eta^2=0$. It has double points in $(x,\eta,\nu,\rho)=(0,0,0,0)$ and $(0,1,0,0)$. There is a third double point in the origin of the chart given by:
\[ x=\mu y \qquad z=\nu y \qquad u=\rho y \]

The three double points are collinear. They lie on the intersection of the exceptional divisor with the strict transform of the plane given by $z=u=0$.

Therefore $X$ has a double point in $x_q$ with three collinear double points infinitely near to it.

\paragraph*{A transversal intersection of a double and a simple branch of $C_4$:} This singularity can be locally given by $z=x^2y=0$. Blowing up $\C^3$ along this curve, gives $uz-x^2y=0$, which has a double line through the origin. Moreover, blowing up the origin, in the chart given by:
\[ y=\eta x \qquad z=\nu x \qquad u=\rho x \]
its strict transform is $\nu\rho-\eta x=0$. It has a double point in the origin. This point does not lie on the double line, which cannot be seen in this chart.

Therefore, $X$ has a double line through $x_q$ and a double point infinitely near to $x_q$.

\paragraph*{A transversal intersection of two double branches of $C_4$:} Note that this is precisely the case in which $C_4$ is two double lines. As already noted, a double line of $C_4$ is mapped to a double line of $X$. Then $x_q$ lies on the intersection of two double lines of $X$. 

This can also be verified blowing up $z=x^2y^2=0$, giving $uz-x^2y^2=0$. This threefold has two double lines intersecting in the origin. The tangent cone is a pair of three-dimensional planes.

\quad

The main results of this section are now collected:

\begin{prop}\label{ed_singularidades}
A singular point of $C_4$ is mapped to a double point of $X$. This includes double components of $C_4$, which are mapped to double curves of $X$. More specifically let $q$ be a singular point of $C_4$ and $x_q$ its image. Then the tangent cone of $X$ in $x_q$ has:
\begin{itemize}
\item rank $4$,  if $q$ is the transversal intersection of two simple branches of $C_4$;
\item rank $3$,  if $q$ is a cuspidal point,  a point of contact of two simple branches or a general point of a double component of $C_4$;
\item rank $2$, that is, it is a pair of three-dimensional planes, if $q$ is the intersection of three simple branches, the intersection of a double and a simple branch or the intersection of two double branches of $C_4$.\end{itemize}

Moreover, $x_q$ is a singular point of a quadric of $\mQ^\prime$ corresponding to a root of multiplicity greater than one in the Segre symbol. 
\end{prop}

\subsection{Summing up}
\label{ed_sum_sec}

Now that we have enough information, we can easily describe the possible varieties $X$. We just consider different possibilities for a curve $C_4$ of type $(2,2)$ in $V$, which are:
\begin{itemize}
\item [\textbf{(E1)}] a smooth elliptic quartic (general case) -- $(2,2)$
\item [\textbf{(E2)}] an irreducible rational nodal quartic -- $(2,2)$
\item [\textbf{(E3)}]an irreducible rational cuspidal quartic -- $(2,2)$
\item [\textbf{(E4)}]a twisted cubic and a transversal line -- $(2,1)+(0,1)$
\item [\textbf{(E5)}]a twisted cubic and a tangent line -- $(2,1)+(0,1)$
\item [\textbf{(E6)}]two transversal conics -- $(1,1)+(1,1)$
\item [\textbf{(E7)}]two tangent conics -- $(1,1)+(1,1)$
\item [\textbf{(E8)}]a conic and two lines intersecting it in two points -- $(1,1)+(1,0)+(0,1)$
\item [\textbf{(E9)}]a conic and two lines intersecting it in one point -- $(1,1)+(1,0)+(0,1)$
\item [\textbf{(E10)}]four lines -- $(1,0)+(1,0)+(0,1)+(0,1)$
\item [\textbf{(E11)}]one double and two simple lines -- $2(1,0)+(0,1)+(0,1)$
\item [\textbf{(E12)}]two double lines -- $2(1,0)+2(0,1)$
\item [\textbf{(E13)}]one double conic -- $2(1,1)$
\end{itemize}

  We will refer to the resulting Bronowski variety by the string on the left in the above list. This leads to the result:

\begin{theorem}
Let $X$ be a Bronowski threefold satisfying hypothesis (H) and with fundamental surface being a smooth quadric. Then $X$ corresponds to one of the $13$ quartic curves listed above.
\end{theorem}

Note that in all cases, $C_4$ is mapped to a (possibly reducible) scroll in conics with degree $12$ and multiplicity four in $x$ (if there is a double curve in $C_4$ its image is counted with multiplicity two). This scroll cuts each quadric of $\mQ^\prime$ in curves of the same type of $C_4$. The only exception is the quadric through $x$, $S_2^x$, which is cut in two lines, which are double lines of the degree $12$ scroll. Since the singular points of $X$ come from $C_4$, the scroll contains these points. 

A description of these threefolds is given in Theorem \ref{ed_the_classification}, together with those obtained in the case in which $V$ is a cone.

\quad

\section{The singular case}

Suppose now that $V$ is a quadric cone with vertex $q$.   According to Lemma \ref{ed_baselocus}, either $\mult_q\X=4$ and all quadrics of $\mQ$ are singular in $q$ or $\mult_q\X=3$ and there is a quadric in $\mQ$  which is a smooth in this point.   In both cases, we have: 
\[ \X\cap V=2C_4+2\l \]
where $C_4$ is the base locus of $\mQ$ and $\l$ is the line joining $p$ and $q$. Note that $\l$ is not a double curve of $\X$. The multiplicity $2$ in  the intersection is due to a tangency of $\X$ and $V$, that is, there is a line $\l^\prime$ infinitely near to $\l$ in the base locus of $\X$.

  As it was remarked in Lemma \ref{ed_imagemquadricas},  the point $p$ is mapped to a quadric cone $S_2^x$ through the smooth point $x$. It belongs to the family $\mQ^\prime$. 

Made these considerations, we analyse the two possibilities.

\subsection{When the multiplicity in $q$ is three}

If $\mult_q \X=3$, there are two possibilities. Either $\mQ$ has a smooth quadric or it is a pencil of cones with moving vertex.

In the first case, the general quadric in $\mQ^\prime$ is smooth, so this family has at most four cones. Then it seems that $x$ is not a general point of $X$, since it lies on one of the four cones of $\mQ^\prime$. In fact:

\begin{lemma}\label{ed_cone_nao_geral}
Suppose $V$ is a quadric cone with vertex $q$ and suppose there is a smooth quadric in $\mQ$. Then $X$ is projectively equivalent to one of the threefolds described in  Section \ref{ed_smooth_sec}.
\end{lemma}
\begin{proof}
The idea is to change coordinates to fit this situation in the smooth case of Section \ref{ed_smooth_sec}.  This is done using the Cremona transformation of Section \ref{ed_jonqui_sec}. Choose a point $\bar{p}$ in $\P^3$ such that the quadric of $\mQ$ containing it is smooth. Apply the Cremona transformation $f$ defined in Section \ref{ed_jonqui_sec} and then apply the inverse transformation associated to the point $f(\bar{p})$ (instead of using the point $p^\prime=f(V)$).

The image of $\X$ by the composition of these two maps is a similar linear system having fixed intersection with a smooth quadric, instead of a cone. The quartic curve in its base locus is of the same type of the original $C_4$.
\\ \end{proof}

As remarked in \cite[p.305]{hodge}, in the second case there is only one possibility, in which $\mQ$ is a pencil of cones tangent to a plane $\Pi$ along a fixed line $r$. The base locus of $\mQ$ is $C_4=2r+C$, where $C$ is a conic intersecting $r$ in one point. This implies that the only reducible quadric in $\mQ$ is the union of $\Pi$ and the plane $\Sigma$ containing $C$. 

Note that, since $\X$ intersects a cone $Q$ of $\mQ$ different from $V$ in the union of $C_4$ and moving plane sections, such quadric is mapped back to a cone, and the vertex of this cone is the image of the vertex of $Q$.

This configuration produces a new Bronowski threefold, which we identify as (E14):

\begin{prop}\label{ed_vertice_movel_prop}
Suppose $\mQ$ is a pencil of  cones with moving vertex and let $q$ be the vertex of $V$. Then the base locus of $\mQ$ is $2r+C$, where $r$ is a line and $C$ is a conic, and the multiplicity of $\X$ in $q$ is three.

The threefold $X$ has a double conic $L$, the image of $r$. This conic is described by the vertices of cones of the one-dimensional family $\mQ^\prime$ contained in $X$. There is exactly one reducible quadric in this family, which intersects $L$ in one point.

The line $r^\prime$  in the base locus of $\mQ$ which is infinitely near to $r$ is mapped to a cubic scroll. The conic $C$ is mapped to a weak Del Pezzo surface of degree six having a double point in $x$.
\end{prop}
\begin{proof}
The first part has already been explained.

By Lemma \ref{ed_baselocus}, $\X$ has multiplicity two in $p\notin r$, $C$, $r$ and in a line $r^\prime$ infinitely near to $r$, corresponding to the plane $\Pi$. Then:
\[ \X\cap\Pi=4r+\{\text{lines}\} \] 
The linear system also contains the line $\l$ through $p$ and $q$ and a line $\l^\prime$ infinitely near to $\l$.

The blow ups that now follow are represented in Figure \ref{ed_figura_vertice_movel}.

\begin{figure}[tb]
\centering
\includegraphics[trim=2 25 65 380,clip,width=1\textwidth]{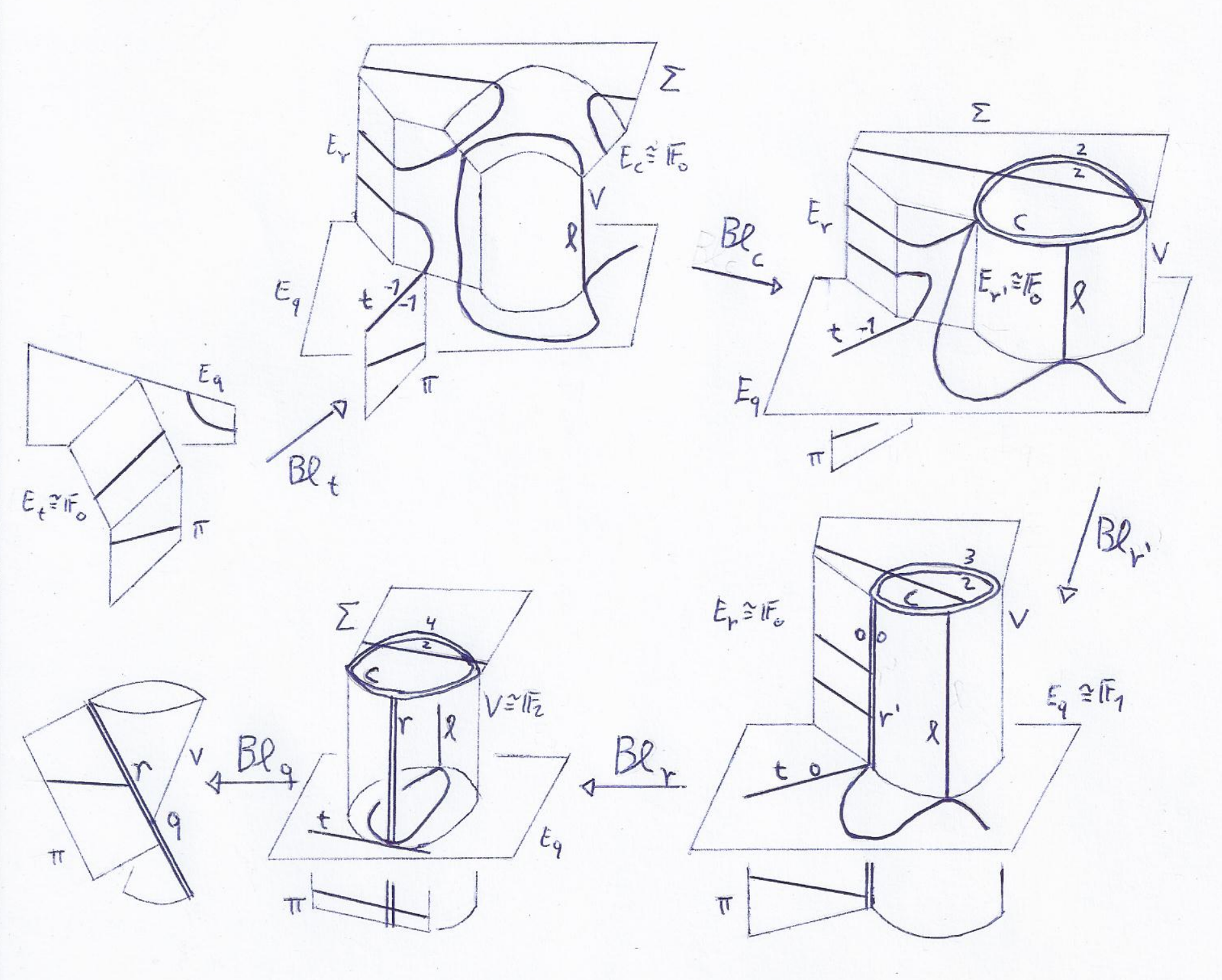}
\caption{$\X$ and the blow ups in Proposition \ref{ed_vertice_movel_prop}}\label{ed_figura_vertice_movel}
\end{figure} 

Start blowing up $q$. The degree three curve $\X_q$ has a double point in $r_q$ and a second double point in ${r^\prime}_q$ infinitely near to it. It also has a simple point in $\l_q$ and a second simple point in ${\l^\prime}_q$ infinitely near to it. Therefore, the line $t=\Pi_q$ through the two double points is a fixed component of $\X_q$. It is the tangent line to the conic $V_q$ in $r_q$. Hence:
\[ \X_q=t+\{\text{conics}\} \] 
where the moving conics have four simple base points, that is, they are tangent to $V_q$ in $\l_q$ and $r_q$. This moving part is equivalent, by a standard quadratic transformation in $E_q$, to lines through a point. Then $E_q$ is mapped to a line. Blowing up $\l$ and $\l^\prime$, it follows that the image of $E_q$ is the line through $x$.

Let $\Pi^\prime$ be a general plane containing $r$, so that $r=\Pi\cap\Pi^\prime$. After the blow up at $q$, we have that $r^2=0$ in both planes. Then its normal bundle is:
\[ N_r = \mO_{\P^1}(0)\oplus \mO_{\P^1}(0) \]

Now blow up $r$. Before the blow ups, the intersection  $\X\cap \Pi^\prime$ consisted of $r$ with multiplicity two and cubic curves. These curves intersected $r$ in $q$ and two moving points. Then, after the blow ups at $q$ and $r$, $\X_r\equiv (2,2)$. But in $E_r\cong \F_0$, the line $r^\prime=V_r$ is a double curve of $\X_r$.  Therefore:
\[ \X_r \equiv (2,2) = 2r^\prime+(2,0) \]
The moving part consists of pairs of fibers over points of $r$, mapping $E_r$ to a conic $L$. Since $r$ is the locus of vertices of cones in $\mQ$, $L$ is the locus of vertices of cones in $\mQ^\prime$. 

In $E_q$, both $t$ and the moving conics intersect $E_r$ in $r^\prime$.

The line $r^\prime=E_r\cap V$ has self-intersection $0$ in $E_r$ and $0$ in $V\cong \F_2$. Then its normal bundle is:
\[ N_{r^\prime} = \mO_{\P^1}(0)\oplus \mO_{\P^1}(0) \]

Blow up $r^\prime$, giving again $\X_{r^\prime}\equiv (2,2)$. This linear system has a double point in $C_{r^\prime}$, since $C$ intersected $r$ in one point and since $C$ is tangent to $\Pi$. It also has a simple point in $t_{r^\prime}$. By Lemma \ref{pr_imagemF1F2}, $\X_{r^\prime}$ corresponds to a linear system of quartic curves in $\P^2$ with three double points and one simple point. After a quadratic standard transformation, we get conics with one base point. Therefore $E_{r^\prime}$ is mapped to a cubic scroll.

As it was already noted before, there is only one reducible quadric in $\mQ$, namely the union of $\Pi$ and $\Sigma$. The linear system $\X$ intersects $\Pi$ in $4r$ and moving lines, mapping it to a plane, and intersects $\Sigma$ in $2C$ plus moving lines, mapping it to a plane too. Therefore the image of this pair of planes is a reducible quadric in $\mQ^\prime$.

The plane $\Sigma$ intersects $E_{r^\prime}$ in the fiber through the double point of $\X_{r^\prime}$, and $\Pi_{r^\prime}\equiv (0,1)$ contains the simple point. Since $(E_r)_{r^\prime}\equiv (0,1)$ contains none of these base points, the image of $\Sigma\cup\Pi$ intersects $L$ in one point.

Next, we investigate the image of $C$. Before the blow ups, the conic $C$ was the complete intersection of $\Sigma$ and $V$. And this is still the case after the blow ups, since $r$ and $r^\prime$ lied in $V\setminus \Sigma$ and $q\notin \Sigma$ (there is no conic in $V$ through $q$, so $q\notin C$). In $\Sigma$, $C$ had two points blown up, so $C^2=2$. In $V$, none of the blow ups has affected its self-intersection, so $C^2=2$. Then:
\[ N_C = \mO_{\P^1}(2)\oplus \mO_{\P^1}(2) \]

Since $\X\cap\Sigma=2C+\{\text{lines}\}$, the blow up at $C$ gives $\X_C\equiv (2,2)$. It has a base point in $\l_C$ and a second base point infinitely near to it. Both points lie on $V_C\equiv (0,1)$. The linear system $\X_C$ corresponds, in $\P^2$, to quartics with two double points and two simple points. A standard quadratic transformation maps it to cubics with three base points. These points lie on the image of $V_C$, which is a line. Hence $E_C$ is mapped to a weak Del Pezzo sextic surface having a double point in $x$.

The line $t$ is the complete intersection of $E_q$ and $\Pi$. In $E_q$, $t^2=-1$, since both $r$ and $r^\prime$ intersected $E_q$ in points of $t$. In $\Pi$, $t$ was the exceptional curve of the blow up at $q$. Since both $r$ and $r^\prime$ lied in $\Pi$, it follows that $t^2=-1$. Then, the normal bundle of $t$ is:
\[ N_t = \mO_{\P^1}(-1)\oplus \mO_{\P^1}(-1) \]

In both $\Pi$ and $E_q$, the moving part of $\X$ does not intersect $t$. Then blowing up $t$ gives $\X_t\equiv (0,1)$ with no base points, and $E_t$ is mapped to a line. This is the directrix line of the cubic scroll, since $E_t$ intersects $E_{r^\prime}$ in the exceptional divisor of the blow up at the simple base point of $\X_{r^\prime}$.

This completes the analysis of the images of the base curves of $\X$. We now proceed to the study of the singularities of $X$. According to Lemma \ref{ed_singularidades_vem_de_C4} (which applies with no change), these singularities lie on $L$, the image of $r$. Let $q^\prime$ be a general point of $r$ and let $\Omega$ be a general plane through $q^\prime$. Then $\X\cap \Omega$ consists of degree five curves with a double point in $q^\prime$, a second double point infinitely near to it (corresponding to $r^\prime$), two double points in $C\cap \Omega$, a simple point in $l\cap \Omega$ and a second simple point infinitely near to it.

Blowing up $q^\prime$, $\X\cap\Omega$ intersects the exceptional curve $e$ in a fixed double point. Blowing up this point, $e$ no longer intersects $\X\cap\Omega$ and $e^2=-2$. Then it is mapped to a double point of the image of $\Omega$. By Lemma \ref{pr_truque_secao_tgente}, it is a double point of $X$. Therefore, $L$ is a double conic of $X$.
\\ \end{proof}

\subsection{When the multiplicity in $q$ is four}

This implies that $C_4$ is a union of four lines through $q$, the intersection of $V$ with another cone with vertex $q$. Since all quadrics of $\mQ^\prime$ are singular and the general one is a cone, there is a curve in $X$ described by these vertices.

\begin{lemma}\label{ed_imagemvertice}
Blowing up  the vertex $q$ of $V$, the plane $E_q$ is mapped to the line $L$ described by the vertices of the cones of $\mQ^\prime$. It is a double line of $X$. The conic $C=V\cap E_q$ is mapped to the line through $x$ in the cone $S_2^x$.
\end{lemma}
\begin{proof}
The following blow ups are represented in Figure \ref{ed_figura_blvertice}.  

Consider the blow up at $q$. Then $\X_q$ is made of quartic curves having four double and two simple points. The simple points are infinitely near. These six points lie on the conic $C=V_q$. So the linear system $\X_q$ has a fixed conic $C$ and then conics through four base points. This moving part is Cremona equivalent to lines through a point. It maps $E_q$ to a line, name it $L$.   

In $E_q$, each conic through the four base points is mapped to a point of $L$. So these points are double points of $X$. These moving conics in $E_q$ are also the intersections of this plane with the quadrics of $\mQ$, which are also singular in $q$. Hence, $L$ is the line described by the vertices of the quadrics $\mQ^\prime$.

Before blowing up  the conic $C$, one should blow up the four (possibly infinitely near) double lines of $\X$, that is, $C_4$, in order to avoid other fixed components. After this blow up, $C$ continues to be a complete intersection of $V$ and $E_q$. But now, its self intersection in $E_q$ has decreased by four, that is, $C^2=0$ in $E_q$. In $V$, $C^2=-2$, as before. So:
\[ N_C=\mO_{\P^1}(-2)\oplus \mO_{\P^1}(0) \] 

Note that, before the blow ups, $\X_q$ intersected $C$ only in the four points coming from the lines of $C_4$. So now, $\X_C$ does not intersect $E_q\cap E_C$, which is the $(-2)$-section of $E_C\cong \F_2$. Therefore $\X_C\equiv e_2+2f_2$.

This linear system has two infinitely near base points, in the intersection with the line $\l$. These two points lie on $V_C$, which is also a section of type $e_2+2f_2$. Hence, $V_C$ is contracted to the point $x$.

By Lemma \ref{pr_imagemF1F2}, $\X_C$ corresponds in $\P^2$ to conics with four base points in two infinitely near pairs. After a standard quadratic map, we get a pencil of lines through a point. So $E_C$ is mapped to a line. It is the one line through $x$ of the cone $S_2^x$.
\\ \end{proof}

\begin{figure}[tb]
\centering
\includegraphics[trim=30 480 10 5,clip,angle=180,width=1\textwidth]{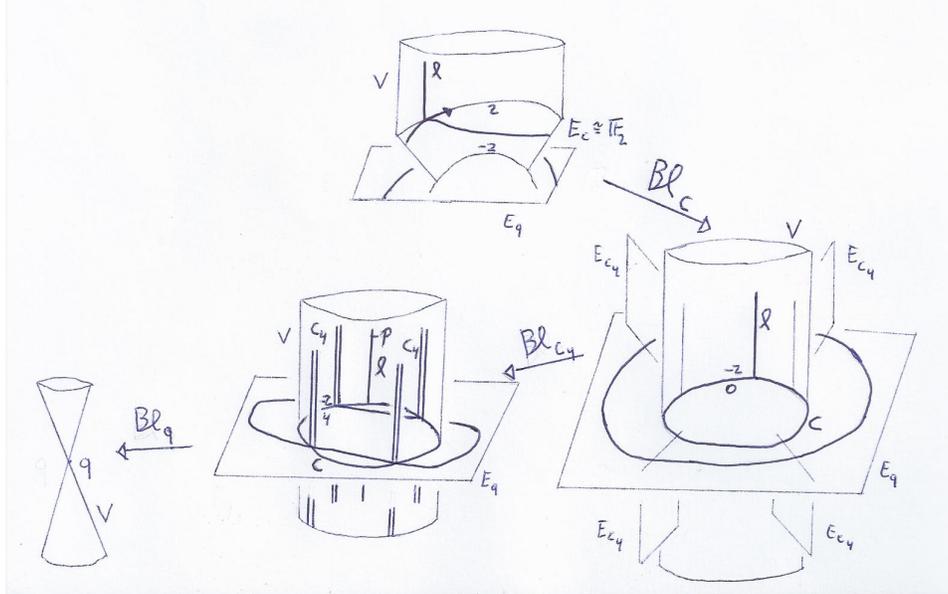}
\caption{Blow up at $q$ and $C$}\label{ed_figura_blvertice}
\end{figure} 

Next, we will find the image of $\l$. 

\begin{lemma}
The lines $\l^\prime$ and $\l$ are mapped respectively to the line through $x$ in $S_2^x$ and to the vertex of this cone.
\end{lemma}
\begin{proof}
The linear system $\X$ cuts a general plane $\Pi$ through $\l$ in this line plus degree four curves with multiplicity three in $q$ and passing through $p$. There is also another base point infinitely near to $q$, which will now be described. Consider the blow up at $q$. Remember that $C=V_q$ is a fixed part of $\X_q$. It intersects the line $\Pi_q$ in two points, one of them in $\l$. The other point is the infinitely near point mentioned above.

After this blow up, the line $\l$ has normal bundle:
\[ N_\l=\mO_{\P^1}(0)\oplus \mO_{\P^1}(0) \]

Remember that $\X$ contains a line $\l^\prime$ infinitely near to $\l$, corresponding to the tangent plane of $V$ at a general point of $\l$. So $\l^\prime$ is a $(0,1)$-curve in $E_\l$, intersecting $E_q$ in a point of the conic $C$. 

Since the degree four curves in $\Pi$ intersected $\l$ in $p$, it follows that:
\[ \X_\l \equiv  (1,1) \equiv \l^\prime + F_p \]
where $F_p$ is the fiber over the point $p\in\l$.

The line $\l^\prime$ is the intersection of $V\cong \F_2$ (since it was blown up at $q$) and $E_\l$. So:
\[ N_{\l^\prime}=\mO_{\P^1}(0)\oplus \mO_{\P^1}(0) \]

In $E_{\l^\prime}$, $E_q$ intersects in a line of type $(1,0)$, while $V$ and $E_\l$ in lines of type $(0,1)$. $\X_{\l^\prime}$ has a base point in the intersection with $C=E_q\cap V$ and another in the intersection with $F_p\subset E_\l$.

Therefore, $\X_{\l^\prime}$ is made of curves of type $(1,1)$ with two simple base points. It maps $E_{\l^\prime}$ to a line through $x$ (the contraction of $V_{\l^\prime}$).

The surface $E_\l$ is mapped to a point. It is the intersection of the line through $x$ and $L$, that is, the vertex of $S_2^x$.

  These blow ups are sketched in Figure \ref{ed_figura_ll'}.  
\\ \end{proof}

\begin{figure}[tb]
\centering
\includegraphics[trim=15 20 70 520,clip,width=1\textwidth]{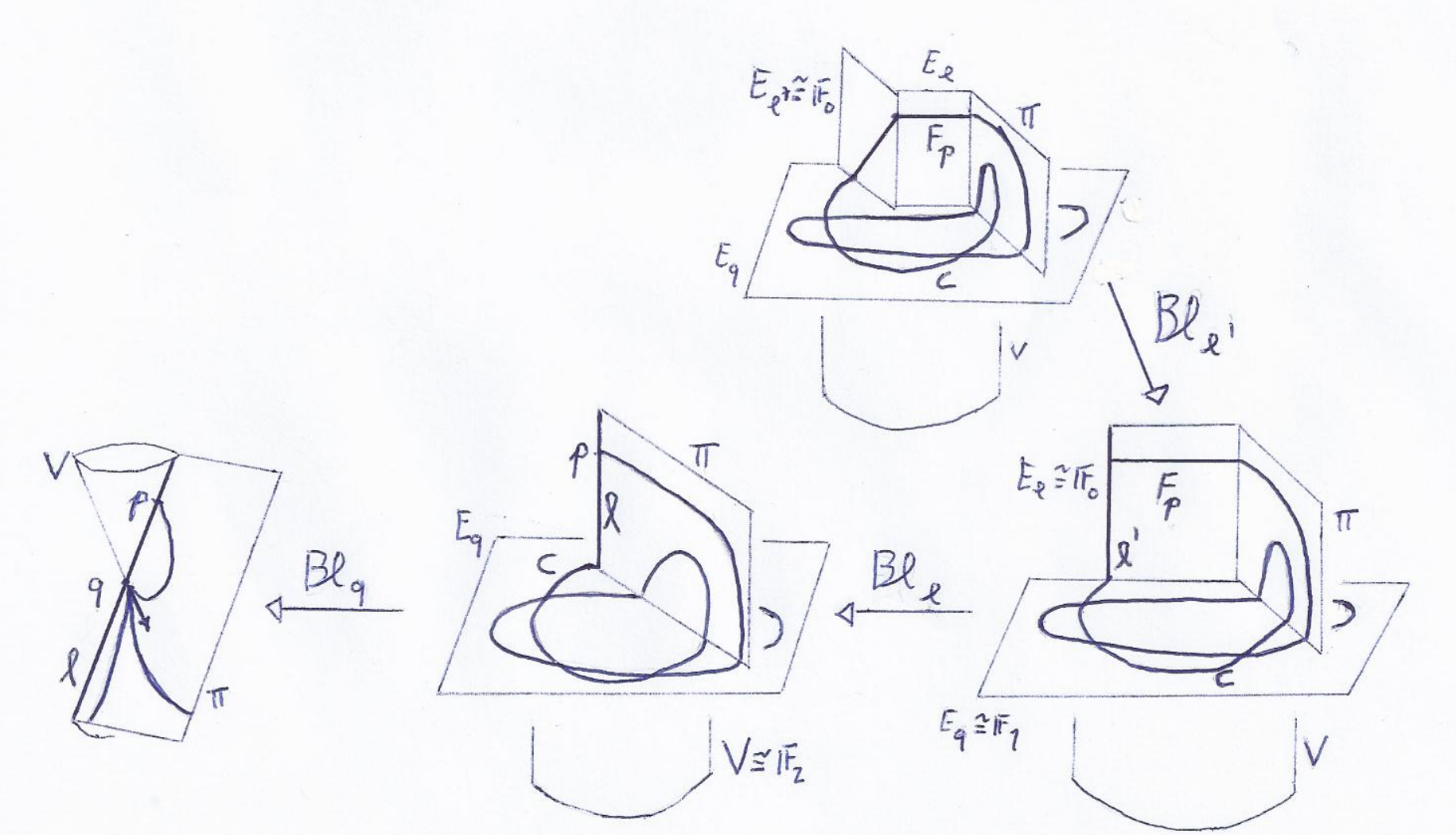}
\caption{Blow up at $\l$ and $\l^\prime$}\label{ed_figura_ll'}
\end{figure}

Note that there is no contradiction with Lemma \ref{ed_imagemvertice}, since, in $E_C$, the blow up at the base points determined by $\l^\prime$ and $\l$ are mapped to the line through $x$ and the vertex of $S_2^x$.

Before entering into the specific cases, let's make a remark on Segre symbols. Since $\mQ$ has no smooth element, the Segre symbol is not defined. But a general plane section of $\mQ$ gives a pencil of conics. Since $V$ is irreducible, there is a smooth conic in this pencil. Moreover, the singular elements in this pencil correspond to the special elements (i.e. pairs of planes and double planes) of $\mQ$.

So I will denote the \emph{Segre symbol of $\mQ$} to be the Segre symbol of the pencil of conics defined by a general plane section of $\mQ$.   To distinguish from the other notation, I will use two brackets.   For example, the special elements of the pencil with symbol $[[1,2]]$ are two pairs of planes, while the one with symbol $[[(21)]]$ has only a double plane.

There are five possibilities for $C_4$, which give five possibilities for the Segre symbol of $\mQ$:
\begin{itemize}
\item [\textbf{(E15)}] four distinct lines -- $[[1,1,1]]$
\item [\textbf{(E16)}] one double and two simple lines -- $[[2,1]]$
\item [\textbf{(E17)}] two double lines -- $[[(11),1]]$
\item [\textbf{(E18)}] a triple and a simple line -- $[[3]]$
\item [\textbf{(E19)}] a line with multiplicity four -- $[[(21)]]$
\end{itemize}

\subsubsection{(E15) Four simple lines}

Let $r$ be a simple line in $C_4$. Then:

\begin{figure}
\centering
\includegraphics[trim=35 670 185 5,clip,width=0.9\textwidth]{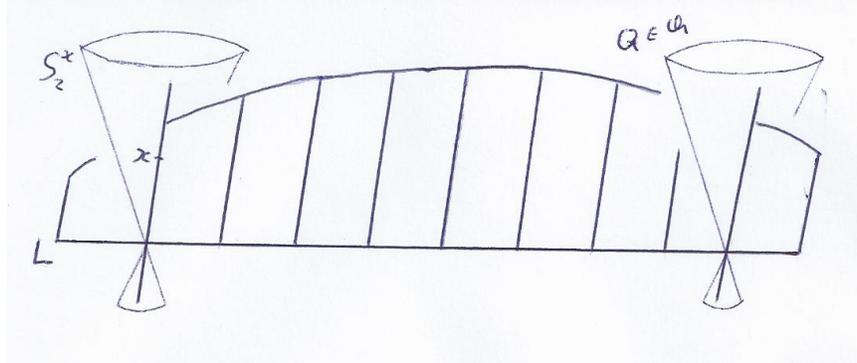}
\caption{The scroll $S_3^x$}\label{ed_figura_cone_S3x}
\end{figure} 

\begin{lemma}\label{ed_cone_4retas_lema}
A simple line $r\subset C_4$ is mapped to a cubic scroll $S_3^x$ through $x$. Its directrix line is $L$ and the cones of $\mQ^\prime$ intersect this scroll in lines of the ruling. Figure \ref{ed_figura_cone_S3x} gives a representation of $S_3^x$.
\end{lemma}
\begin{proof}
The linear system  $\X$ cuts a general plane $\Pi$ through $r$ in $2r$ plus cubics with a double point at $q$. Consider the blow up at $q$, as done in Lemma \ref{ed_imagemvertice}. The line $r$ intersects $E_q$ in a point of the conic $C$ and in a point of the moving part of $\X_q$. Since it is the complete intersection of two planes (blown up at $q$), its normal bundle is:
\[ N_r=\mO_{\P^1}(0)\oplus \mO_{\P^1}(0) \]

Now blow up  $r$. $\X_r$ cuts $E_q\cap E_r\equiv (1,0)$ in two points. One of them lies on $C$, so it is a fixed point. The other is a moving point. Moreover $\X_r$ cuts lines $\Pi_r\equiv (0,1)$ in one moving point.

Therefore, $\X_r$ is made of $(1,2)$ curves with one base point. These correspond to cubics in $\P^2$ with one double and two simple points, which are Cremona equivalent to conics with one base point. Hence $E_r$ is mapped to a cubic scroll $S(1,2)$. Name it $S_3^x$, since $V_r\equiv (0,1)$ is contracted to $x$.

The other quadrics of $\mQ$ also intersect $E_r$ in curves of type $(0,1)$. Then the quadrics of $\mQ^\prime$ intersect $S_3^x$ in lines of its ruling. Since $E_q$ intersects $E_r$ in the line $(1,0)$ through the base point, it is mapped to the directrix line of $S_3^x$. 
\\ \end{proof}

There are four cubic scrolls through the point $x$, which are contracted by the tangential projection $\tau$ to lines. In all four of these scrolls, the directrix line is $L$. 

Note that, blowing up  all four double lines of $\X$, the four exceptional divisors do not intersect. Therefore, the only curves in common to the four scrolls come from contractions of surfaces. So each scroll intersects a cone of $\mQ^\prime$ in a different line. The only exception is the cone $S_2^x$ through $x$, which is intersected by all four in the line through $x$. 

Since $x$ is a general point in $X$, there are cubic scrolls through other general points of $S$. The following Lemma describes the image, via $\tau$, of these surfaces:

\begin{lemma}
Through a general point $y\in X$ there are four scrolls $S(1,2)$ contained in $X$ intersecting quadrics of $\mQ^\prime$ in lines and having directrix line $L$. The image via $\tau$ of these scrolls in $\P^3$ are quadric cones through $\tau(y)$ containing three of the four double lines of $\X$ and $\l$.
\end{lemma}
\begin{proof}
Let $r_1,r_2,r_3$ be three of the four double lines of $\X$. Let $S$ be a quadric surface containing these lines and $\l$. This implies that $S$ is a cone with vertex $q$. Consider the blow up at $q$, so that $S\cong \F_2$. Then:
\[ \X\cap S \equiv 5e_2+10f_2 \equiv 4e_2+2r_1+2r_2+2r_3+\l+ \{e_2+3f_2\} \] 
where the moving part has a base point in $p$, which is a double point of $\X$.

Using Lemma \ref{pr_imagemF1F2}, the moving part is birationally equivalent to cubics in $\P^2$ with one simple point, one double point and a third simple point infinitely near to it. Therefore $S$ is mapped to a cubic scroll $S(1,2)$.

The intersection with quadrics of $\mQ$ is:
\[ \mQ\cap S \equiv 2e_2+4f_2 \equiv 2e_2+r_1+r_2+r_3+ \{f_2\} \] 
which are mapped to lines of the ruling of the cubic scroll.

The directrix line of the scroll is the image of the $(-2)$-section $e_2$, the blow up at $q$, which is mapped to $L$.

The family of quadrics containing these four lines through $q$ forms a pencil, so through a general point in $\P^3$ there is only one of these quadrics. Since there are four possible choices of $r_1$ $r_2$ and $r_3$, there are four quadrics through $\tau(y)$ which are mapped to cubic scrolls in $X$.
\\ \end{proof}
 
A description of case (E15) is given in the following proposition:

\begin{prop}
Let $X$ be the Bronowski threefold corresponding to case (E15). Then $X$ has a double line $L$, described by the vertices of cones of the one-dimensional family $\mQ^\prime$ contained in $X$. There are three reducible quadrics in this family, each is a pair of planes. 

There are four smooth quadric surfaces in $X$ that intersect the cones of $\mQ^\prime$ in lines of one of its rulings.

The four double lines of $\X$ in $\P^3$ are mapped to cubic scrolls $S_3^x$ through $x$. A line of the ruling of each scroll lies on a quadric of $\mQ^\prime$ and the directrix line is $L$.
\end{prop}
\begin{proof}

Given four intersecting lines, there are six planes containing pairs of them. Then there are three reducible quadrics in $\mQ$, that is, three pairs of planes. Together with Lemma \ref{ed_imagemvertice}, this proves the first part. The last part is the content of Lemma \ref{ed_cone_4retas_lema}.

Let $\Pi$ be a plane containing $\l$ and a double line $r$ of $\X$. There are four of such planes.   Then:
\[ \X\cap \Pi=l+2r+ \{\text{conics through } p \text{ and } q \} \]
and $\Pi$ is mapped to a smooth quadric surface. A line in $\Pi$ through $p$ is mapped to a line intersecting   the cones of $\mQ^\prime$. A line through $q$ lies on a cone of $\mQ$, so it is mapped to a line in a cone of $\mQ^\prime$.
\\ \end{proof}

\subsubsection{(E16) A double line and two simple lines}
\label{ed_conedoubleline_sec}

This case is now described:

\begin{prop}
Let $X$ be the Bronowski threefold corresponding to case (E16). In this case,  the one-dimensional family of quadric cones  $\mQ^\prime$ contained in $X$ has two reducible quadrics, both are pairs of planes. One of them corresponds to the root with multiplicity two in the Segre symbol. 

Let $R$ be the singular line of this quadric and let $L$ be the line described by the vertices of cones of $\mQ^\prime$. Then $R$ and $L$ are double lines of $X$ and they intersect each other in one point.

There are three smooth quadric surfaces in $X$ that intersect the cones of $\mQ^\prime$ in lines of one of its rulings.

One of the four double lines of $\X$ in $\P^3$ is mapped to $R$. The other three lines are mapped to cubic scrolls $S_3^x$ through $x$ with directrix line $L$. A line of the ruling of each scroll lies on a quadric of $\mQ^\prime$. One of these scrolls contains $R$. 
\end{prop}
\begin{proof}

Let $r$ be the double line. This means that there is a line $r^\prime$ infinitely near to $r$ in the singular locus of $\X$. The special members of the pencil $\mQ$ are two pairs of planes. Each of these planes contains two lines of the singular locus of $\X$, so it is mapped to a plane in $X$. Let $Q\in\mQ$ be the pair of planes which is singular in $r$. It corresponds to the root with multiplicity two in the Segre symbol. This follows from Proposition \ref{ed_singsegre} applied to the associated pencil of conics in $\P^2$. 

The existence of the three smooth quadric surfaces in $X$ follows, as before, from considering the planes through $\l$ and a double line of $\X$. Since $r^\prime$ is infinitely near to $r$, there are only three of such planes.

The linear system $\X$ intersects  a general plane through $r$ in $2r$ and cubics with a double point in $q$. Now blow up the point $q$, and then the line $r$. So far, the situation is pretty much the same as in the simple line case, and $E_r\cong \F_0$. 

In $E_r$, the line $V_r=r^\prime$, of type $(0,1)$, is a fixed double line of $\X$. It follows then:
\[ \X_r\equiv (1,2)\equiv 2r^\prime + (1,0) \]
where $(1,0)$ is the moving part, mapping $E_r$ to a line $R$ in $X$. It maps each point of $r$ to a point in $R$.   Note that $E_q\cap E_r$ is mapped to a point, so $R$ intersects $L$ in one point.

The line $r^\prime$ is a complete intersection and:
\[ N_{r^\prime}=\mO_{\P^1}(0)\oplus\mO_{\P^1}(0) \]

In $E_{r^\prime}$, both $E_r$ and $V$ intersect in lines of type $(0,1)$. Hence $\X_{r^\prime}$ is of type $(1,2)$. It has a simple point in $E_q\cap E_{r^\prime}$, lying on the curve $C=V_q$. So it is mapped to a cubic scroll $S_3^x$, like the one in the simple line case.

The reducible quadric $Q$ does not contain the plane through $r$ and $r^\prime$. Then $Q_r\equiv (0,2)$ does not contain $V_r$ and $Q$ is mapped to a pair of planes intersecting $S_3^x$ in $R$.

It remains to prove that $R$ is a double line of $X$.

Let $q^\prime$ be a general point of $r$ and let $\Omega$ be a general plane through $q^\prime$. As it was remarked above, a point of $r$ is mapped to a point of $R$.

The intersection of $\X$ with $\Omega$ consists of degree five curves with a double point in $q^\prime$, a second double point $q^{\prime\prime}$ infinitely near to it, two other double points and two simple points. Blowing up $q^\prime$ and $q^{\prime\prime}$, the curve $\Omega_{q^\prime}$ has self intersection $-2$ in $\Omega$ and no intersection with $\X\cap\Omega$. Then $\Omega$ is mapped to a surface with a double point in $\sigma(q^\prime)\in R$. By Lemma \ref{pr_truque_secao_tgente}, it is a double point of $X$. Hence $R$ is a double line of $X$.
\\ \end{proof}

\subsubsection{(E17) Two double lines}

All the information needed here is detailed in the previous section.   Then we get:
\begin{prop}
Let $X$ be the Bronowski threefold corresponding to case (E17). The one-dimensional family of quadric cones  $\mQ^\prime$  in $X$  has a double plane and a pair of planes.

Let $L$ be the line described by the vertices of cones of $\mQ^\prime$, it is a double line of $X$. Two other lines, $R$ and $R^\prime$, which are contained in the double plane of $\mQ^\prime$, are double lines of $X$. These two lines intersect in a point of $L$.

There are two smooth quadric surfaces in $X$ that intersect the cones of $\mQ^\prime$ in lines of one of its rulings.

Two of the four double lines of $\X$ in $\P^3$ are mapped to $R$ and $R^\prime$. The other two lines, infinitely near to those, are mapped to cubic scrolls $S_3^x$ through $x$. A line of the ruling of each scroll lies on a quadric of $\mQ^\prime$ and the directrix line is $L$. One of these scrolls contains $R$, the other contains $R^\prime$. 
\end{prop}

\subsubsection{(E18) A triple and a simple line}
\label{ed_conetripleline_sec}

\begin{prop}
Let $X$ be the Bronowski threefold corresponding to case (E18). The only reducible quadric in $\mQ^\prime$ is a pair of planes. 

Let $R$ be the intersection of these two planes and let $L$ be the line described by the vertices of cones of $\mQ^\prime$. Then $X$ has multiplicity two in $L$, $R$ and in a line $R^\prime$ infinitely near to $R$.

There are two smooth quadric surfaces in $X$ that intersect the cones of $\mQ^\prime$ in lines of one of its rulings.

Two infinitely near double lines of $\X$ in $\P^3$ are mapped to $R$ and $R^\prime$. The other two lines are mapped to cubic scrolls $S_3^x$ through $x$. A line of the ruling of each scroll lies on a quadric of $\mQ^\prime$ and the directrix line is $L$. One of these scrolls contains $R$ and $R^\prime$. 
\end{prop}
\begin{proof}
Let $r$ be the triple line. The situation is very similar to the previous cases. There is a line $r^\prime$ infinitely near to $r$ and a third line $r^{\prime\prime}$ infinitely near to $r^\prime$. These three lines are double lines of $\X$.

Blowing up $q$ and $r$, the moving part of $\X_r$ maps $E_r$ to a line $R$. Blowing up $r^\prime$, the moving part of $\X_{r^\prime}$ maps $E_r^\prime$ to a line $R^\prime$. This line is infinitely near to $R$, since $E_r$ and $E_{r^\prime}$ intersect in a line of type $(1,0)$, while the moving part of $\X$ cut these surfaces in $(0,1)$-lines. 

After the blow up at $r^{\prime\prime}$ the surface $E_{r^{\prime\prime}}$ is mapped to a cubic scroll $S_3^x$. It intersects the one pair of planes of $\mQ^\prime$ in the line $R$.

To compute the multiplicity of $X$ in $R$ and $R^\prime$, let $\Omega$ be a general plane through $q^\prime$, a general point of $r$. Then $\X\cap \Omega$ has a double point in $q^\prime=r\cap \Omega$, a second double point $q^{\prime\prime}$ infinitely near to it (corresponding to $r^\prime$) and a third double point infinitely near to $q^{\prime\prime}$, corresponding to $r^{\prime\prime}$. After the blow up at these three points, there are two intersecting curves with self-intersection $-2$ in $\Omega$. These are mapped to two infinitely near double points in the image of $\Omega$. By Lemma \ref{pr_truque_secao_tgente}, these are double points of $X$. This proves that $R$ and $R^\prime$ are double lines of $X$.

The quadric of $\mQ^\prime$ containing $R$ is the image of the quadric of $\mQ$ which is singular in $r$, namely, the pair of planes. The images of the planes spanned by $\l$ and $r$ and by $\l$ and the other double line of $\X$ are the two smooth quadric surfaces mentioned in the proposition.
\\ \end{proof}

\subsubsection{(E19) A line with multiplicity four}

This case can be easily described with what was studied before.   Therefore:

\begin{prop}
Let $X$ be the Bronowski threefold corresponding to case (E19), in which the only reducible quadric in $\mQ^\prime$ is a double plane. 

Let $L$ be the line described by the vertices of cones of $\mQ^\prime$. Then $L$ is a double line of $X$. There is another double line $R$ in the double plane of $\mQ^\prime$ a third double line $R^\prime$ infinitely near to it and a fourth double line $R^{\prime\prime}$ infinitely near to $R^\prime$.

There is one smooth quadric surface in $X$ that intersects the cones of $\mQ^\prime$ in lines of one of its rulings.

Three double lines of $\X$ in $\P^3$ are mapped to $R$, $R^\prime$ and $R^{\prime\prime}$. The other line is mapped to a cubic scroll $S_3^x$ through $x$. A line of its ruling lies on a quadric of $\mQ^\prime$ and the directrix line is $L$. This scroll contains $R$, $R^\prime$ and $R^{\prime\prime}$. 
\end{prop}

\section{The classification}

Let us now put together the varieties studied in this chapter.

\begin{theorem}\label{ed_the_classification}
Let $X$ be a Bronowski threefold having as fundamental surface a quadric $V$. Suppose that the linear system $\X$ defining the inverse of a general tangential projection of $X$ has degree five. 

Then $\X$ has multiplicity two in the base locus $C_4$ of a pencil of quadrics $\mQ$ and in a point $p\in V$, and multiplicity one in the two (possibly infinitely near) lines of $V$ through $p$. 

The threefold $X$ has the following properties:
\begin{itemize}
\item[(i)] It is an OADP variety;
\item[(ii)] It is the residual intersection of the Segre Embedding of $\P^1\times\P^3$ with a quadric of $\P^7$ containing a $\P^3$ of the ruling;
\item[(iii)] The singularities of $X$ have multiplicity two;
\item[(iv)] The pencil $\mQ$ is mapped to a one-dimensional family of quadric surfaces $\mQ^\prime$ in $X$, which has general smooth member if and only if $V$ is smooth;
\item[(v)]  If $V$ is singular, then $X$ has a double curve $L$ which is described by the vertices of cones of $\mQ^\prime$. If the cones of $\mQ$ have the same vertex, this curve is a line; otherwise it is a conic and $X$ is of type (E14); 
\item[(vi)] Apart from $L$, each  component of the singular locus  of $X$ lies on the singular locus of a quadric of $\mQ^\prime$ corresponding to a root of multiplicity greater than one in the Segre symbol.
\end{itemize}

In Table \ref{ed_tabela}, the singularities of $X$ are described, associated to each possible configuration of $C_4$. Varieties (E1) to (E13) correspond to cases where $V$ is smooth.  In cases (E15) to (E19),  the Segre symbol (with double brackets) refers to the pencil of conics cut by $\mQ$ on a general plane of $\P^3$. 
\end{theorem}
\begin{proof}
The Segre symbol is easily computable in each case,   these computations can be found in \cite[p. 305]{hodge}. Here, it is defined and explained in Section \ref{ed_segre_sec}.

Besides that, Only items $(i)$ and $(ii)$ are in need of a proof, the other results were explained in the previous Sections. 

In all the examples that were studied, the linear system $\X$ is \emph{relatively complete}, so $X$ is linearly normal (cf. Remark \ref{pr_oadpLN}). Then, by Theorem \ref{pr_contidonumscroll}, $X$ lies on a rational normal scroll $Y$, described by the spans of quadrics in $\mQ^\prime$. It is a four-dimensional scroll in $\P^7$ containing a one dimensional family of 3-spaces. 

These 3-spaces are disjoint. Indeed, if they weren't, there would be two quadrics of $\mQ^\prime$ spanning a $\P^6$, none of them containing the general point $x$. This implies that there is a surface of $\X$ containing two quadrics of $\mQ$ different from $V$. Such surface should then be the union of these two quadrics and the plane containing $\l$ and $\l^\prime$. However, this surface has  multiplicity one in $p$, so it cannot belong to $\X$, a contradiction.

Hence $Y$ is the Segre embedding of $\P^1\times \P^3$.  Item $(ii)$ now follows from the fact that $X$ has degree seven and is a divisor on $Y$ that intersects a general $\P^3$ of the ruling in a quadric. 

Now, $(i)$ follows from $(ii)$, using the same argument of \cite[Proposition 2.5]{cmr}.
\\ \end{proof}

\begin{table}
\begin{center}
\begin{tabular}{|c|c|c|c|}
\hline  String & $C_4$ & Segre symbol  & Singularities of $X$  \\ 
\hline (E1) & smooth quartic & $[1,1,1,1]$ & none \\ 
\hline (E2) & nodal quartic &  $[2,1,1]$  & one point \\ 
\hline (E3) & cuspidal quartic &  $[3,1]$ & one point \\ 
\hline (E4) & twisted cubic and  & $[2,2]$ & two points \\ 
& a transversal line  & &\\
\hline (E5) & twisted cubic and & $[4]$  &  two infinitely \\ 
& a tangent line  & & near points \\
\hline (E6) & two transversal conics  & $[(11),1,1]$ & two points\\ 
\hline (E7) & two tangent conics  & $[(21),1]$  & two infinitely \\ 
&&& near points \\
\hline (E8) & one conic and two lines  & $[(11),2]$  & three points \\ 
&	intersecting it in two points && \\
\hline (E9) & one conic and two lines & $[(31)]$ & one point and three  \\
& intersecting it in one point  && infinitely near to it\\ 
\hline (E10) & four lines &   $[(11),(11)]$  & four points\\ 
\hline (E11) & one double and  & $[(22)]$  & one line \\ 
&two simple lines  && \\
\hline (E12) & two double lines  & $[(211)]$ & two lines \\ 
\hline (E13) & one double conic & $[(111),1]$ & one conic \\ 
\hline\hline (E14) & a double line and a conic & -- & one conic  \\ 
\hline (E15) & four lines & $[[1,1,1]]$ & one line \\ 
\hline (E16) & one double and two  & $[[2,1]]$ & two lines \\ 
& simple lines &&  \\
\hline (E17) & two double lines & $[[(11),1]]$ & three lines\\ 
\hline (E18) & one triple and  & $[[3]]$ & three lines: $L,$\\ 
& one simple lines && $R\prec R^\prime$ \\ 
\hline (E19) & one  line with  & $[[(21)]]$ & four lines: $L,$\\ 
& multiplicity four && $R\prec R^\prime\prec R^{\prime\prime}$ \\ 
\hline 
\end{tabular} 
\caption{Varieties satisfying hypothesis (H) and having quadratic fundamental surface.}\label{ed_tabela}
\end{center}
\end{table}

\chapter{The cubic case}
\label{sl_chapter}

In this chapter we consider the case in which the fundamental surface $V$ is a cubic. This is the situation of the degree eight scroll in lines, described in Example \ref{pr_exemplo_sl}. According to hypothesis (H), $\X$ has degree seven. We'll also suppose that the base locus of $\X$ has pure dimension one (see Proposition \ref{sl_baselocus}).

\section{Base locus of $\X$}
\label{sl_blocus_sec}

Since $V$ is a cubic surface, the map $\bar{\tau}:E\cong \P^2\tor V\subset \P^3$ is defined by a linear system of conics with one base point. It is not complete, since the complete linear system maps to the scroll $S(1,2)\subset \P^4$. Then $V$ is a non normal cubic surface, a projection of $S(1,2)$ from an external point.

The proof of the following proposition can be found in \cite{dolga_classical}: 
\begin{prop}\label{sl_projecoes_de_S(1,2)}
Let $V\subset \P^3$ be a projection of the cubic scroll $S(1,2)$ from an external point and let $r\subset V$ be the projection of the directrix line of the scroll. Then $V$ has a double line $s$, projection of a conic of $S(1,2)$ and it is projectively equivalent to one of the following surfaces:
\begin{align*}
(i) &\quad x_0^2x_2 + x_1^2x_3=0 \\
(ii) &\quad x_0^2x_3+x_0x_1x_2+x_1^3=0
\end{align*}

Moreover, in $(i)$ (the general case), $r$ and $s$ are skew lines. In $(ii)$, the conic of $S(1,2)$ that is projected to the double line $s$ is the union of the directrix line and a line of the ruling of the scroll, so $r$ is infinitely near to $s$. The surface $(ii)$ is called \emph{Cayley's ruled cubic}.  
\end{prop}

When computing self-intersections of curves in $V$, we'll consider its normalization $S(1,2)$.

Let $p$ be the base point of $\II_{X,x}$. Then $\bar{\tau}$ maps lines through $p$ to lines of the ruling of $V$. The point $p$ is blown up and mapped to $r$. The image of a general line is a conic. But one of these lines, name it $\hat{s}$, is mapped to the double line $s$.

If $V$ is a general projection of $S(1,2)$, $\hat{s}$ does not contain $p$. Then the map: 
\[ \bar{\tau}\vert_{\hat{s}} : \hat{s}\to s \]
gives a double cover of $\P^1$, ramified over two points  $q_1^s$ and $q_2^s$ in $s$. 

If $V$ is the Cayley's ruled cubic, then $\hat{s}$ is a line through $p$. After the blow up at $p$, the union of $\hat{s}$ and the exceptional divisor is mapped to $s$.

Looking at plane sections of $V$, we have:
\begin{lemma}\label{sl_conicas_em_IIXx}
A conic in $E$ through $p$ lies on the linear system $\II_{X,x}$ if and only if its two points of intersection with $\hat{s}$ are mapped by $\bar{\tau}$ to the same point of $s$.

In the Cayley's ruled cubic case, since $p$ lies on $\hat{s}$, these two points consist of a proper point in $\hat{s}$ and a point infinitely near to $p$.
\end{lemma} 
\begin{proof}
A conic of $\II_{X,x}$ is mapped to a plane section of $V$, which intersects $s$ in a double point. Then the two points of intersection of such conic with $\hat{s}$  are mapped to the same point of $s$. 

Conversely, a conic of $E$ through $p$ is mapped to a cubic curve. If it intersects $\hat{s}$ in two corresponding points, the cubic curve will have a double point in $s$. Hence it is a plane cubic, and therefore a plane section of $V$.

In the Cayley's ruled cubic case, a conic through $p$ intersects $\hat{s}$ in $p$ and in another point. Since $p$ is blown up, the points infinitely near to $p$ are mapped to distinct points, giving the last assertion.
\\ \end{proof}

\label{sl_notacao_divisores_de_V}
Let $\widetilde{E}$ be the plane $E$ blown up at $p$, and write $(a,b)$ for the class of a curve in $\widetilde{E}$ of degree $a$ intersecting the exceptional divisor in $b$ points. So a curve of type $(1,1)$ is mapped by $\bar{\tau}$ to a line of the ruling of $V$, a curve of type $(1,0)$ is mapped to a conic (or to $2s$), a curve of type $(2,1)$ is mapped to a cubic (a plane section of $V$ if it lies on $\II_{X,x}$) and the exceptional curve $(0,-1)$ is mapped to $r$. Abusing notation, the same terminology will be used for curves in $E$ and in $V$.

Note that also in case $(ii)$ the line $\hat{s}$ in $\widetilde{E}$ is of type $(1,0)$. It is the union of a line through $p$ and the exceptional line, so $(1,0)=(1,1)+(0,-1)$.

Below is a description of the linear system $\X$:
\begin{prop}\label{sl_baselocus}
Let $s$ be the double line of $V$ and let $r$ be the image of the blow up at the base point of $\II_{X,x}$. Suppose the base locus of $\X$ has pure dimension 1. Then $\X$ is the linear system of surfaces with degree seven having:
\begin{itemize}
\item multiplicity four in $s$
\item multiplicity two in $C_6$ 
\item multiplicity one in $r$,
\end{itemize} 
where $C_6\subset V$ is a curve of type $(5,4)$, that is, it is the image of a quintic in $E$ having multiplicity four in $p$. In particular, $C_6$ cuts $s$ in five points and $r$ in four points, supposing it does not contain these lines.
\end{prop}
\begin{proof}

Following the notation of Lemma \ref{pr_X'eX''}, the linear system $\Xp$ must desingularize the cubic $V$, so $\mult_s\Xp \geq 1$. This implies that $\mult_s\X \geq 3$. On the other hand, the multiplicity of $\X$ in $s$ is at most $4$, since hypothesis (H) implies that $\Xpp$ has no base locus. Therefore:
\[ m=\mult_s\X \in \{3,4\} \]

As remarked in Lemma \ref{pr_X'eX''}, the moving part of $\Xp\cap V$ defines the inverse of $\bar{\tau}$. Then there is a curve $C\subset V$ such that:
\[ \Xp\cap V = 2(m-2)s+\lbrace\text{conics}\rbrace+C \]
In particular, $C$ is a double curve of $\X$.

If $m=3$, then $C$ has degree eight and $\X$ intersects $V$ in a curve of degree at least $2\cdot 3+2\cdot 8=22$, which is not possible. Therefore, $m=4$, $C=C_6$ has degree six and:
\[ \Xp\cap V = 4s+\lbrace\text{conics}\rbrace+C_6 \]

The intersection of $V$ with a hypersurface of degree $d$ has class $(2d,d)$. Then, for the intersection of $\X^\prime$ with $V$, we have:
\[ (8,4) \equiv \X^\prime\cap V  = 4s+\lbrace\text{conics}\rbrace+C_6 \equiv (2,0)+(1,0)+(a,b)\]
which gives $C_6\equiv (5,4)$. 

The intersection with $\X$ must be a fixed divisor of $V$, then:
\[ (14,7) \equiv \X\cap V = 8s+2C_6+\lbrace \text{fixed line} \rbrace \equiv (4,0)+(10,8)+(a,b)\] 
so the fixed line has type $(0,-1)$, that is, it is $r$. This finishes the proof, since hypothesis (H) implies that the base locus of $\X$ lies on $V$.
\\ \end{proof}

Note however that $C_6$ can contain $r$ or $s$ as a component, as is now described:

\begin{lemma}\label{sl_r_s_em_C6}
If $r\subset C_6$, then there is another line $r^\prime$ infinitely near to $r$ in the base locus of $\X$ and $C_6$ consists of $r$ and five lines of the ruling. In particular:
\[\X\cap V = 2\sum_{i=1}^5 \l_i + 8s + 3r \]
and $C_6$ cuts $s$ in five points.

If $s\subset C_6$, then there is another line $s^\prime$ infinitely near to $s$, which is a double line of $\X$. In this case, $C_6$ is the union of $s$ with multiplicity two and four lines of the ruling and:
\[\X\cap V = 2\sum_{i=1}^4 \l_i + 12s + r \]
In particular, $C_6$ cuts $r$ in four points.
\end{lemma}
\begin{proof}
Suppose first that $r\subset C_6$. Since $r\equiv (0,-1)$, we have that:
\[ C_6 \equiv (5,4) \equiv  (0,-1)+(5,5) \]
So $C_6$ is the union of $r$ and five lines $\l_1,\ldots,\l_5$ of the ruling. 

Note that $\X$ has multiplicity two in $r$. Blowing up $r$, we have that $\X_r\equiv (5,2)$ and $V_r\equiv (2,1)$. But these two curves intersect in five double points, corresponding to the five lines in $C_6$. Since $(5,2)\cdot (2,1)=9$ and $V_r$ is irreducible, it follows that $\X_r$ contains $V_r=r^\prime$. 

Suppose now that $s\subset C_6$. But $s$ is a double line of $V$, so we must have $2s\subset C_6$. Since $2s\equiv (1,0)$, then:
\[ C_6 \equiv (5,4) \equiv  (1,0)+(4,4) \]
and $C_6$ is the union of $2s$ and four lines $\l_1,\ldots,\l_4$ of the ruling. Since $\X$ has multiplicity four in $s$ and two in $C_6$, it follows that:
\[\X\cap V = 2\sum_{i=1}^4 \l_i + 12s + r \]
So blowing up $s$, $V_s\equiv (1,2)$ is irreducible and $\X_s \equiv (3,4)$ contains $2V_s=2s^\prime$, that is:
\[ \X_s \equiv 2s^\prime + \{ (1,0) \} \]
and the result follows.
\\ \end{proof}

We can now give a simple description of the possibilities for the base locus of $\X$:
\begin{coro}\label{sl_decomposicoes_de_C6}
Let $C_6$ be the sextic curve in the base locus of $\X$, given by Proposition \ref{sl_baselocus}. Then one of the following holds:
\begin{itemize}
\item[(i)] $C_6$ is the union of $r$ and five lines of the ruling;
\item[(ii)] $C_6$ is the union of $2s$ and four lines of the ruling;
\item[(iii)] $C_6$ is the union of an irreducible curve of degree $d$, which is smooth outside $s$, and $6-d$ lines of the ruling, for $2\leq d \leq 6$. 
\end{itemize}
\end{coro}
\begin{proof}
Items $(i)$ and $(ii)$ have been proved in Lemma \ref{sl_r_s_em_C6}.

Consider the curve $\bar{\sigma}(C_6)$ in $E\cong \P^2$, having degree five and multiplicity four in $p$, that is, a $(5,4)$ curve. If it contains an irreducible curve $\hat{C}$ of degree $d^\prime$, this curve must have multiplicity $d^\prime-1$ in $p$ and no other singularities, since:
\[ (5,4)-(d^\prime,d^\prime-1)=(5-d^\prime,5-d^\prime) \]

So $\bar{\sigma}(C_6)$ is the union of $\hat{C}$ and $5-d^\prime$ lines through $p$. Then the result follows with $d=d^\prime+1$. 
\\ \end{proof}

Another easy consequence of Lemma \ref{sl_baselocus} and Lemma \ref{sl_r_s_em_C6} is the following:

\begin{coro}\label{sl_X_is_ruled}
If $V$ is a cubic surface, then the associated threefold $X$ is ruled by lines.
\end{coro}
\begin{proof}
Let $y$ be a general point of $X$ and set $y^\prime=\tau(y)$. Consider the plane $\Pi$ spanned by $s$ and $y^\prime$. 

If $C_6$ does not contain $s$, it intersects $\Pi$ in five points in $s$ and a sixth point $p$.  Let $\l$ be the line spanned by $p$ and $y^\prime$. Then $\X$ intersects $\l$ in $p$ with multiplicity two, in $\l\cap s$ with multiplicity four and in a moving point. Therefore $\l$ is mapped to a line through $y$.

Suppose now that $C_6$ contains $s$ and consider the blow up at $s$. Then $\Pi_s\equiv (0,1)$ intersects $s^\prime=V_s\equiv(1,2)$ in one point, lying on the fiber over a point $p\in s$. Then the line through $p$ and $y^\prime$ is intersected by $\X$ in $p$ with multiplicity four and a double point infinitely near to it. Therefore it is mapped to a line in $X$ through $y$.
\\ \end{proof}

\section{The general case}
\label{sl_general_sec}

Throughout this section, suppose $V$ is of type $(i)$ in Proposition \ref{sl_projecoes_de_S(1,2)}, so $r$ and $s$ are skew lines. We will first study specific properties of $V$ and the images of $r$ and $s$ in $X$.

Next, we identify the possible singularities of $X$ based on properties of the base locus of $\X$ and present a family of surfaces contained in $X$ which helps to understand its geometry. Finally, a specific example will be studied in detail.

\subsection{Some properties}

Let's first give more details on $V$:

\begin{lemma}\label{sl_rulingV}
The surface $V$ is ruled by lines that intersect both $s$ and $r$. There is only one line of the ruling through each of the points $q_1^s$ and $q_2^s$ of $s$. Through any other point of $s$ there are two lines of the ruling. Apart from these pairs, the lines of the ruling are disjoint.
\end{lemma}
\begin{proof}
The lines of the ruling are images of lines in $E$ through $p$. Since $p$ does not lie on $\hat{s}$, these cut $\hat{s}$ in a point different from $p$.

Given a general point in $s$, it has two distinct preimages in $\hat{s}$. If the point is $q_1^s$ or $q_2^s$, it has only one preimage in $\hat{s}$. Then the result follows.
\\ \end{proof}

A plane section of $V$ is a cubic with a singular point in $s$. More precisely:
\begin{lemma}\label{sl_conetangente_de_V_e_secoes_planas}
Let $q$ be a point in $s$. If $q\neq q_1^s,q_2^s$, then the tangent cone of $V$ in $q$ is a pair of planes, each containing $s$ and one of the lines of the ruling through $q$. Moreover, a general plane section of $V$ through $q$ is a nodal cubic.

If $q$ is $q_1^s$ or $q_2^s$, the tangent cone of $V$ in $q$ is a double plane, the plane spanned by $s$ and the line of the ruling through $q$. A general plane section of $V$ through $q$ is a cuspidal cubic.
\end{lemma}
\begin{proof}
In both situations, $s$ is a double line of the tangent cone of $V$ in $q$.

In the first case, the tangent cone must contain both lines of the ruling through $q$, so it is the pair of planes containing $s$ and one of these lines. A general plane through $q$ cuts these planes in two distinct lines, which lie on the tangent cone of the cubic curve in $q$. Then it is a nodal point.

Suppose now that $q$ is $q_1^s$ or $q_2^s$. A plane containing $s$ cuts $V$ in $2s$ and a line of the ruling. If such plane lies on the tangent cone of $q$, this intersection has multiplicity three in $q$, so the unique line of the ruling through $q$ is contained in this plane. Then the tangent cone is the claimed double plane, name it $\Pi$.

A general plane section through $q$ is a cubic with tangent cone at $q$ contained in $\Pi$. But the cubic itself cannot lie on $\Pi$, so the double point is a cusp.
\\ \end{proof}

The next step is to determine the images via $\sigma$ of the lines $r$ and $s$. 
By Proposition \ref{sl_baselocus}, if $C_6$ does not contain these lines, it cuts  $r$ in four points and $s$ in five points. Remember that the case in which one of these is contained in $C_6$ is described in Lemma \ref{sl_r_s_em_C6}.

\begin{prop}\label{sl_imagem_r_e_s}
If $r$ is not contained in $C_6$, it is mapped by $\sigma$ to the line $\l_x$ through $x$. This line corresponds to the base point of $\II_{X,x}$, which was mapped to $r$ in the first place. If $r\subset C_6$ it is mapped to a plane and the line $r^\prime$ infinitely near to $r$ is mapped to $\l_x$. These two intersect in a point.

If $s$ is not contained in $C_6$, it is mapped to a weak Del Pezzo surface of degree four $D_4^x$ through $x$.  If $s\subset C_6$, then it is mapped to a line $L$. The line $s^\prime$ infinitely near to $s$ is mapped to a quartic surface $D_4^x$, a projection of the Veronese quartic surface having  multiplicity two along $L$.

Since $r$ and $s$ are disjoint, $\l_x$ and $D_4^x$ intersect only in $x$.
\end{prop}
\begin{proof}
Remember that the base locus of $\X$ has no embedded points. Suppose first that $r$ is not contained in $C_6$.

Blowing up $r$, $\X$ cuts $E_r\cong \F_0$ in curves of type $(6,1)$ having four double points at the intersections with $C_6$. So they break in the fibres over these points plus $(2,1)$ curves with four base points. 

Applying Lemma \ref{pr_imagemF1F2}, the moving part corresponds, in $\P^2$, to cubics with five simple points and one double point. After two standard quadratic maps, we get a linear system of lines through a point. So $E_r$ is mapped to a line. The cubic $V$ cuts the exceptional divisor in a $(2,1)$-curve through the four base points, so it  is contracted to $x$.

The fixed components intersect the moving part in the base points, so they do not affect the image of $E_r$. 

Now suppose $r\subset C_6$. By Lemma \ref{sl_r_s_em_C6}, $C_6$ is the union of $r$ and five lines of the ruling. 

Blow up $r$. As already noted:
\[ \X_r \equiv (5,2) \equiv r^\prime + (3,1)	 \]
and the moving part has five base points in the intersections with the five lines. It maps $E_r$ to a plane.

To find the image of $r^\prime$, first blow up the five lines $\l_i$. Now, $r^\prime$ is the complete intersection of $E_r$ blown up at five points with $V$. In $E_r$, $(r^\prime)^2=4$, so after the blow ups we have $(r^\prime)^2=-1$. In $V$, $(r^\prime)^2=-1$. Hence, blowing up $r^\prime$ gives $E_{r^\prime}\cong \F_0$.

In $E_r$, $\X$ intersects $r^\prime$ in its five base points. So after the blow ups there is no intersection, and $\X_{r^\prime}\equiv (0,1)$. Hence $r^\prime$ is mapped to a line. This is the one line $\l_x$ through $x$, since $V_{r^\prime}\equiv (0,1)$ is contracted to $x$. Moreover, $(E_r)_{r^\prime}\equiv (0,1)$, so the line intersects the plane in one point.

This proves the first part.

For the second part, suppose first that $s$ is not contained in $C_6$.

Consider the blow up at $s$. $\X_s$ is made of curves of bidegree $(3,4)$, having five double points at the intersections with $C_6$. Depending on $C_6$, some of these points can be infinitely near. The linear system is birationally equivalent to cubics in $\P^2$ with five base points. $V$ intersects the exceptional surface in a $(1,2)$ curve through the five points, which is contracted to $x$. Then $s$ is mapped to a weak Del Pezzo quartic surface through $x$. 

Now suppose $s\subset C_6$. By Lemma \ref{sl_r_s_em_C6}, $C_6$ is the union of $2s$ and four lines of the ruling. Blowing up $s$, we have:
\[ \X_s \equiv 2s^\prime + \{ (1,0) \} \equiv (3,4) \]
 where $s^\prime=V_s\equiv (1,2)$. The moving lines of type $(1,0)$ cut the rational curve $s^\prime$ in a linear series of degree two and projective dimension one. This linear series is not complete, otherwise it would have dimension two.

In $E_s$, $(s^\prime)^2=4$. In $V$, $(s^\prime)^2=1$, since the self-intersection of a conic in $S(1,2)$ is $1$.
So blowing up $s^\prime=V\cap E_s$ gives $E_{s^\prime}\cong \F_3$. $\X_{s^\prime}$ cuts $(E_s)_{s^\prime}\equiv e_3$ in two points, cuts a fiber $f_3$ in two points and cuts $V_{s^\prime}\equiv e_3+3f_3$ in four fixed double points. These double points are the intersections of the four double lines of $\X$ with $E_{s^\prime}$. Then:
\[ \X_{s^\prime}\equiv 2e_3+8f_3 \]
having four double base points. 

Using Lemma \ref{pr_imagemF1F2}, $\X_{s^\prime}$ corresponds, in $\P^2$, to degree eight curves with four double points, one point of multiplicity six and two double points infinitely near to it. After three standard quadratic transformations, one maps these curves to conics with no base points. 

As noted above, before the blow up at $s^\prime$, the moving part of $\X_s$ intersected it in a non complete linear series. Then, after this blow up, the restriction of $\X_{s^\prime}$ to $(E_s)_{s^\prime}\equiv e_3$ is a non complete linear series of degree two and dimension one, so it maps the $(-3)$ section to a double line, instead of a conic. Therefore, the base point free linear system of conics we have found is not complete and $\X_{s^\prime}$ maps $E_{s^\prime}$ to a projection of a Veronese surface with a double line in the image of $e_3$.
\\ \end{proof}

\subsection{Singularities of $X$}
\label{sl_singularities_sec}

Note that by hypothesis (H) $X$ is normal, therefore regular in codimension one.  Remember that we are supposing that the base locus of $\X$ has no embedded points. The following result will guide our search for singularities of $X$.

\begin{lemma}\label{sl_singularidades_vem_de_C6}
Let $x_q$ be an isolated singular point or a general point in a singular curve of $X$. Then $x_q$ does not lie on the base locus of $\tau$ and it is mapped to a singular point $q$ of $C_6$ not lying on $r$. This includes general points in non reduced components of $C_6$.
\end{lemma}
\begin{proof}
If $x_q$ is an isolated singular point of $X$, by Terracini's Lemma it cannot lie on a general tangent space of $X$, therefore it does not lie on the base locus of $\tau$. For the same reason, a singular curve cannot lie on a general tangent space, so a general point of it does not lie on the base locus of $\tau$. It may happen though that $T_xX$ contains one or more points of this curve, but it does not contain general points.

Since $\sigma$ (the inverse of $\tau$) restricted to $\P^3\setminus V$ is an isomorphism, $q=\tau(x_q)$  is contained in $V$. Since $\sigma(V)=x$ and $X$ is smooth in $x$, $q$ lies on the base locus of $\X$. Remember that this base locus has no isolated points. In other words, $x_q$ lies on a surface in $X$ that is contracted by $\tau$ to a curve in $V$.

The line $r$ is mapped by $\sigma$ to the line $\l_x$ through $x$, which lies on $T_xX$. Therefore $q\notin r$.

If $q$ is not a singular point of $C_6$, we will prove that $x_q$ is a smooth point of a general tangent hyperplane section of $X$ at $x$ through $x_q$. In particular, $x_q$ is not a singular point of $X$. 

Since $x_q$ is either an isolated singular point or a general point of a singular curve of $X$, a general tangent hyperplane section of $X$ at $x$ through $x_q$ is mapped by $\tau$ to a general plane through $q$.
So, in order to prove the assertion, we study the image of a general plane $\Omega$ through $q$. 

Note that a general plane is cut by $\X$ in degree seven curves having six double points, one point of multiplicity four and one simple base point. This is the case of $\Omega$ if $q$ is a smooth point of $C_6$ not lying on $s$ or a point in $s\setminus C_6$, so this plane is mapped to a smooth hyperplane section of $X$.

Suppose then that $q$ is a smooth point of $C_6$ lying on $s$ (in particular, $s\not\subset C_6$). Then $\X\cap\Omega$ consists of degree seven curves with multiplicity four in $q=s\cap \Omega$, six double points in $C_6\cap \Omega$ and one simple point in $r\cap\Omega$. Since $q\in C_6$, one of the six double points $q^\prime$ lies infinitely near to $q$. This can also be verified by blowing up $q$ and noting that $\X_q$ consists of the line through $s_q$ and $(C_6)_q$ with multiplicity two and two moving lines through $s_q$ (see the proof of Lemma \ref{sl_intersecoes_de_C6_com_2retas_do_ruling}). Since $\Omega_q$ is a general line of $E_q$, $\X_q$ intersects it in a fixed double point and two moving points. This explains the double point $q^\prime$ of $\X\cap\Omega$ infinitely near to $q$. Since $q$ is a smooth point of $C_6$, $\X\cap\Omega$ has no other double points infinitely near to $q$ or $q^\prime$.

Now it easily follows that the image of $\Omega$ is not singular in $x_q$. Blowing up $q$, the exceptional curve $\Omega_q\subset \Omega$ is mapped to a conic. Blowing up $q^\prime$, this second exceptional curve is also mapped to a conic. Then none of these curves are contracted. Moreover, $\sigma$ is an isomorphism outside $V$, so the only curve in $\Omega$ through $q$ that is contracted is $V\cap\Omega$, which is mapped to the smooth point $x$. Hence $x_q$ is a smooth point in the image of $\Omega$.
\\ \end{proof}

We will begin studying the singularities of $C_6$ that do not lie on $s$:

\begin{prop}\label{sl_pontos_duplos_de_X}
Let $q\in C_6$ be a singular point of $C_6$ outside $s$. Then $q$ is mapped to a double point $x_q$ of $X$.  Moreover, if $q$ does not lie on a non reduced component of $C_6$, then the tangent cone of $X$ in $x_q$ has rank four. 
\end{prop}
\begin{proof}
By Corollary \ref{sl_decomposicoes_de_C6}, $C_6$ is either: the union of $r$ and five lines of the ruling of $V$; the union of $2s$ and four lines of the ruling; or the union of a curve of degree $d$, which is smooth outside $s$, and $6-d$ lines of the ruling. As seen in the proof of this result, this curve is of type $(d-1,d-2)$, so it intersects a line of the ruling in one point (possibly in a second point lying on $s$, since $\bar{\tau}$ is $2:1$ in $s$). Therefore each intersection outside $s$ of this degree $d$ curve with a line of the ruling is transversal.

Therefore, if $q$ is a singular point of $C_6$ not lying on $s$, it is either:
\begin{itemize}
\item[(a)] the transversal intersection of a line of the ruling of $V$ and an irreducible curve of degree $d>1$, which is smooth outside $s$;
\item[(b)] the transversal intersection of a line of the ruling of $V$ and $r\subset C_6$;
\item[(c)] a point in a line of the ruling of $V$ which is a non reduced component of $C_6$ (and possibly (a) or (b)).
\end{itemize}

Remember that, by Proposition \ref{sl_baselocus}, $\X$ has multiplicity two in $C_6$. Then in all three cases above, blowing up $q$,  $\X_q$ has degree two and has two or more double points. These points lie on the line $V_q$, since $V$ is smooth outside $s$. Then $\X_q=2V_q$ and $q$ is mapped to a point $x_q\in X$.

To prove that $x_q$ is a double point, we'll use Lemma \ref{pr_truque_secao_tgente}. Let $\Omega$ be a general plane containing $q$. The curves $\X\cap \Omega$ have multiplicity two in $q$. Blowing up $q$, there is another double point in $e=E_q\cap \Omega$, since $\X_q=2V_q$. Blowing up this point, we get $e^2=-2$ in $\Omega$, with $\X\cap \Omega$ having no intersection with $e$. So $e$ is mapped to a double point in the image of $\Omega$. There may be other infinitely near double points. By Lemma \ref{pr_truque_secao_tgente}, $x_q$ is a double point of $X$.

Now we have to prove that, in cases (a) and (b), the tangent cone of $X$ in $x_q$ is a rank four quadric cone. This will be done using Lemma \ref{pr_truqueconetangente}.

Let $\l$ be the line of the ruling through $q$ and let $C$ be the other component of $C_6$ through $q$. In case (b), $C=r$ and $\X$ contains a line $r^\prime$ infinitely near to $r$ (see Proposition \ref{sl_imagem_r_e_s}). Let $\wX$ be the linear system of surfaces in $\X$ having multiplicity three in $q$. 

Blowing up $q$, in both cases $\wX_q$ has multiplicity two in $\l_q$ and $C_q$. In case (b) it has a further base point infinitely near to $C_q$, corresponding to $r^\prime$. Note that $C_q$ and $\l_q$ are not infinitely near, since the intersection of $C$ and $\l$ is transversal. All base points must lie on $t=V_q$, which is a line, since $q\notin s$. Then $\wX_q$ must contain $t$ and:
\[ \wX_q=t+\{ \text{conics through $\l_q$ and $C_q$} \} \]
The moving conics map $E_q$ to a smooth quadric surface.

Since we know that $x_q$ is a double point of $X$, Lemma \ref{pr_truqueconetangente} implies that the projectivization of the tangent cone of $X$ in $x_q$ is a smooth quadric surface. Therefore the rank of the tangent cone is four.
\\ \end{proof}

Next, we discuss the case in which $q$ lies on $C_6\cap s$. We will first study the intersection of a component of $C_6$ with the lines of $V$ through a smooth point of this component.

\begin{lemma}\label{sl_intersecoes_de_C6_com_2retas_do_ruling} 
Let $\l$ and $\l^\prime$ be distinct lines of the ruling of $V$ intersecting in a point $q\in s$. Let $C$ be an irreducible component of $C_6$ having degree $d>1$ and being smooth in $q$. Then, besides $q$, $C$ intersects only one of these lines, say $\l$, in another point.

Moreover, after the blow up at $q$, $E_q$ is mapped to a point in $X$ if and only if $l\subset C_6$. If this is not the case, $\X_q$ has a fixed double line $t^\prime$. On the other hand, if $\l^\prime\subset C_6$ then, blowing up $s$ and $t^\prime$, $E_{t^\prime}$ is mapped to a point in $X$.
\end{lemma}
\begin{proof}
As seen in the proof of Corollary \ref{sl_decomposicoes_de_C6}, $C$ is the image via $\bar{\tau}$ of a curve $\hat{C}$ in $E\cong \P^2$ of degree $d-1$ having multiplicity $d-2$ in $p$. The two lines $\l$ and $\l^\prime$ are images of lines $\hat{\l}$ and $\hat{\l}^\prime$ through $p$. 

Let $\hat{q}=\hat{\l}\cap \hat{s}$ and $\hat{q}^\prime=\hat{\l}^\prime\cap \hat{s}$, which are different points, since $\l$ and $\l^\prime$ are distinct. Then both $\hat{q}$ and $\hat{q}^\prime$ are mapped by $\bar{\tau}$ to $q\in s$. Since $C$ is smooth in $q$, $\hat{C}$ passes through only one of these points, say $\hat{q}^\prime$. Therefore:
\begin{align*}
\hat{C} \cap \hat{\l} &= (d-2)p+\{ \text{one point} \} \\ 
\hat{C} \cap \hat{\l}^\prime &= (d-2)p+ \hat{q}^\prime
\end{align*} 
Then it follows that $C$ intersects $\l$ in $q$ and in a second point and intersects $\l^\prime$ only in $q$. This proves the first part.

By Lemma \ref{sl_conetangente_de_V_e_secoes_planas}, the tangent cone of $V$ in $q$ is the union of the plane $\Pi$ containing $s$ and $\l$, and the plane $\Pi^\prime$ containing $s$ and $\l^\prime$. In particular, the tangent line $T_qC$ of $C$ in $q$ lies on one of these planes. 

Since $C$ is of type $(d-1,d-2)$, it intersects $s$ in $d-1$ points, including $q$. Note that $T_qC$ does not lie on $\Pi\setminus s$. Indeed, this would imply that:
\[ C\cap \Pi = 2q+\{d-2\text{ points}\}+q^\prime \]
where $q^\prime$ is the second point of intersection of $C$ with $\l$. Then $C$ would be contained in $\Pi$. But $V\cap \Pi=2s+l$. Hence $T_qC\in \Pi^\prime$.

Blow up $q$. Set $t=\Pi_q$ and $t^\prime={\Pi^\prime}_q$. Then $\X_q$ has multiplicity four in $s_q=t\cap t^\prime$ and multiplicity two in $C_q\in t^\prime$. 

If $l\subset C_6$, then $\X_q$ has also a double point in $\l_q\in t$, which implies that:
\[ \X_q= 2t+2t^\prime \]
and $E_q$ is mapped to a point.

On the other hand, if $\l$ is not contained in $C_6$ the only other possible components through $q$ are $\l^\prime$ and lines infinitely near to it (see Corollary \ref{sl_decomposicoes_de_C6}). So $\X_q$ will only have further double points in $t^\prime$ and:
\[ \X_q= 2t^\prime+ \{\text{pairs of lines through }s_q\} \]
which maps $E_q$ to a conic.

Now suppose $\l^\prime\subset C_6$. Blow up $s$, which has multiplicity four for $\X$. Now, the line $t^\prime={\Pi^\prime}\cap E_q$ has normal bundle:
\[ N_{t^\prime} = \mO_{\P^1}(-1)\oplus \mO_{\P^1}(0) \]

Blow up $t^\prime$. In $E_{t^\prime}\cong \F_1$, $\X_{t^\prime}$ has no intersection with $(E_q)_{t^\prime}\equiv e_1$, then $\X_{t^\prime}\equiv 2e_1+2f_1$. It has double points in the intersections with $\l^\prime$ and $C$, so it is a fixed double curve and $E_{t^\prime}$ is mapped to a point.
 
Figure \ref{sl_figura_q_em_C.s} illustrates these blow ups when $l\not\subset C_6$, but $l^\prime\subset C_6$.
\begin{figure}
\centering
\includegraphics[trim=10 10 110 560,clip,width=1\textwidth]{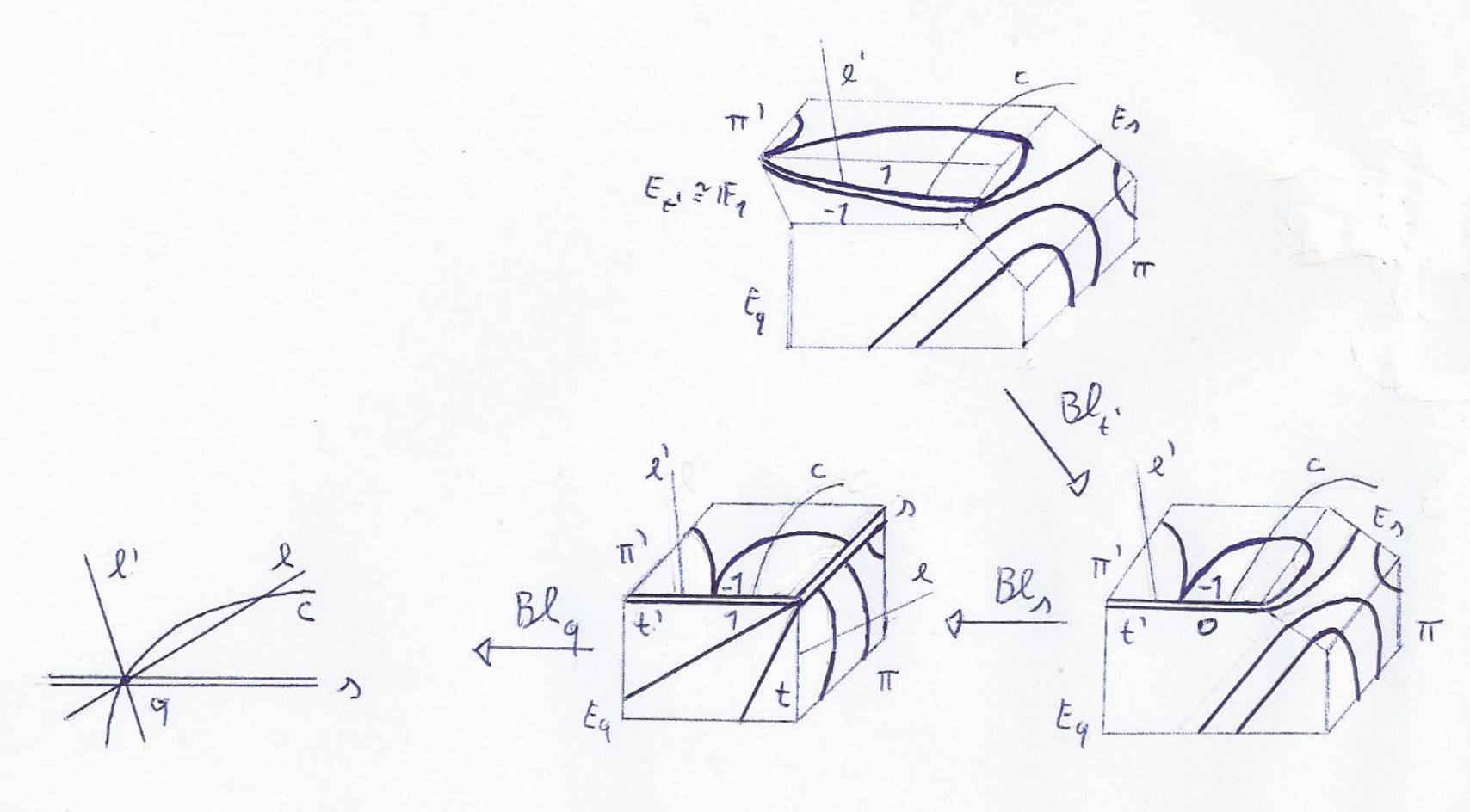}
\caption{Blow up of $\X$ in $q\in C\cap s$ when $\l\not\subset C_6$ and $\l^\prime\subset C_6$}\label{sl_figura_q_em_C.s}
\end{figure}
\\ \end{proof}

Before studying the next type of singularity, we give a Lemma and make some considerations:

\begin{lemma}\label{sl_cubicas_por_C6}
Let $C$ be a non-degenerate reduced irreducible curve of degree six in $\P^3$. Suppose that either $C$ is rational and has a 5-secant line or that $p_a(C)\geq 1$. Then there are at least two distinct cubic surfaces containing $C$.
\end{lemma}
\begin{proof}
The exact sequence:
\[ 0\to \mathcal{I}_{C/\P^3}(3) \to \mO_{\P^3}(3) \to \mO_C(3) \to 0 \]
gives the cohomology exact sequence:
\[ 0\to H^0(\mathcal{I}_{C/\P^3}(3)) \to H^0(\mO_{\P^3}(3)) \to H^0(\mO_C(3)) \to H^1(\mathcal{I}_{C/\P^3}(3))  \to 0 \]

By Riemann-Roch, we have:
\[ h^0(\mO_C(3))=h^1(\mO_C(3))+18+1-p_a(C)=19-p_a(C) \]
since $2p_a(C)-2\leq 6<\deg(\mO_C(3))$ gives $h^1(\mO_C(3))=0$.

On the other hand, $h^0(\mO_{\P^3}(3))=\binom{6}{3}=20$. If $p_a(C)\geq 1$, then: 
\[ h^0(\mathcal{I}_{C/\P^3}) = 20-19+p_a(C) + h^1(\mathcal{I}_{C/\P^3}) \geq 2 \]

If $C$ is rational, then by \cite{glp} $h^1(\mathcal{I}_{C/\P^3})> 0$ if and only if it has a 5-secant line, giving, in this case:
\[ h^0(\mathcal{I}_{C/\P^3}) = 20-19+0 + h^1(\mathcal{I}_{C/\P^3}) >1 \]

Hence in both cases there are at least two different cubic surfaces containing $C$.
\\ \end{proof}

If we are in the general case, that is, if $C_6$ is smooth, then it is rational and has a 5-secant line, namely $s$. Then this Lemma shows that there is a cubic surface other than $V$ containing $C_6$. Therefore there is such a cubic surface in the other special cases (e.g. $C_6$ reducible). We will not give the details here.

Note that any cubic surface containing $C_6$ must also contain the lines $r$ and $s$.  Indeed, if none of these lines is contained in $C_6$, then $C_6$ intersects them in five and four points respectively. If $s\subset C_6$, then $C_6$ contains four lines of the ruling of $V$, intersecting $r$ in four points. And if $r\subset C_6$, there are five lines of the ruling in $C_6$, so it intersects $s$ in five points. 

This implies that, if $S$ is a cubic surface different from $V$ containing $C_6$, then:
\begin{equation}\label{sl_eq_baselocus_mQ}
S\cap V = C_6+2s+r
\end{equation}
 where the multiplicity two in $s$ follows from the fact that $V$ is singular along $s$. Note that the general element of the pencil of cubic surfaces generated by $V$ and $S$ is irreducible, since $V$ is.

Next, the case in which a reduced component $C$ of $C_6$ is singular in $q$ will be studied. As shown in the proof of Corollary \ref{sl_decomposicoes_de_C6}, $C$ is of type $(d-1,d-2)$, so its image in $E\cong \P^2$ via $\bar{\sigma}$ has no singular points, except for $p$. Then $C$ is singular in $q$ if only if $q\in s$ and $\bar{\sigma}(C)$ either contains the two points of $\hat{s}$ that are mapped to $q\notin\{q_1^s,q_2^s\}$ or is tangent to $\hat{s}$ in $\bar{\sigma}(q)$, with $q\in\{q_1^s,q_2^s\}$. In any of these cases, $q$ is a double point of $C$, since the map $\bar{\tau}$ restricted to $\hat{s}$ is $2:1$. 

\begin{lemma}\label{sl_conetangente_de_C}
Let $C$ be an irreducible component of $C_6$ with degree $d>1$ and $q$ be a point in $C\cap s$.

Suppose $q$ is different from $q_1^s$ or $q_2^s$. Let $\Pi_1$ and $\Pi_2$ be the two planes of the tangent cone of $V$ in $q$. After blowing up $q$, set $t_i=\Pi_i\cap E_q$, for $i=1,2$. If $C$ is singular in $q$, then one of the following holds:
\begin{itemize}
\item[(i)] the tangent cone of $C$ in $q$ is a pair of lines different from $s$, one lying on $\Pi_1$ and the other in $\Pi_2$;
\item[(ii)] the tangent cone of $C$ in $q$ is $s$ and another line lying on one of the planes, say $\Pi_1$. Then after blowing up $q$ and $s$, $C$ intersects $E_s\cap E_q$ in $(t_2)_s$;
\item[(iii)] the tangent cone of $C$ in $q$ is $2s$. After blowing up $q$ and $s$, $C$ intersects $E_s\cap E_q$ in $(t_1)_s$ and $(t_2)_s$.
\end{itemize} 

On the other hand, if $q$ is $q_1^s$ or $q_2^s$ and $C$ is singular in $q$, the tangent cone of $C$ in $q$ is a double line; if $C$ is smooth in $q$, it is tangent to the line of the ruling of $V$ through this point.
\end{lemma}
\begin{proof}
Let $\hat{C}$ be the strict transform of $C$ in $E$ via $\bar{\tau}$, a curve of degree $d-1\leq 5$ having multiplicity $d-2$ in $p$ (see the proof of Corollary \ref{sl_decomposicoes_de_C6}).

Suppose $q$ is not $q_1^s$ or $q_2^s$. Then it has two preimages via $\bar{\tau}$, say $\hat{q}_1,\hat{q}_2\in \hat{s}$. Let $\l_1\in\Pi_1$ and $\l_2\in \Pi_2$ be the two lines of the ruling through $q$, such that the strict transform of $\l_i$ via $\bar{\tau}$ passes through $\hat{q}_i$. 

We will first give a correspondence between lines through $\hat{q}_1$ in $E$ and lines through $q$ in $\Pi_1$. 
If the line in $\Pi_1$ is $s$ or $\l_1$, it corresponds to its strict transform in $E$. Given a line $\l$ through $q$ in $\Pi_1$ different from $s$ or $\l_1$, $\l$ is the tangent line at $q$ of a smooth conic $C_\l$. Then $\l$ corresponds to the strict transform of $C_\l$ via $\bar{\tau}$, which is a line $\hat{\l}$ through $\hat{q}_1$.

Now, if a curve $D$ is tangent to $\hat{\l}$ in $\hat{q}_1$, then $\bar{\tau}(D)$ is tangent to $C_\l$ in $q$. In other words, this correspondence associates to a tangent line of a curve in $\hat{q}_1$ a line in the tangent cone at $q$ of the image of this curve via $\bar{\tau}$.

The same reasoning gives a correspondence of lines through $q$ in $\Pi_2$ and lines through $\hat{q}_2$ in $E$. 

Since $\hat{C}$ is smooth outside $p$, $C$ is singular in $q$ if and only if $\hat{C}$ contains both $\hat{q}_1$ and $\hat{q}_2$. If $\hat{C}$ has no tangency with $\hat{s}$ in these two points, its tangent lines in these points correspond to a line in $\Pi_1$ and a line in $\Pi_2$ different from $s$. These two lines form the tangent cone of $C$ in $q$, giving (i).

If $\hat{C}$ is tangent to $\hat{s}$ in $\hat{q}_2$, but not in $\hat{q}_1$, the tangent cone of $C$ in $q$ is the union of $s$ and a line in $\Pi_1$. Blowing up $q$, $C$ intersects $E_q$ in $s_q$ and another point in $t_1$. Blowing up $s$, $C$ intersects $E_q\cap E_s$ (which is the fiber over $s_q$ in $E_s$) in one of the two points intersected by $V$, that is, $(t_1)_s$ or $(t_2)_s$. 

But the tangent line of $\hat{C}$ in $\hat{q}_2$ is a limit of lines through $\hat{q}_2$ different from $\hat{s}$. These lines correspond to lines in $\Pi_2$ and, after blowing up $q$ and $s$, to lines intersecting $E_q$ in $t_2$, with limit point $(t_2)_s$. Therefore $C$ intersects $E_q\cap E_s$ in $(t_2)_s$, giving (ii).

If $\hat{C}$ is tangent to $\hat{s}$ in both $\hat{q}_1$ and $\hat{q}_2$, $C$ has a double point in $q$ with tangent cone $2s$. After blowing up $q$, $C$ has a double point in $s_q$. The reasoning made above shows that blowing up $s$, $C$ intersects $E_q\cap E_s$ in $(t_1)_s$ and  $(t_2)_s$, proving (iii).

Now suppose $q$ is $q_1^s$ or $q_2^s$ and let $\hat{q}$ be the preimage of $q$ via $\bar{\tau}$. Assume first that $C$ is singular in $q$. This happens if and only if $\hat{C}$ is tangent to $\hat{s}$ in $\hat{q}$. The order of the tangency determines the type of the singularity of $C$, which is always a double point. 

If the contact of $\hat{C}$ with $\hat{s}$ is simple, it is defined locally by the contact of a conic and $\hat{s}$. Since such conic intersects $\hat{s}$ in $2\hat{q}$, it belongs to $\II_{X,x}$. Then it is mapped to a plane section of $V$ through $q$, which is a cuspidal cubic by Lemma \ref{sl_conetangente_de_V_e_secoes_planas}. So in this case the tangent cone of $C$ is a double line different from $s$. 

We will now prove that if $\hat{C}$ has a higher order contact with $\hat{s}$ in $\hat{q}$, then the  tangent cone of $C$ is also a double line. For that, suppose that the other components of $C_6$ (which are lines of the ruling of $V$, by Corollary \ref{sl_decomposicoes_de_C6}) do not contain $q$. There is no harm in making this assumption, since we could just define $C_6^\prime$ satisfying it and proceed with the proof for this new curve. 

By Lemma \ref{sl_cubicas_por_C6} (and further remarks), there is a pencil of cubic surfaces containing $C_6$, so let $S$ be a cubic surface containing $C_6$ and having transversal intersection with $V$ in $q$, that is:
\[ \mult_q(S\cap V)=(\mult_q S)\cdot (\mult_q V) \]

Then, by \eqref{sl_eq_baselocus_mQ} and by the assumption made above, $S$ has multiplicity two in $q$. In particular, the tangent cone $\mC_qS$ has degree two. Lemma \ref{pr_interseccao_conetangente} implies:
\[ \mC_q(S\cap V) = \mC_qS \cap \mC_qV \] 

Since $2s$ is part of the intersection $S\cap V$, it is a component of $\mC_q(S\cap V)$. But $\mC_qV$ is a double plane, so $\mC_q(S\cap V)$ is the union of $2s$ and a double line. By \eqref{sl_eq_baselocus_mQ}, this double line is the tangent cone of $C_6$ in $q$. Since $C\subset C_6$ has multiplicity two in $q$, this is the tangent cone of $C$ in $q$ and the assertion is proved.

Finally, suppose $C$ is smooth in $q\in \{ q_1^s,q_2^s\}$. Let $\l$ be the line of the ruling of $V$ through $q$.

Since $C$ is smooth in $q$, the tangent line of $\hat{C}$ in $\hat{q}$ is not $s$ and, by Bezout's theorem, it is not the line through $p$ and $\hat{q}$ ($\hat{C}$ has degree $d-1$ and multiplicity $d-2$ in $p$). Then this tangent line is mapped to a conic in $V$.

The union of this conic and $\l$ forms a plane section of $V$, so the conic cuts $\l$ in two points, one of them being $q$. But the other point must also be $q$, since the strict transform of $\l$ via $\bar{\tau}$ intersects the tangent line of $\hat{C}$ only in $\hat{q}$. Therefore the conic is tangent to $\l$ in $q$, and so is $C$.
\\ \end{proof}

We can now conclude the study of singularities in $C_6\cap s$.

\begin{prop}\label{sl_singularidades_em_C6.s}
Let $q$ be a singularity of $C_6$ in $s$. Then, blowing up $q$,  $E_q$ is contracted to a triple point of $X$ if and only if one of the following holds:
\begin{itemize}
\item[(a)] $q$ is a singularity of an irreducible reduced component of $C_6$;
\item[(b)] there are two distinct lines of the ruling through $q$ contained in $C_6$;
\item[(c)] $q$ is $q_1^s$ or $q_2^s$ and $C_6$ contains the line of the ruling through $q$;
\item[(d)] $C_6$ contains a line of the ruling through $q$ and a curve intersecting this line in two points;
\item[(e)] $s$ is contained in $C_6$.
\end{itemize}

If this is not the case, then $\X_q$ has a fixed curve, which is mapped to a double point of $X$.
\end{prop}
\begin{proof}
Let $q\in s$ be a singular point of $C_6$. Then it is either a singular point of an irreducible reduced component of $C_6$, a point lying on two or more components of $C_6$ or a point of a non reduced component. We will reorganize these possibilities in the following way:
\begin{itemize}
\item[(i)]  $q$ satisfies (a), (b) or (e);
\item[(ii)] $q$ does not satisfy (i) and is the intersection of two or more distinct components of $C_6$;
\item[(iii)] $q$ does not satisfy (i) or (ii) and lies on a non reduced component of $C_6$.
\end{itemize}

By Corollary \ref{sl_decomposicoes_de_C6}, if $q$ satisfies (ii), it is the intersection of a (possibly multiple) line and a curve smooth in $q$, since (a) and (b) are excluded from (ii). By Lemma \ref{sl_intersecoes_de_C6_com_2retas_do_ruling}, this  possibility fits in (c), (d) or in:
\begin{itemize}
\item[(f)] $C_6$ contains a line of the ruling through $q$, a curve intersecting this line in $q$ with multiplicity one and no other component through $q$.
\end{itemize}

Now, if $q$ satisfies (iii), it fits either in (c) or in:
\begin{itemize}
\item[(g)] $q$ is not $q_1^s$ or $q_2^s$, a line of the ruling through $q$ is a non reduced component of $C_6$ and no other component contains $q$.
\end{itemize}

The next step is to prove that $E_q$ or a curve in it is mapped to a point in $X$ and in which cases each occur.

The situation of case (a) is described in Lemma \ref{sl_conetangente_de_C}. If $q$ is not $q_1^s$ or $q_2^s$, the blow up at $q$ gives $V_q=t_1+t_2$, a pair of lines. Since $\X$ has multiplicity four in $s$ and two in $C_6$, by Lemma \ref{sl_conetangente_de_C}, $\X_q$ has multiplicity four in $s_q$ and multiplicity two in two  points of $E_q$, one of them in $t_1$ and the other in $t_2$. These points can be infinitely near to $s_q$. Then $\X_q=2t_1+2t_2$.

If $q$ is $q_1^s$ or $q_2^s$, by Lemma \ref{sl_conetangente_de_C} the tangent cone of $C_6$ in $q$ is a double line. Blow up $q$ and set $V_q=2t$. Then $C_6$ is either tangent to $E_q$ or has a double point in it, that is, $(C_6)_q$ is a double point. Then $\X_q$ has multiplicity four in $s_q$ and in $(C_6)_q$. Therefore $\X_q=4t$, since both $s$ and $C_6$ lie on $V$.

This proves that, in case (a), $E_q$ is mapped to a point. 

Now suppose $q$ satisfies (b). Each of the two lines of the ruling lies on one of the planes of the tangent cone of $V$ in $q$. Blowing up $q$, $\X_q$ has multiplicity four in $s_q$ and multiplicity two in two points, these three being non collinear. Hence it consists of two fixed double lines, and $E_q$ is mapped to a point.

In case (c), let $\l$ be the line of the ruling through $q$. Then either $\l$ is a non reduced component of $C_6$ or it is a simple line in $C_6$ and there is another component $C$ containing $q$.

First, suppose $\l$ is a double line, that is, $C_6=2\l+C_4$. The case in which it has multiplicity greater than two is analogous. Then there is a line $\l^\prime$ infinitely near to $\l$, determined by the surface $V$.

Let $\Omega$ be a general plane containing $\l$. This plane intersects $\X$ in $2\l$ and degree five curves. These curves have a double point in $q=s\cap \Omega$ and three double points in the intersections of $C_4$ with $\Omega$ outside $\l$. There is a fifth double point corresponding to the line $\l^\prime$. Since $\Omega$ intersects $V$ in $\l$ and a conic, which is tangent to $\l$ in $q$ by Lemma \ref{sl_conetangente_de_C}, this fifth double point is infinitely near to $q$, corresponding to the direction of $\l$.

Blow up $q$. Now $\X$ intersects $\Omega$ in $2\l$ and curves having a double point in $E_q\cap \l$. Therefore $\X$ has multiplicity four in $\l_q$. Since it also has multiplicity four in $s_q$, it follows that $\X_q=4t$, where $t=V_q$ is the line through $\l_q$ and $s_q$. Hence $E_q$ is mapped to a point.

Now suppose that $\l$ is a simple line in $C_6$ and there is another component $C$ containing $q$. If $C$ is singular in $q$, it fits in (a). If it is smooth in $q$, then by Lemma \ref{sl_conetangente_de_C} $C$ is tangent to $\l$ in $q$. Blowing up $q$, $C$ and $\l$ intersect $E_q$ in the same point.

Repeating the reasoning made above, $\X$ intersects, after the blow up at $q$, a general plane $\Omega$ through $q$ in $2\l$ and curves having a double point in $\l_q$. 
Therefore, $\X$ has multiplicity four in $s_q$ and $\l_q$, which implies that $\X_q$ is a fixed line with multiplicity four. Then $E_q$ is mapped to a point. 

If $q$ satisfies (d), then by Lemma  \ref{sl_conetangente_de_C} $q$ is not $q_1^s$ or $q_2^s$, and by Lemma \ref{sl_intersecoes_de_C6_com_2retas_do_ruling} $E_q$ is mapped to a point. This same Lemma asserts that in case (f), $\X_q$ has a fixed double line, which is mapped to a point.

In case (e), $q$ can be any point in $s$. As seen in the Proof of Proposition \ref{sl_imagem_r_e_s}, the blow up at $s$ gives:
\[ \X_s \equiv 2s^\prime + \{ (1,0) \} \equiv (3,4) \]
Then $s$ is mapped to a line $L$.  Note that the fiber  over $q$ in $E_s$ intersects the double line $s^\prime\equiv (1,2)$ in two points, which coincide if $q$ is $q_1^s$ or $q_2^s$. Then blowing up $q$ (instead of blowing up $s$), $\X_q$ has multiplicity four in $s_q$ and multiplicity two in two points infinitely near to $s_q$. Therefore $\X_q$ is a pair of (possibly coincident) fixed double lines and $E_q$ is mapped to a point in $L$. 

Finally, suppose $q$ satisfies (g). Again suppose $\l$ is a double line, the other cases being analogous.  Then $C_6=2\l+C_4$. Let $\Pi$ be the plane containing $\l$ and $s$. 
See Figure \ref{sl_figura_q_em_C.s_caso_g} for a sketch of the blow ups that will now be done.

\begin{figure}[tb]
\centering
\includegraphics[trim=17 490 140 105,clip,width=1\textwidth]{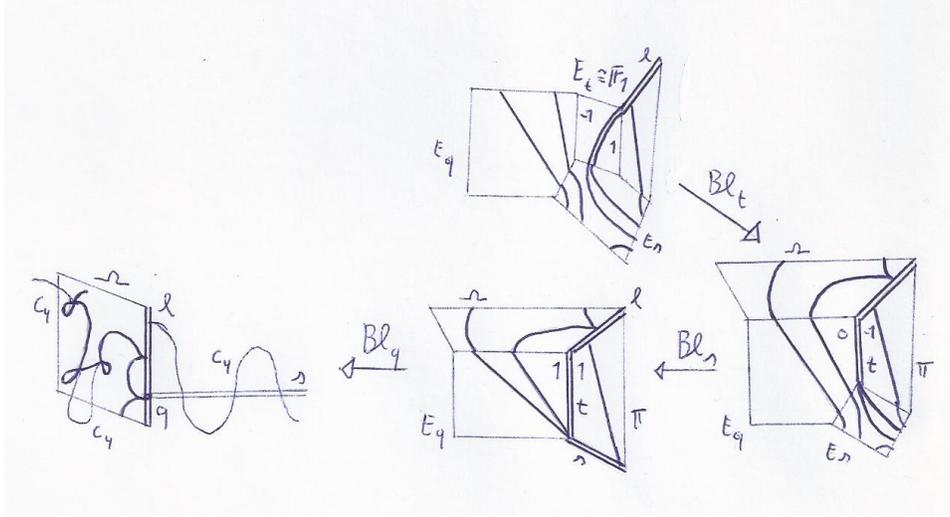}
\caption{$\X$ and the blow ups in case (g) of Proposition \ref{sl_singularidades_em_C6.s}}\label{sl_figura_q_em_C.s_caso_g}
\end{figure}

Let $\Omega$ be a general plane containing $\l$. Then $V\cap\Omega$ consists of $\l$ and a conic intersecting this line in two distinct points. This follows from Lemma \ref{sl_conetangente_de_V_e_secoes_planas} and the fact that $q$ is not $q_1^s$ or $q_2^s$. On the other hand, $\X\cap \Omega$ consists of $2\l$ and degree five curves with three double points in the intersections with $C_4$ not lying on $\l$, a double point in $q$ and a fifth double point. This last point must lie on the conic $V\cap\Omega$. And since it corresponds to a line infinitely near to $\l$, it must lie on $\l$. Then this point is not infinitely near to $q$, otherwise the conic would be tangent to $\l$.

Blowing up $q$, $\X_q$ has multiplicity four in $s_q$ and a double point in the intersection of $\l$ with $E_q$. There is a second double point infinitely near to it, since there is another double line of $\X$ infinitely near to $\l$. But both double points lie on the line $t$ defined by $\l_q$ and $s_q$. 

Note that, unlike case (c), $\X$ does not have multiplicity four in $\l_q$. This follows from the fact that the degree five curves of $\X\cap\Omega$ do not have a double point infinitely near to $q$, as noted above. This gives:
\[ \X_q = 2t +\{\text{pairs of lines through }s_q \} \]
and $E_q$ is mapped to a conic.  Note that $t=\Pi_q$, since $\Pi$ contains both $s$ and $\l$. Then, after blowing up $s$, $t$ has normal bundle: 
\[ N_t = \mO_{\P^1}(-1)\oplus \mO_{\P^1}(0) \]

Blowing up $t$ gives $(E_q)_t\equiv e_1$. Since $\X_t$ does not intersect $E_q$, $\X_t\equiv 2e_1+2f_1$. It has two or more infinitely near double points (associated to $\l$) lying on $V_t\equiv e_1+f_1$. Then $\X_t$ is a double curve and $E_t$ is mapped to a point.

It remains to prove that in cases (a) to (e) the image $x_q$ of $E_q$ is a triple point; and that in cases (f) and (g) the image $x_q$ of the fixed double line in $E_q$ is a double point of $X$. This will be done using Lemma \ref{pr_truque_secao_tgente}. Let $\Omega$ be a general plane trough $q$. The plane $\Omega$ is cut by $\X$ in degree seven curves having, among other base points, multiplicity four in $q$. Since $C_6$ is singular in $q$, these curves have two or more double points infinitely near to $q$.

Blowing up $q$, in cases (a) to (e) $\X$ has two double points in $e=E_q \cap \Omega$. These double points represent the tangent cone of $V$ in $q$ and are infinitely near if $q$ is $q_1^s$ or $q_2^s$.

Blowing up these two points, one gets $e^2=-3$ in $\Omega$ having no intersections with $\X$. Therefore it is mapped to a triple point in the image of $\Omega$. By Lemma \ref{pr_truque_secao_tgente}, $x_q$ is a triple point of $X$.

In cases (f) and (g), $\X$ intersects $e$ in one fixed double point and two moving points. This fixed point is the intersection of $e$ with the fixed double line in $E_q$ mentioned above. Blow up the double point and let $e^\prime$ be the exceptional divisor in $\Omega$. Then $\X$ intersects $e^\prime$ in a fixed double point. Blowing it up, $(e^\prime)^2=-2$ in $\Omega$, so it is contracted to the double point $x_q$ in the image of $\Omega$. By Lemma \ref{pr_truque_secao_tgente},   it is a double point of $X$.
\\ \end{proof}

Note that part (e) of this Lemma and Proposition \ref{sl_imagem_r_e_s} give the following important result:

\begin{coro} \label{sl_s_vai_em_reta_tripla}
Suppose that $s$ is contained in $C_6$. Then the image $L$ of $s$ via $\sigma$ is a triple line of $X$.
\end{coro}

Suppose now that $C_6$ has a non reduced component. By Corollary \ref{sl_decomposicoes_de_C6}, it must be either $s$ or a line of the ruling of $V$. If $C_6$ contains a multiple line $\l$, then there are other lines infinitely near to $\l$ lying on the singular locus of $\X$. We will now study what type of singularities such lines produce on $X$.

\begin{lemma}\label{sl_retas_multiplas}
Let $\l$ be a line of the ruling of $V$ contained in $C_6$ with multiplicity $d\geq 2$, that is, $\l=\l_1\prec \l_2\prec \cdots \prec \l_d$ are double lines of $\X$. Then $\l_1\ldots, \l_{d-1}$ are mapped to double lines $R=R_1,\ldots,R_{d-1}$ of $X$, with $R_i\prec R_{i+1}$ and $\l_d$ is mapped to a cubic scroll through $x$, with $R$ being a line of its ruling.

Moreover, if $q=l\cap s$ is mapped to a triple point of $X$, then this  point lies on $R$, unless $d=2$ and $q$ is $q_1^s$ or $q_2^s$. 
\end{lemma}
\begin{proof}
Denote by $\Pi$ the plane containing $\l$ and $s$. Then:
\begin{equation}\label{sl_eq_retadupla_pi}
\X\cap \Pi=2\l+4s+\{\text{lines}\}
\end{equation}

\begin{figure}
\centering
\includegraphics[trim=7 6 40 550,clip,width=1\textwidth]{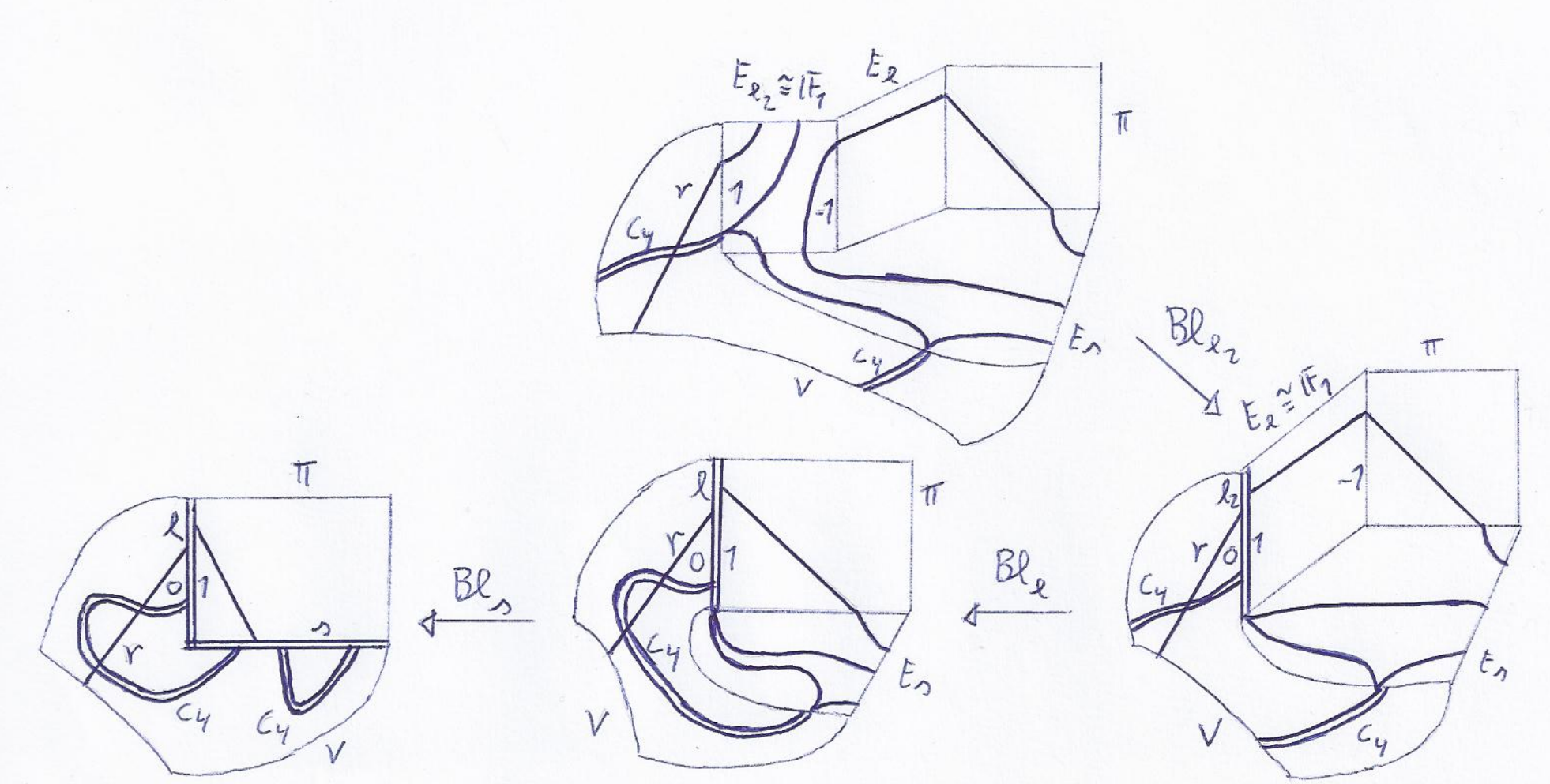}
\caption{The Blow ups for $d=2$}\label{sl_figura_reta_dupla}
\end{figure}

The blow ups that will now be done are illustrated, when $d=2$, in Figure \ref{sl_figura_reta_dupla}. Start blowing up $s$. After this, $\l$ is the complete intersection of $V$ and $\Pi$ and has normal bundle:
\[ N_\l = \mO_{\P^1}(0)\oplus \mO_{\P^1}(1) \]
So blowing up  $\l$ gives $E_{\l}\cong \F_1$ and $V_\l\equiv e_1+f_1$. Then $\l_2$ is the intersection of $E_{\l}$ with $V$. By \eqref{sl_eq_retadupla_pi} and since $\X$ has multiplicity two in $\l$, it follows that:
\begin{equation}\label{sl_eq_retadupla_ag}
\X_{\l} \equiv 2e_1+3f_1 \equiv 2\l_2 + \{ f_1 \}
\end{equation}
 where $\{ f_1\}$ represents the moving part of $\X_\l$, consisting of fibers of $E_\l$. This moving part maps $E_{\l}$ to a line $R$ in $X$. 

The line $\l_2=E_\l\cap V$ has normal bundle:
\[ N_{\l_2} = \mO_{\P^1}(1)\oplus \mO_{\P^1}(0) \]
Blow up  $\l_2$. If $d>2$, $\X$ has multiplicity two in $\l_3=V_{\l_2}\equiv e_1+f_1$ and its intersection with $E_{\l_2}$ is the same as the one given in \eqref{sl_eq_retadupla_ag}. Hence $E_{\l_2}$ is mapped to a line $R_2$. It is infinitely near to $R$, since $\l_2$ was infinitely near to $\l$. And this result repeats for $\l_i$, with $i\leq d-1$, giving $R_{i-1}\prec R_{i}$.

After blowing up $\l_1,\ldots,\l_{d-1}$, blow up $\l_d$, giving again in $E_{\l_d}\cong \F_1$, $E_{\l_{d-1}}\cap E_{\l_d}\equiv e_1$ and $V_{\l_d}\equiv e_1+f_1$.  Then $\X_{\l_d}$ cuts the section $e_1$ in one point, so:
\[ \X_{\l_d}\equiv 2e_1+3f_1 \] 
having one simple point in $r_{\l_d}$ and one double point in the intersection with the other component of $C_6$. Remember that, by Corollary \ref{sl_decomposicoes_de_C6}, this other component can be $r$ itself, $s^\prime\subset E_s$ or an irreducible curve of degree greater than one.

Therefore $\X$ maps $E_{\l_d}$ to a cubic scroll $S(1,2)$ through $x$.  Note that this scroll contains the line $R_{d-1}$, image of $e_1$. It is a line of the ruling not containing $x$. But $R_{d-1}$ is not a proper line in $X$, since it is infinitely near to $R_{d-2}$, so the scroll contains the lines:
\[ R=R_1\prec R_2 \prec \cdots \prec R_{d-1} \]
Hence $R$ is a line of the ruling of the cubic scroll.

The next step is to prove, using Lemma \ref{pr_truque_secao_tgente}, that these are double lines of $X$. So let $p$ be a general point of $\l$, which is mapped to a general point $P$ of $R$, and let $\Omega$ be a general plane in $\P^3$ through $p$.

The plane $\Omega$ is cut by $\X$ in degree seven curves having, among other base points, $d$ double points $p=p_1\prec \cdots \prec p_d$. Blowing up the base points gives a chain of exceptional divisors:
\[ E_1\prec E_2 \prec \cdots \prec E_d \]
with ${E_i}^2=-2$ for $i=1,\ldots,d-1$ having no intersection with $\X$, and ${E_d}^2=-1$ being intersected by $\X$ in two moving points. These are mapped to a singular point of $X$ of type $A_{d-1}$ and a conic. The singular point is $P$ itself, hence the general tangent hyperplane section of $X$ through $P$ has double points in $P=P_1\prec \cdots\prec P_{d-1}$. By Lemma \ref{pr_truque_secao_tgente}, $P$ is a double point of $X$.

Note that $P_i\in R_i$. In fact, if $\Omega$ is a general plane through $q\in \l$, then blowing up $s$ and $\l$ gives $\Omega_\l\equiv f_1$. Therefore, in each $E_{\l_i}$, $\Omega_{\l_i}$ is mapped to a point $P_i\in R_i$. This proves the first part of the Lemma.

Now suppose that $q=l\cap s$ is mapped to a triple point of $X$, that is, $E_q$ is contracted to a point (see Proposition \ref{sl_singularidades_em_C6.s}). Then after the blow up at $s$, the fiber $f^q$ over $q$ in $E_s$ contains the point $\l_s$. Since we have blown up $s$, it is $f^q$ which is contracted to a triple point. 

Note that $V_s$ is tangent to $f^q$ if and only if $q$ is $q_1^s$ or $q_2^s$, since the tangent cone of $V$ in these points is a double plane.  In this case, the point of tangency is $\l_s$.

Suppose $q$ is $q_1^s$ or $q_2^s$ and $d=2$. Blowing up $\l$, $f^q$ intersects $E_\l$ in a point of $\l_2=V_\l$. Blowing up $\l_2$, $E_{\l_2}$ is mapped to a cubic scroll and $E_\l$ is contracted to a line of its ruling, namely $R$. In this case, $f^q$ does not intersect either $V$ or $E_\l$ and its image is not contained in $R$.

Now suppose that $q$ is $q_1^s$ or $q_2^s$ and $d>2$. As it was just remarked, blowing up $s$ and $\l$, $f^q$ intersects $E_\l$ in a point of $\l_2=V_\l$. After these blow ups, $f^q$ and $V_q$ intersect transversally. Since $d>2$, blowing up $\l_2$ gives $E_{\l_2}\cong \F_1$ and:
\[ \X_{\l_2} \equiv 2e_1+3f_1 \equiv 2\l_3 + \{ f_1 \} \]
where $\{f_1\}$ is the moving part of $\X_2$, mapping $E_{\l_2}$ to $R_2$.  After this blow up, $f^q$ and $V_q$ do not intersect, which implies that $f^q\cap E_{\l_2}$ does not lie on $\l_3$. Therefore the image of $f^q$ is contained in $R_2$.

Finally, suppose $q$ is not $q_1^s$ or $q_2^s$. Then, after the blow up at $s$, $V_s$ is not tangent to $f^q$. Therefore, blowing up $\l$, $f^q$ no longer intersects $V_{\l}$ and:
\[ \X_{\l} \equiv 2e_1+3f_1 \equiv 2\l_2 + \{ f_1 \} \]
Since the point $f^q\cap E_\l$ does not lie on $\l_2=V_{\l}$, it is mapped to a point of $R$.
\\ \end{proof}

\subsection{A family of surfaces in $X$}
\label{sl_cubicsfamily_sec} 

Lemma \ref{sl_cubicas_por_C6} and the remarks after it assert that there is a cubic surface, other than $V$, containing $C_6$ and that any cubic surface containing $C_6$ must also contain the lines $r$ and $s$. So let $\mQ$ be the linear system of cubic surfaces in $\P^3$ containing $C_6$, $r$ and $s$. Then by \eqref{sl_eq_baselocus_mQ}  $\mQ$ has fixed intersection with $V$. Suppose that $s$ is not contained in $C_6$, as this case will be considered in Lemma \ref{sl_mQ_com_s_em_C6}.

The exact sequence:
\[ 0 \to \mQ-V \to \mQ \to \mQ\vert_V \to 0 \]
gives:
\[ 0 \to H^0(\P^3,\mQ-V) \to H^0(\P^3,\mQ) \to H^0(\P^3,\mQ\vert_V) \]
where the last map is clearly surjective.  Since $H^0(\P^3,\mQ-V)\cong \C$ and $\mQ$ has fixed intersection with $V$, it follows that $\mQ$ is a pencil.

Since $\mQ$ is generated by $V$ and another cubic, its base locus is given by \eqref{sl_eq_baselocus_mQ}, that is, $C_6+r+2s$. This means that the surfaces in $\mQ$ have in common a curve infinitely near to $s$. It also implies that if $s$ is not contained in $C_6$, $V$ is the only surface in $\mQ$ having a singular curve. 

Note that if $r\subset C_6$, the surfaces in $\mQ$ contain five lines of the ruling, which implies they contain $r$ and $r^\prime$ (see Lemma \ref{sl_r_s_em_C6}).

The following Lemma gives a rational representation of a cubic in $\mQ$ that will be used next:

\begin{lemma}\label{sl_representacao_mQ}
Suppose $s\not\subset C_6$ and let $S_3$ be a cubic in $\mQ$ different from $V$. Then there is a map
\[ \phi: \P^2 \tor S_3 \]
defined by cubic curves through six (possibly infinitely near) points, such that the line $r$ is the exceptional line of the blow up at a point in $\P^2$ and $s$ is the image of a conic through the other five points. 
\end{lemma}
\begin{proof}
Since $s\not\subset C_6$ and $S_3$ is not $V$, $S_3$ is not ruled. Fix a point $p\in S_3$ not lying on a line of $S_3$ and let $\L$ be the linear system of quadric surfaces through $p$ containing $r$ and $s$. These quadrics must also contain the line $\l_p$ through $p$ intersecting both $r$ and $s$. Therefore the base locus of $\L$ is a curve of type $(1,2)$ in each quadric and $\L$ has projective dimension two. Then it defines a map:
\[ \psi: S_3 \tor \P^2 \]

To show that it is birational, take a point $P\in\P^2$. Its preimage via $\psi$ is the set of points in the intersection of $S_3$ with two general quadrics of $\L$ not lying on the base locus of $\L$. Two quadrics of $\L$ intersect in $r$, $s$, $\l_p$ and a line $\l^\prime$ skew to $\l_p$. Since $\l_p\notin S_3$, this line intersects $S_3$ in $p$, a point in $r$ and a point in $s$, all of them lying on the base locus of $\L$. The line $\l^\prime$ intersects $S_3$ in a point in $r$, a point in $s$ and one further point. Then the preimage of $P$ is one point and $\psi$ is birational.

Let $\L^\prime$ be the linear system in $\P^2$ defining the inverse of $\psi$. It is defined by quartic curves.  Indeed, the residual intersection of a quadric of $\L$ with $S_3$ is a quartic curve through $p$ of type $(3,1)$ in this quadric. Then a general plane section of $S_3$ intersects this curve in four points not lying on the base locus. These four points are mapped by $\psi$ to the intersection of a curve in $\L^\prime$ with a line. This shows that $\L^\prime$ is defined by quartic curves.

The map $\psi$ contracts any line in $S_3$ intersecting $r$ and $s$, since such line has fixed intersection with $\L$. If $S_3$ is a smooth cubic, there are five such lines and $\L^\prime$ has five simple base points. If $S_3$ is singular, some of these base points may be infinitely near.

Two other curves are contracted by $\psi$. The residual intersection of $S_3$ with the plane spanned by $p$ and $r$ is a conic intersecting each quadric of $\L$ in $p$, two points in $r$ and one point in $s$. A similar conic can be defined taking the plane through $p$ containing $s$. These conics are contracted by $\psi$ to double points of $\L^\prime$.

Therefore $\L^\prime$ consists of quartic curves having two double points and four (possibly infinitely near) simple points. There are no further base points, since the degree of the image surface is:
\[ 4^2-2\cdot 4-5\cdot 1 = 3 \]

Let us now investigate the image of $r$ and $s$ via $\psi$.

Blow up the line $r$. Then $\L_r\equiv (1,1)$ and the only base point of these curves is $\l_p\cap E_r$. On the other hand, $(S_3)_r\equiv (2,1)$ does not contain this point. Therefore $r$ is mapped by $\psi$ to a cubic curve. It must contain the five base points of $\L^\prime$, since these are  contractions of lines intersecting $(S_3)_r$. It also contains both double points, but it has multiplicity two in one of them, namely the contraction of the conic intersecting $r$ in two points.

The same reasoning applies to $s$. Therefore $r$ and $s$ are mapped to cubic curves $C^r$ and $C^s$ through all simple points having multiplicity one in a double point of $\L^\prime$ and multiplicity two in the other double point.

Now apply in $\P^2$ a standard quadratic transformation $T_1$ centered in the two double points and one simple point of $\L^\prime$. Then $\L^\prime$ is mapped to a linear system $\L^\prime_1$ of cubic curves with six base points. The curves $C^r$ and $C^s$ are mapped to conics $C^r_1$  and $C^s_1$ through five of the six base points. Let $P_1$ be the base point in $C^r_1\setminus C^s_1$. 

Now apply a second quadratic transformation $T_2$ centered in $P_1$ and in two base points of $\L^\prime_1$ in $C^r_1\cap C^s_1$. Then $\L^\prime_1$ is mapped to a linear system $\L^\prime_2$ of cubics with six base points. The image of $C^r_1$ via $T_2$ is a line $C^r_2$ through two base points $P_2$ and $Q_2$. The image of $C^s_1$ is a conic $C^s_2$ containing five of the base points, including $P_2$ and $Q_2$.

Finally apply a third quadratic transformation $T_3$ centered in $P_2$, $Q_2$ and in the base point of $\L^\prime_3$ not lying on $C^s_2$. Again $\L^\prime_2$ is mapped to a linear system $\L^\prime_3$ of cubics with six base points. Now the image of $C^r_2$ is a base point of $\L^\prime_3$ and the image of $C^s_2$ is a conic through the other five points.

Hence $\L^\prime_3$ defines the map $\phi$ described in the statement.
\\ \end{proof}

\begin{lemma}\label{sl_pencil_mQ}
Suppose $C_6$ does not contain $s$. Then the pencil $\mQ$ is mapped by $\sigma$ to a pencil $\mQ^\prime$ of weak Del Pezzo quartic surfaces in $X$. The base locus of $\mQ^\prime$ is a line $L$.

Note that the surface $V\in\mQ$ is contracted to $x$. But it has an extra multiplicity in $s$, so it corresponds to the image of $s$, that is, $D_4^x\in \mQ^\prime$.
\end{lemma}
\begin{proof}
Let $S_3$ be a cubic in $\mQ$ different from $V$. Then it is normal and, since $r$ and $s$ are disjoint, it is not a cone. Hence it is weak Del Pezzo cubic (see \cite{dolga_classical}). 

Let $\phi: \P^2 \tor S_3$ be the birational map given in Lemma \ref{sl_representacao_mQ}, that is, $\phi$ is defined by a linear system of cubics through six points, such that the line $r$ is the exceptional line of the blow up at a  point $p_1\in\P^2$ and $s$ is the image of a conic through the other five base points. 

In $\P^2$, write $(d,m,n)$ for the class of a curve with degree $d$ having multiplicity $m$ at $p_1=\phi^{-1}(r)$ and $n$ at each of the other five points. As usual, $(0,-1,0)$ stands for the exceptional divisor of the blow up at $p_1$, whereas the line $s$ is mapped to a $(2,0,1)$ curve. A plane section of $S_3$ is mapped to a curve of type  $(3,1,1)$.

We will adopt this notation for classes of curves in $S_3$, considering their birational image, via $\phi$, in $\P^2$. Then: 
\[ (9,3,3) \equiv V\cap S_3 \equiv C_6+r+2s \equiv (d,m,n)+(0,-1,0)+(4,0,2) \]
which implies that $C_6$ is a curve of type $(5,4,1)$ in $S_3$.

Now, intersect $S_3$ with $\X$:
\[ (21,7,7)  \equiv 2C_6+r+4s+F_4 \equiv (10,8,2)+(0,-1,0)+(8,0,4)+(d,m,n)\]

So the moving part $F_4$ of $\X\cap S_3$ consists of curves of type $(3,0,1)$, which are mapped via $\phi^{-1}$ to cubics through five points in $\P^2$. This linear system of cubics maps $\P^2$ to a weak Del Pezzo surface of degree four. This surface is the image of $S_3$ via $\sigma$.

To prove the assertion on the base locus, remember that the surfaces in $\mQ$ have  a curve infinitely near to $s$ in common. Then, after the blow up at the base locus of $\X$, the fixed part of $\mQ_s$ is the only curve in the base locus of $\mQ$. Therefore, its image is contained in the base locus of $\mQ^\prime$. But the moving part is mapped to a moving intersection of surfaces in $\mQ^\prime$ with $D_4^x$, the image of $s$. Since $D_4^x\in \mQ^\prime$, it follows that there is no such moving part, that is, $\mQ_s$ is a fixed curve.

Consider the blow up at $s$. Then $\mQ_s$ is a fixed curve $t_L$ of type $(2,1)$. It contains the five points corresponding to the intersections with $C_6$. Remember that $\X_s=(3,4)$ and it has double points in these five  points. Then $\mQ_s$ is mapped to a curve of degree:
\[ (2,1)\cdot (3,4) - 5\cdot 2 = 1 \]
and this line $L$ lies on the base locus of $\mQ^\prime$.

Moreover, some surfaces in $\P^3$ (including exceptional divisors of these blow ups) are contacted by $\sigma$ to curves or points. These are $V$, $E_r$ (or $E_{r^\prime}$ if $r\subset C_6$) and $E_\l$ if $\l$ is a multiple line in $C_6$. Depending on the intersection of $\mQ$ with these surfaces, their images can be contained in the base locus of $\mQ^\prime$. 

After the blow up at the base locus of $\X$, the intersection of $\mQ$ with $V$ consists of $\mQ_s$, which was already analysed. The intersection with $E_r$ is $\mQ_r\equiv (2,1)$ and it contains the four double points of $\X_r\equiv (6,1)$, so each surface of $\mQ^\prime$ intersects the image $\l_x$ of $E_r$ in one point. 

If $r\subset C_6$, $\X_r\equiv (5,2)\equiv r^\prime + (3,1)$ and $\mQ_r = r^\prime$. The moving part intersects $r^\prime$ in five fixed points, corresponding to the five lines of $C_6$. Blowing up these five lines and $r^\prime$ gives $\X_{r^\prime}\equiv \mQ_{r^\prime}\equiv (0,1)$, so again each surface intersects the image of $r^\prime$ in one point (see the proof of Proposition \ref{sl_imagem_r_e_s}).

If $\l$ is a multiple line, let $\Pi$ be the plane containing $\l$ and $s$. Then:
\[ \mQ\cap \Pi=\l+s+\{\text{lines}\} \]
Blowing up $s$ and $\l$ gives $\X_\l\equiv 2e_1+3f_1\equiv 2\l^\prime+\{f_1\}$ (see Lemma \ref{sl_retas_multiplas}) and $\mQ_\l\equiv e_1+2f_1\equiv \l^\prime+\{f_1\}$. Therefore, each surface of $\mQ^\prime$ cuts the image of $E_\l$ in one point.

Hence $L$ is the base locus of $\mQ^\prime$. The second part of the Lemma follows from Proposition \ref{sl_imagem_r_e_s}.
\\ \end{proof}

The case in which $s$ is contained in $C_6$ is now analysed:
\begin{lemma}\label{sl_mQ_com_s_em_C6}
Suppose $s\subset C_6$, so that $C_6=\l_1+\cdots +\l_4+2s$. Define $\mQ$ to be the linear system of cubic surfaces containing $\l_1,\ldots,\l_4,r$ and having multiplicity two in $s$. Then $\mQ$ is a pencil and it is mapped by $\sigma$ to a pencil $\mQ^\prime$ of quartic surfaces in $X$. These surfaces are projections of Veronese quartic surfaces and have a double line $L$, the image of $s$, which is the base locus of $\mQ^\prime$.

Note that the surface $V\in\mQ$ is contracted to $x$. But it contains $s^\prime$, so it corresponds to the image of $s^\prime$, that is, $D_4^x\in \mQ^\prime$.
\end{lemma}
\begin{proof}
For the images of $s$ and $s^\prime$, see Proposition \ref{sl_imagem_r_e_s}.

Let $S$ be a cubic of $\mQ$ different from $V$, then:
\begin{equation}\label{sl_blocus_mQ_s_em_C6} 
S\cap V = \l_1+\cdots+\l_4+r+4s
\end{equation}
which implies that $\mQ$ has fixed intersection with $V$. From the exact sequence:
\[ 0 \to \mQ-V \to \mQ \to \mQ\vert_V \to 0 \]
we have:
\[ 0 \to H^0(\P^3,\mQ-V) \to H^0(\P^3,\mQ) \to H^0(\P^3,\mQ\vert_V) \]
where the last map is clearly surjective.  Since $H^0(\P^3,\mQ-V)\cong \C$ and $H^0(\P^3,\mQ\vert_V)\cong \C$, it follows that $\mQ$ is a pencil.

The cubic $S$ is a non normal cubic of type $(i)$ in Proposition \ref{sl_projecoes_de_S(1,2)}, since it has a double line $s$ and contains a line $r$ skew to $s$. The conics in $S$ give a plane representation similar to the one of $V$ via $\bar{\sigma}$, in which $r\equiv (0,-1)$ and $2s\equiv (1,0)$. 

By \eqref{sl_blocus_mQ_s_em_C6}, $S$ does not contain $s^\prime$. The intersection of $S$ with $\X$ gives:
\[ \sum_{i=1}^4 2\l_i + r + 8s + F_4  \equiv (14,7) \equiv \sum_{i=1}^4 2\cdot (1,1)+(0,-1)+4\cdot (1,0)+(a,b) \]
which implies that $F_4\equiv (2,0)$, which does not contain $s^\prime$.

Now note that the intersection of $F_4$ with $s$ is actually a double point of $F_4$. Indeed, blowing up $s$, $S_s\equiv (1,2)$ intersects $V_s=s^\prime\equiv (1,2)$ in four points, each being the intersection of a line $\l_i$ with $E_s$. And $S_s$ intersects $\X_s$, which is the union of $2s^\prime$ and moving fibers, in these four points and two moving points. Then these two points are the intersection of $F_4$ with $E_s$. Since these points lie on a fiber over a point of $s$, the assertion follows.

Since $F_4\equiv (2,0)$, the moving part of $\X\cap S$ consists of curves birationally equivalent to a linear system of conics in $\P^2$ with no base points. This linear system is not complete, since  these conics intersect the image of $s$ in $\P^2$ in pairs of points that are sent to the same point: the double point of $F_4$ in $s$. Therefore, $S$ is mapped to  a projection of a Veronese quartic surface, having double line $L$, the image of $s$.

The base locus of $\mQ$ is given by \eqref{sl_blocus_mQ_s_em_C6}. After blowing up the base locus of $\X$, $\mQ$ no longer has base curves. Then we must look for surfaces being contracted by $\sigma$, which are $V$, $E_r$, $E_s$ and $E_\l$ if $\l$ is a multiple line in $C_6$.

After the blow ups, $\mQ$ has no intersection with $V$. And as in the previous Lemma, $\mQ_r\equiv (2,1)$ and the surfaces of $\mQ^\prime$ intersect $\l_x$ in moving points.

In $E_s$, $\mQ_s\equiv (1,2)$ containing the four points determined by the lines $\l_i$. As seen in the proof of Proposition \ref{sl_imagem_r_e_s}, $s^\prime\equiv (1,2)$ and:
\[ \X_s = 2s^\prime + \{ \text{moving fibers} \} \equiv (3,4) \]
Then the curves $\mQ_s$ intersect $s^\prime$ in those four points and the moving part of $\X_s$ maps each of these curves to $L$.

Now suppose $\l$ is a multiple line in $C_6$. If $\Pi$ is the plane containing $\l$ and $s$, then:
\[ \mQ\cap \Pi=\l+2s \]
As in the previous Lemma, blowing up $s$ and $\l$ gives, in $E_\l\cong \F_1$, $\Pi_\l\equiv e_1+f_1$. But now $\mQ_\l= \l^\prime\equiv e_1+f_1$. Then the surfaces of $\mQ^\prime$ do not intersect the image of $E_\l$. So one of these surfaces must contain this line.

In conclusion, the base locus of $\mQ^\prime$ is $L$.
\\ \end{proof}

Remember that by Corollary \ref{sl_s_vai_em_reta_tripla}, if $s$ is contained in $C_6$, it is mapped to a triple line $L$. And we have just proved that  all surfaces in $\mQ^\prime$ are singular along $L$ in this case. Here is a slight generalization of these facts: 

\begin{prop}\label{sl_ptos_triplos_mQ'}
The triple points of $X$, if any, lie on the base line $L$ of $\mQ^\prime$. Moreover, the surfaces of $\mQ^\prime$ have multiplicity two in these points.
\end{prop}
\begin{proof}
If $s\subset C_6$, it is mapped to $L$, which is a triple line in this case. Other triple points arise either from the intersection of two lines in $C_6$ or from the intersection of $s$ with a multiple line in $C_6$ through $q_1^s$ or $q_2^s$ (see Proposition \ref{sl_singularidades_em_C6.s} and Corollary \ref{sl_decomposicoes_de_C6}). Let $q$ be this point. Blowing up  $s$, in both cases the triple point is the contraction of the fiber over $q$. By Proposition \ref{sl_imagem_r_e_s}:
\[ \X_s \equiv 2s^\prime + \{ (1,0) \} \]
which implies that the triple point lies on $L$. By Lemma \ref{sl_mQ_com_s_em_C6}, $L$ is a double line of $\mQ^\prime$, so the second assertions is also proved in this case.
 
Now suppose that $C_6$ does not contain $s$. By Proposition \ref{sl_singularidades_em_C6.s}, a triple point of $X$ is the image of the exceptional divisor $E_q$ of the blow up at a singular point $q$ of $C_6$ lying on $s$. As explained in the proof of Lemma \ref{sl_pencil_mQ}, $L$ is the image of a curve $t_L$ of type $(2,1)$ in the exceptional divisor $E_s$ of the blow up at $s$. This curve contains the five intersections of $C_6$ with $s$. 

Let $q\in \P^3$ be a point that is mapped to $x_q$, a triple point of $X$, and let $t^q$ be the fiber over $q$ in $E_s$. Then $q$ fits in one of the cases (a) to (d) of Proposition \ref{sl_singularidades_em_C6.s}. It will be proven that $t_L$ contains $t^q$. 

If $q$ is not $q_1^s$ or $q_2^s$, then the tangent cone of $C_6$ in $q$ have components in both planes $\Pi_1$ and $\Pi_2$, which form the tangent cone of $V$ in $q$.  Then the curve $t_L$ intersects $t^q$ in two points and, therefore, contains this fiber. 

If $q$ is  $q_1^s$ or $q_2^s$, then $V_s$ is tangent to $t_q$. Since $q$ is a singular point of $C_6$, $t_L$ intersects $V_s$ in a point of $t^q$ and another point infinitely near to it. Since $V_s$ is tangent to $t^q$, both points lie on this line, which implies that $t_L$ contains $t^q$.

Therefore, we have that: \[ \mQ_s = t_L = t^q + t_L^\prime \equiv (1,0)+(1,1) \] where $t_L^\prime$ intersects $t^q$ in one point. Then $x_q$ lies on $L$.

We will now prove that $x_q$ is a double point of $D_4^x$, the surface of $\mQ^\prime$ through $x$. Since $x$ is a general point of $X$, $x_q$ is a double point of surfaces of $\mQ^\prime$.

Remember that $D_4^x$ is the image of $E_s$ via $\sigma$ and, since $x_q$ is a triple point of $X$, it is the image of $t^q$, the fiber over $q$ in $E_s$. Therefore, the linear system $\X_s\equiv (3,4)$ has two (possibly infinitely near) double points in $t^q$. 
Blowing up these double points, we get $(t^q)^2=-2$ in $E_s$. Then its image, $x_q$, is a double point of $D_4^x$.
\\ \end{proof}

The pencil  $\mQ^\prime$ will be used to give a geometrical description of the threefold $X$. For that we need another result:

\begin{lemma}\label{sl_numero_quadricas_por_X}
There are at least seven independent quadrics of $\P^7$ containing $X$.
\end{lemma}
\begin{proof}
The exact sequence:
\[ 0 \to \mathcal{I}_{X/\P^7}(2) \to \mO_{\P^7}(2) \to \mO_X(2) \to 0 \]
gives the cohomology exact sequence:
\begin{equation}\label{sl_seq_exata_numero_quadricas_por_X}
0 \to H^0(\mathcal{I}_{X/\P^7}(2)) \to H^0(\mO_{\P^7}(2)) \to H^0(\mO_X(2)) \to H^1(\mathcal{I}_{X/\P^7}(2)) \to 0
\end{equation}

We know that $H^0(\mO_{\P^7}(2))$ stands for the linear system of quadric hypersurfaces of $\P^7$, giving $h^0(\mO_{\P^7}(2))=\binom{9}{7}=36$.

The quadric sections of $X$ correspond, via $\tau$, to the linear system $\L=2\X$ of surfaces of degree $14$ having multiplicity eight along $s$, multiplicity four along $C_6$ and multiplicity two along $r$. Therefore $h^0(\mO_X(2))=h^0(\L)$.

Let $\L-V$ be the linear system of surfaces $S$ of degree $11$ such that  $S+V\in\L$, that is, having multiplicity six along $s$, multiplicity three along $C_6$ and containing $r$. Define analogously $\L-2V$, $\L-3V$ and $\L-4V$.

For $k=0,1,2,3$, the restriction of $\L-kV$ to $V$ gives the exact sequence:
\[ 0 \to \L-(k+1)V \to \L-kV \to (\L-kV)\vert_V \to 0 \]
Passing to cohomology, we have:
\begin{equation}\label{sl_seq_exata_L-kV}
0 \to H^0(\L-(k+1)V) \to H^0(\L-kV) \to H^0((\L-kV)\vert_V) \to 0
\end{equation}
where the exactness on the right follows from the surjectiveness of the last map. 

Note first that $\L$ has fixed intersection with $V$, namely:
\[ \L\vert_V=16s+4C_6+2r \]
Therefore $h^0(\L\vert_V)=1$.

Next, the fixed part of $(\L-V)\vert_V$ is $12s+3C_6+r$. Using the notation of page \pageref{sl_notacao_divisores_de_V} for curves in $V$, the moving part of $(\L-V)\vert_V$ is of type:
\[ (22,11)-(6,0)-(15,12)-(0,-1)=(1,0) \]
Then it corresponds to lines in $\P^2$, giving $h^0((\L-V)\vert_V)=3$.

The linear system $\L-2V$ consists of degree eight surfaces having multiplicity four along $s$ and two along $C_6$. Then the fixed part of $(\L-2V)\vert_V$ is $8s+2C_6$ and the class of its moving part is:
\[ (16,8)-(4,0)-(10,8)=(2,0) \] 
This gives a base point free linear system of conics in $\P^2$, so $h^0((\L-2V)\vert_V)=6$.

Finally, $\L-3V$ consists of degree five surfaces having multiplicity two along $s$ and containing $C_6$. The fixed part of $(\L-3V)\vert_V$ is $4s+C_6$ and the moving part  is of type:
\[ (10,5)-(2,0)-(5,4)=(3,1) \]
Then it corresponds to cubic curves in $\P^2$ with one base point. Therefore $h^0((\L-3V)\vert_V)=9$.

On the other hand, $\L-4V$ is the linear system of quadric surfaces in $\P^3$, giving $h^0(\L-4V)=\binom{5}{3}=10$. Then \eqref{sl_seq_exata_L-kV} gives, for $k=3$:
\[ h^0(\L-3V)= h^0(\L-4V) + h^0((\L-3V)\vert_V) = 10+9=19 \]

Applying the same equation with $k=2$, gives:
\[ h^0(\L-2V)= h^0(\L-3V) + h^0((\L-2V)\vert_V) = 19+6=25 \]

Then, for $k=1$ we have:
\[ h^0(\L-V)= h^0(\L-2V) + h^0((\L-V)\vert_V) = 25+3=28 \]

And finally, for $k=0$:
\[ h^0(\L)= h^0(\L-V) + h^0(\L\vert_V) = 28+1=29 \]

Returning to \eqref{sl_seq_exata_numero_quadricas_por_X}, we have:
\begin{align*}
h^0(\mathcal{I}_{X/\P^7}(2)) &= h^0(\mO_{\P^7}(2)) - h^0(\mO_X(2)) + h^1(\mathcal{I}_{X/\P^7}(2)) \\
&= 36-29+h^1(\mathcal{I}_{X/\P^7}(2))\\
&\geq 7
\end{align*}
\end{proof}

\begin{prop}\label{sl_X_contido_em_F}
Suppose that the linear system $\X$ is relatively complete. Then there is a cone $F$ over the Segre embedding of $\P^1\times \P^2$ with vertex $L$, such that  $X$ is the residual intersection of $F$ with two quadrics containing a $\P^4$ of its ruling. 

In particular, $X$ is an OADP threefold.
\end{prop}
\begin{proof}
According to Remark \ref{pr_oadpLN}, $X$ is linearly normal. Then, by Theorem \ref{pr_contidonumscroll}, $X$ lies on a rational normal scroll $F$, described by the $\P^4$'s spanned by the quartics in $\mQ^\prime$. These $\P^4$'s have a line in common, namely $L$. Therefore $F$ is a cone with vertex $L$ over a scroll $Y$ in $\P^5$. This scroll contains a one-dimensional family of disjoint planes, so $Y$ is the Segre embedding of $\P^1\times \P^2$. In particular, $F$ is the intersection of three quadrics containing it.

Note that $X$ cuts a $\P^4$ of $F$ in a quartic surface. Indeed, if $C_6$ does not contain $s$, then by Lemma \ref{sl_pencil_mQ} this intersection is a weak Del Pezzo surface containing $L$. If $s$ is contained in $C_6$ it is a projection of a Veronese quartic surface from an outside point and has multiplicity two along $L$.

Let $\eta:\Bl_L(\P^7)\to\P^7$ be the blow up of $L$, let $G$ be the strict transform of $F$  and $E$ be the intersection of $G$ and the exceptional divisor. Note that $E$ has dimension three. In $G$, set $H_0$ for the class of the strict transform of a hyperplane section of $F$ and $H_1$ for the class of the strict transform of a $\P^4$ of the ruling. Note that $H_1^2=0$.

By Lemma \ref{sl_numero_quadricas_por_X}, there are at least seven independent quadrics containing $X$. Three of these can be chosen containing $F$. Then there are at least four other quadrics through $X$ which do not contain $F$. Among these, pick two:  $Q_1$ and $Q_2$. Note that $Q_1\cap Q_2$ intersects $\Pi$, a general $\P^4$ of the ruling of $F$, in a quartic surface that coincides with $X\cap\Pi$. If $s\subset C_6$, this surface has multiplicity two along $L$.

Let $Q_1^\prime$ and $Q_2^\prime$ be the total transforms via $\eta$ of $Q_1\cap F$ and $Q_2\cap F$, and let $X^\prime$ be the strict transform of $X\subset F$.

The divisors $Q_1^\prime$ and $Q_2^\prime$ contain both $X^\prime$ and $E$. So let $S$ be residual intersection, that is:
\[ Q_1^\prime\cap Q_2^\prime = X^\prime+E+S \]

If $s$ is contained in $C_6$, then $Q_1\cap Q_2$ has multiplicity two along $L$. In this case write:
\[ Q_1^\prime\cap Q_2^\prime = X^\prime+2E+S \]

In both cases, $S$ does not contain $E$. Since $X$ intersects a $\P^4$ of the ruling of $F$ in a quartic surface, it follows that $S$ consists of threefolds contained in rulings of $G$. Moreover: 
\[ \deg(\eta(S)) = \deg(Q_1\cap Q_2\cap F)-\deg(X)=12-8=4 \]
These considerations imply that $S \equiv 4H_0H_1$. Therefore: 
\[ X^\prime+\delta E\equiv (2H_0)^2-4H_0H_1\equiv (2H_0-H_1)^2 \]
with $\delta\in\{1,2\}$. Taking the image via $\eta$ in $\P^7$, it follows that $X$ is the residual intersection of $F$ with two quadrics containing a $\P^4$ of the ruling.

The proof that $X$ is an OADP threefold is the same done in \cite[Example 2.6]{cmr}.
\\ \end{proof}

\section{Description of the general OADP threefold of degree 8}

This section is dedicated to the study of the case in which the curve $C_6$ of Proposition \ref{sl_baselocus} is general, that is, it is smooth. This defines the general OADP threefold of degree 8. A small description of this threefold is now given using the inverse of its general tangential projection.

\begin{prop}
Let $X$ be the threefold corresponding to the case in which $C_6$ is smooth. Then $X$ is the smooth scroll in lines given in Example \ref{pr_exemplo_sl}. 

The threefold $X$ contains a pencil $\mQ^\prime$ of Del Pezzo quartic surfaces with base locus a line $L$. The surface $D_4^x\subset\mQ^\prime$ through $x$ is the image of the line $s$. The line $r$ is mapped to the line through $x$

The curve $C_6$ in the base locus of $\X$ is mapped to a degree $16$ scroll in conics $S_{16}^x$ having multiplicity $5$ in $x$. It intersects quartics of $\mQ^\prime$ in curves of degree $10$. Its intersection with the quartic of $\mQ^\prime$ through $x$ consists of five conics through this point.

Through $x$ there is a one-dimensional family of conics parametrized by $C_6$. These conics describe the scroll $S_{16}^x$. Besides these, there are other five conics through $x$, which lie on $D_4^x$.
\end{prop}
\begin{proof}
Set $C=C_6$. Since $C$ is smooth, by Lemma \ref{sl_singularidades_vem_de_C6} the associated variety $X$ is smooth, and by Proposition \ref{sl_X_contido_em_F}, it is an OADP threefold. Since it has degree eight, by the classification of smooth OADP threefolds done in \cite{cmr}, it follows that $X$ is projectively equivalent to the variety given in Example \ref{pr_exemplo_sl}. 

Note that the fact that $X$ is ruled by lines also follows from Corollary \ref{sl_X_is_ruled}. Moreover, the description of $X$ as an intersection of divisors of a cone $F$ over the Segre embedding of $\P^1\times \P^2$ with vertex $L$, given  in Example \ref{pr_exemplo_sl}, is also obtained in Proposition \ref{sl_X_contido_em_F}. 

The images of $r$ and $s$ are given by Proposition \ref{sl_imagem_r_e_s} and the assertions on the pencil $\mQ^\prime$ are explained in Lemma \ref{sl_pencil_mQ}. The fact that the general surface in $\mQ^\prime$ is smooth follows from the description of $X$ as an intersection of $F$ with two general quadric hypersurfaces containing a $\P^4$ of the ruling.

Next, we study the image of $C$ via $\sigma$. Remember that $C\equiv (5,4)$, that is, it is the image via $\bar{\tau}$ of a degree five curve in $E\cong \P^2$ having multiplicity four in $p$, the base point of $\II_{X,x}$. In particular, $C$ is rational. The inclusions $C\subset V\subset \P^3$ give an exact sequence:
\[ 0 \to \mathcal{N}_{C/V} \to \mathcal{N}_{C/\P^3} \to \mathcal{N}_{V/\P^3\vert_C} \to 0 \]

Since $C$ is of type $(5,4)$ in $V$, $C^2=9$ and $\mathcal{N}_{C/V}=\mO_{\P^1}(9)$. On the other hand, $V$ is a cubic hypersurface, so $\mathcal{N}_{V/\P^3\vert_C}=\mO_{\P^1}(3)$. Now, since $\mathcal{N}_{C/\P^3}$ is a line bundle over a rational curve, it splits and:
\[ \mathcal{N}_{C/\P^3}= \mO_{\P^1}(9)\oplus \mO_{\P^1}(3) \]

Therefore, blowing up $C$ gives $E_C\cong \F_6$. The surface $V$ has multiplicity one along $C$, except for the five points in $s$, where the multiplicity is two. Then:
\[ V_C\equiv e_6+af_6 \equiv F+e_6+(a-5)f_6 \]
where $F\equiv 5f_6$ is the union of the fibers over the five points in $s$.

Let $S$ be a general surface in $\mQ$. Remember that $V\in\mQ$ and that:
\[ V\cap S=C+r+2s \]
This implies that $S_C\equiv e_6+af_6$ and that it intersects $V_C$ in four simple points corresponding to $r\cap C$ and five double points corresponding to $s\cap C$. Then:
\[ 4+5\cdot 2=V_C\cdot S_C=-6+2a \]
and this gives $a=10$. Set $\widehat{V_C}=V_C-F\equiv e_6+5f_6$.

On the other hand, $\X$ has multiplicity two along $C$, except for the points in $s$, where it is four. Then:
\[ \X_C\equiv  2F+2e_6+bf_6 \]
Let $\widehat{\X_C}$ be the moving part. It has multiplicity two in the points $s_C$ and multiplicity one in $r_C$. We know that $\widehat{\X_C}$ has fixed intersection with $\widehat{V_C}$, which is smooth in these points. Then:
\[ 4+5\cdot 2 = \widehat{\X_C}\cdot \widehat{V_C} = -12+10+b \]
which gives $b=16$. Note that $\widehat{\X_C}$ has no moving intersection with $F$.

Finally, the degree of the image of $E_C$ via $\sigma$ is:
\[ (2e_6+16f_6)^2-4-5\cdot 4 = -24+64-24 = 16 \]
The curve $\widehat{V_C}$ is contracted to the point $x$. 

Note that $(\widehat{V_C})^2=4$ and it contains nine base points of $\widehat{\X_C}$. Then $S_{16}^x$ has multiplicity five in $x$.

The moving intersection of $\widehat{\X_C}$ with $S_C$ is:
\[ (2e_6+16f_6)(e_6+10f_6)-4-2\cdot 5 = -12+20+16-14 = 10 \]
So $S_{16}^x$ intersects a general surface of $\mQ^\prime$ in a degree $10$ curve. Since $s_C$ consists of the five double points of $\widehat{\X_C}$, $S_{16}^x$ intersects $D_4^x$ in five conics through $x$.

To determine all conics through $x$, note that the tangential projection of such conic is a point. Therefore this conic must lie on the image of the exceptional divisors of $\X$.

A general fiber of $E_C$ intersects $\X_C$ in two moving points and intersects $V_C$ in a point outside the base locus of $\X_C$. This gives a one-dimensional family of conics through $x$, which describes $S_{16}^x$ and is parametrized by $C$. 

On the other hand, the Del Pezzo surface $D_4^x$, image of $E_s$, contains ten conics through $x$. As noted above, five of these conics lie on $S_{16}^x$.
\\ \end{proof}

\section{Description of a singular example}
\label{sl_1conica4retas_sec}

In this section, we will describe an example where the curve $C_6$ is the union of a conic and four lines. Two of these lines are infinitely near, that is, it is a double line of $C_6$. In $V$, this means that we are writing:
\[ (5,4) = (1,0)+(1,1)+(1,1)+(2,2) \]

Suppose that the two simple lines $\l_1$ and $\l_2$ intersect in a point $q_1$ (that must lie on $s$). Suppose also that the point of intersection of $C$ and $s$ is $q_1$. Let $\l_3$ be the other line and $\l_4$ be the line infinitely near to it. 

Throughout this section, let $\Sigma$ be the plane containing $C$, let $\Pi_i$ be the plane containing $s$ and $\l_i$, for $i=1,2,3$, and let $\Gamma$ be the plane containing $\l_1,\l_2$ and $r$. Then the tangent cone of $V$ in $q_1$ is the union of $\Pi_1$ and $\Pi_2$. 
By Lemma \ref{sl_intersecoes_de_C6_com_2retas_do_ruling}, $C$ intersects one of the lines through $q_1$, say $\l_2$, in a second point $q_2$. In particular, $\l_2$ lies on $\Sigma$ and the tangent line to $C$ at $q_1$ is $\Sigma\cap \Pi_1$. Set $q_3=\l_3\cap C$ and suppose $\l_3\cap s$ is not $q_1^s$ or $q_2^s$.

All the results of this section refer to this specific configuration. Figure \ref{sl_figura_1conica4retas} illustrates the base locus of $\X$.

\begin{figure}
\centering
\includegraphics[trim=20 650 360 20,clip,width=0.5\textwidth]{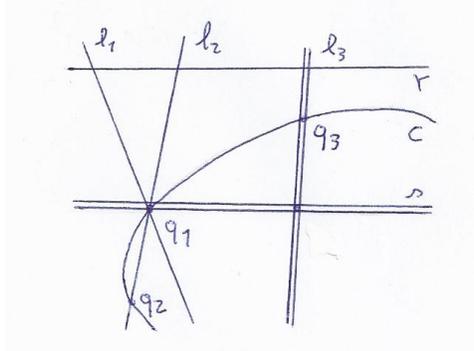}
\caption{The base locus of $\X$ in  Section \ref{sl_1conica4retas_sec}.}\label{sl_figura_1conica4retas}
\end{figure}

The point $q_1$ is an intersection of three components of $C_6$. The following remark will be useful throughout this section:
\begin{remark}\label{sl_1conica4retas_explosao_q1}
Consider the blow up at $q=q_1$. As seen in the proof of Lemma \ref{sl_intersecoes_de_C6_com_2retas_do_ruling}, $\X_q$ is a pair of double lines intersecting in $s_q$.  One of these lines, namely $(\Pi_1)_q$, contains the  intersections of $E_q$ with $C$ and $\l_1$; while the other line $(\Pi_2)_q$ contains the intersection with $\l_2$. See Figure \ref{sl_figura_1co4re_conetg_q1}.
\end{remark}

\begin{figure}
\centering
\includegraphics[trim=340 650 60 40,clip,width=0.6\textwidth]{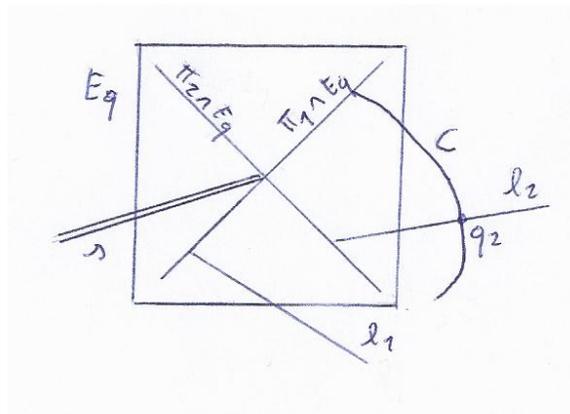}
\caption{The blow up at $q=q_1$}\label{sl_figura_1co4re_conetg_q1}
\end{figure}

\subsection{Singularities of $X$}

By Lemma \ref{sl_singularidades_vem_de_C6}, the singularities of $X$ are images of singularities of $C_6$. Proposition \ref{sl_singularidades_em_C6.s} describes the image of $q_1$ and asserts that $\l_3\cap s$ is not mapped to a triple point. The images of $q_2$ and $q_3$ are explained in Proposition \ref{sl_pontos_duplos_de_X}. The double line $\l_3$ is studied in Lemma \ref{sl_retas_multiplas}. This gives:

\begin{prop}
The points $q_i$ are mapped to points $x_i$ in $X$, for $i=1,2,3$, and $\l_3$ is mapped to a line $R$. The point $x_1$ is a triple point, $x_2$ and $x_3$ are double points and $R$ is a double line of $X$. Among the three singular points, only $x_3$ lies on $R$. Moreover, the projectivization of the tangent cone of $X$ in $x_1$ is the union of a smooth quadric surface and a plane.
\end{prop}
\begin{proof}
Only the two last assertions need a proof. 

The points $q_1$ and $q_2$ do not lie on $\l_3$, so their images do not lie on $R$. And as noticed in the proof of Lemma \ref{sl_retas_multiplas}, after the blow up at $s$ and $\l_3$ the moving part of $\X_{\l_3}$ consists of fibers over points. Since $q_3\notin s$, it is mapped to a point in $R$.

To prove the assertion on the tangent cone, let $\wX$ be the linear system of surfaces in $\X$ having multiplicity five in $q=q_1$. Then, blowing up $q$, the degree five curves $\X_q$ have multiplicity four in $s_q$ and multiplicity two in $C_q$, $(\l_1)_q$ and $(\l_2)_q$. By Remark \ref{sl_1conica4retas_explosao_q1}, the two former lie on $t_1=(\Pi_1)_q$, and the latter lies on $t_2=(\Pi_2)_q$. Then:
\[ \X_q=2t_1+t_2+\{\text{conics}\} \]
where the moving part has two base points: $(\l_2)_q$ and $s_q$. These conics map $E_q$ to a smooth quadric surface.

By Lemma \ref{pr_truqueconetangente}, the projectivization of the tangent cone of $X$ in $x_1$ contains a smooth quadric surface. Since $x_1$ is a point of multiplicity three, the result follows.
\\ \end{proof}

By Proposition \ref{sl_ptos_triplos_mQ'}, $x_1$ is a double point of the degree four surfaces in $\mQ^\prime$. As noted in the proof of this result, the blow up at $s$ gives:
\begin{align}\label{sl_1conica4retas_mQs}
\mQ_s  = t^q + t_L^\prime \equiv (2,1) 
\end{align}
where $t^q$ is the fiber over $q_1$ and $t_L^\prime\equiv (1,1)$ is the curve containing $\l_1\cap E_s$, $\l_3\cap E_s$ and $\l_4\cap E_s$ (which is infinitely near to $\l_3\cap E_s$).

\quad

\subsection{The image of $C_6$}

We will now study the images of the irreducible components of $C_6$ via $\sigma$.

\begin{lemma}\label{sl_1conica4retas_imagem_li}
Let $i\in\{1,2,4\}$. The line $\l_i$ is mapped to a cubic scroll through $x$. This scroll also contains the point $x_i$, if $i=1,2$. If $i=4$, it contains $x_3$ and $R$ is a line of its ruling. It intersects quartics of $\mQ^\prime$ in conics through a fixed point. This point is $x_1$ when $i=1,2$. It lies on $L$ also when $i=4$.
\end{lemma}
\begin{proof}
For $i=1,2$, set $\l=\l_i$, $\Pi=\Pi_i$ and $q=q_1=s\cap \l$. The intersection of $\X$ with $\Pi$ is $2\l+4s$ and moving lines. 
Let $\Omega$ be a plane through $\l$. Then $\X$ cuts $\Omega$ in $2\l$ and quintics intersecting $\l$ in the double point $q$ and three moving points.  If $i=1$, there is another base point infinitely near to $q$, but it does not lie on $\l$. 

First blow up  $s$, to avoid fixed components. By Remark \ref{sl_1conica4retas_explosao_q1}, $\X_s$ intersects the fiber over $q$ in two fixed double points. One of them is $(\l_2)_s$, while the other contains the intersection of $E_s$ with $\l_1$ and $C$.

At this stage, $l=\Omega\cap\Pi$, where at $\Omega$ it had a point blown up. So:
\[ N_\l=\mO_{\P^1}(0)\oplus \mO_{\P^1}(1) \]

Blowing up $l$ gives $E_\l\cong \F_1$, $\Omega_\l\equiv e_1+f_1$ and $\Pi_\l\equiv e_1$. Hence:
\[ \X_\l \equiv 2e_1+3f_1  \]
with one double point in $C_\l$ and one simple point in $r_\l$. Note that there is no other base point, since the blow up at $s$ has separated $\l_1$ and $C$ from $\l_2$ in $q$.

Applying Lemma \ref{pr_imagemF1F2}, this linear system is birationally equivalent, in $\P^2$, to cubics with one double and two simple points. After a standard quadratic transformation, these correspond to conics with one base point. Hence $E_\l$ is mapped to a cubic scroll of type $S(1,2)$.

The section $V_\l\equiv e_1+f_1$ containing the base points is contracted to $x$ and the fiber through the double point $C_\l$ is contracted to $x_i$. The lines of the ruling of the scroll are the images of sections $e_1+f_1$ through $C_\l$. The fiber through $r_\l$ is mapped to the directrix line and $r_\l$ is mapped to the line of the ruling through $x$.

A cubic of $\mQ$ intersects $\Omega$ in $\l$ and a conic, and intersects $\Pi$ in $\l,s$ and a line. Blowing up $s$ gives, by \eqref{sl_1conica4retas_mQs}, $\mQ_s=t^q+t_L^\prime$, which is a fixed curve. Next, blowing up $\l$ gives $\mQ_{\l}\equiv e_1+2f_1$.

We'll now study separately the cases $i=1$ and $i=2$. Both are represented in Figure \ref{sl_figura_1co4re_mQ}, as well as case $i=4$.

\begin{figure}[tb]
\centering
\includegraphics[trim=10 7 40 430,clip,width=1\textwidth]{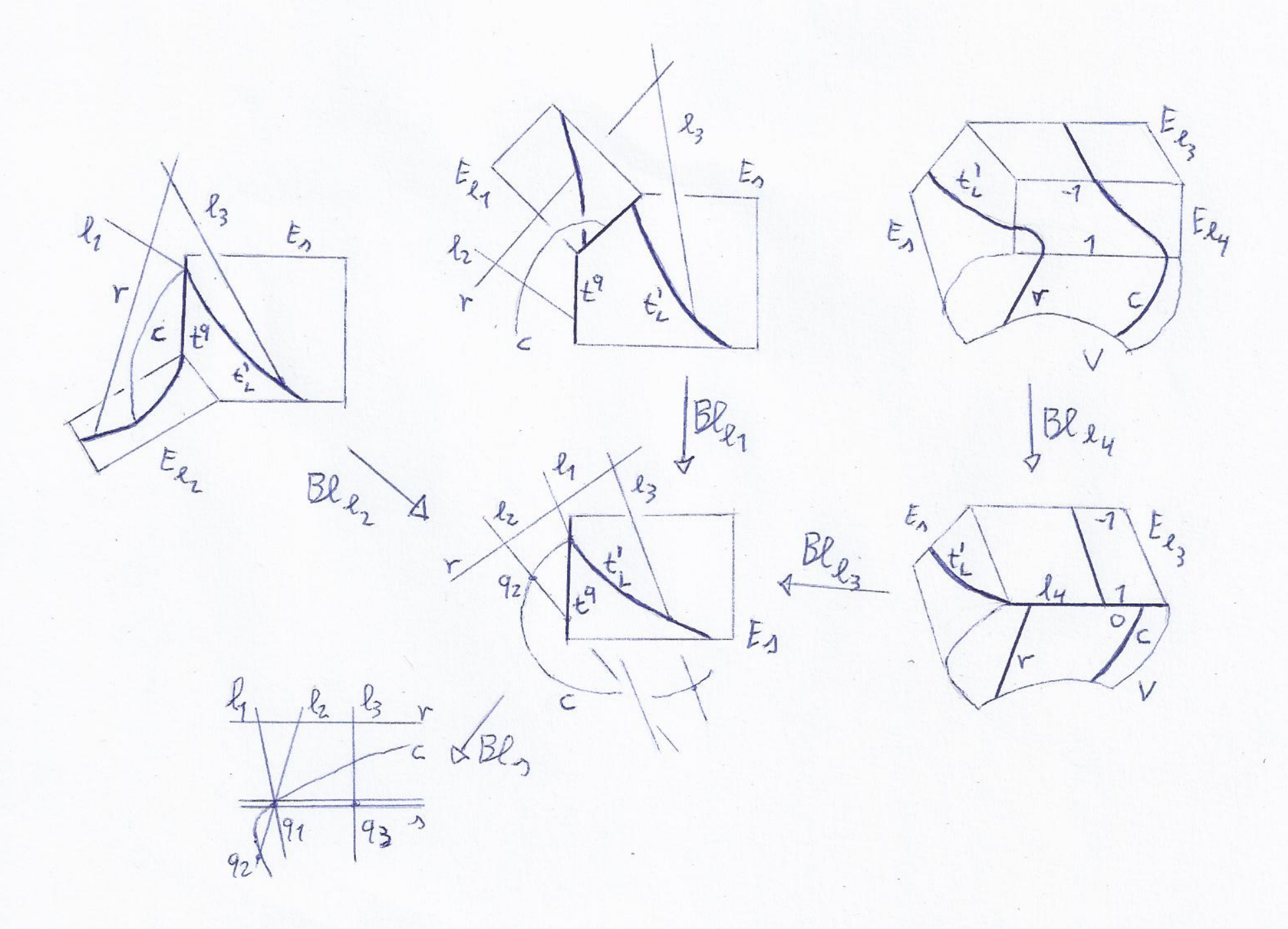}
\caption{$\mQ$ and the blow ups in $\l_i$}\label{sl_figura_1co4re_mQ}
\end{figure}

If $i=2$, $\mQ_{\l_2}$ has three base points. One of them  lies on the intersection with $E_s$, namely, it is $(t^q)_{\l_2}$. Since it is the fiber in $E_s$ over $q_1$, this base point is mapped to $x_1$. Note that $(t^q)_{\l_2}$ is not a base point of $\X_{\l_2}$.

The second is $r_{\l_2}$, a simple point of $\X_{\l_2}$; and the third is $C_{\l_2}$, a double point of $\X_{\l_2}$. There are no other base points, since $(\mQ_{\l_2})^2=3$. Therefore, the cubic scroll intersects a quartic of $\mQ^\prime$ in a curve of degree:
\[ (e_1+2f_1)(2e_1+3f_1) - 1 - 2 = 2 \]
These conics intersect each other in $x_1$.  Indeed, two of the three base points of $\mQ_{\l_2}$ are also base points of $\X_{\l_2}$, while the other one is mapped to $x_1$.

If $i=1$, three of the base points of $\mQ_{\l_1}$ (namely $(t^q)_{\l_1}$, $(t_L^\prime)_{\l_1}$ and $C_{\l_1}$) lie on $(E_s)_{\l_1}\equiv f_1$. Then this is a fixed component and it is mapped to $x_1$. The moving part of $\mQ_{\l_1}$ is of type $e_1+f_1$ and it has a simple base point in $r_{\l_1}$. Then the degree of its image is:
\[ (e_1+f_1)(2e_1+3f_1) - 1 = 2 \]
As before, these conics intersect each other in $x_1$.

This finishes the analysis for $i=1,2$.

For $i=4$, after blowing up $s$, $\l_3$ and $\l=\l_4$ we also have that $ \X_\l \equiv 2e_1+3f_1$. This is explained in the proof of Lemma \ref{sl_retas_multiplas}. In the same Lemma it is showed that $R$ is a line of the ruling of the scroll, image of $\l$. Moreover:
\[ \mQ_{\l_3} \equiv e_1+2f_1 \equiv \l_4+\{f_1\} \]
and $\mQ_\l\equiv e_1+2f_1$ as before. These have base points in $(t_L^\prime)_\l$, $r_\l$ and $C_\l$ and are mapped to conics through a fixed point of $L$, since $(t_L^\prime)_\l$ is not a base point of $\X_\l$. See Figure \ref{sl_figura_1co4re_mQ}.
\\ \end{proof}

Next, we study the image of $C$:

\begin{lemma}
The conic $C$ is mapped to a quartic Veronese surface through $x$, $x_1$, $x_2$ and $x_3$. It cuts each quartic of $\mQ^\prime$ in a conic through $x_1$.

The union of this surface, the plane $\sigma(\Sigma)$ and the cubic scroll, image of $\l_2$, is a tangent hyperplane section of $X$ at $x$. The plane intersects the scroll in a line and the Veronese surface in a conic.
\end{lemma}
\begin{proof}
First note that:
\begin{align}\label{sl_1conica4retas_XeSigma}
\X\cap\Sigma = 2\l_2+2C+\{\text{lines}\} 
\end{align} 

Next, blow up $s$ and then $\l_2$. Now, $C=\Sigma\cap V$. In $V$, $C^2=1$, and in $\Sigma$, $C$ had one point blown up. Hence:
\[ N_C=\mO_{\P^1}(3)\oplus \mO_{\P^1}(1) \]

Blow up $C$, so $E_C\cong\F_2$, $\Sigma_C\equiv e_2$ and $V_C\equiv e_2+2f_2$. Since $\X_C$ cuts $f_2$ in two points, it follows that:
\[ \X_C \equiv 2e_2+6f_2 \]
It has three double points, in the intersections with $\l_1$, $\l_3$ and $\l_4$ (infinitely near to $(\l_3)_C$).

These curves can be birationally mapped to $\P^2$ using Lemma \ref{pr_imagemF1F2}. Their images are  sextic curves with two double points, a third double point infinitely near to one of these, a point of multiplicity four and a double point infinitely near to it. After two standard quadratic transformations, $\X_C$ is mapped to a linear system of conics with no base points.

Therefore $C$ is mapped to a quartic Veronese surface and $V_C$ is contracted to $x$. The points $x_1$ and $x_3$ are the contractions of the fibers through the respective double points. The point $x_2$ is the image of the intersection of $E_C$ with the fiber in $E_{\l_2}$ over $q_2$.

By \eqref{sl_1conica4retas_mQs}, when blowing up  $s$ the cubics of $\mQ$ cut $E_s$ in two fixed curves $t^q$ and $t_L^\prime$. The curve $t_L^\prime$ cuts $\Sigma_s=t^q$ in the point $C_s$.

After blowing up $\l_2$ and $C$, $\mQ_C$ cuts the section $V_C=e_2+2f_2$ in three points, namely $(\l_1)_C$, $(\l_3)_C$ and $(\l_4)_C$. But it cuts the fiber $f_2^q$ of $E_C$ over $q_1$ in three points:  $(\l_1)_C$, $(t^q)_C$ and $(t_L^\prime)_C$. Hence:
\[ \mQ_C = f_2^q + \{ e_2+2f_2 \} \]
and the moving part has two base points: $(\l_3)_C$ and $(\l_4)_C$. It is mapped by $\X_C$ to a conic. The fixed part is mapped to the point $x_1$.

Figure \ref{sl_figura_1co4re_mQ_C} illustrates $\mQ$ under these blow ups.

\begin{figure}
\centering
\includegraphics[trim=10 220 320 130,clip,angle=90,width=1\textwidth]{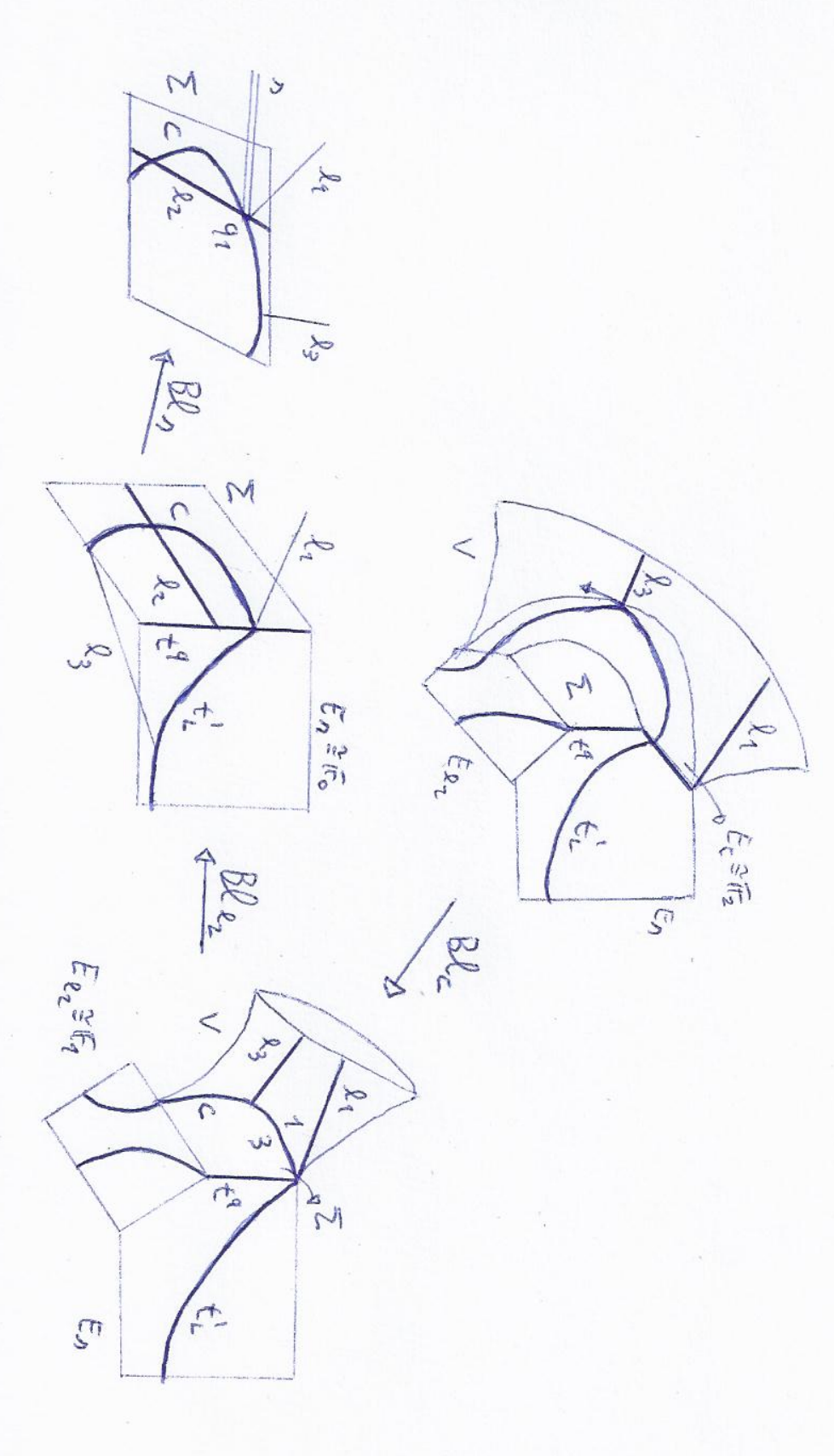}
\caption{$\mQ$ and the blow ups in $s$, $\l_2$ and $C$}\label{sl_figura_1co4re_mQ_C}
\end{figure}

This proves the first part of the Lemma. The second part follows from Lemma \ref{pr_X'eX''} and from \eqref{sl_1conica4retas_XeSigma}.
\\ \end{proof}

Note that $\sigma(\Sigma)$ is not the only plane in $X$. The planes $\Pi_1$ and $\Pi_2$ are also mapped to planes, since:
\[ \X\cap \Pi_i = 4s+2\l_i+\{ \text{lines} \} \]
and these moving lines have no base points. 
For the same reason, the plane containing $s$ and $\l_3$ is mapped to a plane in $X$.

\subsection{Description of the example}

Putting together the results in this example, we have:

\begin{prop}
The threefold $X$ has a double line $R$, a double point $x_2$ and a triple point $x_1$.
It contains a pencil $\mQ^\prime$ of weak Del Pezzo quartic surfaces having a double point in $x_1$. The base locus of this pencil is a line $L$ containing $x_1$.

The tangent cone of $X$ in $x_1$ is the union of a rank four quadric cone and a $\P^3$.

Three lines of the base locus of $\X$ are mapped to cubic scrolls through $x$ and one line is mapped to $R$. The conic $C$ is mapped to a quartic Veronese surface. Each of these surfaces intersects the quartics of $\mQ^\prime$ in conics through a fixed point of $L$. 

In total, counting multiplicities, the curve $C_6$ is mapped to a reducible surface of degree $16$. It has multiplicity $5$ in $x$, multiplicity three in the triple point of $X$ and multiplicity two in the double points, including $R$. It cuts quartics of $\mQ^\prime$ in five conics, three of them through $x_1$.
\end{prop}

\section{The Cayley's ruled cubic case}

In this Section, suppose $V$ is  Cayley's ruled cubic, with equation given in Proposition \ref{sl_projecoes_de_S(1,2)}, $(ii)$. This situation will give us no new varieties, but different projections of varieties having fundamental surface the general cubic $(i)$.

We start by explaining that $r$ is infinitely near to $s$ in $V$ and by computing its image in $X$.

\begin{lemma}\label{sl_r_inf_prox_s}
Blowing up $s$ gives:
\[ V_s=r+s^\prime \equiv (1,1)+(0,1) \]
where $s^\prime$ is the strict transform of the plane intersecting $V$ in $3s$.
\end{lemma} 
\begin{proof}
By Proposition \ref{sl_projecoes_de_S(1,2)}, $V$ can be given by:
\begin{align}\label{sl_eq_cayley}
x_0^2x_3+x_0x_1x_2+x_1^3 = 0
\end{align}

Moreover, the map $\bar{\tau}$ can be defined by:
\[ (t_0:t_1:t_2)\mapsto (t_2^2:t_1t_2:t_0t_2-t_1^2:-t_0t_1) \]
with $p=(1:0:0)\in E$. Let $s_p$ be the line $t_2=0$ through $p$.

Blowing up $p$, the union of $s_p$ and the exceptional line is $\hat{s}$, which is mapped to $s$, the double line of $V$. Then $r$ is infinitely near to $s$.

We'll compute the blow up at $s$ directly from \eqref{sl_eq_cayley}, where $s$ is given by $x_0=x_1=0$. Setting $x_0=uv$, $x_1=v$, $x_2=w$ and $x_3=1$, \eqref{sl_eq_cayley} gives:
\[ v^2(u^2+uw+v) = 0\]
 with $v=0$ being the equation of $E_s$.

Then, restricting to $E_s$ gives $u(u+w)=0$.  The strict transform of the plane $x_0=0$, which intersects $V$ in $3s$, is given by $u=0$. And the curve $v=u+w=0$ is of type $(1,1)$ in $E_s$. It is the intersection of the strict transform of the quadric $x_0x_3+x_1x_2=0$ with $E_s$. 

To find the equation of $r$, we blow up $E$ in $p$. So put $t_0=1$, $t_1=y$ and $t_2=yz$. Then the resolution of $\bar{\tau}$ is:
\[ (1:y:yz)\mapsto (y^2z^2:y^2z:yz-y^2:-y)=(-yz^2:-yz:y-z:1) \]
where $r$ is defined as $y=0$ in the target, giving $x_0=x_1=0$. The blow up at $s$ done above gives:
\[ -yz^2=uv \qquad ; \qquad -yz=v \qquad ; \qquad y-z=w \]
which implies that $y=u+w$. Therefore $r$ is defined infinitely near to $s$ as $u+w=0$, the $(1,1)$ curve in $E_s$.
\\ \end{proof}

\begin{lemma}\label{sl_cayley_imagem_r}
The image of $r\in E_s$ is a line $\l_x$ through $x$.
\end{lemma}
\begin{proof}
Let $Q^r$ be the quadric such that $Q^r\cap E_s = r$.  It  is smooth, otherwise $Q^r$ would be singular in a point of $s$, and $r$ would be reducible. But this is not the case, since $V_s$ would contain a fiber of $E_s$, that is, $V$ would have a triple point: it would be a cone. 

Since $\X$ has multiplicity four in $s$ and one in $r$, it follows that:
\[ \X\cap Q^r=5s + C_9\equiv (5,0)+(2,7) \]
where the equivalence classes refer to curves in the smooth quadric $Q^r$. Then $C_9$ intersects $s$ in seven points. 

Blowing up $s$ gives $\X_s\equiv (3,4)$. Since it contains $r\equiv (1,1)$, the movable part is of type $(2,3)$ and intersects $r$ in five points. By Proposition \ref{sl_baselocus}, four of these points lie on $C_6$ and are base points of the movable part of $\X_s$. They are also base points of $\X\cap Q^r$, so only three of the seven points in $C_9\cap E_s$ are moving.

Since $r=Q^r\cap E_s$, its normal bundle is:
\[ N_r = \mO_{\P^1}(0)\oplus \mO_{\P^1}(2) \]

Therefore, blowing up $r$ gives $(Q^r)_r\equiv e_2+2f_2$ and $(E_s)_r\equiv e_2$. The curves of $\X_r$ intersect a general fiber in one point and have four double points, corresponding to the intersections with $C_6$. Then $\X_r$ contains four fibers of $E_r$. And by what was noted above, the moving part of $\X_r$ intersects $E_s$ in one point and  $Q^r$ in three points. Then: 
\[ \X_r \equiv 4f_2 + \{e_2+3f_2\} \]
where $4f_2$ is the fixed part, consisting of the fibers over the four base points in $r$. Each fiber is intersected by $C_6$ in one point. These are the base points of the moving part. 

The blow ups are represented in Figure \ref{sl_figura_cayley_Er}.

\begin{figure}[tb]
\centering
\includegraphics[trim=40 440 180 100,clip,width=0.9\textwidth]{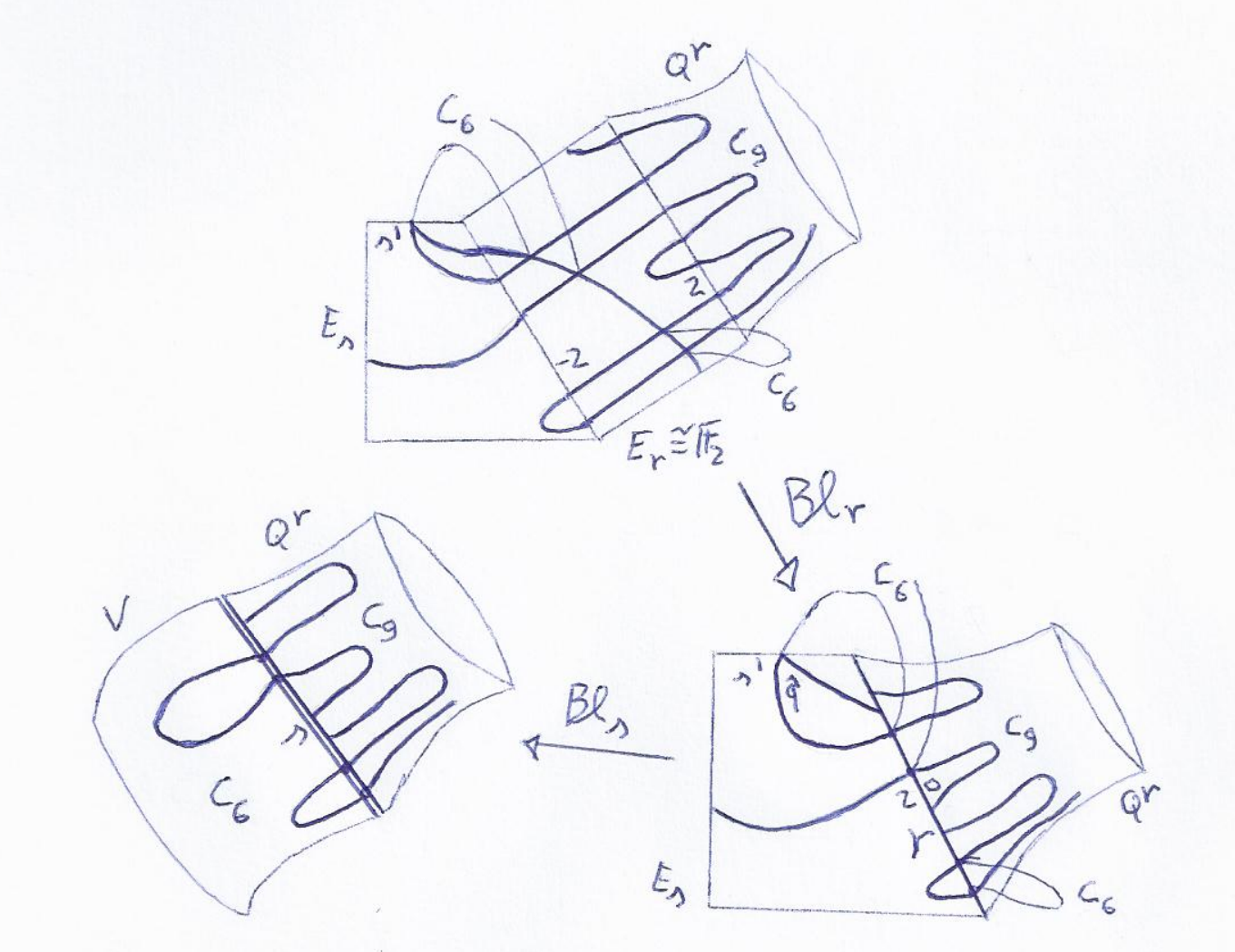}
\caption{$\X$ and the blow ups in $s$ and $r$}\label{sl_figura_cayley_Er}
\end{figure}

By Lemma \ref{pr_imagemF1F2}, the moving part is birationally equivalent to cubics in $\P^2$ with 4 simple points, one double point and another simple point infinitely near to it, which is Cremona equivalent to the linear system of lines with one base point. Moreover, the fixed fibers have no moving intersection with the moving part. Therefore $E_r$ is mapped to a line.

By Lemma \ref{sl_r_inf_prox_s}, $V_s$ has two components: $r$ and $s^\prime$, which intersects $r$ in one point. Then, after the blow up at $r$, $V_r$ intersects $E_s$ in one point. Then $V_r\equiv e_2+3f_2$  and it intersects the moving part of $\X_r$ in its base points. Therefore the image of $E_r$ contains $x$.
\\ \end{proof}

The next result indicates that the center of the tangential projection may not be a general point:

\begin{lemma}\label{sl_cayley_naogeral}
Suppose $V$ is the Cayley's ruled cubic. Then the line $\l_x$ through the point $x\in X$ from which was made the tangential projection intersects another line $L$ of $X$. Moreover, for a general point $y\in X$, the line $\l_y$ through $y$ does not intersect $L$.
\end{lemma}
\begin{proof}
By Proposition \ref{sl_baselocus}, $C_6$ intersects $s$ in five points and $r$ in four points. Blowing up $s$, by Lemma \ref{sl_r_inf_prox_s} $r$ is infinitely near to $s$ of type $(1,1)$. Then $C_6$ cuts $E_s$ in four points in $r$ and one point in $s^\prime\equiv (0,1)$. Let $\hat{q}$ be this point and $t^q\equiv (1,0)$ be the fiber through it in $E_s$. Let $q\in C_6\cap s$ be the point corresponding to this fiber, that is, $\hat{q}$ is infinitely near to $q$.

As it was already noted, the moving part of $\X_s$ is of type $(2,3)$. It has a double point in $\hat{q}$, therefore $t^q$ is mapped to a line, name it $L$. But $t^q$ intersects $r$ in one point, which is mapped to a point of $\l_x$, by Lemma \ref{sl_cayley_imagem_r}. Then $\l_x$ intersects $L$.

Given a general point $y\in X$, it is mapped by $\tau$ to a general point $y^\prime$ in $\P^3$. Since $y^\prime$ is a general point and $C_6$ cuts $s$ in five points, there is a unique line through $y^\prime$ intersecting $C_6$ and $s$ in distinct points. For general $y^\prime$, it does not contain $q$. Therefore it is mapped to the line $\l_y$ through $y$ in $X$, which does not intersect $L$.
\\ \end{proof}

The fact that $x$ is not a general point of $X$ will become clear in the next result.

But before proceeding, note that the proof of Lemma \ref{sl_cubicas_por_C6} can be repeated to show that there is a pencil of cubic surfaces $\mQ$ containing $C_6$, $r$ and $s$, supposing $C_6$ does not contain $s$. Let $S$ be a cubic surface in $\mQ$ different from $V$. Then:
\[  V\cap S = C_6+2s+r \]
which implies that $V$ is the only surface in $\mQ$ with a double curve.

Note that $S$ has a double point in $q$ (notation of Lemma \ref{sl_cayley_naogeral}). In fact, blowing up $q$, $S_q$ contains the non collinear points $(C_6)_q$, $s_q$ and $r_q$ infinitely near to $s_q$, so it is not a line.

In the general setting, $S$ is a general cubic with a point of type $A_1$ in $q$. Then it is the image of a  $\P^2$ blown up at six points lying on a conic $C_q$ via the strict transform of the linear system of cubics through these points. The conic $C_q$ is contracted to the double point $q$ and the six points are mapped to six lines through $q$. Let $p_1\in \P^2$ be the point mapped to $s$. 

For a curve in $S$, write $(d,m,n)$ if its strict transform in $\P^2$ has degree $d$, multiplicity $m$ in $p_1$ and multiplicity $n$ in the other five base points. Therefore $s\equiv (0,-1,0)$ and a general plane section of $S$ is of type $(3,1,1)$. The point $q$ is mapped to $C_q\equiv (2,1,1)$. Then a plane section through $q$ can be written as:
\[ H\cap S\equiv (3,1,1) = (2,1,1)+(1,0,0) \]
Where $(2,1,1)$ corresponds to $C_q$.

Since $S$ contains $C_6$, $r$ and $s$, then:
\[ C_6+3s = V\cap S \equiv (9,3,3) = (4,2,2)+(0,-3,0)+(5,4,1) \]
where $(4,2,2)$ is the class of $2C_q$  and $(0,-3,0)$ represents $3s$. Therefore $C_6\equiv (5,4,1)$.

Intersecting with the linear system $\X$, which has multiplicity four in $q$, gives:
\[ 2C_6+5s+F_4 = \X\cap S \equiv (21,7,7)=(8,4,4)+(0,-5,0)+(13,8,3) \]
with $(8,4,4)$ corresponding to $4C_q$  and $(0,-5,0)$ to $5s$. Then the moving part $F_4$ is of type $(3,0,1)$, that is, it corresponds to cubics in $\P^2$ with five base points.

A similar result can be obtained in the non general setting. If $s$ is contained in $C_6$, the same reasoning done in Lemma \ref{sl_mQ_com_s_em_C6} shows that the moving part of $\X\cap S$, where $S$ is a general cubic of $\mQ$, is birationally equivalent to a non complete linear system of conics in $\P^2$ with no base points. Remember that in this case $S$ is singular in $s$.

These considerations will be used in the proof of the following proposition.

\begin{prop}\label{sl_cayley_finaleira}
Suppose $V$ is Cayley's ruled cubic. Then there is a Cremona transformation in $\P^3$ that maps $\X$ to a linear system defining the inverse of a general tangential projection of $X$, having fundamental surface a general projection of $S(1,2)$ from an outer point.
\end{prop}
\begin{proof}
The linear system $\X$ defines the inverse of $\tau=\tau_x$, the projection of $X$ from $T_xX$. Let $y$ be a general point in $X$, let $\tau_y$ be the projection from $T_yX$. We are looking for the Cremona transformation $T=\tau_y\circ (\tau_x)^{-1}$.

Take a plane in the target $\P^3$. It's preimage via $\tau_y$ is a tangent hyperplane section of $X$ at $y$, that is, a hyperplane section of $X$ with a double point in $y$. This section is mapped by $\tau_x$ to a surface of $\X$ having a double point in $y^\prime=\tau_x(y)$.

Therefore, the linear system $\mathcal{Y}$ of surfaces in $\X$ with a double point in $y^\prime$ defines the Cremona transformation $T$. Note that the line $\l^\prime=\tau(\l_y)$ through $y^\prime$ intersecting $s$ and $C_6$ lies on the base locus of $\mathcal{Y}$, since it intersects these surfaces in eight points.

Blowing up $y^\prime$, the linear system $\mathcal{Y}_{y^\prime}$ consists of conics with one base point, the intersection of $\l^\prime$ with $E_{y^\prime}$. Therefore it maps $E_{y^\prime}$ to a projection $V^\prime$ in $\P^3$ of $S(1,2)$. By Proposition \ref{sl_projecoes_de_S(1,2)}, this projection is not general if and only if the preimage in $E_{y^\prime}$ of the double line of $V^\prime$ contains the base point $(\l^\prime)_{y^\prime}$.

Now consider the cubic $S$ of $\mQ$ through $y^\prime$. By the remarks made above, if $s$ is not contained in $C_6$ the moving part of $\X\cap S$ corresponds in $\P^2$ to cubics with five base points. Then the moving part of $\mathcal{Y}\cap S$ corresponds to cubics with five simple and one double points. This linear system maps $\P^2$ to a line, that is, $S$ is mapped by $T$ to a line.

Note that blowing up the double point, these plane cubics cut the exceptional divisor in a pencil with degree two. Therefore, blowing up $y$,  $\mathcal{Y}$ cuts the line $S\cap E_{y^\prime}$ in a non complete linear system. Hence, the image of  $S$ via $T$ is the double line of $V^\prime$.

Since the base point of $\mathcal{Y}_{y^\prime}$ is $\l^\prime_{y^\prime}$, $V^\prime$ is not a general projection of $S(1,2)$ if and only if $\l^\prime$ lies on the tangent plane of the cubic surface $S$ in $y^\prime$. If this happens, then in fact $S$ contains $\l^\prime$.

But $S$ cannot contain $\l^\prime$, since this implies that $S$ is ruled, by the generality of $y^\prime$. And therefore, $S$ would be a cone or singular along a line. As it was discussed above, $S$ cannot have a singular line. And it cannot be a cone over a smooth cubic, since $q$ is a double point of $S$ not contained in $\l^\prime$. Therefore $E_{y^\prime}$ is mapped to a general projection of $S(1,2)$.

If $s\subset C_6$, the moving part of $\X\cap S$ is birationally equivalent to a non complete linear system of conics with no base points. Then the moving part of $\mathcal{Y}\cap S$ corresponds to a pencil of conics with one double point. Then $S$ is again mapped to the double line of $V^\prime$.

In this case, $S$ is singular along $s$. Moreover $C_6=2s+\l_1+\cdots + \l_4$. If $S$ contains $\l^\prime$, then there are two lines of the ruling of $S$ intersecting in a point not lying on $s$. But this is not possible.
\\ \end{proof}

\section{Conclusion}

An idea of different possibilities of irreducible components of $C_6$ was given in Corollary \ref{sl_decomposicoes_de_C6}. The next Lemma indicates that the total number of different configurations, and hence, of different Bronowski threefolds with cubic fundamental surface is very high.

\begin{lemma}
Suppose that $C_6$ contains a conic. Then there are $49$ different configurations for $C_6$.
\end{lemma}
\begin{proof}
As it was already remarked, $C_6$ is the union of a conic $C$ and four lines $\l_1,\ldots,\l_4$, which can be infinitely near. We divide the configurations in five groups:
\begin{itemize}
\item[(a)] $\l_1,\ldots,\l_4$ are not infinitely near
\item[(b)] $\l_1\prec \l_2$
\item[(c)] $\l_1\prec \l_2 \prec \l_3$
\item[(d)] $\l_1\prec \l_2 \prec \l_3\prec \l_4$
\item[(e)] $\l_1\prec \l_2$ and $\l_3 \prec \l_4$
\end{itemize}
By Lemma \ref{sl_intersecoes_de_C6_com_2retas_do_ruling}, the conic $C$ intersects $s$ in one point and intersects a line $\l$ of the ruling in one  or two points. If $C$ intersects $\l$ in a point of $s$, there are three possibilities: either $C$ intersects $\l$ transversally in one point, or it intersects $\l$ in $s$ and in a second point outside $s$ (so $\l$ and $C$ are coplanar), or it is tangent to $\l$ in a point of $s$ (which is $q_1^s$ or $q_2^s$). If two lines of the ruling intersect in $q\in s$ and this point lies on $C$, then $C$ intersects one of these lines in a second point not in $s$.

Start with group (a). If the four lines are skew, either $C$ intersects none of the lines in $s$, or intersects one of them in one of the three different ways described above. If $\l_1$ and $\l_2$ intersect in $q$, there are the same four possibilities and also $C$ can contain $q$. If $\l_1$ intersects $\l_2$ and $\l_3$ intersects $\l_4$, then there are two possibilities. So group (a) counts $4+5+2=11$ possibilities.

Now consider group (b). In other words, $\l_1$ is a double line in $C_6$. If $\l_1,\l_3$ and $\l_4$ are skew, there are three possibilities for $C$ intersecting the double line in $s$, three for a simple line and one to intersect none in $s$. If $\l_3$ and $\l_4$ intersect, the possibilities are $3+1+1=5$. If $\l_1$ and $\l_3$ intersect, then $C$ can be coplanar to the double line $\l_1$ or to the simple line $\l_3$, or it intersects $s$ in $\l_4$ or outside the three lines. This gives $2+3+1=6$ possibilities. Therefore, group (b) has $7+5+6=18$ possible configurations.

In group (c), $\l_1$ is a triple line. It can be skew to $\l_4$, producing $1+3+3=7$ possibilities. If the two lines intersect, there are three possible configurations. So group (c) counts $7+3=10$ configurations.

I group (d) there is only one line with multiplicity four, which gives four possibilities.

In group (e), the two double lines can be skew or can intersect. If they are skew, there are four options. If they intersect, there are two. So it gives $4+2=6$ possibilities.

Therefore, the total number of possible configurations is:
\[ 11+18+10+4+6=49 \]
\\ \end{proof}

We can now give a small description of the Bronowski varieties having cubic fundamental surface.

\begin{theorem}\label{sl_the_classification}
Let $X$ be a Bronowski threefold having fundamental surface of degree three. Suppose that the linear system $\X$ defining the inverse of a general tangential projection of $X$ has degree seven and that its base locus scheme has pure dimension one, that is, it has no embedded points. The fundamental surface $V$ associated to this projection is a general projection of a scroll $S(1,2)$ from an outer point.

Then $\X$ has multiplicity four in $s$, the double line of $V$, multiplicity two in a degree six curve $C_6$ and multiplicity one in $r$, the line in $V$ skew to $s$. There is a pencil of cubics $\mQ$ containing the base locus of $\X$.

The threefold $X$ has the following properties:
\begin{itemize}
\item[(i)] It is an OADP variety;
\item[(ii)] There is a cone $F$ over the Segre embedding of $\P^1\times \P^2$ with vertex a line $L$, such that  $X$ is the residual intersection of $F$ with two quadrics containing a $\P^4$ of its ruling;
\item[(iii)] The singularities of $X$ have multiplicity two or three, its singular curves are lines and the triple points of $X$ lie on $L$;
\item[(iv)] The pencil $\mQ$ is mapped to a one-dimensional family of quartic surfaces $\mQ^\prime$ in $X$ with base locus $L$. This family is the intersection of $X$ with the $\P^4$'s of the ruling of $F$.
\end{itemize}

In Table \ref{sl_tabela}, some configurations of the base locus of $\X$ are presented, with the singularities of the corresponding OADP threefold $X$ being described.
\end{theorem}
\begin{proof}
The fact that $V$ is a general projection of $S(1,2)$ follows from Proposition \ref{sl_projecoes_de_S(1,2)} and Proposition \ref{sl_cayley_finaleira}. The base locus of $\X$ under the pure dimension hypothesis is  given in Proposition \ref{sl_baselocus}, which also implies that $\X$ is relatively complete.  Then properties (i) and (ii) follow from Proposition \ref{sl_X_contido_em_F}. Property (iii) follows from Section \ref{sl_singularities_sec}, where the singularities of $X$ are discussed. In Section \ref{sl_cubicsfamily_sec}, the family $\mQ^\prime$ is studied, proving (iv).
\\ \end{proof}

\begin{table}
\begin{center}
\begin{tabular}{|>{\centering}m{4.72cm}|>{\centering}m{3.3cm}|>{\centering}m{3.68cm}|}
\hline  $C_6$ & Figure of $C_6$ and $s$ & Singularities of $X$  \tabularnewline 
\hline $r+\l_1+4\l_2$,\qquad \qquad \qquad with $\l_1\cap s \neq \l_2\cap s$ \qquad \qquad  and $\l_2\cap s \neq q_1^s,q_2^s$ 
    & \raisebox{-1\height}{\includegraphics[trim=390 740 90 30,clip,width=0.25\textwidth]{figsl5.pdf}}
    & four double lines \qquad  $R_2\prec R_3\prec R_4\prec R_5$ \qquad  and one double point \tabularnewline
\hline      
	$2s+2\l_1+2\l_2$ 
	& \raisebox{-1\height}{\includegraphics[trim=390 660 90 110,clip,width=0.25\textwidth]{figsl5.pdf}}
	& one triple line \qquad and  two double lines \tabularnewline
\hline $C_2+\l_1+\l_2+2\l_3$, \qquad \qquad as in Section \ref{sl_1conica4retas_sec} 
	& \raisebox{-1\height}{\includegraphics[trim=390 580 90 190,clip,width=0.25\textwidth]{figsl5.pdf}}
	& one double line, \qquad \qquad one triple point \qquad and one double point \tabularnewline 
\hline $C_3+\l_1+2\l_2$, \qquad \qquad \qquad with $C_3$ singular in $q_1^s$, $\l_i\cap s=q_i^s$, for $i=1,2$  
	& \raisebox{-1\height}{\includegraphics[trim=380 480 100 290,clip,width=0.25\textwidth]{figsl5.pdf}}
	& one double line $R$ and two triple points not lying on $R$ \tabularnewline 
\hline $C_4+\l_1+\l_2$, \qquad \qquad \qquad  with $C_4$ smooth, $\l_1\cap s=q$, $C_4\cap s =3q$ and $C_4$ has no other intersection with $\l_1$ 
	& \raisebox{-1\height}{\includegraphics[trim=380 390 100 380,clip,width=0.25\textwidth]{figsl5.pdf}}
	& two double points  \tabularnewline 
\hline $C_5+\l_1$, \qquad \qquad \qquad \qquad with $C_5$ cuspidal in $q_1^s$, $C_5\cap s = 3q_1^s + q$ and $C_5$ intersects $\l_1$ in two points 
	& \raisebox{-1\height}{\includegraphics[trim=370 280 110 490,clip,width=0.25\textwidth]{figsl5.pdf}}
	& two triple and \qquad one double points \tabularnewline
\hline $C_6$, \qquad \qquad \qquad \qquad \qquad with a double point in $q_1$ \qquad and $C_6\cap s=2q_1+3q_2$ 
	& \raisebox{-1\height}{\includegraphics[trim=360 170 120 600,clip,width=0.25\textwidth]{figsl5.pdf}}
	& one triple point \tabularnewline
\hline 
\end{tabular} 
\caption{Example of Bronowski varieties having cubic fundamental surface.}\label{sl_tabela}
\end{center}
\end{table}

\chapter{The quartic case}
\label{sc_chapter}

Finally, we will study the case in which the fundamental surface $V$ has degree four. By \cite{cmr}, there are no smooth OADP threefolds fitting in this situation. Since we are assuming hypothesis (H), $\X$ has degree nine.  A further hypothesis will be made on the base locus of $\X$, see Proposition \ref{sc_base_locus}.

\section{First Considerations}

The map $\bar{\tau}:E\cong \P^2\tor V\subset \P^3$ is defined by a base point free non complete linear system of conics. Then $V$ is a projection of the Veronese quartic surface from a disjoint line.

When computing self-intersections of curves in $V$, we'll consider its normalization, the Veronese surface.

The following proposition describes the three possibilities for these quartic surfaces. This description can be found in \cite{coffman}.

\begin{prop}\label{sc_projecoes_da_veronese}
Let $V\subset \P^3$ be a projection of the Veronese quartic surface from a line disjoint to it. Then $V$ is projectively equivalent to one of the following surfaces:
\begin{align*}
(&i) \quad x_1^2x_2^2-x_0x_1x_2x_3+x_1^2x_3^2+x_2^2x_3^2=0 \\
(i&i) \quad x_0^2x_2^2-2x_0x_1x_2^2-x_0x_1x_3^2+x_1^2x_2^2+x_1^2x_3^2-x_3^4=0 \\
(ii&i) \quad 4x_0^3x_1-13x_0^2x_1^2+14x_0x_1^3-5x_1^4+10x_0^2x_1x_2-22x_0x_1^2x_2+12x_1^3x_2- \\
& -x_0^2x_2^2+ 10x_0x_1x_2^2-10x_1^2x_2^2 -2x_0x_2^3 +4x_1x_2^3- x_2^4 -8x_0^2x_3^2 +20x_0x_1x_3^2- \\
&  -12x_1^2x_3^2-12x_0x_2x_3^2+16x_1x_2x_3^2- 4x_2^2x_3^2-4x_3^4=0
\end{align*}

The surface given by (i) is the general projection and is known as  \textit{Steiner's Roman Surface}. It has a triple point $p=(1:0:0:0)$ and three double lines: 
\[ \l_1:(x_2=x_3=0)\qquad ;\qquad \l_2:(x_1=x_3=0)\qquad ;\qquad \l_3:(x_1=x_2=0) \] 
which intersect in $p$. In the surface (ii), one of the three double lines is infinitely near to another line. In (iii), two lines are infinitely near. 
\end{prop}

The computations that follow will be done for the three cases, the difference being that the singular lines of $V$ may be infinitely near.

The linear system $\II_{X,x}$ in $E\cong \P^2$ has no base points and maps a general line to a conic. There are three lines $\hat{\l}_i\subset E$ that are mapped by $\bar{\tau}$ to the double lines $\l_i$ of $V$. This gives, for each $i$, a double cover of $\l_i\cong \P^1$, ramified in two points. Let $q^i_1,q^i_2\in \l_i$ be the two branch points of this double cover. \label{sc_ramified_sec}

Given a conic in $E$, it lies on $\II_{X,x}$ if and only if its two points of intersection with each line $\hat{\l}_i$ are mapped to the same point in $V$. If the point is one of the ramification points, the conic must be tangent to $\hat{\l}_i$. 

We start with a small result on the preimages of $p\in V$ in $E$.

\begin{lemma}
Consider the three surfaces given in Proposition \ref{sc_projecoes_da_veronese} and the map $\bar{\tau}:E\tor V$ defined by the linear system $\II_{X,x}$ in each case.

If  $\l_1,\l_2,\l_3$ are proper lines of $V$, then $p$ has three preimages in $E$, namely $\hat{\l}_i\cap \hat{\l}_j$. 

If, for instance, $\l_2$ is infinitely near to $\l_1$, then $p$ is a branch point of the double cover of $\l_3$, whereas it has two preimages via $\bar{\tau}$ in $\hat{\l_1}$. 

On the other hand, if $\l_1\prec \l_2\prec \l_3$, then $p$ has only one preimage and it is a branch point in $\l_1$.
\end{lemma}
\begin{proof}
Let $u_0,u_1,u_2$ be projective coordinates in $E\cong\P^2$. The first assertion is clear, since in this case the lines $\hat{\l}_i\subset E$ are not infinitely near.

To prove the second assertion, consider the equation given in Proposition \ref{sc_projecoes_da_veronese}, item (ii). Then the two proper double lines of $V$ are:
\[ \l_3:x_2=x_3=0 \quad ; \quad \l_1:x_0-x_1=x_3=0 \]
and $p=(1:1:0:0)$.

Note that setting $x_0=x_1$ in the equation of $V$ gives $x_3^4=0$, that is, $V$ intersects the plane given by $(x_0-x_1=0)$ in $4\l_1$. Therefore, blowing up $\l_1$, $V$ has a second double line $\l_2$ infinitely near to it given by this plane. Then this notation agrees with the hypothesis that $\l_2$ is infinitely near to $\l_1$.

The map $\bar{\tau}$ can be given by:
\[  (u_0:u_1:u_2)  \mapsto  (u_0^2-u_1^2+u_2^2:u_2^2-u_1^2:u_1u_2:u_0u_1) \]
Then the preimages of $p$ are $(0:0:1)$ and $(0:1:0)$. Moreover:
\[ \hat{\l_3}:u_1=0 \quad ; \quad \hat{\l_1}:u_0=0 \]
so both preimages lie on $\hat{\l_1}$.

The double cover of $\l_3$ is given by:
\[   (u_0:u_2)  \mapsto  (x_0:x_1)=(u_0^2+u_2^2:u_2^2) \]
and $p$ has only one preimage, namely $(0:1)$. This proves the second part.

For the last assertion, consider the variety given by (iii) in Proposition \ref{sc_projecoes_da_veronese}. The double line of $V$ is:
\[ \l_1: x_0-x_1+x_2=x_3=0 \]
and $p=(2:1:-1:0)$. Then $\bar{\tau}$ is given by:
\[  (u_0:u_1:u_2)  \mapsto  (u_0^2+2u_1^2+u_2^2 : 2u_1^2+u_2^2 : u_2^2+2u_0u_2 : u_1u_2+u_0u_1) \]

It follows then that the only preimage of $p$ is $(1:0:-1)$. In particular, $p$ is a branch point of the double cover of $\l_1$.
\\ \end{proof}

For $i\in\{1,2,3\}$, let $\Gamma_i$ be the plane containing the lines $\l_j$, with $j\neq i$. Then:
\[ V\cap\Gamma_i=\l_j+\l_k \]
for $i,j,k$ distinct. If, for example, $\l_2$ is infinitely near to $\l_1$, then the plane $\Gamma_1$ does not exist and $\l_2$ is defined in the exceptional divisor $E_{\l_1}$ of the blow up in $\l_1$ by the intersection with $\Gamma_3$.  If $\l_2$ is infinitely near to $\l_1$ and $\l_3$ is infinitely near to $\l_2$, then there is only the plane $\Gamma_3$.

Consider the blow up at $p$. Then $V_p$ is a degree three curve with three double points, some of them possibly infinitely near. Therefore it is the union of three lines, which may include infinitely near lines. 

If $V$ is the Steiner's Roman Surface, then:
\[ V_p=t_1+t_2+t_3 \]
with $t_i=\Gamma_i\cap E_p$. 

If the line $\l_2$ is infinitely near to $\l_1$, then $V_p$ has two proper and one infinitely near double points. This gives:
\[ V_p=2t_2+t_3 \]
After that, the blow up at $\l_1$ gives $V_{\l_1}\equiv (1,2)$, since $V$ intersects a plane through $\l_1$ in a conic through $p$. Then:
\begin{equation}\label{sc_eq_reta_ip}
V_{\l_1}=f^p+2\l_2\equiv (1,2)
\end{equation}
where $f^p=E_p\cap E_{\l_1}$ and $\l_2=\Gamma_3\cap E_{\l_1}$. In particular, $t_2$ is not a double curve of $V$, that is, blowing up $t_2$ gives $V_{t_2}=E_p\cap E_{t_2}$. To keep the notation of the general case, set $t_1=V_{t_2}$.

If we are in case (iii) of Proposition \ref{sc_projecoes_da_veronese}, say $\l_1\prec\l_2\prec\l_3$, then:
\[ V_p=3t_3 \]
with $t_3=(\Gamma_3)_p$. Blowing up $\l_1$ gives:
\[ V_{\l_1}=f^p_1+2\l_2\equiv (1,2) \]
with $f^p_1=E_p\cap E_{\l_1}$ and $\l_2=\Gamma_3\cap E_{\l_1}$. Since $(\l_2)^2=0$ in both $\Gamma_3$ (which was blown up in $p$) and $E_{\l_1}$, the blow up at $\l_2$ gives $E_{\l_2}\cong \F_0$. The intersections of $V$ with $\Gamma_3$ and $E_{\l_1}$ give $V_{\l_2}\equiv (1,2)$. Then:
\begin{equation} \label{sc_eq_V_l2}
V_{\l_2}=f^p_2+2\l_3\equiv (1,2) 
\end{equation}
where $f^p_2=E_p\cap E_{\l_2}$. Since $V_p=3t_3$, the blow up at $t_3$ gives:
\[ V_{t_3}=t_1=E_p\cap E_{t_3} \]
and the blow up at $t_1$ gives: 
\[ V_{t_1}=t_2=E_p\cap E_{t_1} \]
See Figure \ref{sc_figura_expl_V_1reta}.

\begin{figure}[tb]
\centering
\includegraphics[trim=30 20 100 550,clip,width=1\textwidth]{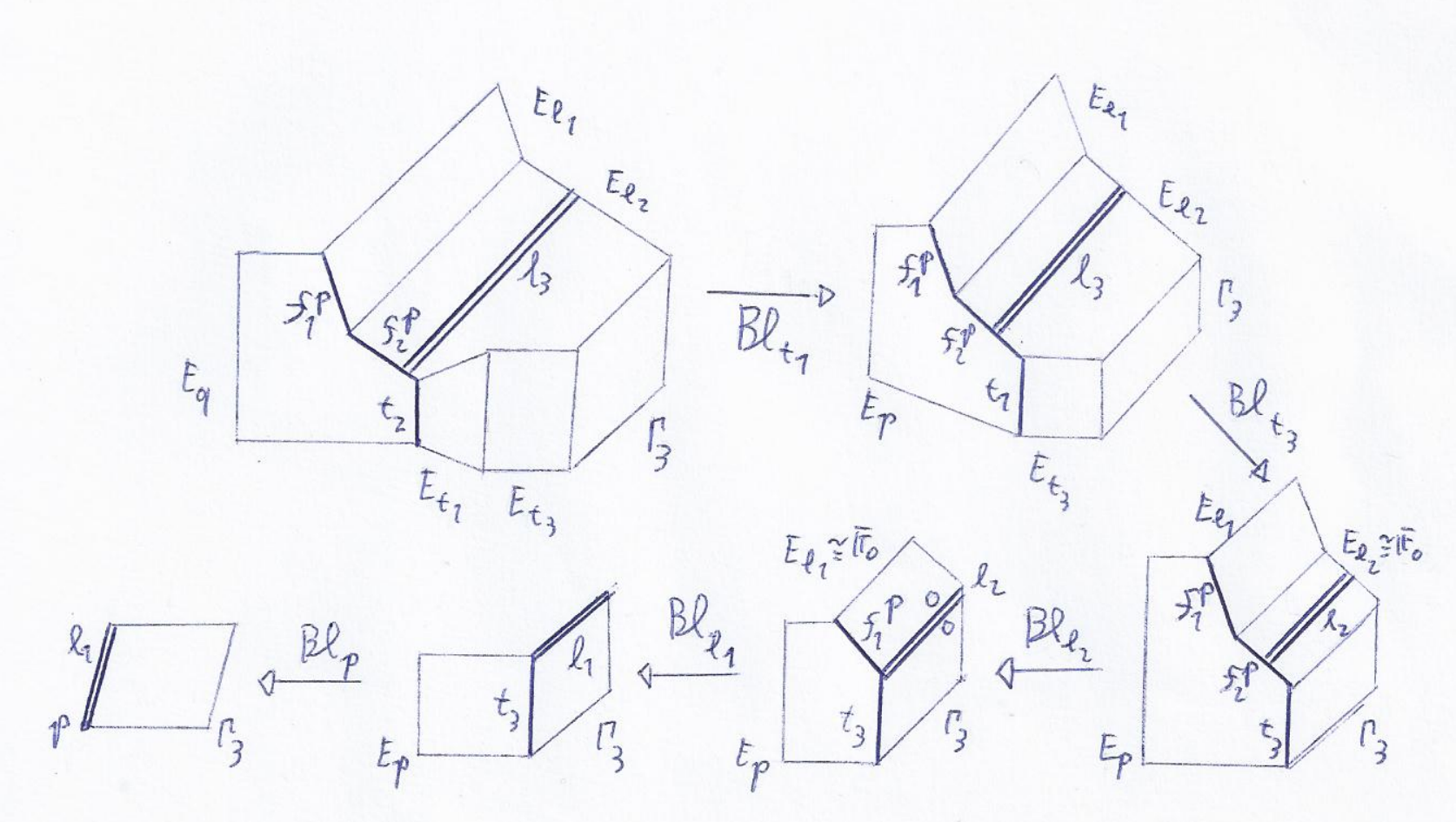}
\caption{Blow ups of $V$ when $\l_1\prec\l_2\prec\l_3$}\label{sc_figura_expl_V_1reta}
\end{figure}

In conclusion, in the three cases we have:
\begin{equation}\label{sc_eq_Vp}
V_p=t_1+t_2+t_3
\end{equation}
where some of these lines can be infinitely near. 

\quad

These considerations explain the tangent cone of $V$ in $p$. Next we study the tangent cone of $V$ in other singular points. These results are very similar to the case in which $V$ has degree three, given in chapter \ref{sl_chapter}. 

\begin{lemma}\label{sc_cones_tangentes_em_V}
Suppose $V$ does not have a double line infinitely near to $\l_i$. Then the tangent cone of $V$ in a point $q\in \l_i\setminus \{p,q^i_1,q^i_2\}$ is a pair of distinct planes containing $\l_i$.  If $q=q^i_j$, the tangent cone of $V$ in $q$ is a double plane containing $\l_i$. In both cases, $V$ intersects one of these planes in $2\l_i$ and a conic through $p$ and $q$.

On the other hand, if for instance $\l_2$ is infinitely near to $\l_1$, then the tangent cone of $V$ in a point $q\in\l_1$ distinct from $p$ is $2\Gamma_3$.
\end{lemma} 
\begin{proof}
In the first case, set $i=1$. Since $\l_1$ is a double line of $V$, the tangent cone of $V$ in a point of $\l_1$ different from $p$ is a pair of planes containing $\l_1$. These planes may coincide or not.

Let $q$ be a point in $\l_1$ different from $p$ and let $\Pi$ be a plane in the tangent cone of $V$ in $q$. Then $\Pi\cap V$ has multiplicity three in $p$ and $q$, so it is the union of $2\l_1$ and a conic $C$ through $p$ and $q$. Note that $2\hat{\l}_1\notin\II_{X,x}$, since the hypothesis that $V$ does not have a double line infinitely near to $\l_1$ implies that $\hat{\l}_1\cap \hat{\l}_j$ is not a ramification point of $\hat{\l}_j$, for $j\in\{2,3\}$. Therefore $4\l_1$ is not a plane section of $V$, whence $C\neq 2\l_1$.

The point $q$ has one or two preimages in $\hat{\l}_1$, and $p$ has two or  three preimages (depending on $V$), namely $\hat{\l}_i\cap \hat{\l}_j$. The strict transform of $C$ cannot be $\hat{\l}_1$, so it is a line through one of the points $\bar{\tau}^{-1}(q)\in \hat{\l}_1$ and the preimage of $p$ not lying on $\hat{\l}_1$. 

If $q$ is not $q^1_1$ or $q^1_2$, there are two such lines, so there are two distinct conics $C$. Therefore the two planes in the tangent cone of $V$ in $q$ are distinct. If $q$ is $q^1_1$ or $q^1_2$, there is only one conic $C$ and the tangent cone of $V$ in $q$ is a double plane. This proves the first part.

Now suppose $\l_2$ is infinitely near to $\l_1$. The blow up at $p$ and $\l_1$ gives:
\[ V_{\l_1}=f^p+2\l_2 \]
where $\l_2=(\Gamma_3)_{\l_1}$ and $f^p=E_{\l_1}\cap E_p$, as explained above. 
Since $q\neq p$, the tangent cone of $V$ in $q$ is $2\Gamma_3$.
\\ \end{proof}

The base locus of $\X$ is now described:\label{sc_blocus_sec}

\begin{prop} \label{sc_base_locus}
Let $\l_1$, $\l_2$ and $\l_3$ be the double lines of $V$ and let $p$ be its triple point. Suppose that the base locus of $\X$ has dimension one, except for $p$. Then $\X$ is the linear system of surfaces with degree nine having:
\begin{itemize}
\item multiplicity six in $p$
\item multiplicity four along each $\l_i$
\item multiplicity two along $C_6$, 
\end{itemize} 
where $C_6$ is a degree six curve in $V$, image via $\bar{\tau}$ of a cubic in $E$. In particular, $C_6$ cuts each $\l_i$ in three points, supposing it does not contain any of these lines.

Moreover, $X$ has degree nine and $p$ is mapped to a point $x_p$ of multiplicity four in $X$. If $C_6$ does not contain $p$, the tangent cone of $X$ in $x_p$ is a cone over a Veronese quartic surface.
\end{prop}
\begin{proof}
Let $d_i$ be the multiplicity of $\X$ in $\l_i$. Since $\Xpp$ is the linear system of planes, $d_i\leq 4$. And since $\Xp$ desingularizes $V$ to $E$, $d_i-2\geq 1$. Therefore $d_i\in\{3,4\}$

The linear system $\Xp$ has multiplicity $d_i-2$ along $\l_i$. Its restriction to $V$ defines the inverse of $\bar{\tau}$. Then:
\[ \Xp \cap V = 2(d_1-2)\l_1+2(d_2-2)\l_2+2(d_3-2)\l_3 + \{ \text{conics} \} +C \]
for a fixed curve $C$. Since $\deg\Xp=5$, this implies:
\begin{equation}\label{sc_eq_d1d2d3}
\deg C = 20- 2(d_1-2)-2(d_2-2)-2(d_3-2)-2
\end{equation}

But $\X$ has multiplicity two along $C$. So its intersection with $V$ is:
\[ \X\cap V = 2d_1\l_1+2d_2\l_2+2d_3\l_3 + 2C+D \]
for a certain curve $D$. This implies:
\[ 36\geq 2d_1+2d_2+2d_3+2\deg C \]
Plugging \eqref{sc_eq_d1d2d3} in the above inequality, gives:
\[ d_1+ d_2+ d_3 \geq 12 \] 

Therefore $d_i=4$ for $i=1,2,3$. In particular there is no curve $D$ and $C=C_6$ has degree six. This proves the assertion on the base locus. 

Therefore the intersection of $\X$ with $V$ is:
\[ \X\cap V = 8\l_1+8\l_2+8\l_3 + 2C_6 \]

A general plane is cut by $\X$ in degree nine curves with three points of multiplicity four and six double points. It is mapped to a general tangent hyperplane section of $X$. Then $X$ has degree $81-3\cdot 4^2-6\cdot 2^2=9$.

To find the multiplicity of $\X$ in $p$, consider the blow up at this point. Then $\X_p$ has multiplicity four in three non collinear points, so $mult_p \X\geq 6$. But since $\Xpp$ has no base points and $p$ is a triple point of $V$, $mult_p \X\leq 6$ and equality holds.

Since $\X_p$ has degree six and three points of multiplicity four, it is the union of three fixed double lines. Some of these lines are infinitely near if $V$ is not the Steiner's Roman Surface. Using the notation of \eqref{sc_eq_Vp}, we have:
\begin{equation}\label{sc_eq_Xp}
\X_p=2t_1+2t_2+2t_3
\end{equation}

Therefore $\X_p$ is a fixed curve and $p$ is mapped to a point $x_p$ in $X$. Let $\Omega$ be a general plane through $p$. Then $\Omega_p$ intersects $\X_p$ in three (possibly infinitely near) fixed double points. After blowing up these points, $\Omega_p$ has no intersection with $\X$ and $(\Omega_p)^2=-4$ in $\Omega$. Then $x_p$ is a point of multiplicity four in the hyperplane section $\sigma(\Omega)$ of $X$. By Lemma \ref{pr_truque_secao_tgente}, $X$ has multiplicity four in $x_p$.

Finally, suppose $p\notin C_6$. Let $\wX$ be the linear system of surfaces in $\X$ having multiplicity seven in $p$. Then, the blow up in $p$ gives:
\[ \wX_p=t_1+t_2+t_3+\{\text{quartic curves}\} \]
The moving part has three double points, namely $t_i\cap t_j=\l_k\cap E_p$, with $i,j,k$ distinct. If $V$ is not the Steiner's Roman surface, some of these points are infinitely near. Since $p \notin C_6$, it has no other base points. A standard quadratic Cremona transformation maps the moving part of $\wX$ to a linear system of conics having no base points. Therefore, by Lemma \ref{pr_truqueconetangente} the projectivization of the tangent cone of $X$ in $x_p$ contains a Veronese surface. Since $X$ has multiplicity four in $x_p$, the result follows.
\\ \end{proof}

Note that the possible irreducible components of $C_6$ are conics, quartics or the double lines $\l_i$. Indeed, $C_6$ is the image via $\bar{\tau}$ of a cubic in $E$. But $\bar{\sigma}$ maps any curve of degree $d$ to a curve of degree $2d$, except for the three lines $\hat{\l}_i$. This proves the assertion.

\section{Images of $\l_i$}
\label{sc_imagens_li_sec}

In this section we study the image of the lines $\l_i$ via $\sigma$. We'll keep the hypothesis of Proposition \ref{sc_base_locus}, that is, $\X$ has degree nine and pure dimension one except for $p$. We will first consider the case in which $C_6$ does not contain any of the lines $\l_i$, and then the other cases.

\begin{prop}\label{sc_imagem_li}
Suppose $C_6$ does not contain any of the lines of $V$ and fix $i\in\{1,2,3\}$. If $V$ has no double line infinitely near to $\l_i$, then $\l_i$ is mapped by $\sigma$ to a weak Del Pezzo surface $^iD_4^x$ of degree four through $x$ having a double point in $x_p$. If there is a double line infinitely near to $\l_i$, then $\l_i$ is mapped to a double line $R_i$ of $X$ not containing $x_p$.
\end{prop}
\begin{proof}
Set $\l=\l_i$. Suppose first that $\l$ is a proper double line of $V$ with no other double line infinitely near to it. Let $\Omega$ be a general plane containing $\l$. Then $\X\cap \Omega$ consists of $4l$ and moving quintic curves, which intersect $\l$ in $p$ with multiplicity two and in three moving points. 

After the blow up at $p$, the normal bundle of $\l$ is:
\[ N_\l = \mO_{\P^1}(0)\oplus \mO_{\P^1}(0) \]

Blowing up $\l$ gives $\X_\l\equiv (3,4)$. Since there is no double line of $V$ infinitely near to $\l$, $\X_\l$ has no fixed components. It has two double points in $f^p_{\l}=E_p\cap E_\l$, given by the intersections with two of the lines $t_i$ (see \eqref{sc_eq_Vp}). One of these points is infinitely near if $V$ is not the Steiner's Roman surface. It also has three double points in the intersections with $C_6$, which does not contain $\l$. The curve of type $(1,2)$ through these five base points is $V_\l$ and is mapped to $x$.

By Lemma \ref{pr_imagemF1F2}, $\X_\l$ is birationally equivalent to degree seven curves in $\P^2$ with one point of multiplicity four, one triple point and five double points. The fiber $f^p_\l$ is mapped to a line through the triple point and two of the double points. After two standard quadratic Cremona transformations, the degree seven curves are mapped to cubics with five base points. Three of these points lie on a line, the image of $f^p_\l$. Since $f^p_\l$ is mapped to $x_p$, $E_\l$ is mapped to a weak Del Pezzo quartic surface through $x$ with a double point in $x_p$.

Now suppose $\l$ is a double line of $V$ infinitely near to $\l^\prime$, which can be proper or infinitely near to another line $\l^{\prime\prime}$. Blowing up $p$ (and $\l^{\prime\prime}$, if it is the case) and $\l^\prime$, by \eqref{sc_eq_reta_ip} and \eqref{sc_eq_V_l2} we have that $\l\equiv (0,1)$ in $E_{\l^\prime}$. As remarked above, $\X_{\l^\prime}\equiv (3,4)$ in both cases. Moreover, it has two double points in $f^p_{\l^\prime}=E_p\cap E_{\l^\prime}$, due to \eqref{sc_eq_Vp} and \eqref{sc_eq_Xp}. Therefore:
\begin{align}\label{sc_eq_reta_dupla_ip}
\X_{\l^\prime}=4\l+ 2f^p_{\l^\prime}+\{\text{moving fibers} \}\equiv (0,4)+(2,0)+(1,0) 
\end{align} 
Then $\l^\prime$ is mapped to a line $R^\prime$.

If $\l^\prime$ is a proper line of $V$ (and $V$ is of type (ii) in Proposition \ref{sc_projecoes_da_veronese}), then $\l=\Gamma\cap E_{\l^\prime}$, with $\Gamma\cong \F_1$ being the strict transform of the plane spanned by $\l^\prime$ and $\l$. 

If $\l^\prime$ is infinitely near to $\l^{\prime\prime}$, then $\l=Q\cap E_{\l^\prime}$, with $Q$ being the strict transform of a quadric containing the three double lines of $V$. This implies that the original quadric $Q$ is singular in $p$, since $Q_p$ contains three non collinear points. After the blow up at $p$ we get, in $Q\cong\F_2$, $\l^{\prime\prime}\equiv f_2$ and $Q_p\equiv e_2$. Then:
\[ V\cap Q=6\l^{\prime\prime}+F \equiv e_2+8f_2 \]
where $F\equiv e_2+2f_2$ is a plane section of $Q$. Blowing up $\l^{\prime\prime}$ and $\l^\prime$ does not affect $Q$, and gives $\l\equiv f_2$ in $Q$. Analogously, the intersection of $\X$ with $Q$ is:
\[ \X\cap Q=12\l^{\prime\prime}+G \equiv 3e_2+18f_2 \]
where $G\equiv 3e_2+6f_2$ is a cubic section of $Q$.

Therefore, in both cases $\l$ has normal bundle:
\[ N_\l = \mO_{\P^1}(0)\oplus \mO_{\P^1}(0) \]
and blowing up $\l$ gives $V_\l\equiv (1,2)$ and $\X_\l\equiv (3,4)$. $\X_\l$ has two double points in $E_p\cap E_\l$, and three double points in $C_6\cap E_\l$, giving the same linear system as before. Then $E_\l$ is  mapped to a weak Del Pezzo surface $D_4^x$.

Next we prove that $R^\prime$ does not contain $x_p$. First note that $R^\prime\subset D_4^x$, since $(E_\l)_{\l^\prime}$ is a curve of type $(0,1)$ in $E_{\l^\prime}$, so it is mapped to $R^\prime$. Now,  one of the two double points in $E_p\cap E_\l$ is $f^p_{\l^\prime}\cap E_\l\subset E_{\l^\prime}\cap E_\l$. The other double point is infinitely near to it if $\l^\prime$ is not a proper line of $V$.

Now note that $E_p\cap E_\l$ and $E_{\l^{\prime}}\cap E_\l$ are transversal and consider the blow up of the point $q=f^p_{\l^\prime}\cap E_\l$, which is the intersection of these two curves. If $\l^\prime$ is a proper line of $V$, $\X_\l$ intersects the exceptional divisor  in two moving points, mapping it to a conic. If $\l^\prime$ is infinitely near to $\l^{\prime\prime}$, $\X_\l$ intersects the exceptional divisor in another double point $q^\prime$, lying on $E_\l\cap E_p$. Blowing up this point, $\X_\l$ intersects the new exceptional divisor in two moving points. So in both cases, the last exceptional divisor is mapped to a conic. One point of this conic is $x_p$, image of $E_\l\cap E_p$. And a distinct point of it lies on the image of $E_{\l^\prime}\cap E_\l$, namely $R^\prime$.  Therefore $x_p\notin R^\prime$, since no other curve of $E_\l$ intersecting $E_p$ and $E_{\l^\prime}$ is contracted.

In Figure \ref{sc_figura_l''l'l} is a sketch of these blow ups when $\l^{\prime\prime}\prec \l^\prime\prec \l$.

\begin{figure}[tb]
\centering
\includegraphics[trim=40 380 10 150,clip,width=1\textwidth]{figsc3.pdf}
\caption{Blow ups when $\l^{\prime\prime}\prec \l^\prime\prec \l$}\label{sc_figura_l''l'l}
\end{figure}

To prove that $\l^\prime$ is mapped to a double line of $X$, suppose it is a proper line of $V$ and let $q$ be a general point of $\l^\prime$ and $\Omega$ be a general plane through $q$. Then $\X\cap \Omega$ has multiplicity four in $q$. Blowing up $q$, by \eqref{sc_eq_reta_dupla_ip} $\X\cap \Omega$ has a fixed point of multiplicity four in $\Omega_q$, corresponding to $\l$. Blowing up this point, $\Omega_q$ has self-intersection $-2$ in $\Omega$ and is not cut by $\X\cap \Omega$. Then it is mapped to a double point in the image of $\Omega$, a hyperplane section of $X$. By Lemma \ref{pr_truque_secao_tgente}, it is a double point of $X$, which implies that the image of $\l^\prime$ is a double line.

Clearly the same reasoning applies if $\l^{\prime\prime}\prec \l^\prime\prec \l$, giving a double line $R^{\prime\prime}$ and a second double line $R^\prime$ infinitely near to it.
\\ \end{proof}

A similar result holds when $C_6$ contains double lines of $V$. We'll study a simple case to illustrate this situation.

First note that $\l_i\subset C_6$ means that the cubic curve $\bar{\sigma}(C_6)\subset E$ contains the line $\hat{\l}_i$. This implies that $C_6=2\l_i+C_4$. Since $\X$ has multiplicity four in $\l_i$, this means that $\X$ has a double curve infinitely near to this line. This double curve is determined by $V$, since $C_6\subset V$. In other words, blowing up $p$ and then $\l=\l_i$, the linear system $\X$ has multiplicity two along the curve $V_{\l}$. The reason we blow up $p$ first is to avoid fixed components, since $\mult_p\X>\mult_\l\X$.

From these considerations, it follows that if, for example, $\l_3$ lies infinitely near to $\l_1$, then  $\l_1\subset C_6$ implies $\l_3\subset C_6$. So we write $C_6=2\l_3+C_4$.

Below is a result similar to Proposition \ref{sc_imagem_li}, supposing $\l_1\subset C_6$.

\begin{lemma}\label{sc_2l1+C4}
Suppose that $C_6=2\l_1+C_4$ and suppose $V$ is the Steiner's Roman Surface. Then $\l_1$ is mapped to a triple line $L_1$ of $X$ through $x_p$, and the double curve of $\X$ infinitely near to $\l_1$ is mapped to a projection of a quartic Veronese surface through $x$ having multiplicity two along $L_1$ and spanning a $\P^4$.
\end{lemma}
\begin{proof}
Set $\l=\l_1$. The blow ups that will now be done are represented in Figure \ref{sc_figura_2l1+c4}.

\begin{figure}[tb]
\centering
\includegraphics[trim=40 10 70 460,clip,width=1\textwidth]{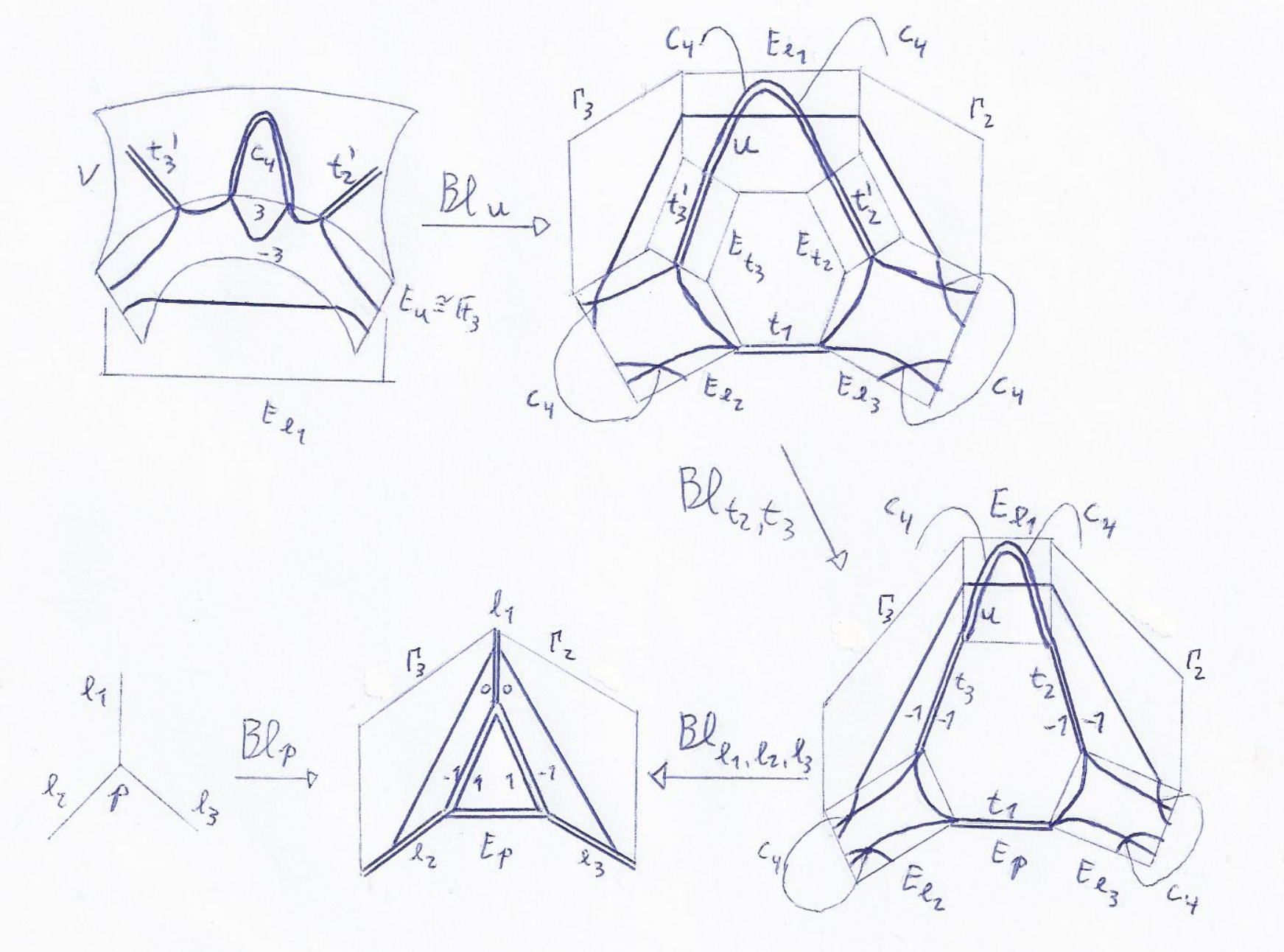}
\caption{Blow ups when $C_6=2\l_1+C_4$}\label{sc_figura_2l1+c4}
\end{figure}

Start blowing up $p$. Set $t_i=(\Gamma_i)_p$. Then:
\[ \X_p=2t_1+2t_2+2t_3=2V_p \]

The line $\l$ is the complete intersection of $\Gamma_2$ and $\Gamma_3$. After the blow up at $p$, we get $\l^2=0$ in both surfaces. Blow up $\l$. We want to study the image of $E_\l\cong \F_0$ via $\sigma$.  As remarked above, since $C_6=2\l+C_4$, $\X_\l$ has multiplicity two along $V_\l$.

A general plane through $\l$ is intersected by $V$ in $2\l$ and a conic through $p$. Therefore $V_\l\equiv (1,2)$, a rational curve. On the other hand, a plane through $\l$ is cut by $\X$ in $4l$ and quintic curves having a double point in $p$. Then $\X_\l\equiv (3,4)$. This implies:
\[ \X_\l=2u+\{\text{fibers}\}\equiv (3,4) \]
where $u=V_\l$ and the moving part consists of fibers in $E_\l$. Therefore $E_\l$ is mapped to a line $L_1$. Since $(E_p)_\l\equiv (1,0)$, $L_1$ contains $x_p$. Note that the moving part of $\X_\l$ cuts $u$ in a degree two non complete linear series.

Before proceeding to the investigation of the image of $u$,  other blow ups will be made in order to bring to light the existence of base curves of $\X$ infinitely near to $t_2$ and $t_3$. First consider the blow up of the lines $\l_2$ and $\l_3$, since $\X$ has multiplicity four along them.

The curve $t_2$ is the complete intersection of $\Gamma_2$ and $E_p$. After the blow ups along the three lines $\l_i$, $t_2$ has self-intersection $(-1)$ in both surfaces. The same holds for $t_3$. Then, for $i=2,3$ the normal bundle of $t_i$ is:
\[ N_{t_i} = \mO_{\P^1}(-1)\oplus \mO_{\P^1}(-1) \]

Set $t=t_2$. Remember that the moving part of $\X\cap\Gamma_2$ does not intersect $t$. Then, blowing up $t$, we get $\X_{t}\equiv (0,2)$. These curves have multiplicity two in the point $u_t$,  therefore $\X_t=2t^\prime_2$, where $t^\prime_2$ is the fiber containing $u_t$. Since $t_2$ did not intersect $t_3$, this blow up did not affect $t_3$. So the same reasoning can be made, giving $\X_{t_3}=2t^\prime_3$.

The curves $t^\prime_2$ and $t^\prime_3$ are actually double curves of $\X$ (infinitely near to $t_2$ and $t_3$). To prove this assertion, let $\Omega$ be a general plane through $p$ and denote intersections with $\Omega$ with an index (for instance, $\X\cap\Omega=\X_\Omega$). We have to repeat the blow ups made above and prove that $\X_\Omega$ has multiplicity two in $(t^\prime_2)_\Omega$ and $(t^\prime_3)_\Omega$. Before the blow ups, $\X_\Omega$ has multiplicity six in $p$ and multiplicity two in four points $(C_4)_\Omega$. After the blow up of $p$, $\X_\Omega$ has three double points on the exceptional divisor, namely $(t_i)_\Omega$, with $i=1,2,3$. Blowing up $\l_1,\l_2,\l_3,t_2$ and $t_3$, we have $t^\prime_2\subset E_{t_2}$ and $t^\prime_3\subset E_{t_3}$. Then the degree of the image of $\Omega$ via $\X$ is:
\[ d=9^2-6^2-4\cdot 2^2-3\cdot 2^2-2\cdot\delta=17-2\delta \]
where $\delta$ corresponds to $(t^\prime_j)_\Omega$, for $j=2,3$. If it is a double point of $\X_\Omega$, then $\delta=4$. If it is a base point with a second infinitely near base point, then $\delta=2$. But the image of $\Omega$ is a tangent hyperplane section of $X$, giving $d=9$ and $\delta=4$. This proves the assertion.

Now let's investigate the normal bundle of $u=E_\l\cap V$. As already noted, $u$ is rational. In $E_\l$, two points of $u$ were blown up, so $u^2=4-2=2$. In $V$,  $u$ is the line $\l$ blown up at $p$, since all the curves that were blown up lied in $V$. To compute $\l^2$, we look at the curve $\bar{\l}$ in the normalization $\bar{V}$ of $V$ that is mapped to $\l$. It is a conic and two of its points are mapped to $p$. The self-intersection of a conic in the Veronese variety $\bar{V}$ is $1$. Then $u^2=1-2=-1$. Therefore the normal bundle of $u$ is:
\[ N_u = \mO_{\P^1}(2)\oplus \mO_{\P^1}(-1) \]

Blowing up $u$ gives $E_u\cong\F_3$, $V_u\equiv e_3+3f_3$ and $(E_\l)_u\equiv e_3$. $\X_u$ intersects $(E_\l)_u$ in two points moving in a non complete linear series and intersects $V_u$ in four fixed double points. These points are $(t^\prime_2)_u$,$(t^\prime_3)_u$ and two points in $(C_4)_u$. Therefore:
\[ \X_u\equiv 2e_3+8f_3 \]
having four double base points. By Lemma \ref{pr_imagemF1F2}, $\X_u$ can be birationally mapped to a linear system in $\P^2$ of degree eight curves with four double points, a point of multiplicity six and two double points infinitely near to it. After applying three standard quadratic maps in $\P^2$, these curves are mapped to conics with no base points. Since $\X_u$ is not complete, this is a non complete linear system. Therefore the image of $E_u$ is a quartic surface, a projection of a Veronese surface. The curve $(E_\l)_u$ is mapped to a double line of it, namely $L_1$, and $V_u$ is contracted to $x$. Since it has no other double lines (proper or infinitely near), this surface spans a $\P^4$.

We are left to prove that $L_1$ is a triple line of $X$. Note that a point in $\l_1$ is blown up and mapped back to a point in $L_1$. So let $q$ be a general point in $\l_1$, let $x_q$ be its image in $X$ and let $\Omega$ be a general plane through $q$. Blowing up $p$ and $\l=\l_1$ as above, we see that $\Omega_\l$ is the fiber over the point $q$, intersecting $u$ in two points. 

Therefore $\X\cap\Omega$ has multiplicity four in $q$ and two double points infinitely near to it. After the blow up at these points, the self-intersection of the exceptional divisor of $q$ in $\Omega$ is $-3$, having no intersection with $\X\cap\Omega$. Then $\Omega$ is mapped to a tangent section of $X$ having multiplicity three in $x_q$. By Lemma \ref{pr_truque_secao_tgente}, $X$ has multiplicity three in $x_q$ and $L_1$ is a triple line.
\\ \end{proof}

If $C_6$ contains other lines, the threefold $X$ will have other triple lines. It may also happen that $C_6=4\l_1+C_2$ or $C_6=6\l_1$, in which cases $X$ has singularities infinitely near to $L_1$. Similar results also hold when $V$ is not the Steiner's Roman Surface.  We will not dwell on this here.

\section{Quartic surfaces in $X$}

A consequence of Proposition \ref{sc_imagem_li} and Lemma \ref{sc_2l1+C4} is that there is at least one quartic surface through  $x$ in $X$ having a double point in $x_p$. If $C_6$ does not contain lines, then it is a weak Del Pezzo surface. Otherwise it can be a surface with a double line. In both cases, this surface spans a $\P^4$.

Since $x$ is a general point of $X$,  these surfaces cover $X$. The following Lemma explains their images in $\P^3$.

\begin{lemma}\label{sc_mQi_pencil}
Fix $i\in \{1,2,3\}$ and suppose $V$ does not have a double line infinitely near to $\l_i$. Let $\mQ_i$ be the linear system of quartic surfaces in $\P^{3}$ containing $C_6$ and $\l_i$ and having multiplicity two in $\l_j$, for $j\neq i$. Then $\mQ_i$ is a pencil and it is mapped to a family $\mQ^{\prime}_i$ of quartic surfaces covering $X$. 
The surface in $\mQ^\prime_i$ corresponding to $V\in\mQ_i$ is the image of $\l_i$.
\end{lemma}
\begin{proof}
The last assertion follows from the fact that $V$ is the surface in $\mQ_i$ with an extra multiplicity along $\l_i$ and from the fact that $V$ is contracted to $x$.

To simplify the notation, set $i=1$ and $\mQ=\mQ_1$. Let us first show that there is a surface different from $V$ in $\mQ$. We'll obtain such surface as the union of a cubic $S_3$ and a plane. By Lemma \ref{sl_cubicas_por_C6}, there is at least a pencil of cubic surfaces in $\P^3$ containing $C_6$. This holds in the general case, where $C_6$ is an elliptic curve, so it must also hold in the degenerated cases.

If $p\notin C_6$, let $S_3$ be the cubic containing $C_6$ and $p$. By Proposition \ref{sc_base_locus}, $C_6$ cuts each of the lines $\l_1,\l_2,\l_3$ in three points, so $S_3$ intersects these lines in four points. Then $S_3$ contains $C_6,\l_1,\l_2$ and $\l_3$.

If $C_6$ is smooth in $p$, then it is the image via $\bar{\tau}$ of a cubic in $E$ through one of the three preimages of $p$, say $\hat{\l}_1\cap\hat{\l}_2$. This cubic intersects $\hat{\l}_1$ and $\hat{\l}_2$ in two other points each, and intersects $\hat{\l}_3$ in three points. This implies that every cubic surface containing $C_6$ intersects $\l_3$ in $p$ and in other three points, so it contains $\l_3$. In particular, these cubic surfaces have fixed tangent plane in $p$. Therefore there is a cubic $S_3$ containing $C_6$ and having a double point in $p$. This cubic contains the three lines $\l_1,\l_2,\l_3$.

Similar considerations prove that if $p$ is a singular point of $C_6$, then there is a cubic surface $S_3$ containing $C_6$, $\l_1$, $\l_2$ and $\l_3$. This can be achieved by doing a case by case analysis in the types of singularities of $C_6$ and its intersections with these lines. 

Now, consider the union of $S_3$ and the plane $\Gamma_1$, which exists since $V$ does not have a double line infinitely near to $\l_1$. This  is a surface in $\mQ$ different from $V$. 

To prove that $\mQ$ is a pencil, let $S$ be a quartic in $\mQ$ different from $V$. Then:
\begin{equation}\label{sc_eq_baselocus_mQ}
S\cap V=2\l_1+4\l_2+4\l_3+C_6 
\end{equation}
that is, the surfaces in $\mQ$ have fixed intersection with $V$. Repeating the argument used in Section \ref{sl_cubicsfamily_sec}, the exact sequence:
\[ 0 \to H^0(\P^3,\mQ-V) \to H^0(\P^3,\mQ) \to H^0(\P^3,\mQ\vert_V) \to 0 \]
shows that $\mQ$ is a pencil.

Next we compute the image of $S$ via $\sigma$. Intersecting with $\X$ gives:
\[ \X\cap S = 4\l_1+8\l_2+8\l_3+2C_6+F_4 \]
where $F_4$ is the moving part. However, note that the fixed part is the intersection of $2V$ with $S$, so $F_4$ is a base point free plane section of $S$. Then the image of $S$ is a quartic surface in $X$.

The surfaces in $\mQ^\prime$ clearly cover $X$, since $\mQ$ covers $\P^3$.
\\ \end{proof}

Two remarks have to be made on the linear systems $\mQ_i$. Note that if, for instance, $\l_2$ is infinitely near to $\l_1$ (i.e., $V$ is not the Steiner's Roman Surface), then the surfaces in $\mQ_1$ have multiplicity two along $\l_2,\l_3$ and multiplicity one along $\l_1$. Applying Enrique's unloading principle (cf. \cite{enrchi}, IV.17), we get surfaces having multiplicity two along $\l_1$ and $\l_3$ and containing $\l_2$. Therefore $\mQ_1$ coincides with $\mQ_2$. To avoid confusion, we will not consider $\mQ_1$ in this case. This explains the hypothesis made in the Lemma.

A second remark is that $C_6$ may contain a  double line of $V$, say $\l_1$. In this case, the surfaces of $\mQ_i$ (for $i=1,2,3$) have in common a curve $u$ infinitely near to $\l_1$, determined by $V$, in order to have \eqref{sc_eq_baselocus_mQ} (and corresponding equations for $i=2,3$). This is the same double curve of $\X$, of type $(1,2)$ in $E_{\l_1}$, as it is explained in the proof of Lemma \ref{sc_2l1+C4}. 

This implies that any surface containing $\l_1$ and $u$ actually has multiplicity two along $\l_1$. Therefore $\mQ_1$ is, in the case where $C_6=2\l_1+C_4$, the linear system of surfaces having multiplicity two along $\l_1,\l_2,\l_3$ and containing $C_4$. This gives, for $S\in\mQ_1$ different from $V$:
\[ S\cap V=4\l_1+4\l_2+4\l_3+C_4 \] 
which agrees with \eqref{sc_eq_baselocus_mQ}. Note that in this case, $u$ is not a base curve of $\mQ_1$, since these surfaces already have multiplicity two along $\l_1$. 

On the other hand, $\mQ_2$ is the linear system having multiplicity two along $\l_1$ and $\l_3$ and containing $C_4$, $u$ and $\l_2$. And the linear system $\mQ_3$ is defined in the obvious way. In all these cases, the proof that $\mQ_i$ is a pencil follows from the same argument given in Lemma \ref{sc_mQi_pencil}.

The next result gives more details on the families $\mQ_i^\prime$.

\begin{lemma}\label{sc_mqi_delpezzo}
The base locus of $\mQ_i^\prime$ is a line $L_i$ through $x_p$. There are possibly other base curves infinitely near to $L_i$, in case $\l_i\subset C_6$. 
Moreover, if $\l_i\not\subset C_6$, then the general surface in $\mQ_i^\prime$ is a quartic weak Del Pezzo surface having a double point in $x_p\in L_i$. If $\l_i\subset C_6$, it is a quartic surface having multiplicity two along $L_i$.
\end{lemma}
\begin{proof}
By the definition of $\mQ_i$, $V$ does not have a double line infinitely near to $\l_i$. By Lemma \ref{sc_imagem_li}, if $\l_i\not\subset C_6$ then $\l_i$ is mapped to a quartic weak Del Pezzo surface through $x$ having a double point in $x_p\in L_i$. This is the surface of $\mQ_i^\prime$ through $x$. Since $x$ is a general point of $X$, this is a general surface of $\mQ_i^\prime$. If $\l_i\subset C_6$, the result follows from Lemma \ref{sc_2l1+C4}.

To prove the assertion on the base locus, suppose $i=1$ and set $\mQ=\mQ_i$. Assume first that $\l_1\not\subset C_6$. Equation \eqref{sc_eq_baselocus_mQ} gives us the base locus of $\mQ$, consisting of base curves of $\X$, and a curve infinitely near to $\l_1$, since a general surface of $\mQ$ is smooth along $\l_1$. 

The base locus of $\mQ^\prime$ arises either from curves in the base locus of $\mQ$ not lying on the base locus of $\X$ or from curves in $\mQ$ that are contracted by $\sigma$. Only the curve infinitely near to $\l_1$ fits in the first case. On the other hand, $\sigma$ contracts $V$ and exceptional surfaces of blow ups. But $\mQ$ intersects $V$ in the base locus of $\X$ and if an exceptional surface is contracted, $\X$ has a base curve in it. And $\mQ$ contains the base curves of $\X$, the only exception happening when $\l_1\subset C_6$. Therefore the image of the base curve of $\mQ$ infinitely near to $\l_1$ is the base locus of $\mQ^\prime$.

Blowing up $p$, the curves $\mQ_p$ have multiplicity two in $(\l_2)_p$ and $(\l_3)_p$ and multiplicity one in $(\l_1)_p$. Then $\mQ_p$ is the union of the line $(\Gamma_1)_p$ and movable conics through these three points. In particular, $\mQ_p$ has  multiplicity three in $p$ and it intersects a general plane through $\l_1$ in $\l_1$ and cubics with a double point in $p$. 

Now blow up $\l_1$. The intersection of $\mQ_1$ with the exceptional divisor $E_{\l_1}\cong \F_0$ is a curve of type $(1,1)$. Since $\l_1\not\subset C_6$, this curve contains the three points $C_6\cap E_{\l_1}$. But fixed three non collinear points in $\F_0$, there is only one curve of type $(1,1)$ through them. And these points are indeed non collinear, since they lie on $V_{\l_1}$, which is irreducible. We have thus shown that:
\[  \mQ_{\l_1}=t_L^1\equiv (1,1) \]
where $t_L=t_L^1$ is a fixed curve. By Lemma \ref{sc_imagem_li}, $\X_{\l_1}\equiv (3,4)$ and it has three double points in $t_L$. Then $t_L$ is mapped to a line $L=L_1$, which is the base locus of $\mQ^\prime$. It contains the point $x_p$, image of $(E_p)_{\l_1}\equiv (1,0)$.

Suppose now that $\l_1\subset C_6$. We'll explain the case $C_6=2\l_1+C_4$, the other cases being similar. As explained in a remark after Lemma \ref{sc_mQi_pencil}, $\X$ has a double curve $u$ infinitely near to $\l_1$, whereas $\mQ$ has multiplicity two along $\l_1$ and does not contain $u$. Then blowing up $p$ gives:
\[ \mQ=t_1+t_2+t_3 \]
keeping the notation of \eqref{sc_eq_Vp}. A general plane through $\l_1$ intersects $\mQ$ in $2\l_1$ and conics through $p$. Then the blow up in $\l_1$ gives $\mQ_{\l_1}\equiv (1,2)$. The movable part of $\X_{\l_1}$, consisting of fibers, contracts $E_{\l_1}$ to the line $L_1$ defined in Lemma \ref{sc_2l1+C4}. Then it follows that this is a double line of $\mQ^\prime$.

As remarked after Lemma \ref{sc_2l1+C4}, if $C_6=4\l_1+C_2$ or $C_6=6\l_1$, $\X$ maps curves infinitely near to $\l_1$ to singular curves infinitely near to $L_1$. In these cases, these curves lie on the base locus of $\mQ^\prime$.

Since the base locus of $\mQ$ is contained in the base locus of $\X$, the base locus of $\mQ^\prime$ is $L_1$ and possibly other curves infinitely near to $L_1$.
\\ \end{proof}

\begin{remark}\label{sc_rmk_retas_Li}
The line $L_i$ in $X$ was defined in two different ways in this section. In the above Lemma, it is defined as the base locus of $\mQ_i$. In Lemma \ref{sc_2l1+C4}, $L_i$ is defined as the contraction of $E_{\l_i}$ when $\l_i\subset C_6$. In the proof of Lemma \ref{sc_mqi_delpezzo} it is explained that these definitions coincide.
\end{remark}

We now explain the first consequence of the results of this section. Unlike Chapter \ref{sl_chapter}, different quartic fundamental surfaces produce different Bronowski threefolds, instead of representing different tangential projections of the same variety.

\begin{prop}
Let $x$ and $y$ be two points of the   Bronowski threefold $X$ such that the tangential projections $\tau_x$ and $\tau_y$ are birational. If the fundamental surfaces associated to $\tau_x$ and $\tau_y$ have degree four, then they are projectively equivalent.
\end{prop}
\begin{proof}
The projective equivalence classes of these quartic surfaces are given by (i), (ii) and (iii) in Proposition \ref{sc_projecoes_da_veronese}. In (i), $V$ is the Steiner's Roman Surface and by Lemma \ref{sc_mQi_pencil}, there are three pencils of quartic surfaces in $X$. Each of  these pencils covers $X$ and each quartic surface spans a $\P^4$. In the same result, it is shown that in case (ii) there are two such pencils in $X$ and in case (i) there is one. The goal is to show that in cases (ii) and (iii), the variety $X$ does not contain another pencil covering $X$ consisting of quartic surfaces, each spanning a $\P^4$.

Note that there is no quartic surface belonging to two different families. This follows from the fact that a quartic surface (in $\P^3$) defines a unique complete linear system containing it (in this case, a pencil).
Then we only need to show that there is no other quartic surface spanning a $\P^4$ through the point $x$ from which is made the tangential projection. This shows that the cases (i), (ii) and (iii) produce different threefolds.

Suppose that, except from the quartic surfaces described in Lemma \ref{sc_mQi_pencil}, there is another quartic surface $S$ through $x$ spanning a $\P^4$. Since $S\subset X$, $T_xS$ is a subspace of $T_xX\cong \P^3$, the center of the projection. Therefore, every hyperplane containing $T_xX$ contains $T_xS$, the induced map on $S$ is defined by a sublinear system of its tangential projection at $x$.  

If $S$ is smooth at $x$, the image of $S$ is either a line $\l$ or a base point of $\X$. 

In the first case, a tangent hyperplane section of $S$ at $x$ (a quartic curve with a double point in $x$) is contracted to a point in $\l$. Therefore, $\X$ has multiplicity four along $\l$, since $\sigma$ maps points of $\l$ to quartic curves. This implies that $\l$ is a double line of $V$. By Proposition \ref{sc_imagem_li} and Lemma \ref{sc_2l1+C4}, $S$ must be one of the surfaces described in these results. 

Suppose now it is a base point of $\X$. But, as seen in Section \ref{sc_imagens_li_sec} and in Proposition \ref{sc_base_locus}, a point of a line $\l_i$ is not mapped to a surface. And neither is a point of $C_6$, since these are mapped to points or curves in $X$.

If $S$ is singular at $x$, the image of $S$ is again a base point of $\X$. As seen above, this is not possible

Then, in both cases, $S$ is a surface in $\mQ_i^\prime$.
\\ \end{proof}

This result also follows from the considerations in the next Section on the singularities of $X$. By Lemma \ref{sc_imagem_li} if $\l_1\prec\l_2$, then $\l_1$ is mapped to a double line $R_1$ of $X$, and if $\l_1\prec\l_2\prec\l_3$, then $X$ has a second double line $R_2$ infinitely near to $R_1$. In the next Section, we'll see this is the only way to produce double lines in $X$.

\section{Singularities of $X$}
\label{sc_singularities_sec}

According to Lemma \ref{sc_base_locus}, the point $p$ is mapped to a point $x_p$ of multiplicity four in $X$. Moreover, Proposition \ref{sc_imagem_li} asserts that if $V$ is not the Steiner's Roman Surface, then $X$ has one or more double lines $R_i$. Apart from these, the other singularities of $X$ will depend, as in the other chapters, on the singularities of $C_6$:

\begin{lemma}\label{sc_singularidades_vem_de_C6}
Let $x_q\neq x_p$ be an isolated singularity of $X$ or a general point on a singular curve of $X$. Then $x_q$ is either a point of a double line $R_i$, image of $\l_i\in V$, or it is mapped by $\tau$ to a singular point of $C_6$. This includes general points in non reduced components of $C_6$.
\end{lemma}
\begin{proof}
Remember first that a double line $R_i$ is the contraction of $E_{\l_1}$, when $V$ has a double line infinitely near to $\l_i$, as it was defined in Lemma \ref{sc_imagem_li}.

As noted in the other chapters,  $x_q$ is mapped to a point $q$ in the base locus of $\X$. 

If $x_q\in \tau^{-1}(p)$, then it lies on the $\P^4$ spanned by $T_xX$ and $x_p$. Varying $x$ in  $X$, we find that $x_q$ cannot be an isolated singular point, since by Terracini's Lemma two general tangent spaces of $X$ are disjoint and $x_q\neq x_p$. By the same reason, if a singular curve of $X$ contains $x_q$, it does not lie on $\tau^{-1}(p)$. Therefore it is projected to a curve in the base locus of $\X$.

This shows that if $x_q$ is an isolated singular point or a general point in a singular curve of $X$, then $q\neq p$.

Suppose $x_q\notin R_i$.  We'll prove that if $q$ is not a singular point of $C_6$, then $x_q$ is a smooth point of $X$. This will be done by showing that  a general tangent hyperplane section of $X$ at $x$ through $x_q$ is smooth on $x_q$. Since $x_q$ is either an isolated singular point or a general point in a singular line, this hyperplane section is mapped by $\tau$ to a general plane through $q$. This is the same technique used in Lemma \ref{sl_singularidades_vem_de_C6}. 

Therefore, let $\Omega$ be a general plane through $q$. If $q\in \l_j\setminus C_6$, then $\X\cap\Omega$ consists of degree $9$ curves having multiplicity four in $q$ and in two other points and having multiplicity two in six points $C_6\cap\Omega$. None of these points are infinitely near to $q$, since $x_q\notin R_j$ and $q\notin C_6$. Since the image of $\Omega$ via $\sigma$ has degree nine, it follows that $\X\cap\Omega$ has no other base points. Then $x_q\in\tau^{-1}(q)$ is a smooth point of $\Omega$.

If $q$ is a smooth point of $C_6$ not lying on any of the double lines, then $\X\cap\Omega$ has three points of multiplicity four, has multiplicity two in $q$ and in other five points. Since $q\notin \l_j$ and since it is a smooth point of $C_6$, none of the other base points is infinitely near to it. Then again the image of $\Omega$ is smooth in $x_q$. 

Suppose then that $q$ is a smooth point of $C_6$ lying on a double line $\l_j$ of $V$ (in particular, $\l_j\not\subset C_6$). Then $\X\cap\Omega$ consists of degree nine curves having multiplicity four in $q=\l_j\cap \Omega$ and in two other points not infinitely near to $q$, and having multiplicity two in six  points of $C_6\cap \Omega$. Since $q\in C_6$, one of the six double points $q^\prime$ lies infinitely near to $q$.  Since $q$ is a smooth point of $C_6$, $\X\cap\Omega$ has no other double points infinitely near to $q$ or $q^\prime$.

Blowing up $q$, the exceptional curve $\Omega_q\subset \Omega$ is mapped by $\sigma$ to a conic. Blowing up $q^\prime$, this second exceptional curve is also mapped to a conic. Then none of these curves is contracted. Moreover, $\sigma$ is an isomorphism outside $V$, so the only curve in $\Omega$ through $q$ that is contracted is $V\cap\Omega$, which is mapped to the smooth point $x$. Hence $x_q$ is a smooth point in the image of $\Omega$. This completes the proof.
\\ \end{proof}

The simplest singularities of $X$ are the images of singular points of $C_6$ not lying on the double lines of $V$:

\begin{lemma}\label{sc_sing_fora_de_li}
Let $q$ be a singular point of $C_6$ not lying on any $\l_i$. Then $q$ is mapped to a double point of $X$ and the tangent cone of $X$ in this point has:
\begin{itemize}
\item rank $4$,  if $q$ is the transversal intersection of two simple branches of $C_6$;
\item rank $3$, if $q$ is a cuspidal point, a point of contact of two simple branches or a general point of a double component of $C_6$;
\item rank $2$, that is, it is a pair of three-dimensional planes, if $q$ is the intersection of three simple branches, the intersection of a cuspidal and a simple branch or the intersection of two cuspidal branches of $C_6$.\end{itemize}
\end{lemma}
\begin{proof}
The curve $C_6$ is the image via $\bar{\tau}$ of a plane cubic and this map is an isomorphism out of the three double lines of $V$. But in Chapter \ref{ed_chapter}, the curve $C_4$ in the base locus of $\X$ is also the image via $\bar{\tau}$ of a cubic. And it is also a double curve of $\X$. Moreover, the singularities of $C_4$ do not lie on the two simple lines of $\X$ (see Lemma \ref{ed_baselocus}), since the tangential projection is general. Then the two maps $\bar{\tau}$ coincide on the point $q$.

Therefore the singularities produced by points of $C_6$ not lying on any of the double lines of $V$ are the same as the ones produced by the curve $C_4$ in Chapter \ref{ed_chapter}. Then the result follows directly from Lemma \ref{ed_singularidades}.
\\ \end{proof}

As described in Section \ref{sc_ramified_sec}, the map $\bar{\tau}$ restricts to a double cover of each proper double line $\l_i$ of $V$. This double cover is ramified at two points, namely $q^i_j$, for $j=1,2$.

Before studying the singularities of $C_6$ lying on a line $\l_i$, we'll describe the behaviour of $\X$ under the blow up at a general point of $\l_i$. To simplify the notation, this will be done for $i=1$.

\begin{lemma}\label{sc_explosao_pt_geral_li}
Suppose $\l_1$ is a proper line of $V$, let $q$ be a point of $\l_1$ distinct from $p$ and let $\{\hat{q}_1,\hat{q}_2\}$ be the preimages of $q$ via $\bar{\tau}$.

If $V$ does not have a double line infinitely near to $\l_1$, then blowing up $q$ gives $V_q=u_1+u_2$, a pair of lines through $(\l_1)_q$. Moreover $\X_q$  is a linear system of quadruples of lines through $u_1\cap u_2$. If $\X_q=2u_1+2u_2$, then $E_q$ is contracted to a point in $L_1$.

If $\l_2$ is infinitely near to $\l_1$ and $\l_3$ is a proper line of $V$, then blowing up $q$ gives $\X_q=4t$, where $t=(\Gamma_3)_q$. After blowing up $\l_1$ and $\l_2$, the blow up at $t$ gives, in $E_t\cong\F_0$, $\X_t\equiv (0,4)$ and $V_t=u_1+u_2\equiv (0,2)$. If $\X_t=2u_1+2u_2$, then $E_t$ is contracted to a point in $R_1$.

If $\l_1\prec\l_2\prec\l_3$, then again blowing up $q$ gives $\X_q=4t^\prime$, where $t^\prime=(\Gamma_3)_q$. After blowing up $\l_1$, $\l_2$, $\l_3$ and $t^\prime$, $V$ and $\X$ intersect $E_{t^\prime}$ in a curve $t$ with multiplicity two and four respectively. The blow up at $t$ gives, in $E_{t}\cong\F_0$, $\X_{t}\equiv (0,4)$ and $V_{t}=u_1+u_2\equiv (0,2)$.  If $\X_t=2u_1+2u_2$, then $E_t$ is contracted to a point in $R_1$. 

In the three cases, the lines $u_1$ and $u_2$ are infinitely near if and only if $q$ is $q^1_1$ or $q^1_2$. Moreover, the indexes can be chosen to satisfy the following property: If a curve in $E=\bar{\tau}^{-1}(V)$ contains $\hat{q}_j$, for $j\in\{1,2\}$, then its image in $V$ intersects $E_t$ in a point of $u_j$.
\end{lemma}
\begin{proof}
Suppose first that $V$ does not have a double line infinitely near to $\l_1$. By Lemma \ref{sc_cones_tangentes_em_V}, the tangent cone of $V$ on $q$ is a pair of planes containing $\l_1$, which coincide if $q$ is $q^1_1$ or $q^1_2$. Then $V_q=u_1+u_2$. Since $\X$ has multiplicity four in $q$ and along $\l_1$, $\X_q$ is the union of four lines through $(\l_1)_q=u_1\cap u_2$.

Remember that the line $L_1$ is the base locus of $\mQ_1^\prime$. Suppose  $\X_q=2u_1+2u_2$, so that $E_q$ is mapped to a point. This implies that $q\in C_6$ and $(C_6)_q$ consists of points in both $u_1$ and $u_2$. Then $\mQ_q$ has three non collinear base points, that is, $q$ is a singular point of $\mQ$.  Now blow up $\l_1$. As shown in the proof of Lemma \ref{sc_mqi_delpezzo}, $\mQ_{\l_1}=t_L^1$ is a fixed curve of type $(1,1)$ and it is  mapped to $L_1$. Since $\mQ$ is singular on $q$, $t_L^1$ contains $E_q\cap E_{\l_1}\equiv (1,0)$. Therefore the contraction of $E_q$ lies on $L_1$. This proves the first part.

Suppose now that $\l_2$ is infinitely near to $\l_1$ and that $\l_3$ is a proper line of $V$. Start blowing up $q$. Then $\X_q$ has multiplicity four in $(\l_1)_q$ and in  $(\l_2)_q$, which is infinitely near to $(\l_1)_q$ in the direction of $t=(\Gamma_3)_q$. Therefore $\X_q=4t$. For the same reason, $V_q=2t$.

Blow up $\l_1$ and $\l_2=\Gamma_3\cap E_{\l_1}$. The self-intersection of $t=\Gamma_3\cap E_q$ in $E_q$ has decreased from $1$ to $-1$ with the blow ups of $\l_1$ and $\l_2$. In $\Gamma_3$, $t$ is the exceptional curve of the blow up at $q$, giving $t^2=-1$. Since $\l_1$ and $\l_2$ lied in $\Gamma_3$, this is not affected by their blow ups. Then the normal bundle of $t$ is:
\[ N_t = \mO_{\P^1}(-1)\oplus \mO_{\P^1}(-1) \]
Therefore, blowing up $t$ gives $V_t\equiv (0,2)$ and $\X_t\equiv (0,4)$. Set $V_t=u_1+u_2$.

As explained in Lemma \ref{sc_imagem_li}, the moving part of $\X_{\l_1}$ consists of moving fibers, which map $E_{\l_1}$ to the line $R_1$.  If $\X_t=2u_1+2u_2$, then $E_t$ is contracted to a point, which coincides with the image of $E_q$. Since $E_q$ intersects $E_{\l_1}$ in a fiber, it follows that the image of $E_t$ lies on $R_1$. This completes the proof of the second part.

Now suppose that $\l_1\prec\l_2\prec\l_3$. This case is illustrated in Figure \ref{sc_figura_expl_q_em_l1}. 

\begin{figure}
\centering
\includegraphics[trim=15 500 100 30,clip,width=1\textwidth]{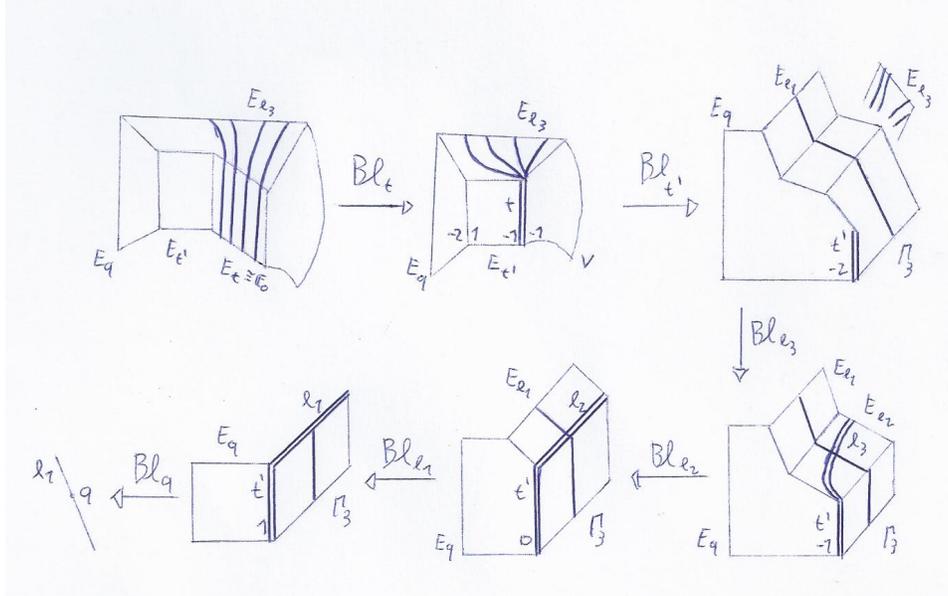}
\caption{The blow up of $\X$ when $\l_1\prec\l_2\prec\l_3$}\label{sc_figura_expl_q_em_l1}
\end{figure}

Blow up $q$, giving again  $\X_q=4t^\prime$ and $V_q=2t^\prime$, with the new notation $t^\prime=(\Gamma_3)_q$. Consider the blow ups along $\l_1$, $\l_2$ and $\l_3$.

After this, $t^\prime$ is no longer the complete intersection of $E_q$ with $\Gamma_3$, but $E_q\cap V=2t^\prime$. In $E_q$, $t^\prime$ is a smooth curve with $(t^\prime)^2=-2$. In $V$, $t^\prime$ is a double curve. But in the normalization of $V$, this double curve is the union of two disjoint smooth curves, the exceptional curves of the blow up at two smooth points. Therefore, $(t^\prime)^2=-1$ and its normal bundle is:
\[ N_{t^\prime} = \mO_{\P^1}(-2)\oplus \mO_{\P^1}(-1) \]

Blow up $t^\prime$. In $E_{t^\prime}\cong\F_1$, clearly $V_{t^\prime}=2t\equiv 2e_1$. On the other hand, $\X$ had no moving intersection with $E_q$, so $\X_{t^\prime}=4t$. The normal bundle of $t$ is:
\[ N_{t} = \mO_{\P^1}(-1)\oplus \mO_{\P^1}(-1) \]

Hence, the blow up at $t$ gives $V_{t}\equiv (0,2)$ and $\X_{t}\equiv (0,4)$, as stated. If $\X_t=2u_1+2u_2$, the surfaces $E_t$, $E_{t^\prime}$ and $E_q$ are contracted to the same point. Since $E_q$ intersects $E_{\l_1}$ in a line, this point lies on $R_1$. The third part is proved.

Now, let $\hat{C}\subset E$ be a line through $\hat{q}_1$, let $C\subset V$ be its image via $\bar{\tau}$ and let $\Sigma$ be the plane containing $C$. Then $\Sigma\cap V=C+C^\prime$, where $C^\prime$ is a conic through $q$. Following the blow ups made above, we get, in the two last cases, $\Sigma_t\equiv (1,0)$. It intersects $V_t=u_1+u_2\equiv (0,2)$ in two points. The same happens in the first case, where $\Sigma_q$ is a line. Then one point, say $\Sigma\cap u_1$, is the intersection with $C$. The other point is the intersection with $C^\prime$. 

By continuity, the lines through $\hat{q}_j$ are mapped to conics intersecting $u_j$, for $j\in\{1,2\}$. Taking tangent lines (or tangent cones) of curves in $E$, the same result follows for curves through $\hat{q}_j$.

The lines $u_1$ and $u_2$ are infinitely near if and only if every plane $\Sigma$ as above intersects $V_t$ (or $V_q$) in two infinitely near points. By the above result, this is equivalent to $\hat{q}_1$ being infinitely near to $\hat{q}_2$, that is, $q$ is $q^1_1$ or $q^1_2$.
\\ \end{proof}

We now consider the singularities of $X$  coming from  a singular point of $C_6$ lying on $\l_i\setminus p$. The study of these singularities is very similar to the singularities described in Chapter \ref{sl_chapter} which lied in the double line $s$ of $V$.

\begin{lemma}\label{sc_sing_em_l1}
Let $q\neq p$ be a singular point of $C_6$ lying on $\l_i$, a proper double line of $V$. Let $\hat{C}$ be the strict transform of $C_6$ and $\{\hat{q}_1,\hat{q}_2\}$ be the preimages of $q$ via $\bar{\tau}$.

If $\hat{C}$ contains both $\hat{q}_1$ and $\hat{q}_2$, then $X$ has a triple point that is mapped to $q$. This point lies on $R_i$ if $V$ has a double line infinitely near to $\l_i$; otherwise it lies on $L_i$.

If $\hat{C}$ contains only one point among $\{\hat{q}_1,\hat{q}_2\}$, then $X$ has a double point that is projected to $q$.

These considerations are also valid if $q=q^i_j$, in which case $\hat{q}_1$ and $\hat{q}_2$ are infinitely near.
\end{lemma}
\begin{proof}
Set $i=1$ and $\l=\l_1$. Suppose that $q$ is not $q^1_j$, so that $\hat{q}_1$ and $\hat{q}_2$ are not infinitely near. If $q=q^i_j$, the same argument can be used, noting that $\hat{C}$ is tangent to $\hat{\l}$ on $\hat{q}=\bar{\tau}^{-1}(q)$.

Suppose first that $V$ does not have a double line infinitely near to $\l_1$. 

By Lemma \ref{sc_explosao_pt_geral_li}, blowing up $q$ gives $V_q=u_1+u_2$ and $\X_q$ is the union of four lines through $u_1\cap u_2$. If $\hat{C}$ contains $\hat{q}_1$ and $\hat{q}_2$, the same Lemma asserts that $C_6$ intersects $E_q$ in points of both $u_1$ and $u_2$, which implies that $\X_q=2u_1+2u_2$. If $\hat{C}$ contains only $\hat{q}_1$, then $C_6$ intersects $E_q$ in points of $u_1$ and $\X_q=2u_1+\{\text{pairs of lines}\}$.

Now repeat the arguments in the proof of Proposition \ref{sl_singularidades_em_C6.s}, of Chapter \ref{sl_chapter}: in the first case $E_q$ is mapped to a triple point of $X$; in the second case $E_q$ is mapped to a curve and $u_1$ is mapped to a double point of $X$ lying on this curve.

By Lemma \ref{sc_explosao_pt_geral_li}, the triple point lies on $R_1$ or $L_1$, as stated.

Suppose now that $V$ has a double line $\l_2$ infinitely near to $\l_1$. The line $\l_3$ can be either proper or infinitely near to $\l_2$. In Lemma \ref{sc_explosao_pt_geral_li}, both cases are considered: If $\l_3$ is a proper line of $V$, then blowing up $q$, $\l_1$, $\l_2$ and a curve $t\in E_q$, gives, in $E_t\cong\F_0$, $\X_t\equiv (0,4)$ and $V_t=u_1+u_2\equiv (0,2)$. If $\l_3$ is infinitely near to $\l_2$, then blowing up $q$, $\l_1$, $\l_2$, a curve $t^\prime\in E_q$ and a curve $t\in E_{t^\prime}$, gives, in $E_{t}\cong\F_0$, $\X_{t}\equiv (0,4)$ and $V_{t}=u_1+u_2\equiv (0,2)$.

So in both cases we have:
\[ V_t=u_1+u_2\equiv (0,2) \]
and $\X_t\equiv (0,4)$, where $t$ is a curve infinitely near to $q$. According again to Lemma \ref{sc_explosao_pt_geral_li}, if a curve in $E=\bar{\tau}^{-1}(V)$ contains $\hat{q}_j$, for $j\in\{1,2\}$, then its image in $V$ intersects $E_t$ in a point of $u_j$.

Therefore, if  $\hat{C}$ contains both $\hat{q}_1$ and $\hat{q}_2$, then:
\[ \X_t=2u_1+2u_2 \] 
and $E_t$ is contracted to a point. It lies on $R_1$, by Lemma \ref{sc_explosao_pt_geral_li}. Since $t$ is infinitely near to $E_q$, this point is projected by $\tau$ to $q$. 

To prove that it is a triple point of $X$, use Lemma \ref{pr_truque_secao_tgente}. Let $\Omega$ be a general plane through $q$. Making the blow ups described in Lemma \ref{sc_explosao_pt_geral_li}, the line $t$ intersects $\Omega$ in one point. Blowing up this point, $\X\cap\Omega$ intersects the exceptional curve $e=E_t\cap \Omega$ in two fixed double points. After the blow up at these two points, we get $e^2=-3$ in $\Omega$ and it is contracted to a triple point of $\sigma(\Omega)$. Then by Lemma \ref{pr_truque_secao_tgente}, it is a triple point of $X$.

If, on the other hand, $\hat{C}$ contains only one point among $\{\hat{q}_1,\hat{q}_2\}$, say $\hat{q}_1$, then:
\[ \X_t=2u_1+\{\text{pairs of lines}\} \] 
where the moving part is of type $(0,2)$. Then $E_t$ is mapped to a curve. The line $u=u_1$ has normal bundle:
\[ N_u = \mO_{\P^1}(0)\oplus \mO_{\P^1}(-1) \]

The blow up at $u$ gives $\X_u\equiv 2e_1+2f_1$, having two or more double points in $(C_6)_u$. Therefore, $\X_u$ is a fixed curve and $E_u$ is mapped to a point. 

To prove that it is a double point of $X$, let $\Omega$ be a general plane through $q$. Perform the blow ups described in Lemma \ref{sc_explosao_pt_geral_li} and the blow ups of $t$ and $u$. Then $\X\cap\Omega$ intersects $e=E_u\cap\Omega$ in one fixed double point. After the blow up at this point, we get $e^2=-2$ in $\Omega$. Then $e$ is mapped to a double point of $\sigma(\Omega)$ and lemma \ref{pr_truque_secao_tgente} concludes the argument.
\\ \end{proof}

Next, the case in which $C_6$ contains a non reduced component is analysed. If this component is a line, this has been studied in Section \ref{sc_imagens_li_sec}. Since $C_6$ is the image via $\bar{\tau}$ of a cubic, the other possibilities are that a non reduced component is a double conic or a triple conic. 

\begin{lemma}
Suppose $C$ is a double conic in $C_6$, that is, $C=C_1\prec C_2$ are double curves of $\X$. Then $C$ is mapped to a double conic of $X$ and $C_2$ is mapped to a Veronese quartic surface.

If $C_6=3C$, that is, $C=C_1\prec C_2\prec C_3$ are contained in $C_6$, then $C_3$ is mapped to a Veronese quartic surface through $x$ and $C_1$ and $C_2$ are mapped to infinitely near double conics of $X$ lying on this surface.

Moreover, if for some $i\in\{1,2,3\}$, $q=C\cap \l_i$ is mapped to a triple point of $X$, then this  point lies on the double conic, unless $C_6\neq 3C$ and $q$ is $q^i_1$ or $q^i_2$.
\end{lemma}
\begin{proof}
Let $\Sigma$ be the plane containing $C$. Suppose first that $C$ does not contain $p$ and let $Q$ be the cone over $C$ with vertex $p$. 

The linear system $\X$ intersects $\Sigma$ in $2C$ and degree five curves, which intersect $C$ in ten points. Among these, three double points of the moving part of $\X\cap\Sigma$ lie on the lines $\l_i$. 
On the other hand, $V$ intersects $\Sigma$ in $C$ and another conic, intersecting $C$ in three points in the lines $\l_i$ and one other point.

The intersection of $\X$ with $Q$ is $2C+4\l_1+4\l_2+4\l_3$ and moving conics. These conics do not contain $p$, since the fixed part has already multiplicity twelve in this point and the tangent cone of $\X$ in $p$ is the union of the planes $\Gamma_i$ with multiplicity two. Therefore, the moving part of $\X\cap Q$ intersects $C$ in two points. The surface $V$ cuts $Q$ in $C+2\l_1+2\l_2+2\l_3$. 

First blow up the  lines $\l_1,\l_2,\l_3$. After that, $C=\Sigma\cap Q$ has normal bundle:
\[ N_C = \mO_{\P^1}(1)\oplus \mO_{\P^1}(2) \]

Blowing up $C$ gives $\Sigma_C\equiv e_1+f_1$ and $Q_C\equiv e_1$. Then $C_2=V_C\equiv e_1+f_1$, which is, by hypothesis, a double curve of $\X_C$. Therefore:
\[ \X_C\equiv 2e_1+4f_1\equiv 2C_2 + \{2f_1\} \]

The moving part maps $E_C$ to a conic in $X$. The curve $C_2=V\cap E_C$ has normal bundle:
\[ N_{C_2} = \mO_{\P^1}(1)\oplus \mO_{\P^1}(1) \]

Blowing up $C_2$ gives $\X_{C_2}\equiv (2,2)$ and $V_{C_2}\equiv (0,1)$. If $C$ is a double curve in $C_6$, $\X_{C_2}$ has one fixed double point, corresponding to the intersection with  the other component of $C_6$, a conic. Then $\X_{C_2}$ is birationally equivalent to a linear system of conics in $\P^2$ with no base points, and $E_{C_2}$ is mapped to a quartic Veronese surface. It contains $x$, the image of $V_{C_2}$.

If $C$ is a triple curve in $C_6$, $\X_{C_2}\equiv 2C_3+(2,0)$, where $C_3=V_{C_2}$. Then $E_{C_2}$ is mapped to a conic. Blowing up $C_3=E_{C_2}\cap V$ gives $E_{C_3}\cong \F_1$, with $V_{C_3}\equiv e_1$. Then $\X_{C_3}\equiv 2e_1+2f_1$ with no base points and again $E_{C_3}$ is mapped to a Veronese surface through $x$.

Now suppose  $C$ contains $p$. The case $C_6=2C+C^\prime$, with $C^\prime$ intersecting $\l_1$ in one point, is illustrated in Figure \ref{sc_figura_2c}. 

\begin{figure}[tb]
\centering
\includegraphics[trim=7 5 110 460,clip,width=1\textwidth]{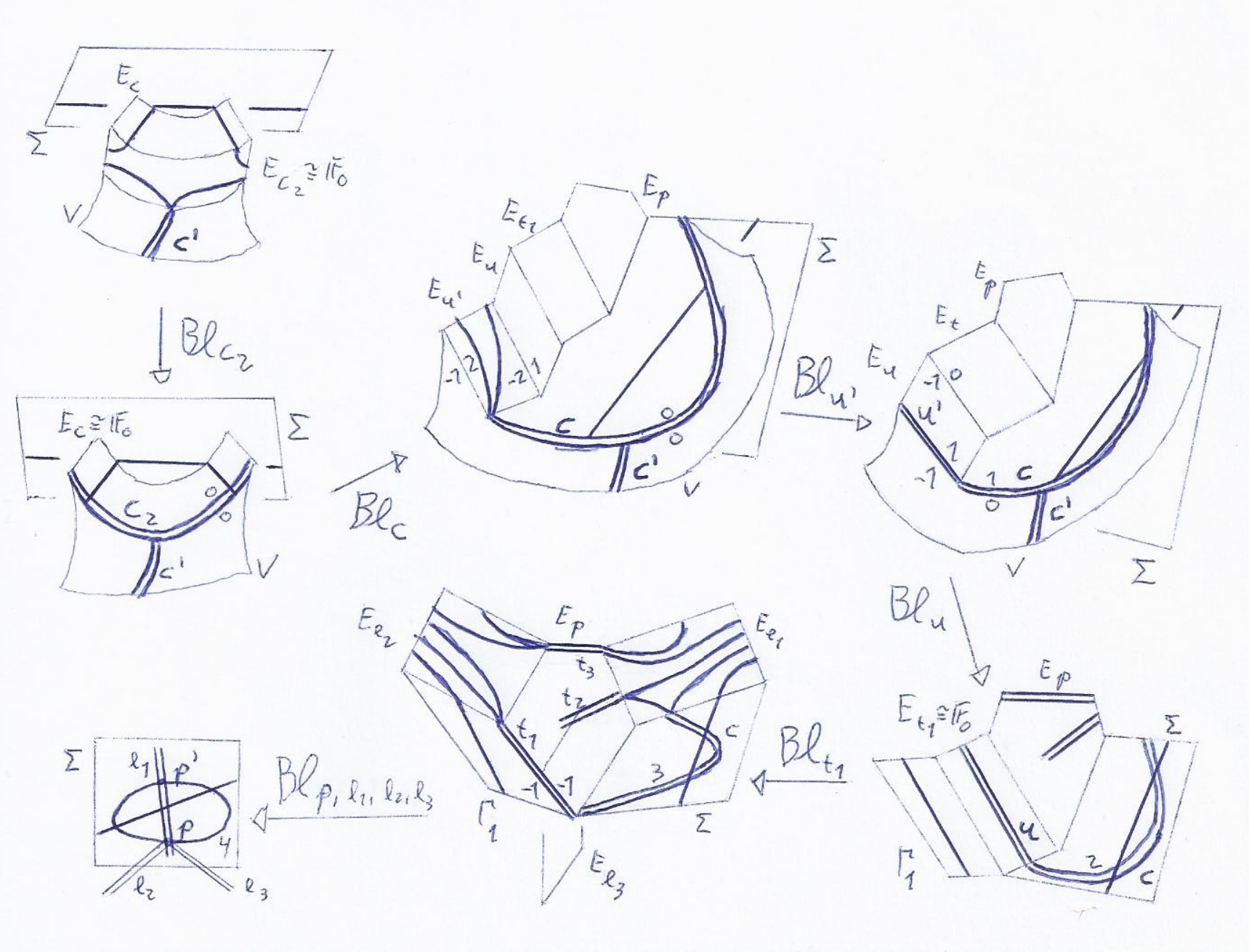}
\caption{Blow ups when $C_6=2C+C^\prime$, with $C$ containing $p$}\label{sc_figura_2c}
\end{figure}

The intersection $\Sigma\cap V$ consists of $C$ and a pair of lines through $p$. Since the only lines in $V$ are its double lines, it follows that $\Sigma\cap V=C+2\l_j$, for some $j$. For simplicity, set $j=1$. This implies that $C$ cuts $\l_1$ in a point $p^\prime$ other than $p$, but cuts $\l_2$ and $\l_3$ only in $p$. Then $C_2\cap \Gamma_1=2p$, and $C_2\cap \Gamma_i=p+p^\prime$, for $i\in \{2,3\}$. Note that $p^\prime\neq p$, otherwise $\Sigma$ would be a component of the tangent cone of $V$ in $p$, so $C$ would be a double line of $V$. 

Moreover:
\[ \X\cap \Sigma=2C+4\l_1+\{\text{lines}\} \]

Start blowing up $p$, giving, as explained in \eqref{sc_eq_Xp}:
\[ \X_p=2t_1+2t_2+2t_3 \]
where $t_i=(\Gamma_i)_p$. Then $C_p$ is a point in $t_1$ and $\Sigma_p$ is the line described by this point and $(\l_1)_p$.

Next blow up the three lines $\l_1$, $\l_2$ and $\l_3$. Before blowing up $C$, we'll blow up $t_1$, in order to show there are base curves of $\X$ infinitely near to it. The normal bundle of $t_1=(\Gamma_1)_p$, after these blow ups, is:
\[ N_{t_1} = \mO_{\P^1}(-1)\oplus \mO_{\P^1}(-1) \]

Therefore, blowing up $t_1$, gives, in $E_{t_1}\cong \F_0$, $\X_{t_1}=2u\equiv (0,2)$, since $C_{t_1}$ is a double point. The normal bundle of $u=E_{t_1}\cap V$ is:
\[ N_{u} = \mO_{\P^1}(0)\oplus \mO_{\P^1}(-1) \]
Then blowing up $u$ gives $V_u\equiv e_1+f_1$ and $\X_u\equiv 2e_1+2f_1$. Again, $C_u$ is a double point of $\X_u$. But there is a second double point, namely $(C_2)_u$, which is infinitely near to $C_u$. Both points lie on $u^\prime=V_u$, so $\X_u=2u^\prime$.

We proceed blowing up $u^\prime=E_u\cap V$, which gives $E_{u^\prime}\cong \F_2$ and $\X_{u^\prime}\equiv 2e_2+4f_2$. If $C$ is a triple curve in $C_6$, then $\X_{u^\prime}$ has three double points: $C_{u^\prime}\prec (C_2)_{u^\prime}\prec (C_3)_{u^\prime}$. All of these points lie on $u^{\prime\prime}=V_{u^\prime}\equiv e_2+2f_2$, giving $\X_{u^\prime}=2u^{\prime\prime}$. The same happens if $C_6=2C+C^\prime$, with $C^\prime$ intersecting $\l_1$ in $p$ and in a second point, giving a third double point of $\X_{u^\prime}$. In both cases, similar considerations show that there are no further base curves of $\X$ infinitely near to $u^{\prime\prime}$.

On the other hand, if $C_6= 2C+C^\prime\neq 3C$, with $C^\prime$ not as above, then $\X_{u^\prime}$ intersects $V_{u^\prime}$ in two double points, so it's not a fixed curve. This is the case illustrated in Figure \ref{sc_figura_2c}.

We now concentrate on $C=\Sigma\cap V$. The following considerations hold for $C_6=2C+C^\prime$ and for $C_6=3C$. In $\Sigma$, the blow up of $p$, $t_1$, $u$ and $u^\prime$ decreased the self-intersection of $C$, giving $C^2=0$. In $V$ (blown up in $p$), we have $C^2=0$. Then the normal bundle of $C$ is:
\[ N_C = \mO_{\P^1}(0)\oplus \mO_{\P^1}(0) \]

Blow up $C$, giving $C_2=V_C\equiv (0,1)$. Since $\X\cap\Sigma$ intersects $E_C$ in two moving points, it follows that:
\[ \X_C=2C_2+\{\text{moving part}\}\equiv (0,2)+\{(2,0)\} \]
and again the moving part maps $E_C$ to a conic.

The curve $C_2=E_C\cap V$ has normal bundle:
\[ N_{C_2} = \mO_{\P^1}(0)\oplus \mO_{\P^1}(0) \]
Blow up $C_2$. Then $\X_{C_2}\equiv (2,2)$ and $V_{C_2}\equiv (0,1)$.

To study the base points of $\X_{C_2}$, suppose first $C_6=2C+C^\prime$, with $C^\prime\neq C$. If $C^\prime$ intersects $\l_1$ in two points, then $(u^{\prime\prime})_{C_2}$ is a double point of $\X_{C_2}$. If $C^\prime$ intersects $\l_1$ only in $p$, then it intersects $\Sigma$ (before the blow ups) in $p$ and in a point lying on $C$. Therefore, $C^\prime$ intersects $E_{C_2}$, and this is a double point of $\X_{C_2}$. Finally, if $C^\prime$ does not contain $p$, then again $(C^\prime)_{C_2}$ is a double point of $\X_{C_2}$.

Then in all cases $\X_{C_2}\equiv (2,2)$ has one double point. By Lemma \ref{pr_imagemF1F2}, this linear system corresponds, in $\P^2$, to curves of degree four having three fixed double points. Applying a standard quadratic map gives a base-point free linear system of conics. Hence $E_{C_2}$ is mapped to a Veronese surface.

Now suppose $C_6=3C$. Then:
\[ \X_{C_2}=2C_3+\{\text{moving part}\}\equiv (0,2)+\{(2,0)\} \]

Then, blowing up $C_3=E_{C_2}\cap V$, gives $E_{C_3}\cong \F_0$ and $V_{C_3}\equiv (0,1)$. The linear system $\X_{C_3}\equiv (2,2)$ has one double point in $(u^{\prime\prime})_{C_3}$. As before, it maps $E_{C_3}$ to a quartic Veronese surface.

\quad

The next step is to show that $X$ has multiplicity two in the conics. Let $q$ be a general point in $C$ and let $\Omega$ be a general plane through $q$. Then $q$ is a double point of $\X\cap \Omega$, which has another double point infinitely near to it. If $C_6=3C$, then there is a third infinitely near double point. 

Blowing up the double points, we get one curve (or two, if $C_6=3C$) with self-intersection $(-2)$ in $\Omega$ having no intersection with $\X\cap\Omega$. Then $\sigma(\Omega)$ has multiplicity two in $\sigma(q)$ and another double point infinitely near to it when $C_6=3C$. Then Lemma \ref{pr_truque_secao_tgente} completes the proof.

\quad

To prove the last part,  set $\l=\l_i$ and suppose that $q=C\cap \l$ is mapped to a triple point of $X$. According to Lemma \ref{sc_sing_em_l1}, $E_q$ is contracted to a point. Then, after the blow up at $\l$, the fiber $f^q$ over $q$ in $E_\l$, which contains the point $C_\l$, is contracted to a triple point. Also note that:
\[ V_\l\equiv (2,2)\equiv f^p+(1,2) \]
where $f^p$ is the fiber over $p$. The curve of type $(1,2)$ is reducible if and only if $V$ has a double line infinitely near to $\l$. In this case, blow up this line in order to find an irreducible curve of type $(1,2)$ in the other exceptional divisor. Then $V_\l$ is tangent to $f^q$ if and only if $q$ is $q^i_1$ or $q^i_2$. In this case, the point of tangency is $C_\l$.

Suppose first that $q$ is $q^i_1$ or $q^i_2$ and $C_6\neq 3C$. Blowing up $C$,  $f^q$ and $V_\l$ intersect $E_C$ in the same point of $C_2=V_C$. Blowing up $C_2$, $E_{C_2}$ is mapped to a quartic surface and $E_C$ is contracted to a conic in it. In this stage, $f^q\subset E_\l$ does not intersect either $V_\l$ or $E_C$. Then it is mapped to a point not lying on the image of $E_C$. Therefore the image of $q$ does not lie on the double conic.

Now suppose that $C_6=3C$. Then, after the blow ups of $\l_1,\l_2,\l_3,C$ and $C_2$, $f^q$ does not intersect $V_\l$, even if $q$ is $q^i_1$ or $q^i_2$. Moreover, as seen before:
\[ \X_{C_2} \equiv (2,2) \equiv 2C_3 + (2,0) \]
and the point $f^q\cap E_{C_2}$ does not lie on $C_3=V_{C_2}$. Hence $f^q$ is mapped to a point in the image of $C_2$, which is a conic infinitely near to the image of $C$. Hence it is a point in the double conic.

Finally, suppose that $q$ is not $q^i_1$ nor $q^i_2$ and $C_6\neq 3C$. Then, after the blow up at $\l$, $V_\l$ intersects $f^q$ transversally in $C_\l$. Blowing up $C$, $f^q$ does not intersect $V_\l$ any more. Then Blowing up $C_2=V_C$, $f^q$ still intersects $E_C$, which is mapped to the double conic containing the image of $f^q$.
\\ \end{proof}

The following proposition summarizes the results on singularities of $X$. 

\begin{prop}\label{sc_singularidades_de_X}
The threefold $X$ has multiplicity four in $x_p$. If $V$ has a double line infinitely near to $\l_i$, then $X$ has a double line $R_i$, which does not contain $x_p$ and is mapped to $\l_i$. The other singularities of $X$ are projected to singularities of $C_6$. 

Let $q$ be an isolated singular point of $C_6$, then it is the image of a singular point $x_q\in X$. If $q$ does not lie on any of the double lines of $V$, then $x_q$ is a double point of $X$. If $q$ lies on $l_i$ and the strict transform of $C_6$ via $\bar{\tau}$ contains both preimages of $q$, then $x_q$ is a triple point of $X$ lying on $R_i$ or $L_i$. If not, it is a double point.

Let $C$ be a non reduced component of $C_6$. If $C=\l_i$, then $X$ has a triple line $L_i$, which is projected to $\l_i$. If $C$ is a conic, then it is mapped to a double conic of $X$. 
\end{prop}
\begin{proof}
These results follow from Proposition \ref{sc_imagem_li}, Lemma \ref{sc_2l1+C4}, Lemma \ref{sc_singularidades_vem_de_C6}, Lemma \ref{sc_sing_fora_de_li} and Lemma \ref{sc_sing_em_l1}.
\\ \end{proof}

We recall a remark made before Lemma \ref{sc_2l1+C4}. If, for instance, $\l_2$ is infinitely near to $\l_1$ and $C_6$ contains $\l_1$, then $\X$ has a double base curve infinitely near to $\l_1$ determined by $V$. But since $\X$ has multiplicity four along $\l_2$, this curve lies infinitely near to $\l_2$. Then in this case we say that $C_6$ contains $\l_2$, instead of $\l_1$. This avoids confusion among the lines $R_i$ and $L_i$.

\section{Description of the general Bronowski threefold of degree 9}
\label{sc_oadpgeral_sec}

In this Section we investigate the Bronowski threefold in which the fundamental surface $V$ is the Steiner's Roman Surface and in which $C_6$ is a smooth curve not containing $p$ and intersecting the lines $\l_i$ transversally. By Proposition \ref{sc_singularidades_de_X}, the only singularity of $X$ is $x_p$, a point of multiplicity four.

The following result is very similar  to Proposition \ref{sl_X_contido_em_F}:

\begin{prop}\label{sc_X_contido_em_3_cones}
Let $X$ be the general Bronowski threefold of degree 9. Then for $i\in\{1,2,3\}$, $X$ lies on the cone with vertex $L_i$ over a Segre embedding of $\P^1\times\P^2$. In each of these cones, let $H_0$ be the class of a plane section and let $H_1$ be the class of a $\P^4$ of the ruling. Then:
\[ X\equiv 4H_0^2-3H_0H_1 \]
\end{prop}
\begin{proof}
For simplicity, set $i=1$. By Remark \ref{pr_oadpLN}, $X$ is linearly normal. Then, by Theorem \ref{pr_contidonumscroll}, $X$ lies on a rational normal scroll $F$, described by the $\P^4$'s spanned by the quartics in $\mQ^\prime_1$ (see Lemma \ref{sc_mqi_delpezzo}). These $\P^4$'s have in common the line $L_1\in X$, so $F$ is the cone with vertex $L_1$ over a scroll $Y$ in $\P^5$. It contains a one-dimensional family of disjoint planes, so $Y$ is the Segre embedding of $\P^1\times \P^2$. 

Since $V$ is the Steiner's Roman surface, $X$ lies on three distinct cones. The vertex of each of these cones is a line $L_i$. Since the cones are distinct and since each cone is the intersection of three quadric hypersurfaces, it follows that $X$ lies on the intersection of $F$ with two quadrics $Q_1$ and $Q_2$. Note that $Q_1\cap Q_2$ intersects a $\P^4$ of the ruling in a surface of $\mQ^\prime_1$. By Lemma \ref{sc_mqi_delpezzo}, this surface is a weak Del Pezzo quartic surface having a double point in $x_p\in L_1$.

Let $\eta:\Bl_{L_1}(\P^7)\to\P^7$ be the blow up of $L_1$, let $G$ be the strict transform of $F$  and $E$ be the intersection of $G$ with the exceptional divisor. Set $H_0$ and $H_1$ for the strict transforms in $G$ of the classes $H_0$ and $H_1$ in $F$. Let $Q_1^\prime$ and $Q_2^\prime$ be the total transforms via $\eta$ of $Q_1\cap F$ and $Q_2\cap F$, and let $X^\prime$ be the strict transform of $X\subset F$.

The divisors $Q_1^\prime$ and $Q_2^\prime$ contain both $X^\prime$ and $E$. So let $S$ be residual intersection, that is:
\[ Q_1^\prime\cap Q_2^\prime = X^\prime+E+S \]
Since $Q_1\cap Q_2$ intersects a $\P^4$ of the ruling of $F$ in a weak Del Pezzo surface, $S$ does not contain $E$. And since $X$ intersects a $\P^4$ of the ruling of $F$ in a quartic surface, it follows that $S$ consists of threefolds contained in rulings of $G$. Moreover: 
\[ \deg(\eta(S)) = \deg(Q_1\cap Q_2\cap F)-\deg(X)=12-9=3 \]
These considerations imply that $S \equiv 3H_0H_1$. Therefore: 
\[ X^\prime+E\equiv 4H_0^2-3H_0H_1 \]
Taking the image via $\eta$ in $\P^7$, the result follows.
\\ \end{proof}

This description gives us an important result:

\begin{theo}
The general Bronowski threefold of degree 9 is an OADP variety.
\end{theo}
\begin{proof}
By Proposition \ref{sc_X_contido_em_3_cones}, $X\equiv 4H_0^2-3H_0H_1$ in a cone $F$ over $Y$ with vertex $L$, where $Y$ is a Segre embedding of $\P^1\times\P^2$. Following \cite[Example 2.6]{cmr},  there is a smooth quadric surface $Q\subset Y$ not intersecting $L$ such that if $X$ intersects  the $\P^5$ spanned by $Q$ and $L$ in an OADP surface, then $X$ is an OADP variety. The quadric $Q$ (and the $\P^5$ spanned by $Q$ and $L$) depends on the choice of a general point in $\P^7$, and $X$ is OADP if there is a unique secant line to $X$ through this point.

Since $Q\subset Y$ and $X\subset F$, the intersection of $X$ with the $\P^5$ is a surface of type $4H_0^2-3H_0H_1$ in the cone $F^\prime$ over $Q$ with vertex $L$, where now $H_0$ is the class of a section of $F^\prime$ and $H_1$ is the class of a $\P^3$ of the ruling. Then the degree of this intersection is:
\[ (4H_0^2 - 3H_0H_1)H_0^2  = 8 - 3 + 0 = 5\]

Since $X$ is non degenerate, this degree five surface also is. This follows from the fact that hyperplane sections of a non degenerate variety are non degenerate. 
 The classification of non degenerate degree five surfaces in $\P^5$ is well known (see \cite{nagata} or \cite{coskun}), giving one of the following:
\begin{itemize}
\item A (possibly weak) Del Pezzo surface
\item A projection to $\P^5$ of a degree five scroll in $\P^6$;
\item A cone over an elliptic curve of degree five in $\P^4$.
\end{itemize}

The two last surfaces are ruled by lines. However, there is no line in $X$ passing through a general point of it. Indeed, this would imply that the linear system $\II_{X,x}$ has  a base point, which is not the case. Therefore, since $Q$ depends on a general point of $\P^7$, it follows that the intersection of $X$ with the $\P^5$ spanned by $Q$ and $L$ is a (possibly weak) Del Pezzo Surface of degree five. And this is an OADP surface (see \cite{cmr}). Then $X$ is an OADP variety.
\\ \end{proof}

The other Bronowski threefolds of degree nine are degenerations of the general one studied in this Section. Then those are also OADP varieties. We will not give the details here.

\end{document}